\documentclass[a4paper,12pt,bibliography=totocnumbered,titlepage,bookmarksnumbered=true]{scrartcl}
\usepackage[utf8]{inputenc}
\usepackage{amsmath}
\usepackage{amssymb}
\usepackage{amsfonts}
\usepackage{amsthm}
\usepackage{stmaryrd}
\usepackage[pdftex]{graphicx,color}
\usepackage[font=small, labelfont=it]{caption}
\usepackage{mathrsfs}
\usepackage[arrow,matrix,curve]{xy}
\usepackage{bbm}
\usepackage{bigints}
\usepackage{epigraph}
\usepackage[]{hyperref}

\newcommand{\veps}{\varepsilon}

\newcommand{\crita}{crit\mspace{-4mu}\left(\mathcal{A}^H\right)}
\newcommand{\ZN}{\left(\begin{smallmatrix} Z\\ N\end{smallmatrix}\right)}
\newtheorem{defn}{Definition}
\newtheorem{theo}[defn]{Theorem}
\newtheorem*{theo*}{Theorem}

\newtheorem{lemme}[defn]{Lemma}
\newtheorem{prop}[defn]{Proposition}
\newtheorem{cor}[defn]{Corollary}
\newtheorem*{cor*}{Corollary}

\theoremstyle{definition}
\newtheorem*{ex}{Examples}
\newtheorem{dis}[defn]{Discussion}
\newtheorem*{rem}{Remark}
\newtheorem*{proper}{Properties}

\title{Rabinowitz-Floer homology\\ on Brieskorn manifolds}
\subtitle{\Large Dissertation\bigskip\\\normalsize zur Erlangung des akademischen Grades\bigskip\\doctor rerum naturalium\\(Dr. rer. nat.)\bigskip\\ Im Fach: Mathematik\bigskip\\ eingereicht an der\bigskip\\ Mathematisch-Naturwissenschaftlichen Fakultät\\ der Humboldt-Universität zu Berlin\bigskip\\von\bigskip\\\large Dipl.-Math. Alexander Fauck\\\normalsize ~\bigskip\bigskip\\ Präsident der Humboldt-Universität zu Berlin\\ Prof. Dr. J. Olbertz\bigskip\bigskip\\Dekan der Mathematisch-Naturwissenschaftlichen Fakultät\\ Prof. Dr. E. Kulke\bigskip\bigskip\\\begin{flushleft}
   \hspace{1cm}Gutachter:\\\hspace{1cm}\phantom{Tag der mündlichen Prüfung:}{\large1. \normalsize Prof. Dr. Klaus Mohnke}\\\hspace{1cm}\phantom{Tag der mündlichen Prüfung:}{\large2. \normalsize Prof. Dr. Urs Frauenfelder}\\\hspace{1cm}\phantom{Tag der mündlichen Prüfung:}{\large3. \normalsize Prof. Dr. Alexandru Oancea}\bigskip\\\normalsize\hspace{1cm}Tag der mündlichen Prüfung: 7. April 2016                                                                                                                                                                                                                                                                                                                                                                                                                                                                                                                                                                                                                                                                                                                                                                                                                                                                                                                                                                                                                                                                                                                                                                                                                                       \end{flushleft}}
\author{}
\date{}
\begin{document}

\maketitle
\null
\thispagestyle{empty}
\newpage
\pagenumbering{roman}
\null\vfill
 \section*{Selbstständigkeitserklärung}
 Hiermit erkläre ich, dass ich die vorliegende Dissertation selbsständig und nur unter Verwendung der angegebenen Literatur und Hilfsmittel angefertig habe.
\vfill
\newpage
\null\newpage
\null\vfill
 \section*{Acknowledgements}
 I wish to thank my supervisor Klaus Mohnke, who introduced me to the field of symplectic and contact geometry. The discussions with him were at the same time inspiring and encouraging. Without him insisting on clear explanations, some parts of this thesis would never have been written.\medskip\\
 I am also very grateful to all the researchers with whom I had many fruitful discussion. In particular, I thank Urs Frauenfelder, Kai Cieliebak and Alexandru Oancea. Special thanks go to Peter Uebele, for countless conversations on Brieskorn manifolds, symplectic homology and symmetries as well as his helpful remarks about this thesis.\medskip\\
 Finally I thank my family and friends for all their support. Among them, a warm, coffee-smelling thank-you goes to Karoline Köhler, for all the breaks we had together and for correcting some of my linguistic errors.\medskip\\
 Last, but not least I thank my wife Cora for all her love and encouragement. 
 \begin{center} I dedicate to her this thesis.\end{center}
\vfill
\newpage
\null
\newpage
\tableofcontents
\newpage
\pagenumbering{arabic}
\epigraph{\textit{``... as simple and still as mysterious and complicated as a simple mathematical formula that can mean all happiness or can mean the end of the world.''}}{Ernest Hemingway, \textit{A moveable feast}}

\section{Introduction}
\subsection{Motivation}
Contact geometry and topology (as well as its symplectic counterpart) is a relative new mathematical field. It dates essentially back to V. Arnold's seminal book ``Mathematical Methods of Classical Mechanics'',\cite{arnold}, from 1974 based on lectures he gave in the sixties. However, the basic structures and questions of the area arise from Newtonian/Hamiltonian mechanics. It is thus closely related to questions on the behaviour of mechanical systems, such as the whole universe. As an example, let us consider the motion of several planets: The total energy of a system of $n$ points $x_j$ of masses $m_1,...\,,m_n$ (e.g.\ a system of $n$ planets) in $\mathbb{R}^3$ is given by
\[H:=\sum_{j=1}^n\frac{1}{2}||\dot{x}_j||^2-\sum_{1\leq j<k\leq n}\frac{m_j m_k}{||x_j-x_k||}.\]
Considering the phase space $V:=\mathbb{R}^{3n}\times\mathbb{R}^{3n}$ and writing $q_j:=x_j$ and $p_j:=\dot{x}_j$, we find that the time evolution of this system is described by
\[\dot{x}_j=\dot{q}_j=\frac{\partial H}{\partial p_j}\qquad\text{ and }\qquad \ddot{x}_j=\dot{p}_j = -\frac{\partial H}{\partial q_j} = m_j\sum_{k\neq j} m_k\frac{x_k-x_j}{||x_k-x_j||^3},\]
where the second equation is Newton's law of gravity. We cannot solve this equation explicitly (not even for only 3 mass-points, a setup which is known as the 3-body problem). However, it is still possible to show that the system of solutions has certain properties.\\
As is well-known, the total energy of the system is time-independent. This implies that every solution of the above differential equation stays, for all time, in a fixed energy hypersurface $\Sigma:=H^{-1}(E_0)\subset V$. A puzzling question is to what extend can $\Sigma$ tell us something about the dynamical behaviour of the solutions of the equation?\\
At first glance, the prospect of answering this seems hopeless. However, there are many classical examples which relate the topology of a space with the behaviour of a dynamical systems on this space. For example, a theorem by Hopf tells us that on $S^2$ every vector field has at least one zero and hence every ordinary differential equation on $S^2$ has at least one constant solution. It is clear that it should be possible to obtain more information if one takes into account not only $\Sigma$ but more geometric features.\\
One possible additional structure in Hamiltonian mechanics is the contact structure $\xi$ on $\Sigma$ (see Section \ref{sec1.1} for a precise definition) and indeed recent theorems in symplectic geometry show more properties of dynamical systems in the presence of $\xi$. A natural question to ask is in how far $\Sigma$ determines $\xi$, or stated differently: Can the additional information be extracted from $\Sigma$ alone? The purpose of this thesis is to show that $\Sigma$ does not uniquely determine $\xi$. In fact, in every dimension greater then 3, we show the existence of various $\Sigma$ which admit infinitely many (fillable) contact structures $\xi$.\\
Another question asks to what extend $\Sigma$ determines the topological/differentiable structure of the whole phase space $V$. As the latter comes with a symplectic structure $\omega$ which induces the contact structure on $\Sigma$, we could also ask if $(\Sigma,\xi)$ determines $(V,\omega)$. The thesis gives a partial answer: The Main Theorem tells us that often we have that either $\Sigma$ does not determine $\xi$ or $(\Sigma,\xi)$ does not determine $(V,\omega)$.\\
Our main tool to achieve these results is Rabinowitz-Floer homology, a variant of symplectic homology which can be thought of as a Morse homology on the (infinite dimensional) loop space of $V$. It is a machinery to obtain information on the structure of a space with the help of the solutions to a gradient like partial differential equation (see equation (\ref{eq3}) in Section \ref{sec1.4}).\\
The following sections of the introduction first provide precise definitions of the objects  that we use in this thesis. Then, we give a sketch of the construction of Rabinowitz-Floer homology before finally presenting our main results and the structure of the text.\\
We invite the reader to consult Appendix \ref{conventions} for any questions concerning the setup, assumptions, sign and grading conventions used in this thesis.

\subsection{Preliminaries}\label{sec1.1}
A \textbf{\textit{symplectic manifold}} $(V,\omega)$ is a smooth $2n$-dimensional manifold $V$ together with a non-degenerate, closed 2-form $\omega$. This means that $d\omega = 0$ and that
\[ \omega_p^\ast:T_p V\rightarrow T^\ast_p V,\quad X\mapsto \omega_p(X,\cdot) \]
 is an isomorphism for all $p\in V$. Equivalently, we may require that $d\omega=0$ and $\omega^n$ is a volume form.  Such an $\omega$ is then called a \textbf{\textit{symplectic structure}} on $V$. One calls $\omega$ \textbf{\textit{exact}}, if there exists a 1-form $\lambda$, such that $d\lambda=\omega$. If the context is clear, we will also write $(V,\lambda)$ to denote an exact symplectic manifold.\\
A function $H\in C^\infty (V)$ on a symplectic manifold $(V,\omega)$ is called a \textbf{\textit{Hamiltonian}}. We define its \textbf{\textit{Hamiltonian vector field $X_H$}} via
\[ dH= -\iota(X_H)\omega = -\omega(X_H,\cdot)=\omega(\cdot,X_H).\]
A \textbf{\textit{contact manifold}} $\Sigma$ is a smooth $(2n-1)$-dimensional manifold together with a completely non-integrable smooth hyperplane distribution $\xi\subset T\Sigma$. The distribution is called a \textbf{\textit{contact structure}}. It is locally defined as $\xi=\ker\alpha$, where $\alpha$ is a local 1-form satisfying $\alpha\wedge(d\alpha)^{n-1}\neq0$ pointwise. If $\alpha$ is globally defined (which we will always assume), it is called a \textbf{\textit{contact form}}. A global $\alpha$ gives rise to a volume form and $\Sigma$ is then orientable. Once an orientation chosen, we will require that $\alpha\wedge(d\alpha)^{n-1} > 0$.\\
The \textbf{\textit{Reeb vector field $R_\alpha$}} of $\alpha$ is the unique vector field satisfying
\[ \iota(R_\alpha)d\alpha = 0 \quad\quad\text{and}\quad\quad \iota(R_\alpha)\alpha = 1.\]
$R_\alpha$ is transverse to $\xi$ and we have therefore $T\Sigma = \xi\oplus \mathbb{R} R_\alpha$. However, this splitting depends on the chosen contact form $\alpha$, as $R_\alpha$ depends on $\alpha$. Reeb trajectories of $(\Sigma,\alpha)$ are the trajectories of the flow of $R_\alpha$, i.e.\ solutions $v\in C^\infty (\mathbb{R},\Sigma)$ of the equation
\begin{equation} \label{eq1}
 \partial_t v(t) - R_\alpha (v(t)) = 0.
\end{equation}
Reeb orbits are the images of Reeb trajectories.

\begin{dis}\label{dis0}
 For two contact forms $\alpha$ and $\alpha\,'$ which define the same contact structure $\xi$, i.e.\ $\ker \alpha= \xi= \ker \alpha\,'$, we find a function $f$ such that $\alpha\,'= f\cdot \alpha$.
In fact, $f$ is given by $f:=\alpha\,'(R_\alpha)$ and has therefore no zeros. We will henceforth always assume that $\alpha'=f\cdot\alpha$ with $f>0$. This is a minor restriction, since changing $f$ to $-f$ replaces $R_{f\alpha}$ by $R_{-f\alpha}=-R_{f\alpha}$. This leaves the Reeb dynamics basically unchanged -- the same Reeb orbits are run through with the same speed but in the opposite direction.
\end{dis}

 An \textbf{\textit{almost complex structure}} $J$ is a bundle-endomorphism of $\xi$ or $TV$, such that $J^2=-Id$. It is called $\alpha$-compatible if $d\alpha(\,\cdot\,,J\,\cdot\,)$ defines a Riemannian metric on $\xi$. Similarly, one defines $\omega$-compatible $J$. The space of $\alpha$- resp. $\omega$-compatible almost complex structures is non-empty and contractible.\\
 Any pair $(\alpha,J)$ of a contact form $\alpha$ and an $\alpha$-compatible almost complex structure $J$ induces a reduction of the structure group of the tangent bundle $T\Sigma$ to the unitary group $1\times U(n-1)$. This reduction is called an \textit{\textbf{almost contact structure}}. The \textit{\textbf{formal homotopy class}} $[\xi]$ of a contact structure $\xi=\ker \alpha$ is the homotopy class of its almost contact structure. It does not depend on the particular choice of $(\alpha,J)$ and is hence well-defined.
 \begin{ex}~
  \begin{itemize}
   \item $\mathbb{R}^{2n}$ with the 2-form $\omega_{std}:=\sum_{k=1}^n dx_k\wedge dy_k$ is the standard model for a symplectic manifold. A $\omega_{std}$-compatible almost complex structure $J$ is given by $J\left(\begin{smallmatrix}x\\y\end{smallmatrix}\right)=\left(\begin{smallmatrix}-y\\\phantom{-}x\end{smallmatrix}\right)$.\linebreak The unit sphere $S^{2n-1}\subset\mathbb{R}^{2n}$ is a contact manifold with contact form $\alpha$ given by the restriction of the 1-form $\lambda:=\frac{1}{2}\sum_{k=1}^n(x_kdy_k-y_kdx_k)$ to $S^{2n-1}$. The standard contact structure is $\xi_{std}:=\ker\alpha$. The Reeb vector field is $R:=2(-y_1,x_1,...\,,-y_n,x_n)$.
   \item A different contact manifold is $\mathbb{R}^{2n+1}$ with contact form $\alpha=dz+\sum_{k=1}^n x_kdy_k$ and Reeb vector field $R=\partial_z$. A $\alpha$-compatible almost complex structure is again given by $J\left(\begin{smallmatrix}x\\y\end{smallmatrix}\right)=\left(\begin{smallmatrix}-y\\\phantom{-}x\end{smallmatrix}\right)$.
   \item More general examples are obtained as follows. Let $N$ be a smooth $n$-dimensional manifold and let $T^\ast N$ be its cotangent bundle. Let $q_k$ be local coordinates on $N$ and let $p_k$ be the associated cotangent coordinates, i.e.\ if $p\in T_q^\ast N$, then \linebreak$p=\sum_{k=1}^n p_k\cdot dq_k$. Define locally a 1-form $\alpha$ and a 2-form $\omega$ on $T^\ast N$ by
   \[\alpha:=-\sum_{k=1}^n p_kdq_k\qquad\text{ and }\qquad\omega:=\sum_{k=1}^n dq_k\wedge dp_k.\]
   Note that $\alpha$ is a primitive of $\omega$. Both definitions are coordinate-independent and hence $\omega$ gives a symplectic structure on $T^\ast N$.\pagebreak\\
   Fix a Riemannian metric $\langle\cdot,\cdot\rangle$ on $N$. This provides a scalar product on $T_q^\ast N$ for every $q$. The unit cotangent bundle $S^\ast N$ is the submanifold of $T^\ast N$ which consists of points $(q,p)$ with $\langle p,p\rangle_q=1$. The 1-form $\alpha$ is a contact form on $S^\ast N$, as the Liouville vector field 
   \[Y_\alpha:=\sum_{k=1}^n p_k\cdot \partial_{p_k}\]
   is transverse to $S^\ast N$ (see Lemma \ref{Lemma1}). The Reeb flow on $S^\ast N$ is the geodesic flow with respect to $\langle\cdot,\cdot\rangle$.
  \end{itemize}
 \end{ex}

\subsection{Exact fillings for contact manifolds} \label{Liou}
There are two important constructions which link contact and symplectic manifolds: 
\begin{itemize}
 \item for us, the \textbf{\textit{symplectization}} of a contact manifold $(\Sigma, \alpha)$ is $\Sigma\times\mathbb{R}$ endowed with the exact symplectic form $\omega=d(e^r\alpha)$, where $r$ is a coordinate on $\mathbb{R}$;
 \item an \textbf{\textit{exact contact hypersurface}} is a hypersurface $\Sigma\subset V$ of an exact symplectic manifold $(V,\lambda)$, where the pull-back $\alpha := i^\ast\lambda$ by the inclusion gives a contact form on $\Sigma$.
\end{itemize}
\begin{defn}
 Let $(V,\lambda)$ be an exact symplectic manifold. The unique vector field $Y_\lambda$ which satisfies $\iota(Y_\lambda)\omega=\lambda$ is the \textbf{\textit{Liouville vector field}} of $\lambda$.
\end{defn}
\begin{lemme} \label{Lemma1}
 A hypersurface $\Sigma\subset V$ is an exact contact hypersurface of $(V,\lambda)$ if and only if $Y_\lambda$ is transverse to $T\Sigma$ along $\Sigma$.
\end{lemme}
\begin{proof}
 The form $\alpha\wedge(d\alpha)^{n-1}$ is non-degenerate if and only if $Y_\lambda\not\in T\Sigma$, since
\begin{equation*}\label{eq1b} 
 \alpha\wedge(d\alpha)^{n-1} = i^\ast\lambda\wedge(di^\ast\lambda)^{n-1} = i^\ast\bigl(\lambda\wedge(d\lambda)^{n-1}\bigr)=i^\ast\bigl(\iota(Y_\lambda)\omega\wedge\omega^{n-1}\bigr)=i^\ast\bigl(\iota(Y_\lambda)\omega^n\bigr).\qedhere
\end{equation*}
\end{proof}
\begin{defn}
 A \textbf{\textit{Liouville domain}} is a compact exact symplectic manifold $(V,\lambda)$ with boundary $\Sigma$, such that the Liouville vector field $Y_\lambda$ points outwards along $\Sigma$.
\end{defn}
It follows from Lemma \ref{Lemma1} that the boundary $\Sigma$ of a Liouville domain $(V,\lambda)$ is an exact contact hypersurface. Moreover, the proof of Lemma \ref{Lemma1} shows that the orientation of $\Sigma$ as the boundary of $V$ coincides with the orientation given by $\alpha\wedge(d\alpha)^{n-1}$.
\begin{dis} \label{dis2}
 We make the following observations for a Liouville vector field $Y_\lambda$, using Cartan's magic formula:
 \begin{align*}
  \mathcal{L}_{Y_\lambda}\omega &= \iota(Y_\lambda)d\omega + d(\iota(Y_\lambda)\omega)=d\lambda = \omega\\
  \mathcal{L}_{Y_\lambda}\lambda &= \iota(Y_\lambda)d\lambda + d(\iota(Y_\lambda)\lambda)=\iota(Y_\lambda)\omega = \lambda
 \end{align*}
since $d\omega=0$ and $\iota(Y_\lambda)\lambda=\omega(Y_\lambda,Y_\lambda)=0$. Hence, the flow $\phi$ of $Y_\lambda$ satisfies 
\[ (\phi^r)^\ast\lambda=e^r\cdot\lambda\qquad\text{ and } \qquad (\phi^r)^\ast\omega=e^r\cdot\omega.\]
If we consider a Liouville domain $(V,\lambda)$, we find that the negative half flow $\phi^r$, \linebreak $r\in(-\infty,0],$ of $Y_\lambda$ is complete as $V$ is compact and $Y_\lambda$ points outwards along $\partial V = \Sigma$, so that $\phi$ can never leave $V$ in negative time. Combining these two facts, we can find in any Liouville domain a collar neighborhood of $\Sigma$ which is symplectomorphic via $\phi$ to $(\Sigma \times(-\infty,0],e^r\cdot\alpha)$, the non-positive symplectization of $(\Sigma,\alpha)$.
\end{dis}
\begin{defn}
 The \textbf{\textit{completion}} $(\hat{V},\hat{\lambda})$ of a Liouville domain $(V,\lambda)$ is obtained by glueing the positive symplectization of $\Sigma=\partial V$ to $V$ along $\Sigma$. In other words, it is the exact symplectic manifold
 \[\hat{V}:= V\cup_\phi (\Sigma\times\mathbb{R}),\qquad \hat{\lambda}:=\begin{cases}\begin{aligned}&\lambda& &\text{on $V$}\\ &e^r\cdot\alpha&&\text{on $\Sigma\times\mathbb{R}$}\end{aligned}\end{cases}\]
 where we identify $\Sigma\times(-\infty,0]$ with the collar of $\Sigma$ in $V$ as described above.
\end{defn}
\begin{ex}~
\begin{itemize}
 \item In a symplectization $(\Sigma\times\mathbb{R},e^r\alpha)$, the Liouville vector field is $\partial_r$ -- the partial derivative with respect to the coordinate on $\mathbb{R}$ -- and its flow $\phi$ is simply given by
\[\phi^t(x,r) = (x,r+t) \qquad \forall r,t\in\mathbb{R}.\] 
 \item The unit ball in $(\mathbb{R}^{2n},\omega_{std})$ is a Liouville domain with contact boundary $(S^{2n-1},\alpha)$. The Liouville vector field is $Y_\lambda=\frac{1}{2}(x_1,y_1,...\,,x_2,y_2)$. Moreover, $(\mathbb{R}^{2n},\omega_{std})$ is the completion of the unit ball.
 \item More generally, the unit disk bundle $D^\ast N$ in $T^\ast N$ (given by pairs $(q,p)$ with $\langle p,p\rangle_q\leq 1$) is a Liouville domain with contact boundary $S^\ast N$ and Liouville vector field $Y_\alpha$. Again, $T^\ast N$ is itself the completion of $D^\ast N$.
\end{itemize}
\end{ex}
\begin{defn} A \textbf{\textit{Liouville isomorphism}} between Liouville domains $(V_1,\lambda_1),\,(V_2,\lambda_2)$ is a diffeomorphism $\varphi: \hat{V}_1\rightarrow \hat{V}_2$ satisfying $\varphi^\ast\hat{\lambda}_2 = \hat{\lambda}_1+dg$ for a compactly supported function $g$.
\end{defn}
\begin{prop}[\cite{Sei}, page 3] \label{LI}
 For any Liouville isomorphism $\varphi: \hat{V}_1\rightarrow \hat{V}_2$ there exists an $R>0$ such that on $\Sigma_1\times[R,\infty)\subset \hat{V}_1$ the map $\varphi$ has the following form:
 \[\varphi(r,x)=(\psi(x),r-f(x)),\]
 where $\psi:\partial V_1 =\Sigma_1\rightarrow \Sigma_2=\partial V_2$ is a contact isomorphism satisfying $\psi^\ast\alpha_2 = e^f\cdot\alpha_1$ for a function $f\in C^\infty(\Sigma_1)$. So near $\infty$, the map $\phi$ is essentially a coordinate change in $r$.
\end{prop}
\begin{proof} 
 As $\varphi^\ast\hat{\lambda}_2=\hat{\lambda}_1+dg$ with $supp(g)$ compact, we can find an $R$ such that $\varphi^\ast\hat{\lambda}_2=\hat{\lambda}_1$ on $\Sigma_1\times[R,\infty)$. This implies $\hat{\omega}_1=d\hat{\lambda}_1=d\varphi^\ast\hat{\lambda}_2=\varphi^\ast(d\hat{\lambda}_2)=\varphi^\ast\hat{\omega}_2$ and hence also $\varphi_\ast Y_{\hat{\lambda}_1}=Y_{\hat{\lambda}_2}$. On $\Sigma_1\times[R,\infty)$, $\varphi$ is therefore compatible with the flows of $Y_{\hat{\lambda}_1}$ resp.\ $Y_{\hat{\lambda}_2}$ in the sense that $\phi^t_{\hat{\lambda}_2}\circ\varphi = \varphi\circ \phi^t_{\hat{\lambda}_1}$.
 This implies that for every $y\in\Sigma_2$ the flow line \linebreak $\{y\}\times\mathbb{R}\subset\hat{V}_2$ with respect to $\phi_{\hat{\lambda}_2}$ is hit at most once by $\varphi(\Sigma_1\times\{R\})$, as $\varphi$ is injective. Since $\varphi(\hat{V}_1\setminus (R,\infty))\subset\hat{V}_2$ is compact, we know that every $\{y\}\times\mathbb{R}$ intersects $\varphi(\Sigma_1\times[R,\infty))$ and by following the flow of $\phi_{\hat{\lambda}_2}$ backwards, we know that $\{y\}\times\mathbb{R}$ even intersects $\varphi(\Sigma_1\times\{R\})$. Therefore, $\varphi(\Sigma\times\{R\})\cap\big(\{y\}\times\mathbb{R}\big)$ contains exactly one element and we may write
 \[\varphi(x,R)=(\psi(x),\tilde{f}(x)),\]
 where $\psi:\Sigma_1\rightarrow \Sigma_2$ is a diffeomorphism and $\tilde{f}\in C^\infty(\Sigma_1)$. The compatibility of $\varphi$ with the two flows shows that
 \[\varphi(x,r)=(\psi(x),\tilde{f}(x)+r-R)\qquad \forall r\geq R.\]
 Note that the 1-form $\lambda_2$ at $\psi(x,\tilde{f}(x))$ is given by $e^{\tilde{f}(x)}\cdot \alpha_2$. Since $\varphi^\ast \hat{\lambda}_2 = \hat{\lambda}_1$ on $\Sigma_1\times[R,\infty)$, we find that
 \[\psi^\ast\left(e^{\tilde{f}(x)}\cdot \alpha_2\right) = e^R\cdot \alpha_1.\]
 Setting $f:=R-\tilde{f}$, we then have $\psi^\ast\alpha_2 = e^f\cdot \alpha_1$ and $ \varphi(x,r)=(\psi(x),r-f(x))$.\qedhere
\end{proof}
Note that, while Liouville isomorphisms preserve the contact structure of the boundary, the contact form may change arbitrarily. More precisely: Consider a Liouville domain $(V,\lambda)$ with contact boundary $(\Sigma,\alpha)$. If $\alpha'=e^f\cdot \alpha$ is another contact form which defines the same contact structure, we may consider the following contact hypersurface in the completion $\hat{V}$:
\[\Sigma' =\{(x,f(x))\,|\,x\in \Sigma\}.\]
Obviously, $\Sigma'$ bounds a compact region $V'\subset \hat{V}$, so that $(V',\lambda':=\hat{\lambda}|_{V'})$ is a Liouville domain, whose completion is also $(\hat{V},\hat{\lambda})$. Hence, $(V,\lambda)$ and $(V',\lambda')$ are Liouville isomorphic with trivial diffeomorphism $\varphi$. This motivates the following definition:
\begin{defn}\label{filling}
 Let $(\Sigma,\xi)$ be a contact manifold. If there exists a Liouville domain $(V,\lambda)$ such that $\partial V=\Sigma$ and $\xi=\ker i^\ast\lambda$, then we call the equivalence class of $(V,\lambda)$ under Liouville isomorphisms an \textbf{\textit{exact contact filling}} of $(\Sigma,\xi)$.
\end{defn}
The discussion above shows that an exact contact filling does not depend on a specific contact form. Therefore, any invariant of exact contact fillings gives  an invariant for contact structures (with the filling). We will later show (Corollary \ref{liouvilleinvariance}), that each exact contact filling of $(\Sigma,\xi)$ possesses a well-defined  Rabinowitz-Floer homology, which is therefore an invariant of the contact structure (together with the filling).

\subsection{Defining Hamiltonians}
The setup in which we define Rabinowitz-Floer homology is the following. Let $(V,\lambda)$ be the completion of a Liouville domain $\tilde{V}$ with contact boundary $M:=\partial\tilde{V}$. Let $\Sigma\subset V$ be an exact contact hypersurface bounding a compact domain $W\subset V$, so that $\Sigma$ is the boundary of the Liouville domain $W$. In particular, we do not require that the completions $\widehat{W}$ and $\widehat{\,\tilde{V}\,}=V$ coincide (nevertheless, $\widehat{W}\subset V$ as symplectic submanifold). 
\begin{defn}
 A \textbf{defining Hamiltonian} for the boundary $\Sigma$ of a Liouville domain $W\subset V$ is a function $H\in C^\infty(V)$, which is constant outside a compact set, whose zero level set $H^{-1}(0)$ equals $\Sigma$ and whose Hamiltonian vector field $X_H$ agrees with the Reeb vector field $R_\alpha$ on $\Sigma$, i.e.\ for the inclusion $i: \Sigma\hookrightarrow V$ holds
 \[\alpha:= i^\ast \lambda \qquad \text{ and }\qquad i_\ast(R_\alpha)=X_H|_\Sigma.\]
 In particular, note that $\Sigma$ is a regular level set of $H$.
\end{defn}
\begin{ex}Let $\beta$ be a smooth monotone cut-off function such that $\beta(x)=\begin{cases}x & x\leq 2\\ 3& x\geq 4\end{cases}$.
 \begin{itemize}
  \item The function $H(p):=\beta\big(||p||^2)-1$ is a defining Hamiltonian for $S^{2n-1}$ in $(\mathbb{R}^{2n},\omega_{std})$.
  \item The function $H(q,p):=\beta\big(\langle p,p\rangle_q\big)-1$ is a defining Hamiltonian for $S^\ast N$ in $T^\ast N$.
 \end{itemize}
\end{ex}
\begin{prop} \label{cor2} For the boundary $\Sigma$ of a Liouville domain $W\subset V$ holds:
 \begin{enumerate}
 \item[i.] The space $\mathscr{H}$ of defining Hamiltonians for $\Sigma$ is non-empty and convex.
 \item[ii.] If $\alpha_0$ and $\alpha_1$ are two contact forms defining the same contact structure on $\Sigma$, then there exists a homotopy of Liouville domains $(W_s,\Sigma_s)\subset(V,\lambda)$ and a corresponding homotopy of defining Hamiltonians $H_s$, such that $\alpha_0=\lambda|_{\Sigma_0}$ and $\alpha_1=\lambda|_{\Sigma_1}$. 
\end{enumerate}
\end{prop}
\begin{proof}
 Since $(V,\lambda)$ is the completion of a Liouville domain, we find that the Liouville vector field $Y_\lambda$ is complete. This allows us to find a symplectic embedding of the symplectization $i:(\Sigma\times\mathbb{R})\rightarrow V$ with $i(\Sigma\times\{0\})=\Sigma$ (see Discussion \ref{dis2}). Hence, it suffices to construct defining Hamiltonians on $\Sigma\times\mathbb{R}$.
 \begin{itemize}
  \item[\textit{@i.}] Here, we consider the function $\qquad\bigg.\tilde{H}(x,r):=e^{\rho(r)}-1,$\\
  where $\rho$ is a smooth monotone increasing function with $\rho(r)=r$ near 0 and $\rho$ constant outside a compact set. This guarantees that $\tilde{H}^{-1}(0)=\{r=0\}=\Sigma\times\{0\}$. Since $d\tilde{H}=e^rdr$ near $\Sigma\times\{0\}$ and $\iota_{R_\alpha}\omega = \iota_{R_\alpha}\big(d(e^r\alpha)\big)=\iota_{R_\alpha}\big(e^rd\alpha +e^rdr\wedge \alpha\big) = -e^rdr$, the Hamiltonian vector field agrees with $R_\alpha$ on $\Sigma\times\{0\}$.\\
  The defining Hamiltonian $H$ for $\Sigma$ is then obtained from $\tilde{H}\circ i^{-1}$ by extending it as constant on $V\setminus i(\Sigma\times\mathbb{R})$. Hence $\mathscr{H}\neq\varnothing$. In general, $H$ is a defining Hamiltonian for $\Sigma$ if and only if it satisfies the following equation:
  \[\mathcal{L}_{Y_\lambda} H|_\Sigma = dH(Y_\lambda)|_\Sigma=\omega(Y_\lambda,R_\alpha)=\lambda(R_\alpha)=1.\]
  Since $H(\Sigma)=0$, this implies that every defining $H$ is positive on $\Sigma\times\mathbb{R}^+$ and negative on $\Sigma\times\mathbb{R}^-$. Hence we have for $H_1,H_2\in\mathscr{H}$ that 
  \[s\cdot H_1 + (1-s)\cdot H_2 \in \mathscr{H} \qquad \text{ for all }\quad s\in[0,1].\]
  \item[\textit{@ii.}] Without loss of generality, we may assume that $\alpha_0=i^\ast\lambda$, where $i:\Sigma\times\{0\}\hookrightarrow V$ is the embedding mentioned above. Since $\alpha_0$ and $\alpha_1$ define the same contact structure on $\Sigma$, we find a function $f\in C^\infty(\Sigma)$, such that $\alpha_1=e^f\cdot\alpha_0$. Then we modify the above construction as follows:
  \[\tilde{H}_s(x,r):=e^{\rho(r-s\cdot f(x))}-1.\]
  Define $H_s$ again by extending $\tilde{H}_s\circ i^{-1}$ as constant on $V\setminus i(\Sigma\times\mathbb{R})$. The homotopy of Liouville domains is then given by $W_s:=H_s^{-1}\big((\infty,0]\big)$ and $\Sigma_s:=H_s^{-1}(0)$.\qedhere
 \end{itemize}
\end{proof}

\subsection{Rabinowitz-Floer homology}\label{sec1.4}
As above, let $(V,\lambda)$ be the completion of a Liouville domain $\tilde{V}$ with contact boundary $M=\partial\tilde{V}$, let $\Sigma\subset V$ be an exact contact hypersurface bounding a compact Liouville domain $W$ and let $H$ be a defining Hamiltonian for $\Sigma$. 
\begin{defn}
 The \textit{\textbf{spectrum}} $spec(\Sigma,\alpha)$ of a contact manifold $\Sigma$ with contact form $\alpha$ is the set of real numbers $\eta\in\mathbb{R}$ such that the ordinary differential equation $\dot{v}=\eta R_\alpha$ has a 1-periodic solution. We denote with $\mathcal{P}(\alpha)\subset C^\infty(S^1,\Sigma)\times\mathbb{R}$ the set of pairs $(v,\eta)$, such that $\eta\in spec(\Sigma,\alpha)$ and $\dot{v}=\eta R_\alpha$.
\end{defn}
\begin{rem}
 We identify a pair $(v,\eta)\in\mathcal{P}(\alpha)$ with the $\eta$-periodic Reeb orbit $\tilde{v}(t):=v(t/\eta)$. Note that we do not exclude the cases $\eta = 0$ and $\eta<0$. For $\eta < 0$, we have that $v$ is again a Reeb orbit with period $|\eta|$, but run through in the opposite direction. For $\eta=0$, the loop $v$ is just constant. The pairs $(v,0)\in\mathcal{P}(\alpha)$ are hence in one-to-one correspondence to the points of $\Sigma$. 
\end{rem}
Let $\mathscr{L}=C^\infty(S^1,V)$ denote the free (smooth) loop space of $V$. The \textbf{\textit{Rabinowitz action functional}} on $V$ associated to $H$ is given by
\[\mathcal{A}^H:\mathscr{L}\times\mathbb{R}\rightarrow\mathbb{R},\qquad \mathcal{A}^H(v,\eta):= \int^1_0 \Big(\lambda(\dot{v}(t))-\eta H(v(t))\Big)dt.\]
One can think of $\mathcal{A}^H$ as the Lagrange multiplier action functional of classical mechanics. In Section \ref{sec2.1}, we will see that the critical points $(v,\eta)$ of $\mathcal{A}^H$ are characterized by
\begin{equation*}
  0=\dot{v}-\eta X_H\qquad\text{ and }\qquad 0=\int^1_0H(v(t))dt.
\end{equation*}
The first equation implies that $\frac{d}{dt}H(v)=dH(\dot{v})=dH(\eta X_H)=0$, so that $H(v)$ is constant and hence 0 by the second equation. As $H$ is a defining Hamiltonian with $H^{-1}(0)= \Sigma$ and $X_{ H}|_{\scriptscriptstyle \Sigma} = R_\alpha$, we find that these equations are therefore equivalent to
\begin{equation} \label{eq2}
 v(t)\in\Sigma \quad\forall\; t\qquad \text{ and }\qquad \dot{v}=\eta R_\alpha.
\end{equation}
This shows that the set of critical points $\crita$ of $\mathcal{A}^H$ consists of the closed Reeb trajectories $v$ on $\Sigma$ with period $\eta$, or equivalently that $\crita=\mathcal{P}(\alpha)$.
\begin{defn}
 A family of $\omega$-compatible almost complex structures
 \[J: S^1\times\mathbb{R}\rightarrow End(TV),\quad (t,n)\mapsto J_t(\cdot,n)\]
 is called \textbf{\textit{admissible}} (of class $C^\ell$/smooth) if:
 \begin{itemize}
  \item as a map with domain $S^1\times V \times \mathbb{R}$ we have that $J$ is of class $C^\ell$/smooth.
  \item $J$ is $t$- and $n$-independent cylindrical at the unbounded end in $V$. This means that on $M\times[R,\infty)$ for $R>0$ sufficiently large $J$ is of the form
\[J|_{\xi_M}=J_0\qquad\text{ and }\qquad J\frac{\partial}{\partial r}= R_\lambda,\]
where $J_0$ is any compatible almost complex structure on the contact structure $\xi_M$ of $M$ and $R_\lambda$ is the vector field whose restriction to the contact hypersurface $M\times\{r\}$ equals the Reeb vector field. Equivalently, we can require that $d(e^r)\circ J = -\lambda$ on $M\times[R,\infty)$.
\item the family is $C^\ell$-bounded, meaning that $\sup_n ||J_t(\cdot,n)||_{C^\ell}<\infty$ with respect to some background norm $||\cdot||$ on $M$.
 \end{itemize}
\end{defn}
\begin{rem}
 The reason for the dependency of $J$ on the additional parameter $n$, not found in the literature until recently in \cite{BourOan3} and \cite{AbbMer}, is based in the transversality problem for Rabinowitz-Floer homology and will become clear in the proof of the local Transversality Theorem \ref{theotrans}.
\end{rem}
In Section \ref{sec2.1}, formula (\ref{metric}), we will see that such a family of almost complex structures $J$ can be used to define a metric $g$ on $\mathscr{L}\times\mathbb{R}$, which yields a gradient $\nabla\mathcal{A}^H$ for $\mathcal{A}^H$. The explicit formula for $\nabla\mathcal{A}^H$ is given in (\ref{gradienta}). 
\begin{defn} \label{gradtraj} An \textbf{$\mathbf{\mathcal{A}^H}$-gradient trajectory} is a solution $(v,\eta)\in C^\infty (\mathbb{R}\times S^1,V)\times C^\infty(\mathbb{R},\mathbb{R})$ of the Rabinowitz-Floer equation, which is the following partial differential equation:
\begin{align} \label{eq3}
\partial_s(v,\eta)=\nabla\mathcal{A}^H(v,\eta) &&\Leftrightarrow&
 \begin{split}
  \partial_s v +J_t (v,\eta)\bigl(\partial_t v-\eta X_H(v)\bigr) &=0 \\
  \partial_s\eta \;+\; \int^1_0 H\bigl(v(s,t)\bigr) dt &=0.
 \end{split}
\end{align}
\end{defn}
The next two lemmas characterize the $\mathcal{A}^H$-gradient trajectories which connect critical points $(v^\pm,\eta^\pm)\in\crita=\mathcal{P}(\alpha)$.
\begin{lemme} \label{lem2a}
 If $(v,\eta)\in \mathcal{P}(\alpha)$, then $\mathcal{A}^H(v,\eta) = \eta$. Thus $\mathcal{A}^H(\mathcal{P}(\alpha))=spec(\Sigma,\alpha)$.
\end{lemme}
\begin{proof}
 Since $H|_{\Sigma}=0$ and $\dot{v}=\eta R$ and $im(v)\subset\Sigma$ if $(v,\eta)\in\mathcal{P}(\alpha)$, we may calculate
\[ \mathcal{A}^H(v,\eta) = \int^1_0 \Big(\lambda(\dot{v})+\eta H(v)\Big)dt = \int^1_0\lambda (\eta R)dt = \int^1_0\eta\, dt=\eta.\qedhere\]
\end{proof}
\begin{lemme} \label{lem2} If $(v,\eta)$ is a non-stationary $\mathcal{A}^H$-gradient trajectory between critical points $(v^\pm,\eta^\pm)\in\mathcal{P}(\alpha)$, i.e.\ where $\displaystyle\lim_{s\rightarrow \pm\infty} (v,\eta)=(v^\pm,\eta^\pm)$, then $\eta^+>\eta^-$.
\end{lemme}
\begin{proof}
With $||\cdot||$ as the norm of the metric $g$ and Lemma \ref{lem2a}, we calculate
\begin{align*}
 \eta_+-\eta_-=\mathcal{A}^H (v^+,\eta^+)-\mathcal{A}^H(v^-,\eta^-)
&=\int^\infty_{-\infty} \frac{d}{ds}\mathcal{A}^H(v,\eta)ds\\
&=\int^\infty_{-\infty} g\bigl(\nabla \mathcal{A}^H(v,\eta),\partial_s(v,\eta)\bigr) ds\\
&\overset{(\ref{eq3})}{=}\int^\infty_{-\infty} ||\nabla\mathcal{A}^H(v,\eta)||^2 ds\\ &> 0.\qedhere
\end{align*}
\end{proof}
\begin{defn}\label{defnenergy}
 The quantity $E(v,\eta):=\int^\infty_{-\infty}||\nabla\mathcal{A}^H(v,\eta)||^2ds\geq 0$ is called the \textbf{\textit{energy}} of the $\mathcal{A}^H$-gradient trajectory $(v,\eta)$. If $\displaystyle\lim_{s\rightarrow\pm\infty}(v,\eta)=(v^\pm,\eta^\pm)\in\mathcal{P}(\alpha)$, then Lemma \ref{lem2} tells us that $E(v,\eta)=\eta^+-\eta^-$.
\end{defn}
Since any pair $(v,\eta)\in\mathcal{P}(\alpha),\,\eta\neq 0$, yields a whole $S^1$-family of points in $\mathcal{P}(\alpha)$ by time shift, $\mathcal{P}(\alpha)$ never consists of isolated points. In order to build a Floer-type homology with basis $\mathcal{P}(\alpha)$, we therefore have to use Morse-Bott techniques. This is why we impose the following non-degeneracy assumption on the Reeb flow $\phi^t$ on $\Sigma$:\bigskip\\\phantom{Here}
\begin{minipage}{12cm}
 \textit{The set $\mathcal{N}^\eta\subset\Sigma$ formed by the $\eta$-periodic Reeb orbits is a closed submanifold for each $\eta\in\mathbb{R}$ and $T_p\,\mathcal{N}^\eta = \ker \left(D_p\phi^\eta - id\right)$ holds for all $p\in\mathcal{N}^\eta$.}
\end{minipage}\hfill(MB)\bigskip\\
Note that we do not assume that the rank $d\lambda|_{\mathcal{N}^\eta}$ is locally constant, as was done in \cite{FraCie}. As far as we see this is not needed. The assumption (MB) is generically satisfied, as shown in \cite{FraCie}, appendix B. Moreover, (MB) implies that $\mathcal{P}(\alpha)$ is a proper submanifold of $\mathscr{L}\times\mathbb{R}$ (see Theorem \ref{theo17}). Note that the components of $\mathcal{P}(\alpha)$ with fixed $\eta\in spec(\Sigma,\alpha)$ correspond to the submanifolds $\mathcal{N}^\eta$ of $\Sigma$ via the map $(v,\eta)\mapsto v(0)$. By abuse of notation, we will write $\mathcal{N}^\eta$ also for the components of $\mathcal{P}(\alpha)$, i.e.\ we consider $\mathcal{N}^\eta$ either as a submanifold of $\mathscr{L}\times\mathbb{R}$ or as a submanifold of $\Sigma$, depending on the context.\medskip\\
There are several approaches to deal with Morse-Bott situations. The one that we will use here, \textit{flows with cascades}, was developed by Urs Frauenfelder, \cite{Fra}, and Frédéric Bourgeois, \cite{Bour}, based on an idea from Piunikhin, Salamon, and Schwarz. It uses flows on the critical manifolds without further perturbation, which makes the computations we have in mind easier. For that, we choose an additional Morse function $h$ and a suitable metric $g_h$ on $\mathcal{P}(\alpha)$ such that the restrictions $h_\eta:=h|_{\mathcal{N}^\eta}$ and $g_\eta:=g_h|_{\mathcal{N}^\eta}$ form a Morse-Smale pair on $\mathcal{N}^\eta$ for every $\eta\in spec(\Sigma,\alpha)$. This is equivalent to $h_\eta$ being a Morse function on $\mathcal{N}^\eta$ for which the stable and unstable manifolds $W^s(p)$ resp.\ $W^u(q)$ of the $\nabla_{g_h} h$-gradient flow intersect transversally for each pair $p,q\in crit(h_\eta)$ .
\begin{defn}\label{Morsetraj}
 An \textbf{h-Morse flow line} $y\in C^\infty \big(\mathbb{R},\crita\big)$ is a solution of \linebreak $\dot{y} = \nabla h(y)$, where $\nabla h$ is the gradient of $h$ with respect to $g_h$.
\end{defn}
\begin{defn}\label{defntrajwithcasc}
 For $c^-,c^+\in crit(h)$ and $m\in\mathbb{N}$, we call a \textbf{trajectory from $c^-$ to $c^+$ with $m$ cascades} a tupel
\begin{equation*}
 (x,t) = \bigl( (x_k)_{1\leq k\leq m},(t_k)_{1\leq k\leq m-1}\bigr),
\end{equation*}
consisting of $\mathcal{A}^H$-gradient trajectories $x_k=(v_k,\eta_k)$ and real numbers $t_k \geq 0$ such that there exist (possibly constant) $h$-Morse flow lines $y_k$, $0\leq k\leq m,$ with
\begin{enumerate}
 \item[i.] \hspace{1cm} $\displaystyle\lim_{s\rightarrow-\infty} y_0(s)=c^-, \mspace{60mu}\lim_{s\rightarrow-\infty} x_1(s)=y_0(0)$,
 \item[ii.] \hspace{1.2cm} $\displaystyle\lim_{s\rightarrow\infty} x_k(s)=y_k(0),\quad\quad \lim_{s\rightarrow-\infty} x_{k+1}(s)=y_k(t_k)$,
 \item[iii.] \hspace{1.1cm} $\displaystyle\lim_{s\rightarrow\infty} x_m(s)=y_m(0),\mspace{40mu}\lim_{s\rightarrow\infty} y_m(s)=c^+$.
\end{enumerate}
 A \textbf{trajectory with \textnormal{\textit 0} cascades from $c^-$ to $c^+$} is an $h$-Morse flow line from $c^-$ to $c^+$. See Figure 1 for an illustration.
\end{defn}
\begin{figure}[ht]
\centering
 \resizebox{15.5cm}{!}{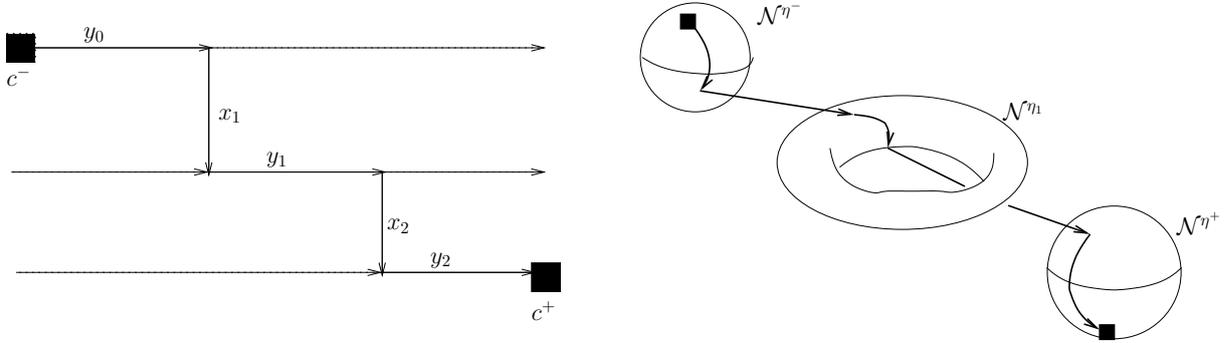} \caption{\label{fig1} A flow line with 2 cascades passing through the critical manifolds $\mathcal{N}^{\eta^-},\,\mathcal{N}^{\eta_1},\,\mathcal{N}^{\eta^+}$.}
\end{figure}
In other words, a trajectory with cascades is an alternating sequence of segments of $h$-Morse flow lines and whole $\mathcal{A}^H$-gradient trajectories. The space of all such trajectories with $m$ cascades from $c^-$ to $c^+$ is denoted by $\widehat{\mathcal{M}}(c^-,c^+,m)$. The moduli space
\[\mathcal{M}(c^-,c^+,m):=\raisebox{.2em}{$\widehat{\mathcal{M}}(c^-,c^+,m)$}\big/\raisebox{-.2em}{$\mathbb{R}^m$}\] is obtained by dividing out the free $\mathbb{R}^m$-action on the $m$ cascades given by time shifts. If $m=0$, we also divide by $\mathbb{R}$ (not by $\mathbb{R}^0\cong 1$), as there is still an $\mathbb{R}$-action by time shift now on the $h$-Morse flow line. In Theorem \ref{theotrans} we show that $\mathcal{M}(c^-,c^+,m)$ is a manifold for generic choices of $J_t(\cdot,n)$. We denote the moduli space of all trajectories with cascades from $c^-$ to $c^+$ by
\[\mathcal{M}(c^-,c^+):= \bigcup_{m\in\mathbb{N}} \mathcal{M}(c^-,c^+,m).\]
 Due to Theorem \ref{theo17}, there are only finitely many non-empty critical manifolds $\mathcal{N}^\eta$ with $\eta\in[\eta^-,\eta^+]$ for each $\eta^\pm\in\mathbb{R}$, so that the above union is in fact finite. Indeed, each cascade reduces the period $\eta$ (cf. Lemma \ref{lem2}) and connects critical manifolds $\mathcal{N}^{\eta_1}$ and $\mathcal{N}^{\eta_2}$, where $\eta^-\leq\eta_1<\eta_2\leq\eta^+$.\pagebreak\\
Theorem \ref{theocompactness2} states that $\mathcal{M}(c^-,c^+)$ still carries the structure of a manifold and Theorem \ref{theoglue} asserts that it is compact. Its zero-dimensional component $\mathcal{M}^0(c^-,c^+)$ is hence a finite set.\medskip\\
Now we define the Rabinowitz-Floer homology. The chain complex $RFC(H,h)$ is given as the $\mathbb{Z}_2$-vector space consisting of formal sums
\[\xi=\sum_{c\,\in crit(h)} \xi_c\cdot c,\]
where the coefficients $\xi_c\in \mathbb{Z}_2$ satisfy the following finiteness condition:
\begin{equation}\label{finite}
 \#\{c\in crit(h)\,|\,\xi_c\neq 0\,\land\,\mathcal{A}^H(c)\geq \kappa\} <\infty\qquad\text{for all $\kappa\in \mathbb{R}$}.
\end{equation}
So $RFC(H,h)$ is a Novikov-completion of the $\mathbb{Z}_2$-vector space generated by the critical points of $h$. Let $\#_2 \mathcal{M}^0(c^-,c^+)\in\mathbb{Z}_2$ denote the cardinality of $\mathcal{M}^0(c^-,c^+)$ modulo 2. The boundary operator $\partial^F$ is then defined to be the (infinite) linear extension of
\begin{equation}\label{eqX}
 \partial^F c^+ = \sum_{c^-\in crit(h)} \#_2 \mathcal{M}^0(c^-,c^+)\cdot c^-,\qquad\qquad c^+\in crit(h).
\end{equation}
To see that $\partial^F$ is well-defined, we have to show that the right hand side still satisfies the finiteness condition (\ref{finite}):\\
Indeed, it follows from Lemma \ref{lem2} that $\partial^F$ reduces the action, i.e.\ $\#_2 \mathcal{M}^0(c^-,c^+)\neq 0$ only if $\mathcal{A}^H(c^-)\leq\mathcal{A}^H(c^+)$. Moreover, we show in Theorem \ref{theo17}, that for any $c^+$ and every $a\in\mathbb{R}$, there are only finitely many $c^-\in crit(h)$ with $a\leq \mathcal{A}^H(c^-)\leq \mathcal{A}^H(c^+)$. Therefore $\partial^F c^+$ satisfies (\ref{finite}) for any $c^+$. In the general case, where $\partial^F\big(\sum \xi_c\cdot c\big):= \sum \xi_c\cdot \partial^F c$, condition (\ref{finite}) follows as both the sum $\sum\xi_c\cdot c$ and $\partial^F c$ satisfy the finiteness condition.\medskip\\
It follows from standard Floer-techniques by considering the 1-dimensional component of $\widehat{\mathcal{M}}(c^-,c^+)$, that $\partial^F\circ\partial^F=0$. Hence, $\partial^F$ is a boundary operator on $RFC(H,h)$. The \textbf{\textit{Rabinowitz-Floer homology}} of $(V,\Sigma)$ with respect to the Hamiltonian $H$ and the Morse-function $h$ is the homology of the chain complex $\big(RFC(H,h),\partial^F\big)$, i.e.\
\[RFH(H,h):=\frac{\ker \partial^F}{\text{im}\; \partial^F}.\]
We prove in Corollary \ref{liouvilleinvariance} that $RFH(H,h)$ only depends on the contact manifold $(\Sigma,\xi)$ and the filling Liouville domain $W$, thus allowing us to write $RFH(W,(\Sigma,\xi))$, where we omit $\xi$ whenever it is clear from the context.\\
In Section \ref{1.5}, we show that $RFH(W,\Sigma)$ can be given a $\mathbb{Z}$-grading under the following assumptions:
\begin{enumerate}
 \item[(A)] \textit{The map $i_\ast : \pi_1(\Sigma)\rightarrow\pi_1(W)$ induced by the inclusion is injective.}
 \item[(B)] \textit{The integral $I_{c_1}: \pi_2(W)\rightarrow \mathbb{Z}$ of the first Chern class $c_1(TW)$ vanishes on spheres.}
\end{enumerate}
For $c=(v,\eta)\in crit(h)\subset \mathcal{P}(\alpha)$, the degree $\mu(c)$ is a half integer given by the formula
\[\mu(c)=\mu_{CZ}(v)+ind_h(c)-\frac{1}{2} dim_c\, \mathcal{N}^\eta +\frac{1}{2},\]
where $ind_h(c)$ is the Morse index of $c$ with respect to $h$, $\mu_{CZ}(v)$ is the (transversal) Conley-Zehnder index of $v$ and $dim_c\,\mathcal{N}^\eta$ is the local dimension of $\mathcal{N}^\eta$ at $c$.

\subsection{Outline of the thesis and main result}
A major part of this thesis is devoted to technical details for the construction of \linebreak Rabinowitz-Floer homology. We do this, since up to now, there is no complete reference for this in the literature. Note that due to the integral term in equation (\ref{eq3}) some delicate adaptations to the standard construction by Floer have to be made.\\
In Section \ref{sec2}, we deal with the transversality problem, i.e.\ we show that $\widehat{\mathcal{M}}(c^-,c^+,m)$ is a manifold for generic $J$. We generalize this to setups where everything is symmetric with respect to a symplectic symmetry of finite order.\\ In Section 3 we present many different properties of the moduli spaces. Due to the vastness of the topic, we do not prove the compactness of $\mathcal{M}(c^-,c^+,m)$. In Section \ref{sec3.1}, we merely provide some estimates which should ensure that the usual compactification due to Gromov works.\\
 In Section \ref{sec.inv}, we show that the Rabinowitz-Floer homology does not depend on the auxiliary choices made for its definition. Even more surprising, we show that $RFH(W,\Sigma)$ actually is invariant under Liouville isomorphisms, thus giving an invariant of the filling by $W$ of $(\Sigma,\xi)$ (in the sense of Definition \ref{filling}). In Proposition \ref{prop16} we then see that under some circumstances $RFH(W,\Sigma)$ is even independent of $W$ and hence yields an invariant of the contact structure.\\
Also in Section \ref{sec.inv}, we show that the action $\mathcal{A}^H$ induces a filtration of the complex $\big(RFC(H,h),\partial^F\big)$. We use this filtration to define truncated homology groups $RFH^{(a,b)}(W,\Sigma)$ and growth rates $\Gamma(W,\Sigma)$. The latter can be used to obtain more information on $RFH(W,\Sigma)$ if it is of infinite dimension. The Sections \ref{1.4} and \ref{1.5} are devoted to the Conley-Zehnder index and the $\mathbb{Z}$-grading of $RFH$.\\
In Section \ref{secalg}, we provide some useful facts about direct and inverse limits. In \ref{secRedFil}, Theorem \ref{reductiontheo}, we show that $RFH(W,\Sigma)$ over field coefficients can be calculated using the singular homology $H_\ast(\mathcal{N}^\eta)$ without knowing an explicit Morse function $h$ on $\mathcal{N}^\eta$ simply by algebraically pretending that we have a perfect Morse function $h$. This  is purely algebraic and as a consequence less intuitive. For a first reading, it can be skipped, as we apply its main result in this thesis solely in situations where we could use standard arguments from spectral sequences.\\ 
In Sections \ref{secsur} and \ref{secsymhom} we introduce symplectic (co)homology and contact surgery (which includes the connected sum construction). In particular, we give a more detailed (and slightly corrected) proof of the fact that symplectic (co)homology is invariant under subcritical surgery, originally due to K. Cieliebak, \cite{Cie}. We make this detour in order to show that Rabinowitz-Floer homology is also invariant under subcritical surgery, which we could not prove directly.\footnote{In a recently published article, \cite{CieOan}, Rem. 9.15, A. Oancea and K. Cieliebak show this invariance within the context of $RFH$.}\\
In Section 7, we introduce the Brieskorn manifolds $\Sigma_a$ and calculate $RFH_\ast(W_\veps,\Sigma_a)$ explicitly for some $\Sigma_a$ with fillings $W_\veps$. In \ref{sec7.3} we then prove our Main Theorem and some corollaries.\\
The appendices deal with technical details which are needed in the text but which are to long or to far of the general discourse. In Appendix \ref{app1}, we give an explicit example of a Morse-Smale pair on the unit cotangent bundle $S^\ast S^n$ which is symmetric with respect to a certain involution. Appendix \ref{appB} recollects facts about convolutions. In Appendix \ref{auttrans}, we show that transversality always holds along constant solutions of (\ref{eq3}). Finally in Appendix \ref{conventions}, we give a short list of all the major assumptions and conventions that we use in this thesis.
\bigskip\\
The main result of this dissertation is the following theorem which states the existence of a rich variety of fillable contact structures or fillings on a differentiable manifold $\Sigma$, with $\dim \Sigma \geq 5$, if it supports at least one fillable contact structure.
\begin{theo*}[Main Theorem]~\\
 Suppose that $\Sigma$ is a differentiable manifold, $\dim \Sigma =2n-1\geq 5$, which supports at least one fillable contact structure with filling for which the conditions (A) and (B) are true. Then $\Sigma$ satisfies at least one of the following alternatives:
 \begin{itemize}
  \item[a)] For every fillable contact structure $\xi$ on $\Sigma$ and any filling $W$ of $(\Sigma,\xi)$, which satisfies (A) and (B), holds true that
  \[\dim_{\mathbb{Z}_2} RFH_\ast(W,(\Sigma,\xi)) =\infty\qquad\qquad\forall\,\ast\in\mathbb{Z}\setminus[-n+1,n].\]
  \item[b)] There is (at least) one contact structure on $\Sigma$  for which there exist infinitely many different fillings.
  \item[c)] There exist infinitely many different fillable contact structures on $\Sigma$.
 \end{itemize}
\end{theo*}
Note the difference to dimension 3, where according to Eliashberg (see \cite{Geiges}) the only fillable contact structure on $S^3$ is the standard one and where due to Gromov, \cite{Gro}, the only filling for the standard contact structure on $S^3$ is the unit ball $(B^4,\omega_{std})$. In particular, $(S^3,\xi_{std})$ does not satisfy a), b) or c).\\
In Section \ref{sec7.3}, we also prove the following dynamical and contact topological consequences.
\begin{cor*}~
 \begin{itemize}
  \item If $\,\Sigma$ satisfies alternative a) of the Main Theorem, then every fillable contact structure on $\Sigma$ has for any generic contact form simple Reeb trajectories of arbitrary length.
  \item If $\,\Sigma$ satisfies alternative b) but not a) of the Main Theorem, then there is at least one contact structure on $\Sigma$ which has simple closed Reeb trajectories of arbitrary length for every generic contact form.
 \end{itemize}
\end{cor*}
\begin{cor*} Every Brieskorn manifold $\Sigma_a$ supports at least 2 non-contactomorphic, exactly fillable contact structures.
\end{cor*}
\newpage

\section{Transversality}\label{sec2}
The aim of this section is to show that $\widehat{\mathcal{M}}(c^-,c^+,m)$ is a finite dimensional manifold. The Subsections \ref{sec2.1} through \ref{secholprop} provide analytic properties that will be needed in the proof of the fundamental Global Transversality Theorem \ref{theotrans} in  \ref{sec2.5}.\\
Throughout this part, we assume $(V,\lambda)$ to be an exact symplectic manifold that is the completion of a compact Liouville domain $\tilde{V}$ with contact boundary $M$. In particular, the symplectization $M\times[0,\infty)$ embeds into $V$ and $V\setminus(M\times(0,\infty))=\tilde{V}$ is compact. Moreover, we assume that $\Sigma\subset V$ is an exact contact hypersurface bounding a compact Liouville domain $W$ and $H$ a defining Hamiltonian for $\Sigma$.

\subsection{The action functional}\label{sec2.1}
In this subsection, we will analyse more closely the Rabinowitz action functional $\mathcal{A}^H$ that we introduced in \ref{sec1.4}. We calculate its gradient $\nabla\mathcal{A}^H$ and Hessian $\nabla^2\mathcal{A}^H$. We show that $\nabla^2\mathcal{A}^H$ is a Fredholm operator of index zero and use this to show that (MB) implies that $\mathcal{A}^H$ is a Morse-Bott functional, i.e.\ that $\crita$ is a submanifold of $\mathscr{L}\times\mathbb{R}$. In particular, we show that the critical manifolds $\mathcal{N}^\eta$ are isolated in the loop space $\mathscr{L}\times\mathbb{R}$. In classical Morse theory this would follow from the Morse Lemma. However, as $\mathscr{L}\times\mathbb{R}$ is infinite dimensional, there is no analogue of the Morse Lemma. Recall that $\mathcal{A}^H$ is given by
\[\mathcal{A}^H:\mathscr{L}\times\mathbb{R}\rightarrow\mathbb{R},\qquad \mathcal{A}^H(v,\eta):= \int^1_0 \lambda(\dot{v}(t))-\eta H(v(t))dt,\]
where $v$ is a smooth loop on $V$. More generally, we do not only consider smooth loops, but also loops that are of Sobolev-class $W^{k,p},\,k\in\mathbb{N},\,1<p<\infty$, where a map $v: S^1\rightarrow V$ is called $W^{k,p}$ if it is $W^{k,p}$ in every chart. We denote the space of such loops by $W^{k,p}(S^1,V)$ and remark that $W^{k,p}(S^1,V)$ is a Banach manifold while $H^k(S^1,V)=W^{k,2}(S^1,V)$ is even a Hilbert manifold (see \cite{kling}, 2.3 ff.). The tangent space $T_v W^{k,p}(S^1,V)$ at any $v\in W^{k,p}(S^1,V)$ is given by the $W^{k,p}$-vector space of $S^1$-sections of $v^\ast TV$, i.e.\ a tangent vector $\mathfrak{v}$ at $v$ is a 1-periodic $W^{k,p}$-map
\[\mathfrak{v}: S^1\rightarrow TV\qquad\text{ satisfying }\qquad \mathfrak{v}(t)\in T_{v(t)}V.\]
A metric on $W^{k,p}(S^1,V)\times\mathbb{R}$ is obtained via a family of $\omega$-compatible almost complex structures $J_t(\cdot,n)$ depending on $(t,n)\in S^1\times \mathbb{R}$. Given $(v,\eta)\in W^{k,p}(S^1,V)\times\mathbb{R}$ and $(\mathfrak{v}_1,\hat{\eta}_1),(\mathfrak{v}_2,\hat{\eta}_2)\in T_{(v,\eta)}(W^{k,p}(S^1,V)\times\mathbb{R})$ the metric $g$ is defined by
\begin{equation}\label{metric}
 g\left(\begin{pmatrix}\mathfrak{v}_1\\ \hat{\eta}_1 \end{pmatrix}\,,\,\begin{pmatrix}\mathfrak{v}_2\\\hat{\eta}_2\end{pmatrix}\right):=\int^1_0\omega\Big(\mathfrak{v}_1(t)\,,\,J_t\big(v(t),\eta\big)\mathfrak{v}_2(t)\Big)dt+\hat{\eta}_1\cdot\hat{\eta}_2.
\end{equation}
In the following, we are going to calculate the first and second variation of $\mathcal{A}^H$ in the form of gradient $\nabla\mathcal{A}^H$ and Hessian $\nabla^2 \mathcal{A}^H$ with respect to $g$. For that, assume that $(v_s,\eta_s)\subset W^{k,p}\times\mathbb{R},\,k\geq 1,$ is a differentiable 1-parameter family depending on $s\in(-\veps,\veps)$. Write 
\[\frac{d}{dt}v=\dot{v},\qquad \left.\frac{d}{ds}v\,\right|_{s=0}=\mathfrak{v}\quad\text{ and }\quad\left.\frac{d}{ds}\eta\,\right|_{s=0} =\hat{\eta}.\]
Let $\nabla$ denote the Levi-Civita connection of the metric $g_{t,n}=\omega(\cdot,J_{t,n}\cdot)$ for fixed $(t,n)$. Observe that, when considered as a derivation of functions, we have
\[\frac{d}{dt}=\nabla_{\dot{v}}\qquad\text{ and }\qquad \frac{d}{ds}=\nabla_{\mathfrak{v}}.\]
Note furthermore that $[\mathfrak{v},\dot{v}]=0$, as for any function $f\in C^\infty(V)$ holds
\[\nabla_{\mathfrak{v}}\nabla_{\dot{v}}f =\frac{d}{ds}\frac{d}{dt}f(v(s,t))=\frac{d}{dt}\frac{d}{ds}f(v(s,t))=\nabla_{\dot{v}}\nabla_{\mathfrak{v}}f.\]
Additionally, we have $\int^1_0\frac{d}{dt}\lambda(\mathfrak{v}{\scriptstyle(t)})dt=0$, as $\mathfrak{v}$ is 1-periodic. This stated, we calculate the first variation of $\mathcal{A}^H$ in the direction of $(\mathfrak{v},\hat{\eta})$ as
\begin{align*}
 \nabla_{(\mathfrak{v},\hat{\eta})}\mathcal{A}^H(v,\eta)&=\left.\frac{d}{ds}\mathcal{A}^H(v_s,\eta_s)\right|_{s=0}\\
 &=\int^1_0\frac{d}{ds}\lambda(\dot{v})-\hat{\eta} H(v)-\eta dH(\mathfrak{v})\,dt\\
 &=\int^1_0d\lambda(\mathfrak{v},\dot{v})+\frac{d}{dt}\lambda(\mathfrak{v})+\lambda([\mathfrak{v},\dot{v}])-\hat{\eta} H(v)-d\lambda(\mathfrak{v},\eta X_H)\,dt\\
 &=\int^1_0d\lambda(\mathfrak{v},\dot{v}-\eta X_H)-\hat{\eta}H(v)\,dt\\
 &=\int_0^1 \omega(\mathfrak{v},J(-J)(\dot{v}-\eta X_H)\, dt - \hat{\eta}\cdot\int_0^1 H(v)\,dt\\
 &=g\left(\begin{pmatrix}\mathfrak{v}\\\hat{\eta}\end{pmatrix}\,,\,\begin{pmatrix}-J(\dot{v}-\eta X_H)\\-{\textstyle\int^1_0} H(v)\, dt\end{pmatrix}\right).
\end{align*}
The gradient of $\mathcal{A}^H$ with respect to $g$ is hence given by
\begin{equation}\label{gradienta}
 \nabla \mathcal{A}^H=-\begin{pmatrix} J(\dot{v}-\eta X_H)\Big. \\\Big. \int^1_0H(v)dt\end{pmatrix}.
\end{equation}
This shows, that Definition \ref{gradtraj} really defines $\nabla \mathcal{A}^H$-gradient trajectories. As stated before, the critical points $(v,\eta)$ of $\mathcal{A}^H$ are solutions of the equations
\begin{equation*}\begin{aligned}
             &&     0&=\dot{v}-\eta X_H\quad& &\text{ and }&\quad\qquad &0=\textstyle\int^1_0H(v(t))dt\\
             &\Leftrightarrow\qquad& \dot{v}&=\eta R& &\text{ and }& &im(v)\subset\Sigma 
                 \end{aligned}
\end{equation*}
Recall that this implied that $\crita =\mathcal{P}(\alpha)$ and that it followed from Lemma \ref{lem2} that $\mathcal{A}^H(\crita)=spec(\Sigma,\alpha)$. Note in particular, that all critical points are in $\mathscr{L}\times\mathbb{R}$, i.e.\ all critical $v$ are smooth even if we consider $\mathcal{A}^H$ on $W^{k,p}(S^1,V)$.\medskip\\
Next, we are going to calculate the Hessian of $\mathcal{A}^H$ at a critical point $(v,\eta)\in \mathscr{L}\times\mathbb{R}$, where $\nabla\mathcal{A}^H(v,\eta)=0$. Actually, we calculate the self-adjoint operator
\[\nabla^2\mathcal{A}^H_{(v,\eta)} : T_{(v,\eta)}\Big( W^{k,p}(S^1,V)\times\mathbb{R}\Big) \rightarrow T_{(v,\eta)}\Big(W^{k-1,p}(S^1,V)\times\mathbb{R}\Big),\]
which satisfies for $x,y\in T_{(v,\eta)}\big(W^{k,p}(S^1,V)\times\mathbb{R}\big)$ that
\[ Hess\,\mathcal{A}^H(x,y) = \nabla_x\, g(y,\nabla\mathcal{A}^H)= g(y,\nabla^2\mathcal{A}^H x).\pagebreak\]
Note that $g(\nabla_x y, \nabla \mathcal{A}^H)=0$ at a critical point and hence $\nabla^2\mathcal{A}^H_{(v,\eta)}x=\nabla_x\nabla\mathcal{A}^H(v,\eta)$. This implies that $Hess\,\mathcal{A}^H$ and $\nabla^2\mathcal{A}^H$ are symmetric, i.e.\ $g(y,\nabla^2\mathcal{A}^H x)=g(\nabla^2\mathcal{A}^H y, x)$, as shown by the following short calculation
\begin{align*}
 Hess\,\mathcal{A}^H_{(v,\eta)}(x,y)&=&\nabla_x\, g(y,\nabla\mathcal{A}^H_{(v,\eta)})&=&\nabla_x\nabla_y\mathcal{A}^H_{(v,\eta)}&=&\\
 &=&\nabla_y\nabla_x\mathcal{A}^H_{(v,\eta)}+\nabla_{[x,y]}\mathcal{A}^H_{(v,\eta)}&=&\nabla_y\nabla_x\mathcal{A}^H_{(v,\eta)}&=&Hess\,\mathcal{A}^H_{(v,\eta)}(y,x).
\end{align*}
For the calculation of $\nabla^2\mathcal{A}^H$, we consider again a differentiable 1-parameter family $(v_s,\eta_s)\subset W^{k,p}\times\mathbb{R}$ with tangent vectors $(\mathfrak{v},\hat{\eta}):=(\mathfrak{v}_s,\hat{\eta}_s)\in T_{(v_s,\eta_s)}\big(W^{k,p}(S^1,V)\times\mathbb{R}\big)$. Then, we trivialize $(v_s,\eta_s)^\ast T\big(W^{k,p}(S^1,V)\times\mathbb{R}\big)$ over $\mathbb{R}\times S^1$, which allows us to calculate
\[{\textstyle\frac{d}{ds}}\nabla \mathcal{A}^H(v_s,\eta_s)\big|_{s=0}=\nabla_{\mathfrak{v},\hat{\eta}}\nabla \mathcal{A}^H(v_s,\eta_s).\]
In the trivialization we differentiate $\nabla\mathcal{A}^H$ at any point $(v,\eta)$, even non-critical ones!  We use this result later in Section \ref{sec2.5}, where the derivative of $\nabla \mathcal{A}^H$ along $(v_s,\eta_s)$ appears in the linearization of the Rabinowitz-Floer equation (\ref{eq3}). Using (\ref{gradienta}), we calculate
\begin{align*}
\nabla_{(\mathfrak{v},\hat{\eta})}\nabla\mathcal{A}^H(v,\eta)&=\nabla_{(\mathfrak{v},\hat{\eta})}\left(\begin{smallmatrix}-J(\dot{v}-\eta X_H)\\-{\int^1_0} H(v)\, dt\end{smallmatrix}\right)\\
&=\begin{pmatrix}-\nabla_{\mathfrak{v}}\big(J(\dot{v}-\eta X_H)\big)-\hat{\eta}(\partial_n J)(\dot{v}-\eta X_H)+J\hat{\eta} X_H\\-\textstyle{\int^1_0}dH(\mathfrak{v})\,dt\end{pmatrix}\\
 \nabla_{\mathfrak{v}}\big(J(\dot{v}-\eta X_H)\big)&=\big(\nabla_{\mathfrak{v}} J\big)(\dot{v}-\eta X_H)-J(\nabla_{\mathfrak{v}}\dot{v}-\nabla_{\mathfrak{v}}\eta X_H)\\
 &=\big(\nabla_{\mathfrak{v}} J\big)(\dot{v}-\eta X_H)+J\big(\nabla_{\dot{v}}\mathfrak{v} + [\mathfrak{v},\dot{v}]-\nabla_{\eta X_H}\mathfrak{v}-[\mathfrak{v},\eta X_H]\big)\\
 &=\big(\nabla_{\mathfrak{v}} J\big)(\dot{v}-\eta X_H)+J\nabla_{(\dot{v}-\eta X_H)}\mathfrak{v}-J[\mathfrak{v},\eta X_H],
 \end{align*}
where we used again $[\mathfrak{v},\dot{v}]=0$. The derivative of $\nabla\mathcal{A}^H$ along $(v_s,\eta_s)$ is hence given by
\begin{equation}\label{secvar}
 \begin{pmatrix}
  -\Big(\nabla_{\mathfrak{v}}J+\hat{\eta}(\partial_n J)\Big)(\dot{v}-\eta X_H)- J\Big(\nabla_{(\dot{v}-\eta X_H)}\mathfrak{v}-[\mathfrak{v}, \eta X_H]-\hat{\eta} X_H\Big)\\\Big.
  -{\textstyle \int_0^1 dH(\mathfrak{v})dt}
 \end{pmatrix}.
\end{equation}
At a critical point, where $\dot{v}-\eta X_H=0$, it takes the form
\begin{equation}\label{nabla2}
 \nabla^2\mathcal{A}^H_{(v,\eta)}(\mathfrak{v},\hat{\eta})= 
 \begin{pmatrix}
 J\big([\mathfrak{v}, \eta X_H]+\hat{\eta} X_H\big)\Big.\\\Big. -\textstyle \int_0^1 dH(\mathfrak{v})dt
 \end{pmatrix}.
\end{equation}
Note that this expression does not depend on the family $(v_s,\eta_s)$, but only on $(\mathfrak{v},\hat{\eta})$ -- its derivative at $s=0$. Let $\phi$ denote the Reeb flow on $\Sigma$. As $X_H$ coincides with the Reeb field on $\Sigma$, we have that the flow $\phi_H$ of $\eta X_H$ is given by $\phi^t_H=\phi^{\eta t}$. At a critical point, we can hence express the term $-[\mathfrak{v},\eta X_H]=[\eta X_H,\mathfrak{v}]$ with the help of $\phi$ as
\[ [\eta X_H,\mathfrak{v}](t_0)=\mathcal{L}_{\eta X_H}\mathfrak{v}(t_0) = \frac{d}{dt}\left(\phi^{\eta t}\right)^\ast\mathfrak{v}(t_0) = \left.\frac{d}{dt}\left(D\phi^{\eta t}\right)^{-1}\mathfrak{v}(t_0+t) \right|_{t=0}.\]
Note that for $\eta=0$, this becomes the ordinary derivative $\frac{d}{dt}$ on $T_{v(0)}V$. Having this in mind, we could write symbolically $\nabla^2\mathcal{A}^H$ at a critical point as
\[\nabla^2\mathcal{A}^H_{(v,\eta)}(\mathfrak{v},\hat{\eta})= 
 \begin{pmatrix}
 J\big(-\frac{d}{dt}\mathfrak{v}+\hat{\eta} X_H\big)\Big.\\\Big. -\textstyle \int_0^1 dH(\mathfrak{v})dt
 \end{pmatrix}.\pagebreak\]
\begin{lemme} \label{lemker}
 Assume that \emph{(MB)} holds true. Then $\ker \nabla^2\mathcal{A}^H$ consists of pairs $(\mathfrak{v},0)$, where $\mathfrak{v}(t)\in T_{v(t)}\mathcal{N}^\eta$ for all $t$ and $\mathfrak{v}$ is constant  with respect to the Reeb flow $\phi$, i.e.\
 \[\mathfrak{v}(t)= \big(D\phi^{\eta t}\big)\mathfrak{v}(0) \qquad \forall\, t.\]
In particular, $\ker \nabla^2\mathcal{A}^H$ is finite dimensional with $\dim \big(\ker \nabla^2 \mathcal{A}^H_{(v,\eta)}\big)=\dim_{v(0)}\mathcal{N}^\eta$.
\end{lemme}
\begin{proof}
 It follows from (\ref{nabla2}), that $(\mathfrak{v},\hat{\eta})\in\ker\nabla^2\mathcal{A}^H_{(v,\eta)}$ if and only if 
 \[ \text{I.}\quad \left.\frac{d}{dt}\left(D\phi^{\eta t}\right)^{-1}\mathfrak{v}(t_0+t) \right|_{t=0}=\hat{\eta}X_H(v(t_0))\qquad\text{ and }\qquad\text{II.}\quad 0=\int^1_0dH(\mathfrak{v})dt.\]
 The proof has now 3 steps:
 \begin{itemize}
  \item First, we show that $\mathfrak{v}(t)\in T_{v(t)}\Sigma$ for all $t\in S^1$. Note that $H, \,dH$ and $X_H$ are invariant under the flow $\phi^{\eta t}$ of $\eta X_H$. This implies that
  \[dH_{v(t)}\big(\mathfrak{v}(t)\big)=dH_{v(0)}\left(\big(D\phi^{-\eta t}\big)\mathfrak{v}(t)\right)\]
  and hence with I. that
  \begin{align*}
   \frac{d}{dt} dH_{v(t)}\big(\mathfrak{v}(t)\big)\Big|_{t=t_0} &=dH_{v(0)}\left(\frac{d}{dt}\big(D\phi^{-\eta (t+t_0)}\big)\mathfrak{v}(t+t_0)\Big|_{t=0}\right)\\
   &=dH_{v(0)}\Big(D\phi^{-\eta t_0}\big(\hat{\eta}X_H(v(t_0))\big)\Big)=dH_{v(0)}\left(\hat{\eta}X_H(v(0))\right)=0.
  \end{align*}
Thus, we have that $dH(\mathfrak{v}(t))$ is constant. Therefore, equation II. implies that
\[0=\textstyle \int_0^1 dH(\mathfrak{v}(t))dt = \int_0^1 dH(\mathfrak{v}(t_0))dt = dH(\mathfrak{v}(t_0)) \qquad \forall\,t_0\in S^1.\]
As $\Sigma$ is a regular level set of $H$, we have $\ker dH_{v(t)}=T_{v(t)}\Sigma$ and hence that $\mathfrak{v}(t)\in T_{v(t)}\Sigma$ for all $t$.
  \item Next, we show $\hat{\eta}=0$. Assume first that $\eta\neq 0$. We then obtain from equation II. 
\begin{align*}
 0&=\int^1_0 dH(\mathfrak{v})\,dt =\int^1_0 d\lambda(\mathfrak{v},X_H)\,dt\\
 &=\int^1_0 \partial_{\mathfrak{v}} \lambda(X_H)-\partial_{X_H}\lambda(\mathfrak{v})-\lambda\big([\mathfrak{v},X_H]\big)\,dt =\int^1_0 0-\frac{1}{\eta}\cdot\frac{d}{dt}\lambda(\mathfrak{v})+\frac{\hat{\eta}}{\eta}\,dt =\frac{\hat{\eta}}{\eta}.
\end{align*}
Here, we used equation I. and $\frac{d}{dt}=\partial_{\dot{v}}=\eta\partial_{X_H}$ and that $\lambda(\mathfrak{v})$ is 1-periodic and $\partial_{\mathfrak{v}}\lambda(X_H)=0$, as $\lambda(X_H)=1$ is constant throughout $\Sigma$ and $\mathfrak{v}\in T\Sigma$. Hence $\hat{\eta}=0$ if $\eta\neq 0$.\\
Now, if $\eta=0$, then $v$ is constant and equation I. becomes the linear differential equation $\partial_t \mathfrak{v}(t_0)=\hat{\eta} X_H(v(0))$ for a map $\mathfrak{v}: S^1\rightarrow T_{v(0)}V$. The solutions to this problem are of the form $\mathfrak{v}(t)=\mathfrak{v}(0)+t\cdot \hat{\eta}X_H(v(0))$. However, we require $\mathfrak{v}(1)=\mathfrak{v}(0)$ and this holds if and only if $\hat{\eta}=0$, as $X(v(0))\neq 0$.
  \item Finally, we prove the lemma. As $\hat{\eta}=0$, equation I. yields
\begin{align*}
 &&\left.\frac{d}{dt}\left(D\phi^{\eta t}\right)^{-1}\mathfrak{v}(t_0+t) \right|_{t=0} &= 0&&\forall t_0\in \mathbb{R}\\
 &\Bigg.\Leftrightarrow& \left(D\phi^{-\eta t}\right)\mathfrak{v}(t) &= \mathfrak{v}(0)&&\forall t\in \mathbb{R}\\
 &\Leftrightarrow& \mathfrak{v}(t) &= (D\phi^{\eta t})\mathfrak{v}(0)&&\forall t\in \mathbb{R}
\end{align*}
Recall that $\mathfrak{v}$ has to be 1-periodic. However, not every solution of the last equation satisfies this. It holds if and only if $\mathfrak{v}(0)=\mathfrak{v}(1)= D\phi^\eta\mathfrak{v}(0)$, in other words if and only if $\mathfrak{v}(0)\in\ker(D\phi^\eta-id)=T_{v(0)}\mathcal{N}^\eta$, by assumption (MB).\qedhere
 \end{itemize}
\end{proof}
In the following, let $(v,\eta)\in\crita$ and $1<p<\infty$ be fixed and abbreviate
 \[W^{k,p}:=W^{k,p}(v,\eta):= T_{(v,\eta)}\left(W^{k,p}(S^1,V)\times\mathbb{R}\right)\]
 for the $W^{k,p}$-vector fields along $(v,\eta)$. Similarly write $L^p = L^p(v,\eta)$ and $C^\infty := C^\infty (v,\eta)$ for the $L^p$- resp.\ $C^\infty$-vector fields. Let us also abbreviate $K:=K(v,\eta):=\ker\nabla^2\mathcal{A}^H_{(v,\eta)}$. Note that in this terminology the operator $\nabla^2\mathcal{A}^H_{(v,\eta)}$ maps $W^{k+1,p}$ to $W^{k,p}$ as in the term $[X_H,\mathfrak{v}]$ the vector field $\mathfrak{v}$ is differentiated once. Now, for any $k\geq 0$ and $1< p <\infty$, the following two lemmas will show that the cokernel of $\nabla^2\mathcal{A}^H$ equals $K$, thus showing that $\nabla^2\mathcal{A}^H$ is self-adjoint and a Fredholm operator of index 0.
\begin{lemme}\label{lemclosedimage}
Assume that (MB) holds. Then it holds that the image $\nabla^2\mathcal{A}^H\big(W^{k+1,p}\big)$ is closed in $W^{k,p}$ for all $p\geq 1$.
\end{lemme}
\begin{proof}~\\
We will show that $\nabla^2\mathcal{A}^H$ has on its image a continuous right inverse, i.e.\ there exists a continuous operator $U:\nabla^2\mathcal{A}^H(W^{k+1,p})\rightarrow W^{k+1,p}$ such that $\nabla^2\mathcal{A}^H\circ U=Id$. Given such a $U$, we prove the lemma as follows: If $(\mathfrak{w}_n,\xi_n)\subset \nabla^2\mathcal{A}^H(W^{k+1,p})$ is a sequence which converges in $W^{k,p}$ to some $(\mathfrak{w},\xi)$, then by continuity of $U$ and completeness of $W^{k+1,p}$ there exists $(\mathfrak{v},\hat{\eta})$ such that $U(\mathfrak{w}_n,\xi_n)\rightarrow(\mathfrak{v},\hat{\eta})$.
Then
\[(\mathfrak{w},\xi)=\lim_{n\rightarrow\infty} (\mathfrak{w}_n,\xi_n)=\lim_{n\rightarrow\infty} \nabla^2\mathcal{A}^H(U(\mathfrak{w}_n,\xi_n))=\nabla^2{A}^H(\mathfrak{v},\eta),\]
by continuity of $\nabla^2\mathcal{A}^H$, which shows that $\nabla^2\mathcal{A}^H(W^{k+1,p})$ is closed in $W^{k,p}$.\medskip\\
 To construct $U$, note that the flow $\phi^{\eta t}$ of $\eta X_H$ provides a trivialization $\Phi$ of $v^\ast TV$
 \begin{align*}
  \Phi : v^\ast TV\rightarrow [0,1]\times T_{v(0)}\qquad\qquad\big(v(t),\xi(t)\big)\mapsto\left(t,\big( D\phi^{\eta t}\big)^{-1}\xi(t)\right).
 \end{align*}
 We can then express $\nabla^2\mathcal{A}^H$ in this coordinates as
 \begin{equation}\label{AHinTriv}
  \nabla^2\mathcal{A}^H(\mathfrak{v},\hat{\eta})=
  -\begin{pmatrix} I\cdot\left(\frac{d}{dt}\mathfrak{v}-\hat{\eta} X\right)\Big.\\\Big. \textstyle \int_0^1 dH_{v(0)}\big(\mathfrak{v}\big)dt\end{pmatrix},
 \end{equation}
where $X(t)=X_H(v(0))$ is a constant vector and $I(t)=\big(D\phi^{\eta t}\big)^{-1} J_{v(t)} \big(D\phi^{\eta t}\big)$ a bounded path of matrices with $I^2=-Id$. Note that $\mathfrak{v}$ is here a map from $[0,1]$ to $T_{v(t_0)}$. In order to guarantee that $\Phi^{-1}\mathfrak{v}$ is 1-periodic, we have to require that $D\phi^\eta\big(\mathfrak{v}(1)\big)=\mathfrak{v}(0)$.\\
Recall that $D\phi^\eta$ restricted to $T_{v(0)}\mathcal{N}^\eta$ is the identity and let $C$ denote any complement to $T_{v(0)}\mathcal{N}^\eta$ in $T_{v(0)}V$. Then $\big((D\phi^\eta)^{-1}-\mathbbm{1}\big)$ restricted to $C$ is invertible. Let $E$ denote the linear map 
\[E: T_{v(0)}V=T_{v(0)}\mathcal{N}^\eta\oplus C\rightarrow T_{v(0)}\mathcal{N}^\eta\oplus C,\qquad E=0\oplus\pi_C\Big(\big(D\phi^\eta\big)^{-1}-\mathbbm{1}\Big)^{-1}.\]
For $(\mathfrak{w},\xi)=\nabla^2\mathcal{A}^H(\mathfrak{v},\hat{\eta})\in\nabla^2\mathcal{A}^H\big(W^{k+1,p}\big)$, we construct $U(\mathfrak{w},\xi)$ using the ansatz
\[\mathfrak{w}=-I\cdot\left(\frac{d}{dt}\mathfrak{v}-\hat{\eta} X\right) \qquad\text{ and }\qquad  \xi = - \int_0^1 dH_{v(0)}\big(\mathfrak{v}\big)dt.\]
Using the symplectic form $\omega=\omega_{v(0)}$ on $T_{v(0)}V$, we solve the first equation for $\hat{\eta}$ to get
\[\widetilde{\eta} := -\omega\Big(IX,\textstyle\int_0^1 I\mathfrak{w}\, dt\Big).\]
Note that $\widetilde{\eta}=\hat{\eta}$. Hence, we can use the fundamental theorem of calculus to solve the first equation for $\mathfrak{v}$ to
\[ \widetilde{\mathfrak{v}}(t):=\textstyle\int_0^t( I\mathfrak{w}+\widetilde{\eta}X) dt + E\big(\textstyle\int_0^1 (I\mathfrak{w}+\widetilde{\eta}X) dt\big).\]
The $E$-term is needed, as we shall see, to ensure that $D\phi^\eta\big(\widetilde{\mathfrak{v}}(1)\big)=\widetilde{\mathfrak{v}}(0)$. Assuming this, we have that 
\[U: \nabla^2\mathcal{A}^H\big(W^{k+1,p}\big)\rightarrow W^{k+1,p},\qquad (\mathfrak{w},\xi)\mapsto (\,\widetilde{\mathfrak{v}},\widetilde{\eta})\]
is a well-defined linear operator. The continuity of $U$ is obvious. To see that $U$ is a right inverse of $\nabla^2\mathcal{A}^H$, i.e.\ $\nabla^2\mathcal{A}^H\circ U=Id$ on $im\big(\nabla^2\mathcal{A}^H\big)$, let $(\mathfrak{w},\xi)=\nabla^2\mathcal{A}^H(\mathfrak{v},\hat{\eta})$ and calculate
\begin{align*}
 \widetilde{\eta} = -\omega\Big(IX,\int_0^1 I \mathfrak{w}\, dt\Big) &= -\omega\Big(IX,\int_0^1 I\circ -I\big(\textstyle\frac{d}{dt}\mathfrak{v}-\hat{\eta}X\big)\, dt\Big)\\
 &= -\omega\Big(IX,\int_0^1 \textstyle\frac{d}{dt}\mathfrak{v}-\hat{\eta}X dt\Big)\\
 &= -\omega\Big(IX,\mathfrak{v}(1)-\mathfrak{v}(0) - \hat{\eta} X\Big) = \hat{\eta},
\end{align*}
where the last line follows as $D\phi^\eta X=X$ so that the $X$-part of $\mathfrak{v}(1)-\mathfrak{v}(0)$ is zero. Moreover
\begin{align*}
  \widetilde{\mathfrak{v}}(t)&=\int_0^t I\circ -I\big(\textstyle\frac{d}{dt}\mathfrak{v}-\hat{\eta}X\big)+\hat{\eta}X\, dt + E\displaystyle\int_0^1 I\circ -I\big(\textstyle\frac{d}{dt}\mathfrak{v}-\hat{\eta}X\big)+\hat{\eta}X\, dt\\
 &= \int_0^t \textstyle\frac{d}{dt}\mathfrak{v}\,dt +E\displaystyle\int_0^1\textstyle\frac{d}{dt}\mathfrak{v}\,dt\\
  &= \mathfrak{v}(t)-\mathfrak{v}(0) + E\big(\mathfrak{v}(1)-\mathfrak{v}(0)\big)\\
  &= \mathfrak{v}(t)-\mathfrak{v}(0) + E\big((D\phi^1)^{-1}\mathfrak{v}(0)-\mathfrak{v}(0)\big)\\
  &= \mathfrak{v}(t)-\mathfrak{v}(0) + E\big((D\phi^1)^{-1}-\mathbbm{1})\mathfrak{v}(0)\big)\\
  &= \mathfrak{v}(t)-\pi_{T_{v(0)}\mathcal{N}^\eta}\big( \mathfrak{v}(0)\big).
\end{align*}
Here, the last line implies that $(\widetilde{\mathfrak{v}},\widetilde{\eta})$ is in fact $(\mathfrak{v},\hat{\eta})$ minus $\big(\pi_{T_{v(0)}\mathcal{N}^\eta}\big( \mathfrak{v}(0)\big),0\big)$, which is an element in the kernel of $\nabla^2\mathcal{A}^H$. This shows that $\nabla^2\mathcal{A}^H(\widetilde{\mathfrak{v}},\widetilde{\eta})=\nabla^2\mathcal{A}^H(\mathfrak{v},\hat{\eta})=(\mathfrak{w},\xi)$, i.e.\ $\nabla^2\mathcal{A}^H\circ U= Id$. It also shows that $D\phi^\eta\big(\widetilde{\mathfrak{v}}(1)\big)=\widetilde{\mathfrak{v}}(0)$, as $D\phi^\eta|_{T\mathcal{N}^\eta}=Id$ and hence
\[D\phi^\eta\big(\widetilde{\mathfrak{v}}(1)\big)=D\phi^\eta\Big(\mathfrak{v}(1)-\pi_{T_{v(0)}\mathcal{N}^\eta}\big(\mathfrak{v}(0)\big)\Big)=\mathfrak{v}(0)-\pi_{T_{v(0)}\mathcal{N}^\eta}\big(\mathfrak{v}(0)\big)=\widetilde{\mathfrak{v}}(0).\qedhere\]
\end{proof}
\begin{lemme} \label{lemcoker}
 Assume that \emph{(MB)} holds. Then 
 \[W^{k,p}=\ker \nabla^2\mathcal{A}^H\oplus\nabla^2\mathcal{A}^H\big(W^{k+1,p}\big).\]
 In particular $\text{\emph{coker}}\,\nabla^2\mathcal{A}^H=\ker\nabla^2\mathcal{A}^H$ and $\nabla^2\mathcal{A}^H$ is a Fredholm operator of index 0.
\end{lemme}
\begin{proof}
  Let $\perp$ denote the $L^2$-orthogonal complement with respect to the metric $g$. As we integrate over the compact 1-dimensional domain $S^1$, it follows from Rellich's Theorem that $W^{k,p}$ embeds into $L^2$ for $k\geq 1$ and all $p\in(1,\infty)$ or for $k=0$ and $p\geq 2$. Thus it makes sense for these $k$ and $p$ to consider the closed space
  \[\big(\nabla^2\mathcal{A}^H(W^{k+1,p})\big)^\perp\subset L^2.\]
  \textit{\underline{Claim} :} $\big(\nabla^2\mathcal{A}^H(W^{k+1,p})\big)^\perp=\ker \nabla^2\mathcal{A}^H$.\medskip\\
  We prove this claim below, but first let us show how the claim implies the lemma.\\
  For $k\geq 1$ or $k=1$ and $p\geq 2$, we argue as follows. As $\ker\nabla^2\mathcal{A}^H$ consists of smooth elements, we find that
  \[\ker\nabla^2\mathcal{A}^H=\ker\nabla^2\mathcal{A}^H\cap W^{k,p}=\big(\nabla^2\mathcal{A}^H(W^{k+1,p})\big)^\perp\cap W^{k,p}.\]
  As $\nabla^2\mathcal{A}^H(W^{k+1,p})$ is closed in $W^{k,p}$, we hence have that
  \[W^{k,p}=\ker\nabla^2\mathcal{A}^H\oplus\nabla^2\mathcal{A}^H(W^{k+1,p}).\]
  For $k=0$ and $1<p\leq 2$, we use the dual space $(L^p)^\ast=L^q$, where $1/p+1/q=1$ so that $q\geq 2$ and $L^q$ embedds into $L^2$. Then we consider the annihilator $\big(\nabla^2\mathcal{A}^H(W^{1,p})\big)^0\subset L^q$ and find again (by repeating the proof of the claim) that
  \[\big(\nabla^2\mathcal{A}^H(W^{1,p})\big)^0=\ker\nabla^2\mathcal{A}^H\]
  and hence $L^p=\ker \nabla^2\mathcal{A}^H\oplus\nabla^2\mathcal{A}^H(W^{1,p})$.\medskip\\
 \textit{Proof of the claim:}\\
 Basically, the statement follows from elliptic regularity as $\nabla^2\mathcal{A}^H$ is an elliptic operator of order 1. However, we give here for convenience an explicit proof. Recall from the proof of Lemma \ref{lemclosedimage} that the flow $\phi^{\eta t}$ of $\eta X_H$ provides a trivialization of $v^\ast TV\cong [0,1]\times T_{v(0)}V$ under which elements of $W^{k,p}$ become $W^{k,p}$-maps $\mathfrak{v}:[0,1]\rightarrow T_{v(0)}V$ such that $D\phi^\eta\big(\mathfrak{v}(1)\big)=\mathfrak{v}(0)$. We expressed $\nabla^2\mathcal{A}^H$ in this trivialization by (\ref{AHinTriv}), where $I$ is a $t$-dependent matrix such that $I^2=-Id$. As $\omega$ is invariant under the flow $\phi^{\eta t}$, we find that the metric $g$ on $[0,1]\times T_{v(0)}$ is just
 \[g\big((\mathfrak{v}_0,\hat{\eta}_0),(\mathfrak{v}_1,\hat{\eta}_1)\big) = \int_0^1\omega_{v(0)}\big(\mathfrak{v}_0, I\mathfrak{v}_1\big)dt + \hat{\eta}_0\cdot\hat{\eta}_1.\]
 In this framework, we can now describe $\big(\nabla^2\mathcal{A}^H(W^{k+1,p})\big)^\perp$. It is obvious that \linebreak $\ker\nabla^2\mathcal{A}^H\subset \big(\nabla^2\mathcal{A}^H(W^{k+1,p})\big)^\perp$. For the opposite inclusion let $(\mathfrak{v},\hat{\eta})$ be any element in the complement. We then have for all $(\mathfrak{w},\xi)\in W^{k+1,p}$ with $D\phi^\eta\big(\mathfrak{w}(1)\big)=\mathfrak{w}(0)$ that
 \[0=g\left(\begin{pmatrix}\mathfrak{v}\\\hat{\eta}\end{pmatrix},\nabla^2\mathcal{A}^H\begin{pmatrix}\mathfrak{w}\\\xi\end{pmatrix}\right) \overset{(\ref{AHinTriv})}{=} \int_0^1\omega\big(\mathfrak{v},\textstyle \frac{d}{dt}\mathfrak{w}-\xi\cdot X\big)dt+\hat{\eta}\cdot\displaystyle\int_0^1dH(\mathfrak{w})dt.\]
 Considering in particular $\xi=0$ and $\xi=1$, we find that this holds if and only if
 \[0 =\int_0^1\omega\big(\mathfrak{v},\textstyle\frac{d}{dt}\mathfrak{w}\big)dt + \hat{\eta}\cdot\displaystyle\int_0^1\omega\big(\mathfrak{w},X\big)dt\qquad\text{ and }\qquad 0=\int_0^1\omega\big(\mathfrak{v},X\big)dt.\]
 As the Liouville vector field $Y_\lambda$ is also preserved under the flow of $\eta X_H$ (at least on $\Sigma$), we find that $\mathfrak{w}_0(t):=Y_\lambda(v_0)$ satisfies $D\phi^\eta\big(\mathfrak{w}_0(1)\big)=\mathfrak{w}_0(0)$. For this map we have $\frac{d}{dt}\mathfrak{w}_0=0$ and hence we conclude that \[0=\hat{\eta}\cdot\int_0^1\omega(\mathfrak{w}_0,X)dt=\hat{\eta}\cdot\int_0^1\omega(Y_\lambda,X)dt=\hat{\eta}.\]
 Thus, we find that $\mathfrak{v}$ has to satisfy the following two equations:
 \[I.\quad 0=\int_0^1\omega(\mathfrak{v},X)dt\qquad\text{ and }\qquad II.\quad 0=\int_0^1\omega\big(\mathfrak{v},\textstyle\frac{d}{dt}\mathfrak{w}\big)dt,\]
 where $\mathfrak{w}$ may still be any $W^{k+1,p}$-map $\mathfrak{w}:[0,1]\rightarrow T_{v(0)}V$ with $D\phi^{\eta t}\big(\mathfrak{w}(1)\big)=\mathfrak{w}(0)$. The second condition on $\mathfrak{w}$ is automatically satisfied if $\text{supp } \mathfrak{w}\subset(0,1)$. In particular, we can take $\mathfrak{w}=\rho_\delta\ast\mathfrak{u}$, where $\mathfrak{u}$ is any test-function with support in $(\delta,1-\delta)$ and $\rho_\delta$ is a smooth bump function with $\text{supp}\,\rho_\delta\subset(-\delta,\delta),\; \int_{-\infty}^\infty\rho_\delta\,dt=1$ and $\rho_\delta(t)=\rho_\delta(-t)$ (see Appendix \ref{appB}). Then we have from II. and Corollary \ref{corconvolution} that
 \begin{align*}
  0=\int_0^1\omega\big(\mathfrak{v},{\textstyle\frac{d}{dt}}(\rho_\delta\ast\mathfrak{u})\big)dt = \int_0^1\omega\big(\mathfrak{v},({\textstyle\frac{d}{dt}\rho_\delta})\ast\mathfrak{u}\big)dt&=\int_0^1\omega\big(({\textstyle\frac{d}{dt}\rho_\delta})\ast\mathfrak{v},\mathfrak{u}\big)dt\\
  &= \int_0^1\omega\big({\textstyle\frac{d}{dt}}(\rho_\delta\ast\mathfrak{v}),\mathfrak{u}\big)dt.
 \end{align*}
 As this equation holds for any test function $\mathfrak{u}$ with $\text{supp }\mathfrak{u}\subset(\delta,1-\delta)$, we infer that $\rho_\delta\ast\mathfrak{v}$ is constant on $(\delta,1-\delta)$. As $\rho_\delta\ast\mathfrak{v}\rightarrow\mathfrak{v}$ in $L^p$ (see Lemma \ref{lemconvconv}), we conclude that $\mathfrak{v}$ is also constant and hence smooth. From equation I., we then conclude that the $Y_\lambda$-component of $\mathfrak{v}$ has to be zero. Now as $D\phi^\eta\big(\mathfrak{v}(0)\big)=D\phi^\eta\big(\mathfrak{v}(1)\big)=\mathfrak{v}(0)$, we know that $\mathfrak{v}(t)=\mathfrak{v}(0)\in\ker (D\phi^\eta-Id)$ and this implies that $(\mathfrak{v},\xi)=(\mathfrak{v},0)$ lies in $\ker\nabla^2\mathcal{A}^H$.
 \qedhere
\end{proof}

\begin{theo} \label{theo17}
 If (MB) is satisfied, then the following equivalent statements are true:
 \begin{itemize}
  \item The spectrum $spec(\Sigma,\alpha)=\mathcal{A}^H\big(\mathcal{P}(\alpha)\big)$ is closed and discrete.
  \item $\mathcal{P}(\alpha)=\crita$ is a submanifold of $\mathscr{L}\times\mathbb{R}$ with disjoint components $\mathcal{N}^\eta$, i.e.\ $\mathcal{A}^H$ is a Morse-Bott functional.
  \item For any real numbers $-\infty<a<b<\infty$ there are only finitely many $\eta\in[a,b]$ such that $\mathcal{N}^\eta\neq\emptyset$.
 \end{itemize}
\end{theo}
\pagebreak
\begin{proof}
The equivalence of the 3 statements is obvious. We therefore only show the first one. That $spec(\Sigma,\alpha)$ is closed can be seen as follows: Suppose that $(\eta_n)\subset spec(\Sigma,\lambda)$ is a sequence with $\lim \eta_n=\eta$. This means that there exists a sequence of Reeb trajectories $(v_n)$ whose periods are $\eta_n$. As $\Sigma$ is compact, we may assume by the Arzela-Ascoli theorem that $v_n$ converges uniformly to a Reeb trajectory $v$, whose period has to be $\eta$. Hence $\eta\in\text{spec}(\Sigma,\lambda)$.\\
The proof of $\text{spec}(\Sigma,\lambda)$ being discrete is more involved. In what follows, it suffices to restrict to the Hilbert spaces $H^k=W^{k,2}$. It follows from Lemma \ref{lemker} and \ref{lemcoker} that $\nabla^2\mathcal{A}^H : H^1\rightarrow L^2$ is a Fredholm operator of index zero. Let again $K=\ker\nabla^2\mathcal{A}^H=\text{coker}\,\nabla^2\mathcal{A}^H$. We abbreviate for a moment the orthogonal complements of $K$ in $H^1$ resp.\ $L^2$ by $S:=K^{\perp_{H^1}}\subset H^1$ and $R:=K^{\perp_{L^2}}\subset L^2$. The restriction $\nabla^2\mathcal{A}^H: S\rightarrow R$ is bijective and due to the Open Mapping Theorem therefore a (continuous) isomorphism. This implies in particular that $\nabla^2\mathcal{A}^H$ is bounded from below, i.e.\ there exists a constant $C>0$, depending continuously on $(v,\eta)$, such that
\[||\nabla^2\mathcal{A}^H(\mathfrak{v},\hat{\eta})||^2_{L^2}\geq C\cdot||(\mathfrak{v},\hat{\eta})||^2_{H^1},\]
for all $(\mathfrak{v},\hat{\eta})\in S$. If we denote by $\pi:H^1\rightarrow S$ the orthogonal projection onto $S$, we may rewrite this more elegantly as
\[||\nabla^2\mathcal{A}^H(\mathfrak{v},\hat{\eta})||^2_{L^2}\geq C\cdot||\pi(\mathfrak{v},\hat{\eta})||^2_{H^1}\]
for all $(\mathfrak{v},\hat{\eta})\in H^1$. As $C$ depends continuously on $(v,\eta)$, it can be chosen globally for all $(v,\eta)\in\mathcal{N}^\eta$, as $\mathcal{N}^\eta\subset\Sigma$ is closed and hence compact. Now consider the function
\[ e: H^k(S^1,V)\times\mathbb{R} \rightarrow \mathbb{R},\qquad e(x)=||\nabla\mathcal{A}^H(x)||^2_{L^2}\]
on the Hilbert manifold $H^k(S^1,V)\times\mathbb{R}$ of pairs $(v,\eta)$ consisting of $H^k$-loops $v$ and real numbers $\eta$.
We find that $e(x)=0$ if and only if $x\in\crita$. Moreover, for $x\in\crita$ and $X,Y\in T_x\big(H^k(S^1,V)\times\mathbb{R}\big)$, one easily calculates that
\[e(x)=0,\qquad D e(x)=0\qquad\text{ and }\qquad D^2 e(x)[X,Y]=2g\big(\nabla^2\mathcal{A}^H(x)X,\nabla^2\mathcal{A}^H(x)Y\big).\]
Let $(v_a,\eta_a)\subset H^k(S^1,V)\times\mathbb{R}$ be a $k$-times differentiable family in $a$ such that $(v_0,\eta_0)\in\mathcal{N}^{\eta_0}\subset\crita$ and $\frac{d}{da}(v_a,\eta_a)|_{a=0}=(\mathfrak{v},\hat{\eta})$. Then, the Taylor formula at $(v_0,\eta_0)$ yields
\[e(v_a,\eta_a)=||\nabla^2\mathcal{A}^H_{(v_0,\eta_0)}(\mathfrak{v},\hat{\eta})||^2_{L^2}\cdot a^2 + O(a^3).\]
Using the above estimate, we find constants $c,d>0$ such that at least for small $a$ we have
\begin{equation} \label{e-norm}
 e(v_a,\eta_a)\geq c\cdot a^2\left(||\pi(\mathfrak{v},\hat{\eta})||^2_{L^2}-d\cdot a\right).
\end{equation}
Note that $c$ and $d$ depend again continuously on $(v,\eta)$ and may therefore be chosen globally on $\mathcal{N}^\eta$. If $(\mathfrak{v},\hat{\eta})\not\in\ker\nabla^2\mathcal{A}^H$, it follows that $e(v_a,\eta_a)>0$ near (not at) $a=0$ and hence that $(v_0,\eta_0)$ is the only critical point of $\mathcal{A}^H$ on $(v_a,\eta_a)$ near $a=0$.\\
We recall that $\mathcal{N}^\eta\subset H^k(S^1,V)\times\mathbb{R}$ is assumed to be a submanifold. Note that $H^k(v,\eta)$ and $K(v,\eta)$ are the tangent spaces to $H^k(S^1,V)\times\mathbb{R}$ resp.\ $\mathcal{N}^\eta$ at points $(v,\eta)\in \mathcal{N}^\eta\subset H^k(S^1,V)\times\mathbb{R}$. According to Theorem 1.3.5 in \cite{kling}, there exists in Hilbert manifolds around each point $(v,\eta)$ in a submanifold $\mathcal{N}^\eta$ a submanifold chart, i.e.\ there exists a closed linear subspace $E\subset H^k(v,\eta)$ such that $H^k(v,\eta)=K(v,\eta)\oplus E$ and open neighborhoods $U\subset H^k(S^1,V)$ around $(v,\eta)$, $V'\subset K(v,\eta), V''\subset E$ each around 0 together with a diffeomorphism
\[\phi: V'\times V''\rightarrow U\qquad\text{ with }\qquad \phi(V'\times\{0\})=\mathcal{N}^\eta\cap U.\]
Fix $y\in V'$. Then any $(\mathfrak{v},\hat{\eta})\in E$ with $||(\mathfrak{v},\hat{\eta})||^2_{L^2}=1$ yields a well-defined path
\[(v_a,\eta_a):=\phi\big(y+a\cdot(\mathfrak{v},\hat{\eta})\big)\]
at least for $|a|<\delta_E$, with $\delta_E$ so small that $B_{\delta_E}(0)\subset V''$. Due to the following Lemma \ref{lem3}, there exists a constant $k>0$, such that $||\pi(\mathfrak{v},\hat{\eta})||^2_{L^2}>k$ for all $(\mathfrak{v},\hat{\eta})\in E$. Set $\veps_E=\min\{\delta_E,k/d\}$, with $d$ being the constant from (\ref{e-norm}). Then we find that the only critical point on $(v_a,\eta_a)$ for $|a|<\veps_E$ is $\phi(y)\in\mathcal{N}^\eta$ at $a=0$. Hence we see for the open set $\phi\big(V'\times B_{\veps_E}(0)\big)\subset H^k(S^1,V)\times\mathbb{R}$ that 
\[\phi\big(V'\times B_{\veps_E}(0)\big)\cap\crita =\mathcal{N}^\eta\cap U.\]
By covering $\mathcal{N}^\eta$ with a finite number of charts $\phi$, we obtain that $\mathcal{N}^\eta\subset\crita$ is isolated.
\end{proof}
\begin{lemme}\label{lem3}
 Let $H$ be a Hilbert space, $K\subset H$ a finite dimensional subspace and $E\subset H$ a closed subspace such that $H=K\oplus E$, i.e.\ E is any closed complement of $K$. Denote by $\pi: H\rightarrow K^\perp$ the orthogonal projection onto the orthogonal complement of $K$. Then there exists a constant $k>0$, depending on $E$, such that for all $x\in E$ holds 
 \[||\pi(x)||^2\geq k\cdot||x||^2.\]
\end{lemme}
\begin{proof}
 Obviously, it suffice to show the claim for all $x\in E$ with $||x||^2=1$. Assume the contrary. Then there exists a sequence $(x_n)\subset E,\;||x_n||^2=1$ with $\lim ||\pi(x_n)||^2=0$. Note that
 \[||x_n-\pi(x_n)||^2=||x_n||^2-||\pi(x_n)||^2=1-||\pi(x_n)||^2\leq 1,\]
 as $\pi$ is an orthogonal projection. Hence, we see that $x_n-\pi(x_n)\in B_1(0)\subset K$ lies in the unit ball of $K$, which is compact, as $K$ has finite dimension. By considering a subsequence, still denoted by $x_n$, we may assume that $x_n-\pi(x_n)$ converges in $K$ to $y\in K$. Now we have
 \[||x_n-y||\leq ||x_n-\pi(x_n)-y||+||\pi(x_n)||\rightarrow 0.\]
 In other words, $x_n$ converges in $H$ to $y\in K$. As $E$ is closed, this means that $y\in E$. Hence $y\in E\cap K=\{0\}$ and therefore $y=0$, contradicting $||y||=\lim ||x_n||=1$. Thus, the assumption is false and the lemma follows.
\end{proof}

\subsection{Asymptotic estimates}\label{2.4}
In this section, we show the following theorem, whose most important statement is that convergent $\mathcal{A}^H$-gradient flow lines always converge exponentially.
\begin{theo}\label{asympequiv} Let $(v,\eta): \mathbb{R}\times S^1\rightarrow V\times\mathbb{R}$ be a solution of the Rabinowitz-Floer equation (\ref{eq3}). Then the following statements are equivalent:
 \begin{enumerate}
  \item $E(v,\eta)<\infty$ and $(v,\eta)$ stays in a compact region in $V\times\mathbb{R}$.
  \item $|\partial_s(v,\eta)(s,t)|\rightarrow 0$ and $dist.\big((v,\eta)(s,t),\mathcal{N}^{\eta^\pm}\big)\rightarrow 0$ for some $\eta^\pm\in spec(\Sigma,\alpha)$, where both limits are uniform in $t$ for $s\rightarrow \pm\infty$.
  \item There exist constants $\delta,c>0$ and $v^\pm\in\mathcal{N}^{\eta^\pm}$ such that $|\partial_s (v,\eta)(s,t)|\leq c\cdot e^{-\delta|s|}$ and $dist.\left(\left(\begin{smallmatrix}v(t)\\\eta\end{smallmatrix}\right)(s),\left(\begin{smallmatrix}v^\pm(t)\\\eta^\pm\end{smallmatrix}\right)\right)\leq c\cdot e^{-\delta|s|}$ for all $(s,t)\in\mathbb{R}\times S^1$.
 \end{enumerate}
 \end{theo}
\begin{rem}~
\begin{itemize}
 \item The term $E(v,\eta)$ is the energy of $(v,\eta)$ as defined in Definition \ref{defnenergy} by
 \[
 E(v,\eta)=\int_{\mathbb{R}\times S^1} \left|\partial_s \left(\begin{smallmatrix} v\\\eta\end{smallmatrix}\right)(s,t)\right|^2 dt\,ds=\int^\infty_{-\infty}||\nabla\mathcal{A}^H(v,\eta)||^2ds=\int^\infty_{-\infty} \frac{d}{ds} \mathcal{A}^H(v,\eta)\,ds.
\]
 Here, we denote for $(Y,n)\in L^2(S^1,TV)\times L^2(\mathbb{R})$ by $|\left(\begin{smallmatrix}Y\\n\end{smallmatrix}\right)|^2$ the pointwise norm given by $\omega(Y,JY)+n^2$, while $||\left(\begin{smallmatrix}Y\\n\end{smallmatrix}\right)||^2$ is either the $L^2$-norm $\int_0^1 |Y|^2dt + n^2$ over $S^1$ or by abuse of notation the $L^2$-norm $\int_{-\infty}^{+\infty}\big(\int_0^1 |Y|^2dt+ n^2\big)\,ds $ over $\mathbb{R}\times S^1$.
 \item It follows from 3.\ using local coordinates that we have for any convergent $\mathcal{A}^H$-gradient trajectory $(v,\eta)$ and any $p>1$ 
 \begin{align*}
  \big(v(s,t),\eta(s)\big)-\big(v^\pm(t),\eta^\pm\big)&\in W^{1,p}(\mathbb{R};e^{r|s|}dt\,ds)\\
  \partial_s(v,\eta)&\in L^p(\mathbb{R};e^{r|s|}dt\,ds)
 \end{align*}
for $r$ small enough. For the precise definition of these weighted Sobolev spaces, see Definition \ref{defnWeighSob} below.
\end{itemize}
\end{rem}
The proof of Theorem \ref{asympequiv} is given in several steps. It is obvious that 3. implies 1. and the verification is left to the reader. The proof that 1. implies 2. is given in Proposition \ref{asympconv}. It uses many ideas from Salamon, \cite{Sal}. That 2. implies 3. is proved in Proposition \ref{asympexpconv} and is inspired
by a similar proof due to Bourgeois and Oancea in \cite{BourOan1}. 
\begin{prop}\label{asympconv}
 Let $(v,\eta)$ be an $\mathcal{A}^H$-gradient trajectory with $E(v,\eta)<\infty$ and $(v,\eta)$ staying in a compact region $W\subset V\times\mathbb{R}$. Then, $\displaystyle\lim_{s\rightarrow \pm\infty}\partial_s(v,\eta)=0$ and there exists $\eta^\pm\in spec(\Sigma,\alpha)$ with $dist.\Big((v,\eta)(s),\mathcal{N}^{\eta^\pm}\Big)\rightarrow 0$ and all limits are uniform in $t$.
\end{prop}
\pagebreak
\begin{proof} We only prove the case $s\rightarrow+\infty$, as $s\rightarrow -\infty$ is completely analogue.
\begin{enumerate}
 \item First, we show that $E(v,\eta)<\infty$ and $(v,\eta)$ staying in a compact set implies that $\displaystyle \lim_{s\rightarrow \infty}\partial_s(v,\eta)=0$ uniformly in $t$. The proof relies on the following a priori estimate:
 \begin{equation}\label{aprioriest}
 \begin{aligned}
  && \int_{B_r(s,t)} \left|\partial_s \left(\begin{smallmatrix} \textstyle v\\\big.\textstyle\eta\end{smallmatrix}\right)\right|^2&<\delta \quad\forall t\in S^1\\
  &\Rightarrow&  \exists\, t^\ast\in S^1 :\forall t\in S^1:\;\left|\partial_s \left(\begin{smallmatrix} \textstyle v\\\big.\textstyle\eta\end{smallmatrix}\right)(s,t)\right|^2&\leq \frac{Ar^2}{2}+\frac{8}{\pi r^2}\int_{\textstyle B_r(s,t^\ast)} \left|\partial_s \left(\begin{smallmatrix} \textstyle v\\\big.\textstyle\eta\end{smallmatrix}\right)\right|^2
  \end{aligned}
 \end{equation}
 for solutions of the Rabinowitz-Floer equation (\ref{eq3}) which stay in a compact region $W\subset V\times\mathbb{R}$. Here, $A>0$ and $\delta>0$ are constants depending on $W$, $\omega$, $J$ and $H$, but not on $(v,\eta)$. This estimate is proven in Lemma \ref{lemaprioriest}  and \ref{eq3meanvalue}. Assuming (\ref{aprioriest}), we show the uniform convergence as follows: As $E(v,\eta)$ is finite and the energy density $\left|\partial_s \left(\begin{smallmatrix} v\\\eta\end{smallmatrix}\right)\right|^2$ always non-negative, we can choose for any $\veps$ with $0<\veps<\sqrt{\delta}$ an $s_0>0$ so large such that
 \[\int_{s_0}^\infty \int_0^1 \left|\partial_s \left(\begin{smallmatrix} \textstyle v\\\big.\textstyle\eta\end{smallmatrix}\right)\right|^2\; dt\,ds\leq \veps^2.\]
 Then we may apply (\ref{aprioriest}) with $r=\sqrt{\veps}$ to obtain $|\partial_s \left(\begin{smallmatrix} v\\\eta\end{smallmatrix}\right)(s,t)|^2\leq \left(A/2+8/\pi\right)\veps$ for $s\geq s_0+\sqrt{\veps}$. This shows that $\partial_s\left(\begin{smallmatrix}v\\\eta\end{smallmatrix}\right)$ converges to zero uniformly as $s\rightarrow \infty$.
 \item It remains to show that $(v,\eta)(s)$ lies uniformly arbitrarily close to $\crita$ as $s\rightarrow\infty$. Recall that
 \[\partial_s \begin{pmatrix} v\\\eta\end{pmatrix}(s,t)=-\begin{pmatrix}J(\partial_t v - \eta X_H\\\int_0^1 H(v)dt\end{pmatrix}.\]
 As $\partial_s \left(\begin{smallmatrix} v\\\eta\end{smallmatrix}\right)\rightarrow 0$ uniformly, we find in particular that $\partial_t v-\eta X_H(v)$ converges uniformly to zero. As $(v,\eta)$ stays in the compact region $W$, we find that $v$ and $\eta$ stay in compact regions and that $\partial_t v$ is bounded. We may hence apply the Arzela-Ascoli Theorem which shows that we can extract from any sequence $s_k\rightarrow \infty$ a subsequence (still denoted $s_k$), such that $(v,\eta)(s_k)$ converges uniformly to some $(v^+,\eta^+)\in\crita$. All possible values for $\eta^+$ have to lie in $spec(\Sigma,\alpha)=\mathcal{A}^H(\crita)$. As $spec(\Sigma,\alpha)$ is closed and discrete, we conclude that all $\eta(s_k)$ converge to the same $\eta^+$, which shows that $\eta(s)\rightarrow \eta^+$. A similar argument shows that for every sequence $(s_k)$ such that $v(s_k)$ converges holds that the limit is a point on $\mathcal{N}^{\eta^+}$. Using some auxiliary metric on $V$, we find that this implies that $dist.\big(v,\eta)(s,t),\mathcal{N}^{\eta^+}\big)\rightarrow 0$ uniformly in $t$.\qedhere
\end{enumerate}
\end{proof}
To complete the proof of Proposition \ref{asympconv}, it remains to show (\ref{aprioriest}). We will do so by applying the following general a priori estimate (\ref{genapioriest}) to $w=|\partial_s\left(\begin{smallmatrix}v\\ \eta\end{smallmatrix}\right)|^2$. The estimate (\ref{genapioriest}) is a variation on a similar estimate by Dietmar Salamon (see \cite{Sal}, Prop.\ 1.21). The main difference is that in Rabinowitz-Floer theory the Laplacian $\Delta w$ is not bounded from below by a pointwise constraint $-A-Bw^2$, but by $-A-B(w^2+\int_{S^1} w^2)$, which involves the values of $w$ on a whole circle. This is due to the fact that the Rabinowitz-Floer equation is only semi-local. The $t^\ast$-shifted center of integration on the right side of our estimate (\ref{genapioriest}) pays tribute to this.
\begin{lemme} \label{lemaprioriest}
 Assume that $w:\mathbb{R}^2\rightarrow \mathbb{R}$ is a $C^2$-function satisfying
 \begin{itemize}
  \item $w\geq 0$.
  \item $w(s,t)=w(s,t+T)$ for some $T> 0$ and all $(s,t)\in\mathbb{R}^2$.
  \item $\Delta w \geq -A - B\big(w^2+\frac{1}{T}\int_0^T w^2 dt\big)$,\medskip\\ where $\Delta w=\partial_s^2 w + \partial_t^2 w$ is the Laplacian and $A,B\geq 0$ are constants.
  \item $\displaystyle \int_{B_r(0,t)} w \leq \frac{\pi}{32 B}\qquad\forall\, t\in [0,T]$,\medskip\\
  where $B_r(0,t)\subset\mathbb{R}^2$ is the closed ball with radius $r$ around the point $(0,t)$.
 \end{itemize}
 Then there exists $t^\ast\in[0,T]$ such that for all $t$ holds
 \begin{align}\label{genapioriest}
  w(0,t)\leq \frac{Ar^2}{2}+\frac{8}{\pi r^2}\int_{\textstyle B_r(0,t^\ast)} w.
 \end{align}
\end{lemme}
\begin{proof}
 Following \cite{DuffSal}, the proof is devided into five steps.\medskip\\
 \underline{Step 1}: \textit{The lemma holds (even stronger) with $B=0$, i.e.\ if $\Delta w\geq -A$ then}
 \[w(s_0,t_0)\leq \frac{Ar^2}{8}+\frac{1}{\pi r^2}\int_{B_r(s_0,t_0)} w \qquad \forall\, (s_0,t_0)\in\mathbb{R}^2.\]
 This is the mean value inequality for the subharmonic function
 \[\widetilde{w}(s,t)=w(s,t)+A\frac{(s-s_0)^2+(t-t_0)^2}{4}.\]
 For completeness, we give the following proof from \cite{DuffSal}. Assume without loss of generality that $s_0=t_0=0$. Then by the divergence theorem, we have
 \[0\leq\frac{1}{\rho}\int_{B_\rho(0)}\Delta\widetilde{w} =\frac{1}{\rho}\int_{\partial B_\rho(0)}\frac{\partial\widetilde{w}}{\partial\nu} = \int_0^{2\pi}\frac{d}{d\rho}\widetilde{w}(\rho e^{i\theta})d\theta =\frac{d}{d\rho}\left(\frac{1}{\rho}\int_{\partial B_\rho(0)}\widetilde{w}\right),\]
 where $\nu$ is the outer normal vector field. Hence, we have for $0<\rho<r$
 \[\frac{1}{2\pi\rho}\int_{\partial B_\rho(0)}\widetilde{w}\leq \frac{1}{2\pi r}\int_{\partial B_r(0)}\widetilde{w}.\]
 The term on the left converges to $\widetilde{w}(0)$ as $\rho$ tends to zero. Hence
 \[2\pi r\widetilde{w}(0)\leq\int_{\partial B_r(0)}\widetilde{w}.\]
 Integrating this inequality from $0$ to $r$ gives the mean value inequality for $\widetilde{w}$. Hence
 \[ w(0)=\widetilde{w}(0)\leq \frac{1}{\pi r^2}\int_{B_r (0)} \widetilde{w} = \frac{Ar^2}{8}+\frac{1}{\pi r^2}\int_{B_r(0)}w.\]
 \underline{Step 2}: \textit{It suffices to prove the lemma for $r=1$}\medskip\\
 Suppose that $w$ satisfies the assumptions and define $\widetilde{w}$ and $\widetilde{A},\widetilde{B}$ and $\widetilde{T}$ by
 \[\widetilde{w}(s,t):=w(rs,rt),\quad \widetilde{A}:=A\cdot r^2,\quad \widetilde{B}:=B\cdot r^2\quad\text{ and }\quad\widetilde{T}:=\frac{1}{r} T.\]
 Then $\widetilde{w}$ is positive, $\widetilde{T}$-periodic and we have
 \begin{align*}
  &&\Delta\widetilde{w} &= r^2\Delta w\geq -r^2 A -r^2B\cdot\big(w^2+\textstyle\frac{1}{T}\int_0^T w^2 \,dt\big)\\
  &&&=-\widetilde{A}-\widetilde{B}\big(\widetilde{w}^2+\textstyle\frac{1}{\widetilde{T}}\int_0^{\widetilde{T}}\widetilde{w}^2 \,dt\big)\\
  &\text{and}& \int_{B_1(0,t)} \widetilde{w} &= \frac{1}{r^2}\int_{B_r(0,rt)} w \leq \frac{\pi}{32\widetilde{B}} \quad\forall\,t.
 \end{align*}
Hence, assuming the lemma for $r=1$, we obtain
\[w(0,t)=\widetilde{w}(0,{\textstyle\frac{1}{r}}t)\leq \frac{\widetilde{A}}{2}+\frac{8}{\pi}\int_{B_1(0,t^\ast)} \widetilde{w} = \frac{Ar^2}{2}+\frac{8}{\pi r^2}\int_{B_r(0,rt^\ast)} w.\]
\underline{Step 3}: \textit{It suffices to prove the lemma for $B=1$.}\medskip\\
Suppose that $w$ satisfies the assumptions and define $\widetilde{w}$ and $\widetilde{A}$ by
\[\widetilde{w}(s,t):=B\cdot w(s,t),\quad \widetilde{A}=B\cdot A.\]
Then $\widetilde{w}$ is still positive and $T$ periodic. Moreover
\[\Delta\widetilde{w}\geq -\widetilde{A} -\big(\widetilde{w}^2+\textstyle\frac{1}{T}\int_0^T\widetilde{w}^2\,dt\big) \qquad\text{ and }\qquad \int_{B_1(0,t)}\widetilde{w}\leq \frac{\pi}{32}\quad\forall\, t.\]
Hence assuming the lemma for $B=1$, we obtain
\[w(0,t)={\textstyle\frac{1}{B}}\widetilde{w}(0,t)\leq \frac{\frac{1}{B} \widetilde{A}}{2}+\frac{8}{\pi}\int_{B_1(0,t^\ast)}{\textstyle\frac{1}{B}}\widetilde{w} = \frac{A}{2}+\frac{8}{\pi}\int_{B_1(0,t^\ast)} w.\]
\underline{Step 4}: \textit{The Heinz Trick: Assume $B=r=1$ and define $f: [0,1]\times[0,T]\rightarrow\mathbb{R}$ by
\[f(\rho,t)=(1-\rho)^2\cdot\max_{B_\rho(0,t)} w.\]
As $f$ is non-negative, continuous and $f(1,t)=0$, there exists $\rho^\ast\in[0,1)$ and $t^\ast\in[0,T]$ and $c>0$ and $z^\ast\in B_{\rho^\ast}(0,t^\ast)$ such that
\[f(\rho^\ast,t^\ast)=\max_{\rho,\, t} f(\rho,t)=(1-\rho^\ast)^2\cdot c,\qquad c=w(z^\ast)=\max_{B_{\rho^\ast}(0,t^\ast)} w.\]
Write $\veps=\frac{1-\rho^\ast}{2}$. For $0\leq \rho\leq \veps$ then holds that}
\[w(z^\ast)=c\leq\frac{A\rho^2}{8} + 4c^2\rho^2 + \frac{1}{\rho^2\pi}\int_{\textstyle B_1(0,t^\ast)} w.\tag{$\ast$}\]
To see this, note that $\rho^\ast+\veps=\frac{1+\rho^\ast}{2}<1$. Since $B_\veps(z)\subset B_{\rho^\ast+\veps}(0,t)$ for all $z=(s,t)\in[-\rho^\ast,\rho^\ast]\times[0,T]$, it holds for these $z$ that
\begin{align*}
 &&\max_{B_\veps(z)}w&\leq \max_{B_{\rho^\ast+\veps}(0,t)} w = \frac{f(\rho^\ast+\veps,t)}{(1-\rho^\ast-\veps)^2} = \frac{4f(\rho^\ast+\veps,t)}{(1-\rho^\ast)^2} \leq \frac{4f(\rho^\ast,t^\ast)}{(1-\rho^\ast)^2}= 4c.\tag{$\ast\ast$}\\
 &\Rightarrow& \Delta w &\geq - A - w^2 -{\textstyle{\frac{1}{T}}}\int_0^T w^2 dt\geq -A-16c^2-\textstyle \frac{T}{T} 16c^2 = -A-32c^2.
\end{align*}
Now, $(\ast)$ follows from step 1 at $w(z^\ast)=c$ with $r=\rho\leq\veps$ and $A$ replaced by $A+32c^2$. (Note that $B_\rho(z^\ast)\subset B_1(0,t^\ast)$!)\medskip\\
\underline{Step 5}: \textit{The lemma holds for $r=1$ and $B=1$.}\medskip\\
If $c\leq A/8$, then $w(0,t)\leq 4c\leq A/2$ by $(\ast\ast)$ and this implies the lemma. Hence we may assume that $c\geq A/8$. We prove that this implies $4c\cdot \veps^2\leq \frac{1}{2}$. Suppose otherwise that $\veps^2\geq \frac{1}{8c}$. Then in $(\ast)$, we can choose $\rho=\sqrt{\frac{1}{8c}}\leq \veps$ and obtain
\begin{align*}
 && c&\leq \frac{A}{8}\cdot\frac{1}{8c}+4c^2\cdot\frac{1}{8c}+\frac{8c}{\pi}\int_{B_1(0,t^\ast)} w\\
 && &\leq c\cdot\veps^2+\frac{c}{2}+\frac{8c}{\pi}\int_{B_1(0,t^\ast)} w\\
 &\Leftrightarrow & \frac{c\cdot\left(\frac{1}{2}-\veps^2\right)\pi}{8c}&\leq \int_{B_1(0,t^\ast)} w\\
 &\Rightarrow& \frac{\pi}{32} &\leq \int_{B_1(0,t^\ast)} w,
\end{align*}
where we used that $\veps=\frac{1-\rho^\ast}{2}\leq\frac{1}{2}$. But the last inequality is a contradiction to the fourth assumption. Hence $4c\cdot\veps^2\leq\frac{1}{2}$. Now consider $(\ast)$ with $\rho=\veps$
\begin{align*}
 && c&\leq \frac{A\veps^2}{8} + 4c^2\veps^2 + \frac{1}{\veps^2\pi}\int_{B_1(0,t^\ast)} w& &\leq \frac{A}{32} + \frac{c}{2}+\frac{1}{\veps^2\pi}\int_{B_1(0,t^\ast)} w\\
 &\Rightarrow& \frac{c}{2}&\leq \frac{A}{32}\cdot\frac{1}{4\veps^2} + \frac{1}{\veps^2\pi}\int_{B_1(0,t^\ast)} w\\
 &\Rightarrow& 4c\veps^2&\leq \frac{A}{16}+\frac{8}{\pi}\int_{B_1(0,t^\ast)} w& &\leq \frac{A}{2}+\frac{8}{\pi}\int_{B_1(0,t^\ast)} w.
\end{align*}
\[\text{Hence}\qquad w(0,t)=f(0,t)\leq f(\rho^\ast,t^\ast)=(1-\rho^\ast)^2\cdot c = 4\veps^2\cdot c\leq \frac{A}{2}+\frac{8}{\pi}\int_{B_1(0,t^\ast)} w.\qedhere\]
\end{proof}

\begin{lemme} \label{eq3meanvalue}
 Assume that $(v,\eta)\in C^\infty(\mathbb{R}\times S^1,V)\times C^\infty(\mathbb{R},\mathbb{R})$ is a solution of the Rabinowitz-Floer equation (\ref{eq3}) which stays in a compact region $W\subset V\times\mathbb{R}$. Then there exist constants $A,\delta>0$ depending on $W$ and $\omega,H, J$ such that
 \begin{equation*}
 \begin{aligned}
  && \int_{B_r(s,t)} \left|\partial_s \left(\begin{smallmatrix} \textstyle v\\\big.\textstyle\eta\end{smallmatrix}\right)\right|^2&<\delta \quad\forall t\in S^1\\
  &\text{implies }&  \exists\; t^\ast\in S^1\;:\;\left|\partial_s \left(\begin{smallmatrix} \textstyle v\\\big.\textstyle\eta\end{smallmatrix}\right)(s,t)\right|^2&\leq \frac{Ar^2}{2}+\frac{8}{\pi r^2}\int_{\textstyle B_r(s,t^\ast)} \left|\partial_s \left(\begin{smallmatrix} \textstyle v\\\big.\textstyle\eta\end{smallmatrix}\right)\right|^2 \quad\forall t\in S^1.
  \end{aligned}
 \end{equation*}
\end{lemme}
\begin{rem}
 The Maximum Principle, Proposition \ref{maxprinc}, and the a priori estimates on $\eta$, Corollary \ref{cor3.7}, show that $W$ only depends on $E(v,\eta)$ (and on $H_s$ in the homotopy case).
\end{rem}
\begin{proof}
 As announced, we prove the statement by applying Lemma \ref{lemaprioriest} to the function
 \[w: \mathbb{R}\times S^1 \rightarrow \mathbb{R}\qquad w(s,t):=\frac{1}{2}\left|\partial_s\left(\begin{smallmatrix}\textstyle v\\\textstyle\big.\eta\end{smallmatrix}\right)(s,t)\right|^2.\] As $w\geq 0$ and $w$ being 1-periodic is obvious, we only have to show that there exist constants $A,B>0$ such that
 \[\Delta w \geq -A -B(w^2+\textstyle \int_0^1 w^2 dt).\]
 The constant $\delta$ is then given by $\pi/16B$. This estimate is somewhat technical but straightforward. The trick is that we can use the equation $\partial_s v = -J(\partial_t-\eta X_H)$ to rewrite the second order terms in $\Delta w$ as first order derivatives of $J$ and $X_H$. Let us abbreviate $\mathfrak{v}:=\partial_s v,\,\dot{v}:=\partial_t v,\, \dot{\eta}=\partial_s\eta$ and $X:=X_H$. Then, the Rabinowitz-Floer equation (\ref{eq3}) translates to
 \[\mathfrak{v} = -J(\dot{v}-\eta X)\quad\Leftrightarrow\quad \dot{v} =J\mathfrak{v} + \eta X,\qquad \dot{\eta}=\textstyle \int_0^1 H(v) dt.\]
 As $\frac{d}{ds}\frac{d}{dt}f(v)=\frac{d}{dt}\frac{d}{ds}f(v)$ for every function $f$, we have as in Section \ref{sec2.1} that, $[\mathfrak{v},\dot{v}]=0$ and hence $\nabla_\mathfrak{v}\dot{v}=\nabla_{\dot{v}}\mathfrak{v}$, where $\nabla$ denotes the Levi-Civita connection of the metric $g=\omega(\cdot,J\cdot)$.\\
 Now, $w$ is given by $w(s,t)=\frac{1}{2}\big(|\mathfrak{v}|^2+\dot{\eta}^2\big)=\frac{1}{2}\big(g(\mathfrak{v},\mathfrak{v})+\dot{\eta}^2\big)$ and the Laplacian satisfies
 \[ \Delta w = |\nabla_\mathfrak{v} \mathfrak{v}|^2+(\partial_s\dot{\eta})^2+|\nabla_{\dot{v}}\mathfrak{v}|^2+\kappa.\]
 Here, $\kappa$ is an error term given by
 \begin{align*}
  \kappa &= g(\mathfrak{v},\nabla_{\mathfrak{v}}\nabla_{\mathfrak{v}} \mathfrak{v}) +g(\mathfrak{v},\nabla_{\dot{v}}\nabla_{\dot{v}} \mathfrak{v}) +\dot{\eta}\cdot(\partial_s\partial_s\dot{\eta}).
 \end{align*}
 In the following, we estimate the components of $\kappa$. First we have
 \begin{align*}
  \nabla_{\mathfrak{v}} \mathfrak{v} &= -\nabla_{\mathfrak{v}}\big(J(\dot{v}-\eta X)\big) = -(\nabla_{\mathfrak{v}}J)(\dot{v}-\eta X)-J(\nabla_{\mathfrak{v}}\dot{v}-\dot{\eta}X-\eta\nabla_{\mathfrak{v}}X)\\
  \nabla_{\dot{v}} \dot{v} &= \nabla_{\dot{v}}\big(J\mathfrak{v}+\eta X\big) = (\nabla_{\dot{v}} J)\mathfrak{v} + J(\nabla_{\dot{v}}\mathfrak{v}) + \eta \nabla_{\dot{v}}X\\
  \nabla_{\mathfrak{v}}\dot{v} &= \nabla_{\mathfrak{v}}\big(J\mathfrak{v}+\eta X\big) = (\nabla_{\mathfrak{v}} J)\mathfrak{v} + J(\nabla_{\mathfrak{v}}\mathfrak{v}) + \dot{\eta} X + \eta\nabla_{\mathfrak{v}}X
 \end{align*}
which implies with $\nabla_\mathfrak{v}\dot{v}=\nabla_{\dot{v}}\mathfrak{v}$ that
\[\nabla_{\mathfrak{v}} \mathfrak{v} +\nabla_{\dot{v}} \dot{v} = (\nabla_{\dot{v}} J)\mathfrak{v} -(\nabla_{\mathfrak{v}} J)(\dot{v}-\eta X) + J(\dot{\eta} X + \eta\nabla_{\mathfrak{v}} X) + \eta\nabla_{\dot{v}} X\tag{$\ast$}\]
which implies
\begin{align*}
 \nabla_{\mathfrak{v}}\big(\nabla_{\mathfrak{v}} \mathfrak{v} +\nabla_{\dot{v}}\dot{v}\big) &=\phantom{+\,} (\nabla_{\dot{v}} J)\nabla_{\mathfrak{v}} \mathfrak{v} - (\nabla_{\mathfrak{v}} J)(\nabla_{\mathfrak{v}} \dot{v} -\dot{\eta} X - \eta \nabla_{\mathfrak{v}} X)\\
 &\phantom{\,=\,}+ (\nabla_{\mathfrak{v}}\nabla_{\dot{v}} J)\mathfrak{v} - (\nabla_{\mathfrak{v}}\nabla_{\mathfrak{v}} J)(\dot{v}-\eta X)\\
 &\phantom{\,=\,}+ (\nabla_{\mathfrak{v}} J)(\dot{\eta}X + \eta\nabla_{\mathfrak{v}} X) + J\big((\partial_s \dot{\eta}) X + \dot{\eta}\nabla_{\mathfrak{v}} X + \dot{\eta}\nabla_{\mathfrak{v}} X + \eta\nabla_{\mathfrak{v}}\nabla_{\mathfrak{v}}X\big)\\
 &\phantom{\,=\,}+ \dot{\eta}\nabla_{\dot{v}} X + \eta\nabla_{\mathfrak{v}}\nabla_{\dot{v}} X.
\end{align*}
Additionally, we have
\begin{align*}
 \nabla_{\mathfrak{v}}\nabla_{\mathfrak{v}} \mathfrak{v} +\nabla_{\dot{v}}\nabla_{\dot{v}} \mathfrak{v} &= \nabla_{\mathfrak{v}}\big(\nabla_{\mathfrak{v}}\mathfrak{v} + \nabla_{\dot{v}} \dot{v}\big) + \nabla_{\dot{v}} \nabla_{\mathfrak{v}} \dot{v} -\nabla_{\mathfrak{v}}\nabla_{\dot{v}}\dot{v} =\nabla_{\mathfrak{v}}\big(\nabla_{\mathfrak{v}}\mathfrak{v} + \nabla_{\dot{v}} \dot{v}\big) + R(\dot{v},\mathfrak{v};\dot{v}),
\end{align*}
where $R$ is the curvature tensor of $\nabla$. Note that $im(v,\eta)$ lying in the compact region $W$ guarantees that the number $\eta$ and the tensor fields $J, \nabla J, \nabla\nabla J, X, \nabla X$ and $R$ are bounded. Moreover, we have for any tensor $T$ the estimate 
\[|\nabla_\mathfrak{v}\nabla_\mathfrak{v} T|\leq||\nabla\nabla T||\cdot|\mathfrak{v}|^2+||\nabla T||\cdot|\nabla_\mathfrak{v}\mathfrak{v}|.\]
Using this, the expression for $\nabla_{\mathfrak{v}}\dot{v}$ calculated above and the estimate $|\dot{v}|\leq||J||\cdot |\mathfrak{v}|+|\eta X|$, we find that there exists a constant $\tilde{C}>0$ (depending on $H,\omega,J$ on $W$) such that
\begin{align*}
 |\nabla_{\mathfrak{v}}\nabla_{\mathfrak{v}} \mathfrak{v} +\nabla_{\dot{v}}\nabla_{\dot{v}}\mathfrak{v}|\leq \tilde{C}\Big(|\mathfrak{v}|\big(|\nabla_{\mathfrak{v}}\mathfrak{v}|+|\nabla_{\dot{v}}\mathfrak{v}|+|\dot{\eta}|\big)+|\nabla_{\mathfrak{v}}\mathfrak{v}|+ |\mathfrak{v}|^3+|\mathfrak{v}|^2 + |\mathfrak{v}|+|(\partial_s \dot{\eta})|\Big).
\end{align*}
Using Young's inequality, we thus find constants $C,D>0$ such that
\begin{align*}
 &\phantom{\geq\;} g\big(\mathfrak{v},\nabla_{\mathfrak{v}}\nabla_{\mathfrak{v}} \mathfrak{v} +\nabla_{\dot{v}}\nabla_{\dot{v}}\mathfrak{v}\big)\\
 &\geq - \tilde{C}\Big(\big(|\mathfrak{v}|^2+|\mathfrak{v}|\big)|\nabla_{\mathfrak{v}}\mathfrak{v}|+ |\mathfrak{v}|^2\cdot|\nabla_{\dot{v}}\mathfrak{v}|+|\mathfrak{v}|^4+|\mathfrak{v}|^3 + |\mathfrak{v}|^2\cdot|\dot{\eta}| +|\mathfrak{v}|\cdot|(\partial_s \dot{\eta})|\Big)\\
 &\geq - D -C\left|\left(\begin{smallmatrix}\textstyle\big. \mathfrak{v}\\\textstyle\dot{\eta}\end{smallmatrix}\right)\right|^4 - \frac{1}{2}\Big|\nabla_{\mathfrak{v}}\mathfrak{v}\Big|^2-\frac{1}{2}\Big|\nabla_{\dot{v}}\mathfrak{v}\Big|^2.\tag{$\ast\ast$}
\end{align*}
Next, we estimate the $\dot{\eta}$ part of $\kappa$. We have
\begin{align*}
 \partial_s\partial_s\, \dot{\eta} &= \partial_s\partial_s\int_0^1 H(v)\,dt =\partial_s\int_0^1 dH(\mathfrak{v})\,dt =\int_0^1 (\nabla_{\mathfrak{v}} dH)(\mathfrak{v})+dH(\nabla_{\mathfrak{v}}\mathfrak{v})\,dt.
\end{align*}
As $dH(\dot{v})$ is a 1-periodic function, we find
\[ 0 = dH(\dot{v})(1)-dH(\dot{v})(0)=\int_0^1\frac{d}{dt} dH(\dot{v})\,dt = \int_0^1(\nabla_{\dot{v}} dH)(\dot{v}) + dH(\nabla_{\dot{v}}\dot{v})\,dt\]
which implies with $(\ast)$ that
\begin{align*}
 \partial_s\partial_s\,\dot{\eta} &= \int_0^1 (\nabla_{\mathfrak{v}} dH)(\mathfrak{v})+dH(\nabla_{\mathfrak{v}}\mathfrak{v})+(\nabla_{\dot{v}} dH)(\dot{v}) + dH(\nabla_{\dot{v}}\dot{v})\,dt \\
 &=\int_0^1\left(\begin{aligned} &\phantom{+\,} (\nabla_{\mathfrak{v}} dH)(\mathfrak{v})+(\nabla_{\dot{v}} dH)(\dot{v})\\&+dH\Big((\nabla_{\dot{v}} J)\mathfrak{v} -(\nabla_{\mathfrak{v}} J)(\dot{v}-\eta X) + J(\dot{\eta} X + \eta\nabla_{\mathfrak{v}} X) + \eta\nabla_{\dot{v}} X\Big)\end{aligned}\right)\,dt
\end{align*}
and hence there is another constant $\tilde{C}'>0$ such that
\[|\dot{\eta}\cdot\partial_s\partial_s\dot{\eta}|\leq |\dot{\eta}|\int_0^1 \tilde{C}'\big(1+|\mathfrak{v}|+|\mathfrak{v}|^2+|\dot{\eta}|\big)dt.\]
Using Jensen's inequality for integrals and Young's inequality again, we find constants $C',D'>0$ such that
\[\dot{\eta}\cdot\partial_s\partial_s\dot{\eta}\geq - D'-C'\left(\left|\left(\begin{smallmatrix}\textstyle\big. \mathfrak{v}\\\textstyle\dot{\eta}\end{smallmatrix}\right)\right|^4+\int_0^1\left|\left(\begin{smallmatrix}\textstyle\big. \mathfrak{v}\\\textstyle\dot{\eta}\end{smallmatrix}\right)\right|^4\,dt\right).\tag{$\ast\ast\ast$}\]
Combining $(\ast\ast)$ and $(\ast\ast\ast$) and using the fact that $\left|\left(\begin{smallmatrix}\mathfrak{v}\\\dot{\eta}\end{smallmatrix}\right)\right|^2=\left|\partial_s\left(\begin{smallmatrix}v\\\eta\end{smallmatrix}\right)\right|^2=w$, we find constants $A,B>0$ such that
\[\Delta w = |\nabla_\mathfrak{v} \mathfrak{v}|^2+(\partial_s\dot{\eta})^2+|\nabla_{\dot{v}}\mathfrak{v}|^2+\kappa\geq -A-B\big(w^2+\textstyle\int_0^1 w^2\,dt\big).\]
\end{proof}
In the remainder of this subsection, we prove the second part of Theorem \ref{asympequiv}, namely that $|\partial_s(v,\eta)(s,t)|$ converges exponentially to 0 and that $(v,\eta)$ converges exponentially to some $(v^\pm,\eta^\pm)\in\crita$. We start by describing the structure of the manifold $V$ near points $v^\pm\in\mathcal{N}^{\eta^\pm}$ more explicitly.
\begin{lemme}[cf. \cite{BoElHoWyZe}, Lem.\ A.1]\label{lemcoord}
 Assume that $\mathcal{A}^H$ satisfies (MB). Let $\mathcal{N}^\eta\subset\Sigma\subset V$ denote the submanifold of $V$ which is covered by $\eta$-periodic orbits of the Reeb field $R$. Let $v$ be a non-constant $\eta$-periodic Reeb trajectory. Then
 \begin{itemize}
  \item[a)] if $\eta\neq 0$ is the minimal period of $v$, there exists a tubular neighborhood $U\subset V$ of $im(v)$ such that $U\cap\mathcal{N}^\eta$ is invariant under the flow of $\eta R$ and one finds coordinates
  \begin{align*}
  & \big(\vartheta,z_1,\dots,z_k,z_{k+1},\dots, z_{2n-1}\big)\in S^1\times\mathbb{R}^{2n-1}, \quad k=\dim\mathcal{N}^\eta -1\bigg.\\
  \text{such that }\quad& U\cap\mathcal{N}^\eta=\{z_{k+1}=\dots=z_{2n-1}=0\},\qquad U\cap\Sigma=\{ z_{2n-1}=0\}\\
   \text{and }\quad& \eta R|_{\mathcal{N}^\eta} = \frac{\partial}{\partial \vartheta},\qquad\qquad Y_\lambda|_U=\frac{\partial}{\partial z_{2n-1}}.
  \end{align*}
  \item[b)] if $v$ is an $m$-multiple of a trajectory $v_0$ of minimal period $\frac{\eta}{m}\neq 0$, there exists a tubular neighborhood $\tilde{U}$ of $im(v_0)$ such that its $m$-fold cover $U$ together with all the structures induced by the covering map $\pi: U \rightarrow \tilde{U}$ from  corresponding objects on $\tilde{U}$ satisfies the properties of part a).
 \end{itemize}
\end{lemme}
\begin{rem} If $\eta=0$, i.e.\ if $v$ is constant and $\mathcal{N}^0=\Sigma$, we can trivially find such coordinates, but there is no distinct direction $\vartheta$ and thus $\vartheta=z_0\in\mathbb{R}$ instead of $S^1$.
\end{rem}
\pagebreak
\begin{proof} Let $N_XY$ denote the normal bundle of a submanifold $Y\subset X$.
\begin{itemize}
 \item As $\Sigma\subset V$ is a compact codimension 1 submanifold and the Liouville vector field $Y_\lambda$ transverse to $\Sigma$, we can find a tubular neighborhood of $\Sigma$ given by an embedding 
 \[\pi_V : \Sigma\times(-\veps,\veps)\hookrightarrow V,\qquad (p,z_{2n-1})\mapsto \pi_V(p,z_{2n-1})\] with $\pi_\ast(\partial_{\textstyle z_{2n-1}})=Y_\lambda$ and $\pi(\Sigma\times\{0\})=\Sigma$. Note that $Y_\lambda$ trivializes $N_V\Sigma\cong\Sigma\times\mathbb{R}$.
 \item As $\mathcal{N}^\eta\subset\Sigma$ is a compact submanifold, we can find a tubular neighborhood
 \[\tilde{U}\subset N_{\Sigma}\mathcal{N}^\eta,\qquad \pi_\Sigma : \tilde{U}\hookrightarrow \Sigma,\qquad (q,z_{k+1},...\,,z_{2n-2})\mapsto\pi_\Sigma(q,z_{k+1},...\,,z_{2n-2}),\]
 such that $\pi_\Sigma(\mathcal{N}^\eta\times\{0\})=\mathcal{N}^\eta$. Here, $z_{k+1},...\,,z_{2n-2}$ are coordinates in the normal direction, which are well-defined only locally.
 \item In case a), the compact Lie group $S^1$ acts freely on $\mathcal{N}^\eta$ near $v$ by the flow of $\eta R$. Hence there exists by the Slice Theorem (see \cite{silva}, Thm.\ 23.5) a tubular neighborhood of $im(v)$
 \[\pi_\mathcal{N} : S^1\times (-\veps,\veps)^k\hookrightarrow \mathcal{N}^\eta,\qquad (\vartheta, z_1,...\,,z_k)\mapsto\pi_{\mathcal{N}}(\vartheta, z_1,...\,,z_k),\]
 such that $\pi_\mathcal{N}(S^1\times\{0\})=im(v)$ and $(\pi_\mathcal{N})_\ast(\partial_\vartheta)=R$. Note that $N_{\mathcal{N}^\eta}\big(im(v)\big)$ is a trivial bundle over $S^1$ as it is orientable.
 \item Combining $\pi_V,\,\pi_\Sigma$ and $\pi_\mathcal{N}$ then gives the desired tubular neighborhood $U$ and the coordinates. Note that $N_\Sigma\mathcal{N}^\eta$ is trivial over $im(\pi_\mathcal{N})$, as the latter is homotopy equivalent to $S^1$. This ensures that the coordinates $z_{k+1},...\,,z_{2n-2}$ are globally defined on $U$.
 \item In case b), construct a tubular neighborhood $U$ of $v_0$ as in case a). Taking its $m$-fold covering then clearly satisfies the lemma.\qedhere
\end{itemize}
\end{proof}
For the following, we choose a finite cover of $\mathcal{N}^{\eta^\pm}$ by neighborhoods $U_j$ as in Lemma \ref{lemcoord}. We remark that for the asymptotic estimates in $U_j\cap\mathcal{N}^{\eta^\pm}$ it is irrelevant whether we are in a neighborhood of $im(v^\pm,\eta^\pm)$ or on a covering. Let us abbreviate
\[z_{in}:=(\vartheta,z_1,\dots,z_k),\qquad z_{out}:=(z_{k+1},\dots,z_{2n-1}),\qquad z:=(\vartheta,z_1,\dots,z_{2n-1})\]
and let $n\in\mathbb{R}$ be a coordinate for the $\eta$-direction such that $n=0$ corresponds to $\eta^\pm$. Using the coordinates from Lemma \ref{lemcoord}, we can express the Rabinowitz-Floer equation (\ref{eq3}) near $(v^\pm,\eta^\pm)$ as
\begin{equation}\label{eqfloerloc}\begin{aligned}
 \partial_s Z + J_t(v,\eta)\left(\partial_t Z + \frac{\partial}{\partial\vartheta}-(N+\eta^\pm)\cdot X_H(v)\right) &= 0\\
 \partial_s N + \int^1_0 H(v) dt &=0,
 \end{aligned}
\end{equation}
where $Z(s,t):=\big(\vartheta\circ v(s,t)-t, z\circ v(s,t)\big)$ and $N(s):=\eta(s)-\eta^\pm$ are chosen such that $(Z,N)\rightarrow(0,0)$ (which will become clear at the end of this paragraph). Let $Z_{in},\,Z_{out}$ resp.\ $Z_{2n-1}$ denote the $z_{in}$- resp.\ $z_{out}$- resp.\ $z_{2n-1}$-part of $Z$.\pagebreak\\
 As $\eta^\pm X_H=\frac{\partial}{\partial\vartheta}$ on $\{z_{out}=0\}$ and $H\equiv 0$ on $\{z_{2n-1}=0\}$ we will see in the Lemma below that equation (\ref{eqfloerloc}) can be rewritten as
\begin{equation}\label{eqfloerloc2}
 \begin{aligned}
  \partial_s Z + J_t(v,\eta)\left[\partial_t Z + S(v,\eta)\cdot\begin{pmatrix}Z_{out}\\ N\phantom{ou}\end{pmatrix}\right] &= 0\\
  \partial_s N + \int^1_0 \big(h(v)\cdot Z_{2n-1}\big)\, dt &= 0
 \end{aligned}
\end{equation}
for some functions $S$ and $h$ on $S^1\times\mathbb{R}^{2n}$ resp.\ $S^1\times\mathbb{R}^{2n-1}$ with values in $2n\times(2n-k)$-matrices resp.\ $\mathbb{R}$. Using $S$ and $h$, we define an $s$-dependent operator $A(s)$ by
\begin{equation}\begin{aligned}\label{secvar2}
A(s) &: W^{k+1,p}(S^1,\mathbb{R}^{2n})\times\mathbb{R}\rightarrow W^{k,p}(S^1,\mathbb{R}^{2n})\times\mathbb{R}\\
 A(s)\begin{pmatrix}z_{in}(t)\\ z_{out}(t)\\ n\end{pmatrix} &= \begin{pmatrix} J_t((v,\eta)(s))\left[\displaystyle \frac{d}{dt}\begin{pmatrix} z_{in}(t)\\ z_{out}(t)\\ n\end{pmatrix} + S\big((v,\eta)(s)\big)\begin{pmatrix} z_{out}(t)\\ n\end{pmatrix}\right]\\ \bigg.\int_0^1 \big(h(v(s))\cdot z_{2n-1}(t)\big) dt\end{pmatrix}.
 \end{aligned}
\end{equation}
Note that we can write more precisely $A(s)=A\big((Z,N)(s,t)\big)$ as we have actually a family of operators depending on $t$ and points in $V\times\mathbb{R}$.
\begin{lemme}
 Equation (\ref{eqfloerloc2}) holds true. Moreover, for $(Z,N)$ holds that 
 \[\Big|\Big|A(s)+\nabla^2\mathcal{A}^H_{(Z_{in},0,0)}\Big|\Big|\rightarrow 0\qquad \text{ as }s\rightarrow \pm\infty.\]
\end{lemme}
\begin{rem}
 The operators $A(s)$ are of course only defined in $U_j$. The last statement is hence to be understood as follows: No matter in which $U_j$ the image of $(v,\eta)(s)$ lies, the operators $A(s)$ become close to the operator $-\nabla^2\mathcal{A}^H$.
\end{rem}
\begin{proof}
 Equation (\ref{eqfloerloc2}) follows basically from Haddamard's Lemma, which states for a function $f:\mathbb{R}^{n}\rightarrow\mathbb{R}$ with $f|_{\mathbb{R}^k\times\{0\}}\equiv 0$ that
 \begin{align*} &f(x)=\sum_{j=k+1}^{n} g_j(x)\cdot x_j,\quad\text{ where }\quad g_j(x):=\int_0^1\frac{\partial f}{\partial x_j}\big(x_1,...\,,x_k,rx_{k+1},...\,,rx_n\big)\,dr.\\
 &\text{In particular, note that }\quad g_j(x_1,...\,,x_k,0)=\frac{\partial f}{\partial x_j}\big(x_1,...\,,x_k,0\big). \tag{$\ast$}
 \end{align*}
 The role of $f$ for $S$ is played by the vector valued function $\mathbb{S}:=\frac{\partial}{\partial \vartheta}-(n+\eta^\pm)\cdot X_H$ and for $h$ by the function $H$. The functions $S$ and $h$ are then explicitly given by
 \begin{itemize}
  \item $\displaystyle h(z):=\int_0^1 \frac{\partial H}{\partial z_{2n-1}}\Big(\vartheta,z_1,...\,,z_{2n-2},r\cdot z_{2n-1}\Big) dr,$\\
  where only the $z_{2n-1}$-derivative contributes, as $H\equiv 0$ on $\{z_{2n-1}=0\}$.
  \item $\displaystyle S_{ij}(z,n):= \int^1_0 \frac{\partial \mathbb{S}^i}{\partial z_{k+j}}\Big(z_{in},r z_{out},r n\Big)dr = -\int^1_0 \frac{\partial\, (n+\eta^\pm)\big(X_H\big)^i}{\partial z_{k+j}} \Big(z_{in},r z_{out},r n\Big)dr$\\
  where $\big(X_H\big)^i$ is the $i$-th component of $X_H$ and $\partial z_{2n}=\partial n$. Note again that only the last $2n-k$ derivatives contribute, as $\mathbb{S}\equiv 0$ on $\{z_{out}=0, n=0\}$\pagebreak.
 \end{itemize}
We prove the formula only for $h$, as the proof for $S$ is entirely similar. Let $z\in S^1\times\mathbb{R}^{2n-1}$ be arbitrary. As $H\equiv 0$ on $\{z_{2n-1}=0\}$, we have $H(\vartheta,z_1,...\,,z_{2n-2},0)=0$. Let \linebreak $g(r):= H(\vartheta,z_1,...\,,z_{2n-2},r\cdot z_{2n-1})$. Then we have
\begin{align*}
 && g'(r)&=\frac{\partial H}{\partial z_{2n-1}}\Big(\vartheta, z_1,...\,,z_{2n-2},r\cdot z_{2n-1}\Big)\cdot z_{2n-1}\\
 &\Rightarrow& H(z)=g(1)-g(0) =\int_0^1 g'(r)dr &= \int_0^1 \frac{\partial H}{\partial z_{2n-1}}\Big(\vartheta, z_1,...\,,z_{2n-2},r\cdot z_{2n-1}\Big)dr\cdot z_{2n-1}\\
 &&&= h(z)\cdot z_{2n-1}.
\end{align*}
For the asymptotic behaviour of $A(s)$, we first note that on $\{z_{out}=0,n=0\}$ we have due to $(\ast)$ that $h$ and $S$ are given by
\[\begin{aligned}
   h(z_{in},0)&=\frac{\partial H}{\partial z_{2n-1}}\big(z_{in},0\big)\\
   S_{ij}(z_{in},0,0)&= -\frac{\partial\, (n+\eta^\pm)\big(X_H\big)^i}{\partial z_{k+j}} \Big(z_{in},0,0\Big).
  \end{aligned}\]
  On the other hand, we recall that $\nabla^2\mathcal{A}^H_{(v,\eta)}(\mathfrak{v},\hat{\eta})$ at a critical point $(v,\eta)\in\crita$ was defined by taking a differentiable 1-parameter family $(v_a,\eta_a)$ with $(v_0,\eta_0)=(v,\eta)$ and $\frac{d}{da}(v_a,\eta_a)\big|_{a=0}=(\mathfrak{v},\hat{\eta})$ and setting $\nabla^2\mathcal{A}^H_{(v,\eta)}(\mathfrak{v},\hat{\eta}):=\frac{d}{da}\nabla\mathcal{A}^H(v_a,\eta_a)\big|_{a=0}$.\\
  Recall that $\nabla\mathcal{A}^H(v_a,\eta_a)=-\big(J(\dot{v}_a-\eta_a X(v_a))\,;\,\int_0^1 H(v_a)dt\big)$. In our local coordinates, any map of the form $(v_0,\eta_0): S^1\rightarrow S^1\times\mathbb{R}^{2n},\; t\mapsto (t,z_1,...\,,z_k,0,...\,,0,0)$ for $z_1,...\,,z_k$ fixed corresponds to an element in $\crita$. So if we take an arbitrary family $(v_a,\eta_a)$ with $(v_0,\eta_0)=(z_{in},0,0)$ and $\frac{d}{da}(v_a,\eta_a)\big|_{a=0}=(\mathfrak{v},\hat{\eta})$, we get
  \begin{align*}
   \nabla^2\mathcal{A}^H_{(z_{in},0,0)}(\mathfrak{v},\hat{\eta}) &=-\frac{d}{da}\left.\begin{pmatrix}J\big(\frac{d}{dt}v_a(t)-(\eta_a+\eta^\pm) X_H(v_a(t))\big)\\\int_0^1H(v_a(t))dt\end{pmatrix}\right|_{a=0}\\
   &=-\begin{pmatrix}J\big(\frac{d}{dt}\mathfrak{v}(t)-\sum_{j=0}^{2n}\frac{\partial(n+\eta^\pm X_H)^i}{\partial z_j}(z_{in},0,0)[\mathfrak{v}^j]\big)\\ \int_0^1\sum_{j=0}^{2n-1}\frac{H}{\partial z_j}(z_{in},0)[\mathfrak{v}^j] dt\end{pmatrix}\\
   &=-\begin{pmatrix} J\big(\frac{d}{dt}\mathfrak{v}(t)-\sum_{j=k+1}^{2n}\frac{\partial(n+\eta^\pm X_H)^i}{\partial z_j}(z_{in},0,0)[\mathfrak{v}^j]\big)\\ \int_0^1\frac{H}{\partial z_{2n-1}}(z_{in},0)\cdot\mathfrak{v}^{2n-1} dt\end{pmatrix},
  \end{align*}
  where the last line follows as $(n+\eta^\pm)X_H=\frac{\partial}{\partial\vartheta}$ is constant on $\{z_{out}=0,n=0\}$ and as $H$ is constant on $\{z_{2n-1}=0\}$. Note that no derivatives of $J$ appear as the corresponding term is zero due to $\frac{d}{dt}v_0(t)-(0+\eta^\pm) X_H(v_0)=0$ for $(v_0,\eta^\pm)\in\crita$.\\
  These calculations show that on $\{z_{out}=0, n=0\}$ the operator $-\nabla^2\mathcal{A}^H_{(z_{in},0,0)}$ has the same form as the point depending operator
  \[A_{(z_{in},0,0)}\left(\begin{smallmatrix}X_{in}\\X_{out}\\X_n\end{smallmatrix}\right):=\begin{pmatrix} J\big(\frac{d}{dt}X + S(z_{in},0,0)\left(\begin{smallmatrix}X_{out}\\X_n\end{smallmatrix}\right)\big)\\ \int_0^1 h(z_{in},0)\cdot X_{2n-1} dt\end{pmatrix}.\]
  Recall that Proposition \ref{asympconv} implies for a solution $(Z,N)$ of (\ref{eqfloerloc}) that $(Z_{out},N)\rightarrow(0,0)$ uniformly in $t$ as $s\rightarrow \pm\infty$. This shows 
  \[\big|\big| A(s)+\nabla^2\mathcal{A}^H_{(Z_{in},0,0)}\big|\big|=\big|\big|A_{(Z_{in},Z_{out},N)}-\big(-\nabla^2\mathcal{A}^H_{(Z_{in},0,0)}\big)\big|\big|\rightarrow 0\qquad \text{uniformly in $t$.}\qedhere\]
\end{proof}
Let us write for the moment $\nabla^2\mathcal{A}^H(z_{in}):=\nabla^2\mathcal{A}^H_{(z_{in},0,0)}$. We know by Lemma \ref{lemker} that the kernels of these operators have finite dimension $k+1=\dim\mathcal{N}^{\eta^\pm}$. Note that the coordinates $z_{in}$ are by Lemma \ref{lemcoord} coordinates on $\mathcal{N}^{\eta^\pm}$ and invariant under the flow of $\eta^\pm R$. We hence obtain that $\ker \nabla^2\mathcal{A}^H(z_{in})$ is for any $z_{in}$ spanned by the following constant $S^1$-families of vectors corresponding to vectors spanning $T\mathcal{N}^{\eta^\pm}$:
\[e_0(t)=(1,0,...\,,0),\quad e_1(t)=(0,1,0,...\,,0),\quad ...\,,\quad e_k(t)=(\underbrace{0,...\,,0}_{k-times},1,0,...\,,0),\]
 Recall that we have by Lemma \ref{lemcoker} an orthogonal splitting $H^k\times\mathbb{R}=\ker \nabla^2\mathcal{A}^H(z_{in})\oplus im\, \nabla^2\mathcal{A}^H(z_{in})$. We denote by $Q$ the orthogonal projection onto $im\, \nabla^2\mathcal{A}^H(z_{in})=\linebreak\big(\ker \nabla^2\mathcal{A}^H(z_{in})\big)^\perp$. Note that $Q$ does not depend on $z_{in}$, as $H^k\times\mathbb{R}$ and $\ker \nabla^2\mathcal{A}^H(z_{in})$ do not depend on $z_{in}$. Moreover, $\nabla^2\mathcal{A}^H(z_{in})$ restricted to $im(Q)=im\,\nabla^2\mathcal{A}^H(z_{in})$ is continuously invertible and we have the formulas \label{Qformulas}
\begin{itemize}
 \item $\partial_t Q = \partial_t$, as $\ker Q = \ker \nabla^2\mathcal{A}^H(z_{in})$ consists of constant vectors,
 \item $A(s)=A(s)Q$, as $A(s)|_{\ker \nabla^2\mathcal{A}^H(z_{in})}\equiv 0$ (see (\ref{secvar2})),
 \item $\big(\partial_t A(s)\big)=\big(\partial_t A(s)\big)Q$, as for all vector fields $X$ we have\medskip\\
 $\begin{aligned}\partial_t(A(s)X)=\partial_t(A(s)QX)&\Rightarrow\big(\partial_t A(s)\big)X+A(s)\partial_t X=\big(\partial_t A(s)\big) QX +A(s)\underbrace{\partial_t QX}_{=\partial_t X}\\ &\Rightarrow \big(\partial_t A(s)\big) X=\big(\partial_t A(s)\big)QX,\end{aligned}$
 \item $\partial_sQ = Q \partial_s$, as $Q$ is $s$-independent (as $Q$ does not depend on $Z_{in}$),
 \item $\big(\partial_s A(s)\big)=\big(\partial_s A(s)\big)Q$, as for all vector fields $X$ we have\medskip\\
 $\begin{aligned} \partial_s(A(s)X)=\partial_s(A(s)QX)&\Rightarrow (\partial_s A(s))X+A(s)\partial_s X=(\partial_sA(s))QX+A(s)\partial_s QX\\
 &\phantom{\Rightarrow (\partial_s A(s))X+A(s)\partial_s X}\;=(\partial_sA(s))QX+\underbrace{A(s)Q}_{=A(s)}\partial_s X\\ &\Rightarrow(\partial_s A(s))X=(\partial_s A(s))QX,\end{aligned}$
 \item $\partial_s\ZN + A(s)\ZN =0$, by the Rabinowitz-Floer equation (\ref{eqfloerloc2}),
 \item $Q\ZN(s)\rightarrow 0$ for $s\rightarrow \pm\infty$, as $(Z_{out},N)\rightarrow 0$ and $\partial_tv\rightarrow \eta X_H$, which implies in particular that $Z_{in}$ becomes close to being constant with respect to the flow of $\eta X_H$. In our coordinates, this means that $Z$ becomes close to maps of the form
 \[t\mapsto (t,z_1,...\,,z_k,0,...\,,0),\qquad z_1,...\,,z_k \text{ fixed},\]
 which lie all in the kernel of $Q$.
\end{itemize}
The following Proposition \ref{asympexpconv} will prove the remaining part of Theorem \ref{asympequiv}. In order to give the statement precisely, let us make the following definition.
\begin{defn}\label{defnWeighSob}
 Fix a smooth cut-off function $\beta$ such that $\beta(s)=1$ for $s\geq 1$ and $\beta(s)=-1$ for $s\leq -1$ and define for $\delta>0$ the \textbf{exponential weight function $\gamma_\delta$} by
 \[\gamma_\delta: \mathbb{R}\rightarrow\mathbb{R},\qquad \gamma_\delta(s)= e^{\delta\cdot\beta(s)s}.\]
 Let $I\subset \mathbb{R}$ be an unbounded interval. For $\Omega=I$ or $\Omega=I\times S^1$ let $||\cdot||_{k,p,\delta}$ be the following norm for locally  $p$-integrable functions $f:\Omega\rightarrow \mathbb{R}$ with weak derivatives up to order $k$
 \[||f||_{k,p,\delta}:=\sum_{|i|=0}^k ||\gamma_\delta\cdot \partial_i f||_p,\]
 where $i$ denotes a multi-index of the (possibly) two variables $s$ and $t$. The \textbf{$\gamma_\delta$-weighted Sobolev space $W^{k,p}_\delta(\Omega)$} is then defined by
 \[ W^{k,p}_\delta(\Omega):=\left\{\left. f\in W^{k,p}(\Omega)\,\right|\, ||f||_{k,p,\delta}<\infty\right\} = \left\{\left. f\in W^{k,p}(\Omega)\,\right|\,\gamma_\delta\cdot f\in W^{k,p}(\Omega)\right\}.\]
 For $k=0$, we also write $L^p_\delta(\Omega):= W^{0,p}_\delta(\Omega)$.
\end{defn}
\begin{prop}[cf. \cite{BourOan1} A.1]\label{asympexpconv}
 Let $\eta^\pm$ be fixed. There exist constants $C,\rho>0$ such that for all $\mathcal{A}^H$-gradient trajectories $(v,\eta)$ with $\displaystyle \lim_{s\rightarrow \pm\infty} \partial_s(v,\eta) = 0$ and $dist.\big((v,\eta)(s),\mathcal{N}^{\eta^\pm}\big) \rightarrow 0$ holds
 \begin{align*}
  \left|\partial_s \left(\begin{smallmatrix}\textstyle v\big.\\ \textstyle \eta\end{smallmatrix}\right)(s,t)\right| =
  \left|\partial_s \left(\begin{smallmatrix}\textstyle Z\big.\\ \textstyle N\end{smallmatrix}\right)(s,t)\right| \leq C\cdot e^{-\rho|s|}
 \end{align*}
for $|s|\geq s_0>0$ sufficiently large. This implies in particular that $(v,\eta)$ actually converges to some $(v^\pm,\eta^\pm)\in\crita$ and for coordinates $(\vartheta,z)$ around $v^\pm$ as in Lemma \ref{lemcoord} holds
 \begin{align*}
  \vartheta\circ v(s,t)-t&\in W^{1,p}_\delta\big((-\infty,-s_0]\times S^1,\mathbb{R}\big)\\
  z\circ v(s,t)&\in W^{1,p}_\delta\big((-\infty,-s_0]\times S^1,\mathbb{R}^{2n-1}\big)\\
  \eta-\eta^- &\in W^{1,p}_\delta\big((-\infty,-s_0],\mathbb{R}\big)\phantom{\lim_{s}}\\\phantom{\sum^n}
  \vartheta\circ v(s,t)-t&\in W^{1,p}_\delta\big([s_0,\infty)\times S^1,\mathbb{R}\big)\\
  z\circ v(s,t)&\in W^{1,p}_\delta\big([s_0,\infty)\times S^1,\mathbb{R}^{2n-1}\big)\\
  \eta-\eta^+ &\in W^{1,p}_\delta\big([s_0,\infty),\mathbb{R}\big)
 \end{align*}
for some $s_0$ sufficiently large and $\delta/p <\rho$.
\end{prop}
\begin{proof}
 It suffices to make the proof for $s\rightarrow+\infty$, the other case  being entirely similar. Write as above $Z(s,t)=\big(\vartheta\circ v(s,t)-t,z\circ v(s,t)\big)$ and $N(s)=\eta(s)-\eta^+$ for coordinates $(\vartheta,z)$ on some open sets $U_j$ covering $\mathcal{N}^{\eta^+}$ as in Lemma \ref{lemcoord}. We will show below, that there exist constants $\rho,\tilde{C}>0$ such that for $s\geq s_0$ sufficiently large and on any $U_j$ holds
 \[ \left|\left|Q \left(\begin{smallmatrix}Z\\ N\end{smallmatrix}\right)(s)\right|\right|_1\leq \tilde{C}\cdot e^{-\rho s}\tag{$\ast$}\]
for the $H^1$-norm $||\cdot||_1$. As $0=\partial_s\left(\begin{smallmatrix}Z\\N\end{smallmatrix}\right) + A(s)\left(\begin{smallmatrix}Z\\N\end{smallmatrix}\right) = \partial_s\left(\begin{smallmatrix}Z\\N\end{smallmatrix}\right) + A(s)Q\left(\begin{smallmatrix}Z\\N\end{smallmatrix}\right)$, this yields
\[||\partial_s\left(\begin{smallmatrix}Z\\N\end{smallmatrix}\right)(s)||^2_0
\leq||A(s)||^2\cdot\left|\left|Q \left(\begin{smallmatrix}Z\\ N\end{smallmatrix}\right)(s)\right|\right|_1^2
\leq \tilde{C}\cdot e^{-2\rho s},\]
 for a bigger constant $\tilde{C}$ and $s\geq s_0$. As $(Z,N)$ satisfies the partial differential equation (\ref{eqfloerloc}), which is similar to (\ref{eq3}), it satisfies also the mean value inequality from Lemma \ref{eq3meanvalue}, i.e.\ there exist constants $A,\delta>0$ such that
 \begin{align*}
  && \int_{B_r(s,t)} \left|\partial_s \left(\begin{smallmatrix} Z\\N\end{smallmatrix}\right)\right|^2&<\delta \quad\forall t\in S^1\\
  &\text{implies }&  \exists\; t^\ast\in S^1\;:\;\left|\partial_s \left(\begin{smallmatrix} Z\\N\end{smallmatrix}\right)(s,t)\right|^2&\leq \frac{Ar^2}{2}+\frac{8}{\pi r^2}\int_{\textstyle B_r(s,t^\ast)} \left|\partial_s \left(\begin{smallmatrix} Z\\N\end{smallmatrix}\right)\right|^2 \quad\forall t\in S^1.
  \end{align*}
 If necessary, increase $s_0$ such that $\tilde{C}\cdot e^{-2\rho(s_0-1)}\leq \delta$ holds and set $r=e^{-\rho s/2}$ for $s\geq s_0$. Then $s-r\geq s-1\geq s_0-1$ so that the assumption of the mean value inequality is satisfied and we get
 \begin{align*}
  && \left|\partial_s \left(\begin{smallmatrix} Z\\N\end{smallmatrix}\right)(s,t)\right|^2&\leq \frac{A}{2}e^{-\rho s} + \frac{8\tilde{C}}{\pi e^{-\rho s}}\cdot e^{-2\rho (s-1)}=\left(\frac{A}{2}+\frac{8\tilde{C}e^{2\rho}}{\pi}\right)e^{-\rho s}\\
  &\Rightarrow& \left|\partial_s \left(\begin{smallmatrix} Z\\N\end{smallmatrix}\right)(s,t)\right| &\leq \tilde{C}\cdot e^{-\rho s},
 \end{align*}
 for $\tilde{C}$ even bigger. Now, we can show that $(v,\eta)$ truly converges. For any interval $[s_1,s_2]$ such that $(v,\eta)\big([s_1,s_2]\big)\subset U_j$ for a fixed $j$, we obtain by integration
 \begin{align*}
  \left|\begin{pmatrix}Z\\N\end{pmatrix}(s_2,t)-\begin{pmatrix}Z\\N\end{pmatrix}(s_1,t)\right|=\left|\int_{s_1}^{s_2} \partial_{\mathfrak{s}}\begin{pmatrix}Z\\ N\end{pmatrix}(\mathfrak{s},t)\,d\mathfrak{s}\right|&\leq\int_{s_1}^{s_2} \tilde{C}\cdot e^{-\rho\mathfrak{s}}\,d\mathfrak{s}\\
  &= \tilde{C}\cdot\frac{1}{\rho}\cdot e^{-\rho s_1}\big(1-e^{-\rho (s_2-s_1)}\big)\\
  &\leq \tilde{C}\cdot\frac{1}{\rho}\cdot e^{-\rho s_1}.
 \end{align*}
 As the $U_j$ form a finite open cover of the compact set $\mathcal{N}^{\eta^+}$, we find that the maximal distance  $dist.(p,\partial U_j)$ over all $U_j$ for any $p\in\mathcal{N}^{\eta^+}$ is bounded from below. So for $s_1$ large enough, we know from the above estimate that $(v,\eta)(s)$ stays in one $U_j$ for all $s\geq s_1$. Moreover, $(v,\eta)(s)$ is a Cauchy sequence, which implies that it converges to some $(v^+,\eta^+)\in U_j\cap\mathcal{N}^{\eta^+}$. Taking once more with Lemma \ref{lemcoord} a neighborhood $U$ around $(v^+,\eta^+)$ such that in these coordinates $\lim\ZN=0$, we finally obtain
 \[\left|\begin{pmatrix}Z\\N\end{pmatrix}(s,t)\right|=\left|\int_s^\infty \partial_{\mathfrak{s}}\begin{pmatrix}Z\\ N\end{pmatrix}(\mathfrak{s},t)\,d\mathfrak{s}\right|\leq\int^\infty_s \tilde{C}\cdot e^{-\rho\mathfrak{s}}\,d\mathfrak{s} = \tilde{C}\cdot\frac{1}{\rho}\cdot e^{-\rho s},\]
 which proves the proposition with $C:=\tilde{C}/\rho$.\pagebreak\\
 It remains to show the exponential $H^1$-estimate $(\ast)$ on $Q\ZN$. In the following, we abbreviate $||\cdot||:=||\cdot||_0$ for  the $L^2$-norm, write $||\cdot||_k$ for the $H^k$-norm and $\langle\cdot,\cdot\rangle=g(\cdot,\cdot)$ for the scalar product on $L^2$. We first show that $||Q\ZN||$ is exponentially bounded. For that define 
\[f(s):=\frac{1}{2}\left|\left|\Big.Q\ZN(s)\right|\right|^2.\]
We show below the existence of a constant $c$ such that $f''$ satisfies for $s$ sufficiently large
\[f''(s)\geq 4(c-2\veps)\cdot f(s),\tag{$\ast\ast$}\]
where $\veps>0$ is an arbitrarily small constant (for $s$ sufficiently large). Set $\rho:=\sqrt{c-2\veps}$, such that $f''\geq 4\rho^2\cdot f$.
For $s_0$ large, define furthermore $g(s):=f(s_0)e^{-2\rho(s-s_0)}$. Then
\[g''=4\rho^2 g,\qquad (f-g)''\geq 4\rho^2(f-g),\qquad (f-g)(s_0)=0\quad\text{ and }\quad \lim_{s\rightarrow\infty} f(s)-g(s)=0.\]
The last statement holds as $g(s)\rightarrow 0$ and $f(s)\rightarrow 0$. To see this for $f$, recall that $Q\ZN\rightarrow 0$.
Then it follows that $f-g\leq 0$ on $[s_0,\infty)$, as it cannot have a strictly positive maximum. Therefore we obtain an exponential bound as
\[\left|\left|Q \left(\begin{smallmatrix}Z\\ N\end{smallmatrix}\right)(s)\right|\right|=f(s)\leq g(s)= \left|\left|Q \left(\begin{smallmatrix}Z\\ N\end{smallmatrix}\right)(s_0)\right|\right|e^{-\rho(s-s_0)}.\]
To show $(\ast\ast)$, consider the operator $A(s)$ and recall that  $||A(s)+\nabla^2\mathcal{A}^H(Z_{in})||\rightarrow 0$. As all $\nabla^2\mathcal{A}^H(Z_{in})$ restricted to $im (Q)$ are continuously invertible, we find for $s$ sufficiently large that the operators $A(s)$ and $Q A(s)$ are also invertible when restricted to $im(Q)$. Hence, there exists for such $s$ constants $c(s)>0$ such that for all $\left(\begin{smallmatrix} z\\ n\end{smallmatrix}\right)\in H^k(S^1,\mathbb{R}^{2n})\times\mathbb{R}$ holds
 \[ \left|\left|\Big.A(s)Q \left(\begin{smallmatrix}z\\n\end{smallmatrix}\right)\right|\right|^2_{k-1}\geq \left|\left|\Big.Q A(s)Q \left(\begin{smallmatrix}z\\n\end{smallmatrix}\right)\right|\right|^2_{k-1}\geq c(s)\cdot \left|\left|\Big.Q \left(\begin{smallmatrix}z\\n\end{smallmatrix}\right)\right|\right|^2_k\geq c(s)\cdot \left|\left|\Big.Q \left(\begin{smallmatrix}z\\n\end{smallmatrix}\right)\right|\right|^2_{k-1}.\]
 Note that $c(s)$ can be chosen arbitrarily close to the smallest non-zero square of an eigenvalue of $\nabla^2\mathcal{A}^H(Z_{in})$ (while increasing $s$). As $\nabla^2\mathcal{A}^H(z_{in})$ depends continuously on $z_{in}\in\mathcal{N}^{\eta^+}$ and $\dim \big(\ker \nabla^2\mathcal{A}^H(z_{in})\big)=\dim\mathcal{N}^{\eta^+}$ is constant on the compact set $\mathcal{N}^{\eta^+}$, we can choose one $c$ independent of $s$ for which the above estimate holds for all $s$ sufficiently large. Note that $c$ can in particular be chosen independent from $U_j$.\\
 The estimate $(\ast\ast)$ now follows from simple but tedious estimates. The main idea is to replace $\partial_s\ZN$ by $-A(s)\ZN$ via the Rabinowitz-Floer equation and then to apply the above estimate. In order to make the calculations more readable, we will from now on write $A$ instead of $A(s)$. Using the formulas for $Q$ (see page \pageref{Qformulas}), we calculate for $f$
\begin{align*}
 f''(s)&=||\partial_s Q\ZN||^2+\langle Q\ZN,\partial_s^2 Q\ZN\rangle +\mu\\
 &=||QAQ\ZN||^2-\langle Q\ZN,\partial_s QA\ZN\rangle + \mu\\
 &= ||QAQ\ZN||^2-\langle Q\ZN,Q(\partial_s A)Q\ZN-QA^2Q\ZN\rangle+\mu\\
 &\begin{aligned}=||QAQ\ZN||^2&-\langle Q\ZN,Q(\partial_sA)Q\ZN\rangle\\
 &+ \langle(A^\ast-A)Q\ZN,AQ\ZN\rangle+||AQ\ZN||^2+\mu\end{aligned}\\
 &\geq2c||Q\ZN||^2_1-||\partial_sA||\cdot||Q\ZN||^2_1-||A^\ast-A||\cdot||A||\cdot||Q\ZN||^2_1+\mu\\
 &\geq 2(c-\veps)||Q\ZN||^2_1+\mu.
\end{align*}
Here, $A^\ast$ denotes the adjoint of $A$ and we used the facts that $||A^\ast-A||\rightarrow 0$ (as the limit operators $\nabla^2\mathcal{A}^H(Z_{in})$ are selfadjoint), that $||\partial_s A||\rightarrow 0$ (as $\partial_s A\ZN=(DA)[\partial_s\ZN]$ and $\partial_s\ZN\rightarrow 0$) and that $||A||$ is uniformly bounded.\\
With $\mu$ we summarized all terms of $f''$ which involve derivatives of the metric $g$. Note that we cannot avoid this, as $J$ and hence $g$ depends on the point $\ZN(s)$. The estimate of $\mu$ is similar and goes as follows:
\begin{align*}
 \mu &= {\textstyle\frac{1}{2}}\big(\partial_s^2g\big)\big(Q\ZN,Q\ZN\big)+\big(\partial_s g\big)\big( Q\ZN,\partial_s Q\ZN\big)\\
 &={\textstyle\frac{1}{2}}\big(\partial_s(Dg)[\partial_s\ZN]\big)\big(Q\ZN,Q\ZN\big)+\big(Dg[\partial_s\ZN]\big)\big(Q\ZN,QAQ\ZN\big)\\
 &\begin{aligned}=&\phantom{+}{\textstyle\frac{1}{2}}\big(D^2g[\partial_s\ZN,\partial_s\ZN]+Dg[(\partial_s A)Q\ZN+A^2Q\ZN]\big)\big(Q\ZN,Q\ZN\big)\\
 &+\,\big(Dg[\partial_s\ZN]\big)\big(Q\ZN,QAQ\ZN\big)
  \end{aligned}\\
 &\begin{aligned}\geq &-{\textstyle\frac{1}{2}}||D^2 g||\cdot ||\partial_s \ZN||^2\cdot ||Q\ZN||^2_1-{\textstyle\frac{1}{2}}||Dg||\big(||\partial_s A||+||A^2||\big)||Q\ZN||_1\cdot ||Q\ZN||^2_1\\
  &-||Dg||\cdot ||QA||\cdot ||\partial_s\ZN||\cdot||Q\ZN||^2_1 
  \end{aligned}\\
 &\geq-2\veps||Q\ZN||^2_1.
\end{align*}
Here, we write $Dg$ for the total differential of $g$, which is in coordinates well-defined. The last line follows then for $s$ sufficiently large, as $||Q\ZN||_1,||\partial_s\ZN||\rightarrow 0$, while all operator norms are uniformly bounded. Combining the two estimates, we obtain
\[f''(s)\geq 2(c-2\veps)||Q\ZN||^2_1\geq 2(c-2\veps)||Q\ZN||^2=4(c-2\veps)f(s).\]
To complete the proof of $(\ast)$, we also have to show that $||\partial_t Q\ZN||=||\partial_t \ZN||$ is exponentially bounded. We define again a function
\[f(s):=\frac{1}{2}\Big|\Big|\partial_t(s) \ZN \Big|\Big|^2\]
and show the existence of another constant $\tilde{c}>0$ such that $f''$ satisfies for $s$ sufficiently large
\[f''(s)\geq 4(\tilde{c}-3\veps)\cdot f(s) - Ke^{-2\rho s},\tag{$\ast\ast\ast$}\]
where $\veps>0$ is again arbitrarily small, $K$ is some constant and $\rho$ is as above. By choosing $\tilde{\rho}\leq\min\{\sqrt{\tilde{c}-3\veps},\rho\}$, we then get $f''\geq 4\tilde{\rho}^2\cdot f -K e^{-2\tilde{\rho}s}$. For $s_0$ sufficiently large, define 
\[g(s):=\left(\frac{sK}{4\tilde{\rho}}-\frac{s_0K}{4\tilde{\rho}}+f(s_0)e^{2\tilde{\rho}s_0}\right)e^{-2\tilde{\rho}s}.\]
Then we have $g''=4\tilde{\rho}^2 g-Ke^{-2\tilde{\rho}s},\;g(s_0)=f(s_0)$ and $\displaystyle \lim_{s\rightarrow \infty} g(s)=0$ and hence
\[ (f-g)''\geq 4\tilde{\rho}^2(f-g),\quad(f-g)(s_0)=0\quad\text{ and }\quad \lim_{s\rightarrow \infty} f(s)-g(s)=0.\]
Then it follows again that $f-g\leq0$ on $[s_0,\infty)$, as it cannot have  a strictly positive maximum. Therefore, we obtain
\[||\partial_t\ZN||\leq \left(\frac{sK}{4\tilde{\rho}}-\frac{s_0K}{4\tilde{\rho}}+f(s_0)e^{2\tilde{\rho}s_0}\right)e^{-2\tilde{\rho}s}\leq \tilde{C}(s+1)e^{-2\tilde{\rho}s}\]
for some large constant $\tilde{C}$.\pagebreak\\
As $(s+1)e^{-\delta s}\rightarrow 0$ for $s\rightarrow \infty$ and any $\delta>0$, we get by decreasing $\tilde{\rho}$ and increasing $s_0$ further, that for all $s\geq s_0$ holds
\[||\partial_t\ZN||\leq \tilde{C}\cdot e^{-\tilde{\rho}s}.\]
To show $(\ast\ast\ast)$, we consider the operator $\partial_t$ on $H^1$. Its image $im(\partial_t)$ is closed in $L^2$, as it has a right inverse by integration. Moreover, $im(\partial_t)\cap \ker \nabla^2\mathcal{A}^H(z_{in})={0}$ for all $z_{in}$, as integrating a constant vector field $x(t)=x_0$ over $S^1$ yields $X(t)=tx_0$, which is 1-periodic only if $x_0=0$. As $\nabla^2\mathcal{A}^H(z_{in})(im(\partial_t))$ is also closed, we find that $\nabla^2\mathcal{A}^H(z_{in})$ restricted to $im(\partial_t)$ is a bijective operator between Banach spaces and hence continuously invertible. It follows that $A(s)$ for $s$ sufficiently large is also continuously invertible over $im(\partial_t)$, which gives us for $s$ sufficiently large a constant $\tilde{c}$ such that for all $\left(\begin{smallmatrix} z\\ n\end{smallmatrix}\right)\in H^k(S^1,\mathbb{R}^{2n})\times\mathbb{R}$ holds
\[||A(s)\partial_t\left(\begin{smallmatrix} z\\n\end{smallmatrix}\right)||^2_{k-1}\geq \tilde{c}\cdot||\partial_t\left(\begin{smallmatrix} z\\n\end{smallmatrix}\right)||^2_{k-1}.\]
 Note that we can choose $\tilde{c}$ globally, as all $\nabla^2\mathcal{A}^H(z_{in})$ restricted to $im(\partial_t)$ are injective. For the proof of $(\ast\ast\ast)$, we use estimates similar to whose for $||Q\ZN||$
\begin{align*}
 f''(s)&=||\partial_s\partial_t\ZN||^2+\langle\partial_t\ZN,\partial_s^2\partial_t\ZN\rangle + \mu\\ &= ||\partial_t\partial_s\ZN||^2+\langle \partial_t,\partial_t\partial_s^2\ZN\rangle + \mu\\ &=||\partial_t A\ZN||^2-\langle \partial_t\ZN,\partial_t(\partial_s A)\ZN\rangle + \langle\partial_t\ZN,\partial_t A^2\ZN\rangle + \mu\\
 &=||A\partial_t\ZN||^2-\langle\partial_t\ZN,(\partial_s A)\partial_t\ZN\rangle + \langle\partial_t\ZN,A^2\partial_t\ZN\rangle+\mu+\kappa,
\end{align*}
where $\mu$ contains again all terms with derivatives of the metric and $\kappa$ is an error term which we give explicitly below. Using similar estimates as for $||Q\ZN||$, we find that
\[f''(s)\geq 2(\tilde{c}-2\veps)||\partial_t\ZN||^2+\kappa,\]
with $\veps>0$ arbitrarily small and $s$ sufficiently large. The error term $\kappa$ is given by
\begin{align*}
   \begin{aligned}\kappa\\\phantom{.}\end{aligned} &\begin{aligned}=&\phantom{+}\;||(\partial_t A)\ZN||^2+ 2\langle A\partial_t\ZN,(\partial_t A)\ZN\rangle\\
            &-\langle\partial_t\ZN,(\partial_t\partial_s A)\ZN\rangle + \langle\partial_t\ZN,(\partial_t A^2)\ZN\rangle
           \end{aligned}\\
 &\begin{aligned}\geq &-2||A||\cdot||\partial_t\ZN||\cdot||\partial_t A||\cdot||Q\ZN||\\
       &-||\partial_t\ZN||\cdot||\partial_t\partial_sA||\cdot||Q\ZN||-||\partial_t\ZN||\cdot ||\partial_t A^2||\cdot||Q\ZN||
      \end{aligned}\\
 &\geq -2\veps||\partial_t\ZN||^2-K\cdot||Q\ZN||^2.        
\end{align*}
Here, $K$ is some large constant depending on $\veps$. The last line follows from Young's inequality and the fact that all operator norms are bounded. Using the exponential bound for $||Q\ZN||$, we hence get
\[f''(s)\geq 2(\tilde{c}-3\veps)||\partial_t\ZN||^2_1-K\cdot e^{-2\rho s}\geq 4(\tilde{c}-3\veps)f(s)-K\cdot e^{-2\rho s}.\qedhere\]
\end{proof}
\begin{rem}
By making similar estimates for the higher $t$-derivatives of $Q\ZN$, we could get an exponential bound on the $H^k$-norm of $\partial_s\ZN$. The pointwise estimates would then follow from Sobolev's embedding theorem.
\end{rem}

\newpage

\subsection{Unique continuation and injective points}\label{secholprop}
In this section, we present certain properties of solutions of the Rabinowitz-Floer equation (\ref{eq3}) which they have in common with holomorphic curves -- in particular unique continuation and the existence of so called regular or injective points. These properties for Floer/symplectic homology were first shown by Floer, Hofer, Salamon in \cite{FlHoSa}. Recently, Bourgeois, Oancea, \cite{BourOan3}, and Abbondandolo, Merry, \cite{AbbMer}, generalized them to Rabinowitz-Floer homology. I extend them here to situations, where we have on $V$ a symplectic symmetry of finite order.\medskip\\
That solutions of (\ref{eq3}) can be to some extend considered as holomorphic curves is due to the Carleman similarity principle (see \cite{FlHoSa}, Thm.\ 2.2, or \cite{DuffSal}, section 2.3). It states that every sufficiently regular solution $u$ with $u(0)=0$ of a partial differential equation on the disc $B_\veps(0)\subset \mathbb{C}$ of the form
\[\partial_s u(z) + J(z)\partial_t u(z) + C(z)u(z)=0\]
 is conjugated to a holomorphic curve. Here, $J,C$ are supposed to be $\mathbb{R}$-linear, sufficiently differentiable and $J^2=-Id$. Note that even solutions of the ``normal'' Floer equation (\ref{eqast}) do not satisfy this equation, as $X_H(p)\neq 0$ in general for any $p\in V$. However, the difference of two solutions $u, v$ of (\ref{eqast}) satisfies this equation, provided that $u$ and $v$ agree to infinite order at a point $p$. It follows then from the unique continuation for holomorphic maps that $u$ and $v$ coincide on an open neighborhood of $p$ and thus everywhere.\\
 Unique continuation can also be obtained from Aronszajn's Theorem (see \cite{DuffSal}, 2.3), which states that every (sufficiently regular) function $u$ on $B_\veps(0)$ is equal to 0 if it vanishes to infinite order at 0 and satisfies point-wise almost everywhere
 \[|\Delta u|\leq c\cdot\left(|u|+|\partial_s u|+ |\partial_t u|\right).\]
 Unfortunately both results, Carleman similarity principle and Aronszajn's Theorem, do not apply to solutions of the Rabinowitz-Floer equation, as it involves an integral and is hence not completely local. However, Bourgeois and Oancea generalized in \cite{BourOan3} Aronszajn's Theorem to situations like this and used it then to prove the following two Theorems \ref{uniquecont} and \ref{uniquecontpt}.\medskip\\
 To state the theorems, we write $I_h(s):=(s-h,s+h)\subset\mathbb{R}$ and $V_h(s,t):=(s-h,s+h)\times \linebreak (t-h,t+h)\subset\mathbb{R}\times S^1$ for $h>0$.
 \begin{theo}[\textbf{Unique Continuation}, cf. \cite{BourOan3}, Prop.\ 3.5]\label{uniquecont}~\\
  Let $v_i: I_h(s)\times S^1\rightarrow V,\; \eta_i: I_h(s)\rightarrow\mathbb{R},\; i=0,1,$ be two smooth functions satisfying the ($s$-dependent) Rabinowitz-Floer equation, i.e.\
  \begin{equation*}
  \left.\begin{aligned}
   \partial_s v + J_t(v,\eta)\big(\partial_t v -\eta X_H(v)\big)&=0\\
   \partial_s \eta + \textstyle\int_{S^1} H(v)dt &= 0
   \end{aligned}\right\}(\ref{eq3})\;\,\text{or}\quad \left.\begin{aligned}
   \partial_s v + J_t(v,\eta)\big(\partial_t v -\eta X_{H_s}(v)\big)&=0\\
   \partial_s \eta + \textstyle\int_{S^1} H_s(v)dt &= 0
   \end{aligned}\right\}(\ref{eqRabFlHom}).
  \end{equation*}
If $(v_0,\eta_0)$ and $(v_1,\eta_1)$ coincide at a point $p\in I_h(s)\times S^1$ to infinite order, then they coincide on $I_h(s)\times S^1$. In particular, this applies then $(v_0,\eta_0)$ and $(v_1,\eta_1)$ agree on some open set $U\subset I_h(s)\times S^1$.
 \end{theo}
\begin{theo}[cf. \cite{BourOan3}, Lem.\ 4.5, or \cite{FlHoSa}, Lem.\ 4.2]\label{uniquecontpt}~\\
Suppose that $U_i=(v_i,\eta_i),\, i=0,1$, are smooth functions on $I_{h_0}(s)\times S^1,\, h_0>0,$ satisfying the Rabinowitz-Floer equation (\ref{eq3}). Assume that
\begin{equation*}
 \begin{aligned}
  U_0(s_0,t_0)=U_1(s_0,t_0),\qquad\qquad
  \partial_s v_0(s_0,t_0)\neq 0,\qquad\qquad \partial U_1(s_0,t_0)\neq 0.
 \end{aligned}
\end{equation*}
Assume also that for any $0<h'\leq h_0$ there exists $0<h\leq h_0$ such that
for any $(s,t)\in V_h(s_0,t_0)$ there exists $(s',t)\in V_{h'}(s_0,t_0)$ such that $U_0(s,t)=U_1(s',t)$. Then $U_0=U_1$.
\end{theo}
For the following generalized results on injective points, let us consider the following situation: Suppose we have on $V$ a smooth (exact) symplectic symmetry $\sigma$ of finite order, i.e.\ $\sigma : V\rightarrow V$ is a diffeomorphism such that $\sigma^k=Id$ for some $k\in\mathbb{N}$ and $\sigma^\ast\lambda=\lambda$. Moreover, suppose that $H$ and $J$ are $\sigma$-invariant, i.e.\  $H(\sigma(p))=H(p)$ for all $p\in V$ and $\sigma^\ast J=J$. Note that this implies that for any solution $(v,\eta)$ of (\ref{eq3}) we have that $(\sigma\circ v,\eta)$ is also a solution of (\ref{eq3}). Let $V_{fix}:=\{p\in V\,|\,\sigma(p)=p\}$ denote the fixed point set of $\sigma$.
\begin{lemme}\label{notfixdense} Let $U=(v,\eta)$ be a solution of the Rabinowitz-Floer equation (\ref{eq3}). Suppose that $im(v)\not\subset V_{fix}$. Then, the following set is open and dense in $\mathbb{R}\times S^1$:
 \[F(U):=\{(s,t)\in\mathbb{R}\times S^1\,|\,v(s,t)\not\in V_{fix}\}.\]
\end{lemme}
\begin{proof}
It is easy to see that $F(U)$ is open in $\mathbb{R}\times S^1$, as we may write it  equivalently as
\[F(U):=\big\{(s,t)\in\mathbb{R}\times S^1 \,\big|\,dist.\big(v(s,t),(\sigma\circ v)(s,t)\big)>0\big\},\]
where the condition $dist.(\cdot,\cdot)>0$ is obviously open. To show that $F(U)$ is dense, we suppose the contrary. Then there exists an open set $W\subset\mathbb{R}\times S^1$ such that $v(W)\subset V_{fix}$. It follows that $(v,\eta)$ and $(\sigma\circ v,\eta)$ coincide on $W$ and hence by unique continuation that $(v,\eta)=(\sigma\circ v,\eta)$ everywhere. But this implies that $im(v)\subset V_{fix}$, a contradiction to the premise of the lemma.
\end{proof}
\begin{lemme}\label{Lemnoshift} Suppose that $(v,\eta)$ is a solution of (\ref{eq3}) with $\displaystyle\lim_{s\rightarrow\pm\infty}(v,\eta)=(v^\pm,\eta^\pm)\in\crita$. If $(\partial_s v,\partial_s \eta)\not\equiv (0,0)$, then there is no constant $s_0\in\mathbb{R}\setminus\{0\}$ such that $(v,\eta)$ is an $s_0$-shift of itself, i.e.\ it cannot hold for every $(s,t)\in\mathbb{R}\times S^1$ that
 \[\big(v(s+s_0,t),\eta(s+s_0)\big)=\big(v(s,t),\eta(s)\big).\]
 In the $\sigma$-symmetric case, there is also no constant $s_0\in\mathbb{R}\setminus\{0\}$ such that $(\sigma\circ v,\eta)$ is an $s_0$-shift of $(v,\eta)$. If $im(v)\not\subset V_{fix}$, then $s_0=0$ is also impossible.
\end{lemme}
\begin{proof}
 If there were such a constant $s_0\neq 0$, then $(v,\eta)$ would be $s_0$-periodic and hence
 \begin{align*}
  \big(v(s,t),\eta(s)\big) &= \big(v(s+k\cdot s_0,t),\eta(s+k\cdot s_0)\big)& &\forall\, k\in\mathbb{Z},\,\forall\,s,t\in\mathbb{R}\times S^1.\\
  &= \lim_{k\rightarrow \pm \infty} \big(v(s+k\cdot s_0,t),\eta(s+k\cdot s_0)\big)\\
  &= \big(v^\pm(t),\eta^\pm\big). 
 \end{align*}
This implies that $\partial_s (v,\eta)\equiv (0,0)$ -- a contradiction to our assumption.\pagebreak\\
 If there were such a constant $s_0$ in the $\sigma$-symmetric case, then we find by applying $\sigma$ iteratively $k$ times that $(v,\eta)$ would be $k\cdot s_0$ periodic. By repeating the same arguments as before, we find again a contradiction. If $im(\sigma)\not\subset V_{fix}$, then we have $\sigma\circ v\neq v$, which shows that $s_0=0$ is also impossible.
\end{proof}
Following \cite{FlHoSa} and \cite{BourOan3}, we define for a solution $U=(v,\eta)$ of (\ref{eq3}) with $\displaystyle\lim_{s\rightarrow\pm\infty} (v,\eta)=(v^\pm,\eta^\pm)\in\crita$ the set of regular points as
\[\mathcal{R}(U):=\left\{(s,t)\in\mathbb{R}\times S^1\;\left|\;\begin{aligned}\partial_s\big(v(s,t),\eta(s)\big)&\neq (0,0)\\ (v(s,t),\eta(s))&\neq (v^\pm(t),\eta^\pm)\\ (v(s,t),\eta(s))&\neq \big(v(s',t),\eta(s')\big), \forall s'\in\mathbb{R}\setminus\{s\}\end{aligned}\right.\right\}.\]
Together with Peter Uebele, \cite{Ueb}, we define in the $\sigma$-symmetric case when $im(v)\not\subset V_{fix}$ the set of symmetric regular  points by
\[\mathcal{S}_\sigma(U):=\left\{(s,t)\in\mathbb{R}\times S^1\;\left|\;\begin{aligned}\partial_s\big(v(s,t),\eta(s)\big)&\neq (0,0)\\ (v(s,t),\eta(s))&\neq (\sigma\circ v^\pm(t),\eta^\pm)\\ (v(s,t),\eta(s))&\neq \big(\sigma\circ v(s',t),\eta(s')\big),\forall s'\in\mathbb{R}\end{aligned}\right.\right\}.\]
Note that we require in particular for $(s,t)\in \mathcal{S}_\sigma(U)$ that $v(s,t)\neq (\sigma\circ v)(s,t)$.
\begin{prop}[cf. \cite{FlHoSa}, Thm.\ 4.3, or \cite{BourOan3}, Prop.\ 4.3]\label{regpts}~\\
 Assume that $\partial_s U\not\equiv 0$. Then, the set $\mathcal{R}(U)$ is open and dense in the non-empty open set $\{(s,t)\in \mathbb{R}\times S^1\,|\,\partial_s v(s,t)\neq 0\}$. In the symmetric case, the set $\mathcal{S}_\sigma(U)$ is likewise open and dense in the same set.
\end{prop}
\begin{proof}~\\
\underline{1) All conditions are open}
\begin{itemize}
 \item The set $\{(s,t)\in \mathbb{R}\times S^1\,|\,\partial_s v(s,t)\neq 0\}$ is clearly open. If it were empty, then $\partial_s v\equiv 0$ and $v(s,\cdot)=v^\pm(\cdot)\in\Sigma=H^{-1}(0)$ for all $s$. Then, the second equation of (\ref{eq3}) implies that $\partial_s \eta\equiv 0$ and hence $\partial_s U\equiv 0$, which contradicts the assumption of the proposition.
 \item The first and second condition for $\mathcal{R}(U)$ resp.\ $\mathcal{S}_\sigma(U)$ are clearly open. We need to show that the third condition for $\mathcal{R}(U)$ is open as well. Arguing by contradiction, we find a point $(s_0,t_0)\in\mathcal{R}(U)$, a sequence $(s^\nu,t^\nu)\rightarrow(s_0,t_0)$ and a sequence $s'^\nu\neq s^\nu$ such that $U(s'^\nu,t^\nu)=U(s^\nu,t^\nu)$. As $\partial_s U(s_0,t_0)\neq 0$, we can find $h>0$ such that $U(\cdot,t_0)$ is an embedding on $I_h(s_0)$ and $U(\cdot,t^\nu)$ is an embedding on $I_h(s^\nu)$ for $\nu$ large enough. Thus we can assume without loss of generality that $s'^\nu$ is bounded away from $s_0$ (otherwise $s'^\nu\in I_h(s^\nu)$ for $\nu$ large enough, a contradiction). Since $U$ converges at $\pm\infty$ to $(v^\pm,\eta^\pm)$ and $U(s_0,t_0)\neq(v^\pm(t_0),\eta^\pm)$ by assumption, we infer the existence of some $T>0$ such that $s'^\nu\in[-T,T]$ for all $\nu$. We can therefore extract a convergent subsequence, still denoted by $s'^\nu$, such that $s'^\nu\rightarrow s_0'\neq s_0$. Then $U(s_0',t_0)=U(s_0,t_0)$, which contradicts the assumption $(s_0,t_0)\in\mathcal{R}(U)$.
 \item For $\mathcal{S}_\sigma(U)$, the third condition can be written as
 \[dist\big(v(s,t),\eta(s)\big),(\sigma\circ v)(\mathbb{R},t),\eta(\mathbb{R})\big)>0,\]
 since $(v(s,t),\eta(s))\neq(v^\pm(t),\eta^\pm)$. This is clearly an open condition.
\end{itemize}
\underline{2) Density}\medskip\\
 It suffices to show for every $(s_0,t_0)\in\mathbb{R}\times S^1$ with $\partial_s v(s_0,t_0)\neq 0$ that it can be approximated by a sequence of points $(s^\nu,t^\nu)\in\mathcal{R}(U)$ resp.\ $\in\mathcal{S}_\sigma(U)$. As $im(\sigma)\not\subset V_{fix}$, we know by Lemma \ref{notfixdense} that any $(s_0,t_0)\in\mathbb{R}\times S^1$ can be approximated by a sequence $(s^\nu,t^\nu)$ satisfying $v(s^\nu,t^\nu)\not\in V_{fix}$. Hence we may assume without loss of generality that $(v_0,t_0)\not\in V_{fix}$. Thus, any sequence $(s^\nu,t^\nu)\rightarrow(s_0,t_0)$ satisfies $v(s^\nu,t^\nu)\not\in V_{fix}$ for $\nu$ large enough. Hence, we may assume for $S_\sigma(U)$ that $v(s,t)\neq(\sigma\circ v)(s,t)$. The remaining conditions for $\mathcal{S}_\sigma(U)$ are very similar to those of $\mathcal{R}(U)$ and the proof of the density of $\mathcal{R}(U)$ and $\mathcal{S}_\sigma(U)$ is similar as well. In order to give them both at the same time, we will write $W(s',t)=(w(s',t),\eta(s'))$, where $W$ and $w$ are either $U$ and $v$ or $\sigma\circ U$ and $\sigma\circ v$.\\
 As $\partial_s v(s_0,t_0)\neq 0$, we may choose $h>0$ so small that $\partial_s v\neq 0$ on $V_h(s_0,t_0)$ and $I_h(s_0)\rightarrow V, s\mapsto v(s,t)$ is an embedding for all $t\in I_h(t_0)$. Then $I_h(s_0)\rightarrow V\times\mathbb{R}, s\mapsto U(s,t)$ is a fortiori also an embedding for all $t\in I_h(t_0)$. Thus, every point $(s,t)\in V_h(s_0,t_0)$ can be approximated by a sequence $(s^\nu,t^\nu)$ satisfying $U(s^\nu,t^\nu)\neq(v^\pm(t^\nu),\eta^\pm)$ and $U(s^\nu,t^\nu)\neq(w^\pm(t^\nu),\eta^\pm)$. Hence we can assume without loss of generality that
 \[\forall (s,t)\in V_h(s_0,t_0)\; : \; U(s,t)\neq (v^\pm(t),\eta^\pm),(w^\pm(t),\eta^\pm).\tag{$\ast$}\]
 Let us denote $C(W):=\{(s,t)\in\mathbb{R}\times S^1\,|\,\partial_s W(s,t)=0\}$. As $\partial_s U\not\equiv 0$ this set has empty interior by unique continuation.\medskip\\
 \textbf{\textit{Claim:}} \textit{$(s_0,t_0)$ can be approximated by a sequence $(s^\nu, t_0)$ such that for all $\nu$ and all $s'\in\mathbb{R}\setminus\{s^\nu\}$ with $U(s^\nu,t_0)=W(s',t_0)$ holds that $(s',t_0)\not\in C(W)$.}\medskip\\
 Assuming the claim, we can suppose without loss of generality that for each $s'\in\mathbb{R}$ with $U(s_0,t_0)=W(s',t_0)$ holds that $(s', t_0)\not\in C(W)$. Moreover, after further diminishing $h>0$, we can assume without loss of generality that
 \[\forall (s,t)\in V_h(s_0,t_0), \forall s'\in\mathbb{R} \,:\, U(s,t)=W(s',t)\;\,\Rightarrow\;\,(s',t)\not\in C(W).\tag{$\ast\ast$}\]
 Indeed, if this would fail for all $h>0$, we could find a sequence $(s^\nu,t^\nu)\rightarrow(s_0,t_0)$ and a sequence $s'^\nu$ such that $(s'^\nu,t^\nu)\in C(W)$ and $U(s^\nu,t^\nu)=W(s'^\nu,t^\nu)$. Due to $\lim_{s\rightarrow \pm\infty}W(s,t) = (w^\pm(t),\eta^\pm)$ uniformly, we deduce from $(\ast)$ the existence of a constant $T>0$ such that $|s'^\nu|\leq T$. Thus, up to a subsequence, we have $s'^\nu\rightarrow s'\in [-T,T],\; t^\nu\rightarrow t_0$ and $U(s_0,t_0)=W(s',t_0)$ with $(s',t_0)\in C(W)$ contradicting our last assumption on $(s_0,t_0)$ obtained by the claim.\medskip\\
 \textit{\textbf{Proof of the claim:}} Let us choose a neighborhood $\mathcal{V}$ of $U(I_h(s_0),t_0)$ in $V\times\mathbb{R}$ of the form $I_h(s_0)\times \mathbb{R}^{2n}$. This is possible, as $s\mapsto v(s,t_0)$ is an embedding. Let $pr_1$ denote the projection to the first coordinate. Consider the function $f:=pr_1\circ W(\cdot,t_0)$ with 
 \[f: dom(f):= W(\cdot,t_0)^{-1}(\mathcal{V})\rightarrow I_h(s_0).\]
 Write $C(W)_{t_0}:=\{s\in\mathbb{R}\,|\,(s,t_0)\in C(W)\}$. Then $f\big(C(W)_{t_0}\cap dom(f)\big)$ is contained in the critical values of $f$. This is a nowhere dense set in $I_h(s_0)$ by Sard's Theorem and the claim follows.\medskip\\
 Now assume by contradiction the existence of a point $(s_0,t_0)$ satisfying $\partial_s v(s_0,t_0)\neq 0$ and $(\ast)$ and $(\ast\ast)$, which cannot be approximated by points in $\mathcal{R}(U)$ resp.\ $\mathcal{S}_\sigma(U)$. Then there exists an $0<\veps< h$ such that
 \[\forall\, (s,t)\in V_\veps(s_0,t_0)\;\; \exists\, s'\neq s\; : \; U(s,t)=W(s',t). \tag{$\ast\ast\ast$}\]
 As above, we find a constant $T>0$ such that $|s'|\leq T$ for all $s'$. This implies that for any $(s,t)\in V_\veps(s_0,t_0)$, there is only a finite number of values $s'\in\mathbb{R}$ such that $U(s,t)=W(s',t)$. If not, we could find an accumulation point $s'\in[-T,T]$ where $U(s,t)=W(s',t)$ and $\partial_s W(s',t)=0$, a contradiction with $(\ast\ast)$. Let $s_1,...\,, s_N\in [-T,T]$ be the points such that $U(s_0,t_0)=W(s_j,t_0), \, j=1,...\,,N$.\\
 Now we claim that for any $r>0$ there exists $\delta>0$ such that
 \[\forall\, (s,t)\in V_\delta(s_0,t_0)\;\exists\,(s',t)\in \bigcup_{j=1}^N V_r(s_j,t_0)\;\; :\;\; U(s,t)=W(s',t).\]
 If this would fail, we could find $r>0$ and a sequence $(s^\nu,t^\nu)\rightarrow (s_0,t_0)$ such that for all $\nu$ and for all $(s', t^\nu)\in \bigcup_{j=1}^N V_r(s_j,t_0)$ we have $U(s^\nu,t^\nu)\neq W(s',t^\nu)$. On the other hand by $(\ast\ast\ast)$, there exists $s'^\nu\in[-T,T]$ such that $U(s^\nu,t^\nu)=W(s'^\nu,t^\nu)$ and now in particular $|s'^\nu-s_j|\geq r$ for all $j=1,...\,,N$. Up to a subsequence we have $s'^\nu\rightarrow s'$ and $t^\nu\rightarrow t_0$ with $U(s_0,t_0)=W(s',t_0)$ and $s'\neq s_j,\; j=1,..,N$, a contradiction. Define
 \[\Sigma_j:=\left\{(s,t)\in \overline{V}_\delta(s_0,t_0)\, \left|\,\exists(s',t)\in \overline{V}_r(s_j,t_0)\,:\, U(s,t)=W(s',t)\right.\right\}.\]
 Then $\Sigma_j$ is closed and $\overline{V}_\delta(s_0,t_0)=\Sigma_1\cup...\cup\Sigma_N$. It follows from Baire's Theorem that one of the $\Sigma_j$, say $\Sigma_1$, has non-empty interior. Let $(\bar{s},\bar{t})\in int(\Sigma_1)$ and denote by $(\bar{s}',\bar{t})$ a preimage $W^{-1}(U(\bar{s},\bar{t}))$ in $V_r(s_1,t_0)$. Let $0<r_1<r$ be such that $V_{r_1}(\bar{s}',\bar{t})\subset V_r(s_1,t_0)$ and $0<\delta_1<\delta$ be such that $V_{\delta_1}(\bar{s},\bar{t})\subset\Sigma_1$ and such that for all $(s,t)\in V_{\delta_1}(\bar{s},\bar{t})$ there exists $(s',t)\in V_{r_1}(\bar{s}',\bar{t})$ with $U(s,t)=W(s',t)$. It follows from our construction that for all $0<h'\leq r_1$ there exists $0<h\leq \delta_1$ such that for all $(s,t)\in V_h(\bar{s},\bar{t}), $ there exists $(s',t)\in V_{h'}(\bar{s}',\bar{t})$ such that $U(s,t)=W(s',t)$. We can therefore apply Theorem \ref{uniquecontpt} with $(s_0,t_0)=(\bar{s},\bar{t}),\; U_0=U,\, U_1=W(\cdot+\bar{s}'-\bar{s})$ and $h_0=r_1$, to obtain $U_0=U_1$. This implies that $W$ is an $(\bar{s}'-\bar{s})$-shift of $U$, $\bar{s}'-\bar{s}\neq 0$, which contradicts Lemma \ref{Lemnoshift}.
\end{proof}
\begin{rem}
 Write $v(\pm\infty,t):=v^\pm(t)$ and $\overline{\mathbb{R}}:=\mathbb{R}\cup\{\pm\infty\}$. By applying Proposition \ref{regpts} to the symplectic symmetries $\sigma^1,\sigma^2,...\,,\sigma^k=id$, we find for a solution $(v,\eta)$ of (\ref{eq3}) with $im(v)\not\subset V_{fix}(\sigma^l)$ for $l=1,...\,,k-1$ that the set
 \[\mathcal{S}(U)=\left\{(s,t)\in\mathbb{R}\times S^1\;\left|\;\begin{aligned}(\partial_s v(s,t),\partial_s \eta(s))&\neq (0,0)&&\\
 (v(s,t),\eta(s))&\not\in \big(\sigma^l\circ v(\overline{\mathbb{R}},t),\eta(\overline{\mathbb{R}})\big),& l&=1,...\,,k\\
 \text{except}\quad (v(s,t),\eta(s))&=\big(\sigma^k\circ v(s,t),\eta(s)\big)&&\end{aligned}\right.\right\}\]
\end{rem}
is open and dense in $\{(s,t)\in\mathbb{R}\times S^1\,|\,\partial_s v(s,t)\neq 0\}$. This follows from Baire's Theorem as the above set is the finite intersection of the open and dense sets $\mathcal{R}(U)$ and $\mathcal{S}_\sigma(U),\mathcal{S}_{\sigma^2}(U),...\,,\mathcal{S}_{\sigma^{k-1}}(U)$.

\subsection{Transversality}\label{sec2.5}
In this subsection, we show that $\widehat{\mathcal{M}}(c^-,c^+,m)$ is a manifold for generic choices of $J$. We do so by describing this space as the zero-set of a Fredholm section $\mathcal{F}$ in a suitable Banach bundle $\mathcal{E}\rightarrow\mathcal{B}$. Then, we generalize this result to situations with a symplectic symmetry $\sigma$ of finite order and show again that $\widehat{\mathcal{M}}(c^-,c^+,m)$ is a manifold, now for generic choices of $J$ in the space of $\sigma$-symmetric almost complex structures.\bigskip\\
Suppose we have chosen a Morse function $h$ and a metric $g_h$ on $\crita$ such that $(h,g_h)$ is Morse-Smale. Let $\phi$ denote the gradient flow of $h$ with respect to $g_h$. It defines an $\mathbb{R}$-family of diffeomorphisms $T_h(t)\in \text{\textit{Diff}}\big(\crita\big)$ by
\[T_h(t)(x):=\phi^t(x).\]
A trajectory with $m$ cascades from $c^-$ to $c^+$, $c^\pm\in\text{crit}(h)$ is by Definition \ref{defntrajwithcasc} a tupel
\[(x,t)=\big((x_k)_{1\leq k\leq m},\,(t_k)_{1\leq k\leq m-1}\big),\]
where $x_k=(v_k,\eta_k)$ are non-constant $\mathcal{A}^H$-gradient trajectories and $t_k\geq 0$ non-negative real numbers satisfying some asymptotic and connectedness conditions. Using $T_h$, we can rewrite these conditions as follows:
\begin{align*}
 &(\text{\textbf{\textit{Asymptotics}}})& &\forall\, k\,\exists\, x_k^\pm\in \crita, s.t. \lim_{s\rightarrow \pm\infty} x_k = x_k^\pm\\ 
 &(\text{\textbf{\textit{Connectedness}}})& &\lim_{t\rightarrow -\infty} T_h(t)(x_1^-)=c^-,\; \lim_{t\rightarrow \infty} T_h(t)(x_m^+) = c^+,\; T_h(t_k)(x_k^+) = x_{k+1}^-.
\end{align*}
Note that the map $t\mapsto T_h(t)(x_k^+), t\in[0,t_k],$ is exactly the $h$-gradient trajectory $y_k$ of the original definition. We call a trajectory with $m$ cascades \textit{\textbf{stable}}, if for all $t_k$ holds $t_k>0$. In particular, trajectories with 0 or 1 cascade are always stable. We denote the space of stable trajectories with $m$ cascades from $c^-$ to $c^+$ by $\widehat{\mathcal{M}}_s(c^-,c^+,m)$ and the corresponding moduli space of unparametrized trajectories by $\mathcal{M}_s(c^-,c^+,m)$.\bigskip\\
The dimension of $\widehat{\mathcal{M}}_s(c^-,c^+,m)$ is given in terms of Morse- and Conley-Zehnder indices. In order to define the Conley-Zehnder index as simple as possible, we assume throughout this section that the following map between fundamental groups induced by inclusion is injective
\[ i_\ast : \pi_1(\Sigma)\rightarrow \pi_1(V).\]
 Moreover, we consider only contractible closed Reeb orbits. Alternatively, we could assume that $\Sigma$ is simply connected, i.e.\ $\pi_1(\Sigma)=0$. Under these assumptions, we can choose for any closed Reeb orbit $v\in\mathcal{P}(\alpha)=\crita$ a map $\bar{v}\in C^\infty(D,\Sigma)$ from the unit disc $D=\big\{z\in\mathbb{C}\,\big|\,|z|\leq 1\big\}$ to $\Sigma$ such that $\bar{v}(e^{2\pi it})=v(t)$. We call such a $\bar{v}$ a \textbf{\textit{capping}} for $v$. Given a capping, we define in Section \ref{1.4} the (transversal) Conley-Zehnder index $\mu(v,\bar{v})$ of the pair $(v,\bar{v})$.\\
 Now we are ready to state the following fundamental theorem.\pagebreak
\begin{theo}[Global Transversality Theorem]\label{theotrans}~\\
Given a Morse-Smale pair $(h,g_h)$ on $\crita$, there exists a set of second category $\mathcal{J}_{reg}$ of admissible families of smooth almost complex structures $J_t(\cdot,n)$ such that for every $J\in \mathcal{J}_{reg}$ the space of stable $\mathcal{A}^H$-trajectories $\widehat{\mathcal{M}}_s(c^-,c^+,m)$ has the structure of a finite dimensional manifold for every $c^-,c^+\in crit(h)$ and $m\in\mathbb{N}$. Its local dimension at a trajectory with cascades $(v,t)$ is given by
\begin{align*}
 \dim_{(v,t)}\,\widehat{\mathcal{M}}_s(c^-,c^+,m) =&\phantom{\,+\,}\Big(\mu_{CZ}(c^+,\bar{c}^+)+ind_h(c^+)-\frac{1}{2}\dim_{c^+}(\crita)\Big)\\
 &-\Big(\mu_{CZ}(c^-,\bar{c}^-)+ind_h(c^-)-\frac{1}{2}\dim_{c^-}(\crita)\Big)\\
 &+m-1+\sum_{k=1}^m 2c_1(\bar{v}_k^-\#v_k\#\bar{v}_k^+)\\
 =&\phantom{\,+\,} \mu(c^+,\bar{c}^+)-\mu(c^-,\bar{c}^-)+m-1+\sum_{k=1}^m 2c_1(\bar{v}_k^-\#v_k\#\bar{v}_k^+),\\
 \text{where }\qquad \mu(c,\bar{c}):=&\phantom{\,+\,} \mu_{CZ}(c,\bar{c})+ind_h(c)-\frac{1}{2}\dim_{c}(\crita)+\frac{1}{2}.
\end{align*}
Here, $\bar{v}_k^\pm$ are cappings for $v^\pm_k=\displaystyle \lim_{s\rightarrow \pm\infty} v_k$ and $\bar{c}^\pm$ are cappings for $c^\pm$ such that \linebreak $\bar{v}_k^+=\bar{v}_{k+1}^-\#\mathscr{C}_{k,k+1}$ where $\mathscr{C}_{k,k+1}: S^1\times[0,t_k]\rightarrow\mathcal{N}^{\eta_k}$ is the cylinder given by \linebreak $t\mapsto T_h(t)(v_k^+)$. Moreover $\bar{c}=\bar{v}_1\#\mathscr{C}_{c^-,1}$ and $\bar{v}_m^+=\bar{c}^+\#\mathscr{C}_{m,c^+}$, where $\mathscr{C}_{c^-,1}$ and $\mathscr{C}_{m,c^+}$ are the analogue cylinders at the both ends. Finally $\bar{v}_k^-\#v_k\#\bar{v}_k^+$ is the sphere obtained by capping the cylinder $v_k$ with $v_k^\pm$, $c_1$ is the first Chern-class of $TV$ and $\dim_{c^\pm}(\crita)$ is the local dimension of $\crita$ near $c^\pm$.
\end{theo}
\begin{rem}~
\begin{itemize}
 \item The condition $\bar{v}_k^+=\bar{v}_{k+1}^-\#\mathscr{C}_{k,k+1}$ guarantees $\mu_{CZ}(v_k^+,\bar{v}_k^+)=\mu_{CZ}(v_{k+1}^-,\bar{v}_{k+1}^-)$.
 \item For $m=0$, the dimension formula is not quite right. In this case $m-1$ has to be replaced by zero, so that the dimension is given by $ind_h(c^+)-ind_h(c^-)$ just as in ordinary Morse theory.
\end{itemize} 
\end{rem}
The proof of Theorem \ref{theotrans} will make use of Fredholm theory and the implicit function theorem on infinite-dimensional Banach manifolds (see \cite{DuffSal}, app. A). We start by describing the analytic setup that we need.\\
Fix connected components $C^\pm$ of $\crita$, fix $J$ and choose $\delta:=\delta(C^-,C^+)>0$ such that $\delta/p <\rho$, where $\rho$ is the constant for $(C^-,C^+)$ from Proposition \ref{asympexpconv}. We define
\[\mathcal{B}:=\mathcal{B}(C^-,C^+):=\mathcal{B}^{1,p}_\delta(C^-,C^+),\qquad (v,\eta)\in \mathcal{B},\; v:\mathbb{R}\times S^1\rightarrow V,\;\eta: \mathbb{R}\rightarrow\mathbb{R},\]
to be the Banach manifold of maps $v$ which are locally in $W^{1,p}$, converge at both ends in $C^\pm$ and are in the weighted Sobolev spaces $W^{1,p}_\delta$ (see Definition \ref{defnWeighSob}), i.e.\
\begin{enumerate}
 \item $(v,\eta)$ converges uniformly as $s\rightarrow \pm\infty$ to $(v^\pm,\eta^\pm)\in C^\pm$
 \item there exist tubular neighborhoods $U^\pm$ of $v^\pm$ together with smooth parametrizations $\psi^\pm : U^\pm\times\mathbb{R}\rightarrow S^1\times\mathbb{R}^{2n-1}\times \mathbb{R}$ such that for $s_0>0$ sufficiently large holds
 \begin{align*}
  &\psi^\pm\circ (v^\pm(t),\eta^\pm) = (t,0)\\
  &\psi^+\circ (v,\eta)-\psi^+\circ (v^+,\eta^+)\in W^{1,p}_\delta\big([s_0,\infty)\times S^1\big)\times W^{1,p}_\delta\big([s_0,\infty)\big)\\
  &\psi^-\circ(v,\eta)-\psi^-\circ (v^-,\eta^-)\in W^{1,p}_\delta\big((-\infty,-s_0]\times S^1\big)\times W^{1,p}_\delta\big((-\infty,-s_0]\big).
 \end{align*}
\end{enumerate}
Note that $\delta/p<\rho$ implies together with the asymptotic estimates of Section \ref{2.4} that any $\mathcal{A}^H$-gradient trajectory $(v,\eta)$ which converges at the ends to $C^\pm$ is a point in $\mathcal{B}$. Using the exponential map with respect to any metric on $V$, it is not difficult to see that $\mathcal{B}$ is a Banach manifold whose differentiable structure does not depend on this metric.  Observe that there are two natural smooth evaluation maps 
\[ev^\pm:\mathcal{B}\rightarrow C^\pm,\qquad\qquad ev^\pm(v,\eta)=(v^\pm,\eta^\pm).\]
Let $(v,\eta)\in\mathcal{B}$. The tangent space $T_{(v,\eta)}\mathcal{B}$ at $(v,\eta)$ can be identified with tuples
\[(\mathfrak{v},\hat{\eta})=(\mathfrak{v}_0,\hat{\eta},\mathfrak{v}^-,\mathfrak{v}^+)\in W^{1,p}_\delta(\mathbb{R}\times S^1, v^\ast TV)\oplus W^{1,p}_\delta(\mathbb{R},\mathbb{R})\oplus T_{(v^-,\eta^-)} C^-\oplus T_{(v^+,\eta^+)} C^+.\label{nonnatural}\]
The first two summands are naturally given as $\ker d(ev^-)\cap \ker d(ev^+)$, while the identification of the complement with $T_{(v^-,\eta^-)} C^-\oplus T_{(v^+,\eta^+)} C^+$ depends on the choice of submanifold charts for $C^\pm$.\\
Let $\mathcal{E}$ be the Banach bundle over $\mathcal{B}$ whose fiber at $(v,\eta)\in\mathcal{B}$ is given by
\[\mathcal{E}_{(v,\eta)}:= L^p_\delta(\mathbb{R}\times S^1,v^\ast TV)\times L^p_\delta(\mathbb{R},\mathbb{R}).\]
Consider the $J$-depending section
\[\mathcal{F}:\mathcal{B}\rightarrow\mathcal{E},\qquad (v,\eta)\mapsto\big(\partial_s(v,\eta)-\nabla\mathcal{A}^H(v,\eta)\big)=\begin{pmatrix}\partial_s v + J_t(v,\eta)\big(\partial_t v- \eta X_H(v)\big)\\\partial_s \eta  +\int_0^1H(v)dt\end{pmatrix}.\]
Note that $\mathcal{F}(v,\eta)=0$ means that $(v,\eta)$ satisfies the Rabinowitz-Floer equation (\ref{eq3}) for an $\mathcal{A}^H$-gradient trajectory. Thus, we find that $\mathcal{F}^{-1}(0)$ is exactly the space of all $\mathcal{A}^H$-gradient trajectories which converge at both ends in $C^\pm$. Let us denote by $\widehat{\mathcal{M}}(C^-,C^+)$ the space of all $\mathcal{A}^H$-gradient trajectories $(v,\eta)$ with asymptotics $(v^\pm,\eta^\pm)$ on $C^\pm$. Then we have apparently that $\widehat{\mathcal{M}}(C^-,C^+)=\mathcal{F}^{-1}(0)$.
\begin{theo}[Local Transversality Theorem]\label{theoloctrans}~\\
 There exists a set $\mathcal{J}_{reg}$ of second category of admissible families of almost complex structures $J_t(\cdot, n)$, such that for all $J\in \mathcal{J}_{reg}$ holds that the space $\widehat{M}(C^-, C^+)$ is a manifold for any connected components $C^\pm\subset\crita$. Its local dimension at $(v,\eta)\in \widehat{M}(C^-, C^+)$ is given by
 \[\dim_{(v,\eta)}\,\widehat{M}(C^-, C^+) = \mu_{CZ}(v^+,\bar{v}^+)-\mu_{CZ}(v^-,\bar{v}^-) +2c_1(\bar{v}^-\#v\#\bar{v}^+) + \frac{\dim\, C^- + \dim\, C^+}{2},\]
 where $\bar{v}^-, \bar{v}^+$ are cappings for $v^\pm = \displaystyle \lim_{s\rightarrow \pm \infty} v$.
\end{theo}
\begin{proof}
We will show for all $J\in\mathcal{J}_{reg}$ that $\mathcal{F}$ and the zero section in $\mathcal{E}$ intersect transversally. Then it follows (as in the finite dimensional case) from the implicit function theorem that $\mathcal{F}^{-1}(0)$ is a manifold. The proof is split in 3 parts
 \begin{enumerate}
  \item First, we show that the vertical differential $D_{(v,\eta)}$ of $\mathcal{F}$ at $(v,\eta)\in\mathcal{F}^{-1}(0)$ given by
  \begin{align*}
   D_{(v,\eta)}:&=D^V\mathcal{F}(v,\eta) : T_{(v,\eta)}\mathcal{B}\rightarrow \mathcal{E}_{(v,\eta)}\\
   D_{(v,\eta)}:&\phantom{\,\rightarrow\,} W^{1,p}_{\delta}(\mathbb{R}\times S^1, v^\ast TV)\times W^{1,p}_{\delta}(\mathbb{R},\mathbb{R})\times T_{v^-}C^-\times T_{v^+} C^+ \\
   &\rightarrow L^{1,p}_{\delta}(\mathbb{R}\times S^1, v^\ast TV)\times L^{1,p}_{\delta}(\mathbb{R},\mathbb{R})
  \end{align*}
is a Fredholm operator of index
\[ind\big(D_{(v,\eta)}\big)=\mu_{CZ}(v^+,\bar{v}^+)-\mu_{CZ}(v^-,\bar{v}^-)+2c_1(\bar{v}^-\#v\#\bar{v}^+)+\frac{\dim\,C^-+\dim\, C^+}{2}.\]
\begin{proof}~\\
We consider the restriction $\bar{D}_{(v,\eta)}$ of $D_{(v,\eta)}$ to $W^{1,p}_\delta(\mathbb{R}\times S^1,v^\ast TV)\times W^{1,p}_\delta(\mathbb{R},\mathbb{R})$. Then $\bar{D}_{(v,\eta)}$ is the linearization (the first variation) of the differential operator in the\linebreak Rabinowitz-Floer equation (\ref{eq3}). Hence it may be written as
\[\bar{D}_{(v,\eta)} = \partial_s + A(s),\]
where $A^\pm:=\displaystyle \lim_{s\rightarrow \pm\infty} A(s) = \nabla^2\mathcal{A}^H(v^\pm,\eta^\pm)$ and $A(s)\left(\begin{smallmatrix}\mathfrak{v}\\\hat{\eta}\end{smallmatrix}\right)$ is given in (\ref{secvar}) as
\[ A(s)\left(\begin{smallmatrix}\mathfrak{v}\\\hat{\eta}\end{smallmatrix}\right)=\begin{pmatrix}
  -\Big(\nabla_{\mathfrak{v}}J+\hat{\eta}(\partial_n J)\Big)(\dot{v}-\eta X_H)- J\Big(\nabla_{(\dot{v}-\eta X_H)}\mathfrak{v}-\eta[\mathfrak{v}, X_H]-\hat{\eta} X_H\Big)\\\Big.
  -{\textstyle \int_0^1 dH(\mathfrak{v})dt}
 \end{pmatrix}\]
Let $\gamma_\delta(s)=e^{\delta\cdot\beta(s)\cdot s}$ be the weight functions from the definition of $W^{1,p}_\delta$. They define a continuous isomorphism
\[\phi: W^{1,p}_\delta\rightarrow W^{1,p},\quad f\mapsto \gamma_\delta \cdot f.\]
Let $\tilde{D}_{(v,\eta)}$ be the conjugated operator between the unweighted Sobolev spaces, i.e.\
\begin{align*}
 &&\tilde{D}_{(v,\eta)}: W^{1,p}(\mathbb{R}\times S^1,v^\ast TV)\times W^{1,p}(\mathbb{R},\mathbb{R})&\rightarrow L^p(\mathbb{R}\times S^1,v^\ast TV)\times L^p(\mathbb{R},\mathbb{R})\\
 &\text{with}& \tilde{D}_{(v,\eta)} &= \phi\bar{D}_{(v,\eta)}\phi^{-1}.
\end{align*}
As $\tilde{D}_{(v,\eta)}$ and $\bar{D}_{(v,\eta)}$ are conjugated, they are simultaneously Fredholm and if so, they have the same index. For $\zeta=(\mathfrak{v}_0,\hat{\eta})\in W^{1,p}(\mathbb{R}\times S^1,v^\ast TV)\times W^{1,p}(\mathbb{R},\mathbb{R})$ we calculate
\begin{align*}
 \tilde{D}_{(v,\eta)}(\zeta) = \phi\bar{D}_{(v,\eta)}\phi^{-1}(\zeta) &= \phi\bar{D}_{(v,\eta)}(\gamma_{-\delta}\cdot \zeta)\\
  &= \phi\big(\partial_s(\gamma_{-\delta}\cdot\zeta)+ A(s)(\gamma_{-\delta}\cdot \zeta)\big)\\
  &= \gamma_\delta\Big(\gamma_{-\delta} (\partial_s\zeta)+\gamma_{-\delta}\big(\delta\beta(s)+\delta\partial_s\beta(s)s\big)\zeta+\gamma_{-\delta}\cdot A(s)\zeta\Big)\\
  &= \partial_s\zeta+\Big[A(s)+\delta\big(\beta(s)+\partial_s\beta(s)s\big)\cdot Id\Big]\zeta.
\end{align*}
Hence, $\tilde{D}_{(v,\eta)}$ is given by $\tilde{D}_{(v,\eta)} = \partial_s +  B(s)$, where $B(s)$ is the operator \linebreak $B(s) = A(s) + \delta\big(\beta(s)+\partial_s\beta(s)s\big)Id$. Then $B^\pm := \displaystyle \lim_{s\rightarrow \pm\infty} B(s) = A^\pm \pm \delta\cdot Id$.\pagebreak\\
As $A^\pm=\nabla^2\mathcal{A}^H(v^\pm,\eta^\pm)$, they are symmetric and so are $B^\pm$. Moreover, for $|\delta|$ smaller than the absolute value of the smallest non-zero eigenvalue of $A^\pm$, it follows that the $B^\pm$ are invertible. If $p=2$, it is shown in \cite{RoSa2} that $\tilde{D}$ is then a Fredholm operator. For general $p$, the ideas in \cite{Sal} can be used to give a proof. The index formula for $D_{(v,\eta)}$ is proved in \cite{FraCie}, Sec.\ 4.1.
\end{proof}
\item Let $\mathcal{J}^\ell:=\mathcal{J}^\ell(\rho)$ be the open subset in the Banach manifold of admissible \linebreak 1-periodic families of complex structures $J_t(\cdot, n)$ of class $C^\ell$ for which we may choose $\rho$ as the constant from Proposition \ref{asympexpconv}\footnote{Although $\rho$ is arbitrary close to the absolute smallest non-zero eigenvalue of $\nabla^2\mathcal{A}^H$, it cannot be chosen globally for all $J$, as $\nabla^2\mathcal{A}^H$ depends on the metric $g$, which itself depends on $J$.}. We consider the universal section 
\[\mathcal{F}:\mathcal{J}^\ell\times\mathcal{B}\rightarrow\mathcal{E},\qquad \mathcal{F}(J,(v,\eta)) = \begin{pmatrix} \partial_s v + J_t(v,\eta)\big(\partial_t v - \eta X_H(v)\big)\\ \partial_s\eta + \int_0^1H(v)dt\end{pmatrix}\]
and we prove that the universal moduli space
\[\mathcal{U}^\ell:=\mathcal{U}^\ell(C^-,C^+):=\left\{\left.\big(J,(v,\eta)\big)\in\mathcal{J}^\ell\times\mathcal{B}\,\right|\,\mathcal{F}(J,(v,\eta)) = 0\right\}\]
is a separable Banach manifold of class $C^\ell$, as $\mathcal{F}$ intersects the zero section transversally. For that, it suffices to show that the vertical differential
\begin{align*}
 && &D_{(J,(v,\eta))} : T_J\mathcal{J}^\ell\times T_{(v,\eta)}\mathcal{B}\rightarrow \mathcal{E}_{(v,\eta)}\\
 &\text{given by}& &D_{(J,(v,\eta))}\Big(Y,(\mathfrak{v},\hat{\eta})\Big) = D_{(v,\eta)}(\mathfrak{v},\hat{\eta})+\begin{pmatrix} Y_t(v,\eta)\cdot(\partial_t v - \eta X_H)\\ 0\end{pmatrix}
\end{align*}
is surjective for every $(J,(v,\eta))\in\mathcal{U}^\ell$. The tangent space $T_J\mathcal{J}^\ell$ consists of matrix valued maps $Y: S^1\times\mathbb{R}\rightarrow End(TV)$ of class $C^\ell$ satisfying the conditions
\begin{enumerate}
 \item[(1)] $\omega(Yv,w)+\omega(v,Yw) = 0 \qquad \forall\, v,w\in TV$
 \item[(2)] $J_t(p,n)Y_t(p,n)+Y_t(p,n)J_t(p,n) = 0 \qquad \forall\, p\in V, n\in\mathbb{R}, t\in S^1$
 \item[(3)] $Y_t(p,n) = \left(\begin{smallmatrix} \textstyle Y_\xi & 0 \\ 0 & 0 \end{smallmatrix}\right)$ on the cylindrical end of $V$, i.e.\ for $p\in[R,\infty)\times M$. Here, $Y_\xi$ is independent of $n$ and $t$ and the block structure is with respect to the splitting $\xi\oplus span(R_\lambda,Y_\lambda)$
 \item[(4)] $\displaystyle\sup_{n\in\mathbb{R}} ||Y_t(\cdot, n)||_{C^\ell}<\infty$.
\end{enumerate}
We have shown in 1.) that $D_{(v,\eta)} = D_{(J,(v,\eta))}|_{0\times T_{(v,\eta)}\mathcal{B}}$ is a Fredholm operator and has hence a finite dimensional cokernel. Therefore, $D_{J,(v,\eta)}$ has a finite dimensional cokernel as well and thus a closed range. Hence, for surjectivity it only remains to prove that the range is dense.\\
For that, consider $q$ with $1/p+1/q=1$. The dual space $\big(\mathcal{E}_{(v,\eta)}\big)^\ast$ is then given by
\[\big(\mathcal{E}_{(v,\eta)}\big)^\ast = L^q_{-\delta}\big(\mathbb{R}\times S^1,v^\ast TV\big)\times L^q_{-\delta}\big(\mathbb{R},\mathbb{R}\big).\]
To show the density of the range, it suffices to show that any $(\zeta,\hat{\mu})\in\big(\mathcal{E}_{(v,\eta)}\big)^\ast$ which annihilates the range is identically zero, i.e.\ we show that $(\zeta,\hat{\mu})=0$, if for all $J\in Y\in T_J\mathcal{J}^\ell$, $\mathfrak{v}\in W^{1,p}_\delta(\mathbb{R}\times S^1,v^\ast TV)$ and $\hat{\eta}\in W^{1,p}_\delta(\mathbb{R},\mathbb{R})$ holds that
\begin{align*}
 \mspace{-13mu}\int_{\mathbb{R}\times S^1}\left\langle D_{J,(v,\eta)}(Y,\mathfrak{v},\hat{\eta}),(\zeta,\hat{\mu})\right\rangle ds\,dt = 0 \Leftrightarrow  \begin{cases}\int\left\langle D_{(v,\eta)}(\mathfrak{v},\hat{\eta}),(\zeta,\hat{\mu})\right\rangle ds\,dt &= 0\\
 \int\left\langle Y(\partial_t v- \eta X_H),\zeta\right\rangle ds\,dt &=0.
 \end{cases}\tag{$\ast$}\label{eqasttrans}
\end{align*}
The first equation implies that $(\zeta,\hat{\mu})$ is a weak solution of $D_{(v,\eta)}^\ast (\zeta,\hat{\mu}) = 0$, where $D_{(v,\eta)}^\ast$ is the formal adjoint of $D_{(v,\eta})$ which is of the same form as $D_{(v,\eta)}$. As $D_{(v,\eta)}$ is elliptic, so is $D_{(v,\eta)}^\ast$ and it follows from elliptic regularity that hence $(\zeta,\hat{\mu})$ is of class $C^\ell$ and has the unique continuation property (in the sense of Theorem \ref{uniquecont}).\\ 
As $(v,\eta)$ is non-constant, we find that $\partial_s v\neq 0$. In Proposition \ref{regpts}, we showed that the set $\mathcal{R}(U)$ of regular points of $U=(v,\eta)$ is open and dense in the open non-empty set $\{(s,t)\in\mathbb{R}\times S^1\,|\,\partial_s(v,\eta)\neq 0\}$. Recall that $\mathcal{R}(U)$ is given by
\[\mathcal{R}(U):=\left\{(s,t)\in\mathbb{R}\times S^1\;\left|\;\begin{aligned}(\partial_s v(s,t),\partial_s \eta(s))&\neq (0,0)\\ (v(s,t),\eta(s))&\neq (v^\pm(t),\eta^\pm)\\ (v(s,t),\eta(s))&\neq \big(v(s',t),\eta(s')\big), \forall s'\in\mathbb{R}\setminus\{s\}\end{aligned}\right.\right\}.\]
Now, we consider the following open subset of $\mathcal{R}(U)$:
\[\Omega:=\left\{(s,t)\in\mathcal{R}(U)\,\left|\, \Big. v(s,t)\not\in crit(H)\right.\right\}.\]
As $\Sigma=H^{-1}(0)$ is a regular level set of $H$ and $\displaystyle\lim_{s\rightarrow \pm\infty}v(s,\cdot)\in\Sigma$, it follows that $\Omega$ is also a non-empty open set. We will show that $(\zeta,\hat{\mu})$ vanishes identically on $\Omega$ and hence by the unique continuation property everywhere. Note that $H$ is constant on $[R,\infty)\times M$ and hence $[R,\infty)\times M\subset crit H$. Thus $v(\Omega)\subset V\setminus\big([R,\infty)\times M\big)$ and hence we have all freedom to perturb $J$ on $\Omega$, i.e.\ we do not have to pay attention to condition (3).\bigskip\\
First, we prove that $\zeta\equiv 0$ on $\Omega$. Suppose by contradiction that there exists $(s_0,t_0)\in\Omega$ such that $\zeta(s_0,t_0)\neq 0$. Set $p:=v(s_0,t_0)$ and $n:=\eta(s_0)$ and choose a linear map $Y_p: T_pV\rightarrow T_pV$ such that at the point $(t_0,p,n)$ condition (1) and (2) for $Y\in T_J\mathcal{J}$ are satisfied and
\[\omega_p\big(Y_pJ\partial_s v(s_0,t_0),\zeta(s_0,t_0)\big)>0.\]
See for instance \cite{DuffSal}, Lem.\ 3.2.2, for an explicit construction of such a $Y_p$. Now choose an element $\widetilde{Y}\in T_J \mathcal{J}^\ell$ such that $\widetilde{Y}_{t_0}(p,n)=Y_p$. As $(s_0,t_0)$ belongs to $\mathcal{R}(U)$, one can choose a smooth cutoff function $\beta: S^1\times V\times \mathbb{R}\rightarrow [0,1]$ supported near $\big(t_0,v(s_0,t_0),\eta(s_0)\big) = (t_0,p,n)$ such that for $Y:=\beta\cdot \widetilde{Y}\in T_J\mathcal{J}^\ell$ we have
\[\int_{\mathbb{R}\times S^1}\big\langle YJ\partial_s v,\zeta\big\rangle ds\,dt>0.\]
This contradicts the second equation in $(\ast)$ and hence $\zeta$ has to vanish on $\Omega$.\pagebreak\\
Secondly, we prove that also $\hat{\mu}\equiv 0$ on $\Omega$. Suppose that $\mathfrak{v}\in W^{1,p}_\delta(\mathbb{R}\times S^1,v^\ast TV)$ is supported in $\Omega$. By the first step, we may assume that $\zeta \equiv 0$ on $\Omega$ and this implies
\begin{align*}
 \int_{\mathbb{R}\times S^1}\Big\langle D_{(v,\eta)}(\mathfrak{v},0),(\zeta,\hat{\mu})\Big\rangle ds\,dt &=\phantom{-} \int_{\Omega}\Big\langle D_{(v,\eta)}(\mathfrak{v},0),(0,\hat{\mu})\Big\rangle ds\,dt\\
  &\overset{(\ref{secvar})}{=} -\int_{\Omega}\hat{\mu}(s)\cdot dH\big(v(s,t)\big)\big[\mathfrak{v}(s,t)\big] \, ds\,dt.
\end{align*}
Now if there exists $(s_0,t_0)\in \Omega$ such that $\hat{\mu}(s_0,t_0)\neq 0$, then we can find a vector $\mathfrak{v}_p\in T_pV$ satisfying $dH_p(\mathfrak{v}_p)<0$, as $v^{-1}(crit H)\cap\Omega = \emptyset$. Using the above equation, we find that a suitable extension $\mathfrak{v}$ of $\mathfrak{v}_p$ gives 
\[\int_{\mathbb{R}\times S^1}\Big\langle D_{(v,\eta)}(\mathfrak{v},0),(0,\hat{\mu})\Big\rangle ds\,dt>0.\]
Thus, $\hat{\mu}$ also has to vanish on $\Omega$. Now, it follows from the unique continuation property that $(\zeta,\hat{\mu})=0$ everywhere. Therefore, $D_{J,(v,\eta)}$ is surjective for all $(J,v,\eta)\in\mathcal{U}^\ell$. Thus, $\mathcal{F}$ is transversal to the zero section in $\mathcal{E}$ and it follows from the implicit function theorem on Banach spaces due to Smale (see for example \cite{DuffSal}, App.\ A) that $\mathcal{U}^\ell =\mathcal{F}^{-1}(0)$ is a separable $C^\ell$-Banach manifold.
\item We prove the theorem. Consider the projection
\[\pi: \mathcal{U}^\ell\rightarrow\mathcal{J}^\ell,\qquad (J,v,\eta)\mapsto J.\]
Its differential at a point $(J,v,\eta)$ is just the projection
\[d\pi (J,v,\eta) : T_{J,(v,\eta)}\,\mathcal{U}^\ell\rightarrow T_J\mathcal{J}^\ell,\qquad (Y,\xi,\hat{\eta})\mapsto Y.\]
The kernel of $d\pi(J,v,\eta)$ is $0\times (\ker D_{(v,\eta)})$ as $T_{J,(v,\eta)}\,\mathcal{U}^\ell=\ker D_{J,(v,\eta)}$ and $D_{J,(v,\eta)}$ restricted to $0\times T\mathcal{B}$ is $D_{(v,\eta)}$. Moreover, it follows from linear algebra that $im(d\pi)$ has the same (finite) codimension as $im(D_{(v,\eta)})$. Indeed, if $Y\subset X$ is a subspace and $f:X\rightarrow Z$ a linear map, then there exists a natural ismorphism 
\[\raisebox{.2em}{$f(X)$}\left/\raisebox{-.2em}{$f(Y)$}\right.\cong\raisebox{.2em}{$X$}\left/\raisebox{-.2em}{$(\ker f + Y)$}\right..\]
Applied to our situation, we have on one hand
\begin{align*}
 \text{coker}\,d\pi = \raisebox{.2em}{$T\mathcal{J}^\ell$}\left/\raisebox{-.2em}{$d\pi\big(T\,\mathcal{U}^\ell\big)$}\right.&=\raisebox{.2em}{$d\pi\big(T\mathcal{J}^\ell\times T\mathcal{B}\big)$}\left/\raisebox{-.2em}{$d\pi\big(\ker D_{J,(v,\eta)}\big)$}\right.\\
  &\cong \raisebox{.2em}{$T\mathcal{J}^\ell\times T\mathcal{B}$}\left/\raisebox{-.2em}{$\big(0\times T\mathcal{B}+\ker D_{J,(v,\eta)}\big)$}\right.,
 \end{align*}
 and as $D_{J,(v,\eta)}:T\mathcal{J}^\ell\times T\mathcal{B}\rightarrow \mathcal{E}$ is surjective, we have on the other hand
 \begin{align*}
  \text{coker}\,D_{(v,\eta)} = \raisebox{.2em}{$T\mathcal{E}$}\left/\raisebox{-.2em}{$D_{(v,\eta)}\big(T\mathcal{B}\big)$}\right. &=\raisebox{.2em}{$D_{J,(v,\eta)}\big(T\mathcal{J}^\ell\times T\mathcal{B}\big)$}\left/\raisebox{-.2em}{$D_{J,(v,\eta)}\big(0\times T\mathcal{B}\big)$}\right.\\
  &\cong\raisebox{.2em}{$T\mathcal{J}^\ell\times T\mathcal{B}$}\left/\raisebox{-.2em}{$\big(0\times T\mathcal{B}+\ker D_{J,(v,\eta)}\big)$}\right..
\end{align*}
Thus, $d\pi(J,v,\eta)$ is a Fredholm operator of the same index as $D_{(v,\eta)}$. In particular, $\pi$ is a Fredholm map and it follows from the Sard-Smale theorem for $\ell$ sufficiently large, that the set $\mathcal{J}^{l,C^-,C^+}_{reg}$ of regular values of $\pi$ is a set of second category.\pagebreak\\
Note that $J\in \mathcal{J}^{l,C^-,C^+}_{reg}$ is a regular value of $\pi$ exactly if $D_{(v,\eta)}$ is surjective for every $(v,\eta)\in\mathcal{F}^{-1}(0)$. In other words, for every $J\in\mathcal{J}^{l,C^-,C^+}_{reg}$ is $\widehat{\mathcal{M}}(C^-,C^+)=\pi^{-1}(J)$ a manifold of dimension $ind(D_{(v,\eta)})$, again by the implicit function theorem.\\
It is possible to show that the set $\mathcal{J}_{reg}^{C^-,C^+}(\rho):=\mathcal{J}^{l,C^-,C^+}_{reg}\cap \mathcal{J}^\infty(\rho)$ of smooth regular $J$ is of second category in $\mathcal{J}(\rho)$ with respect to the $C^\infty$-topology via a standard trick due to Taubes (see \cite{DuffSal}, Thm.\ 3.1.6 or \cite{Fra}, Thm.\ A.13 for details).\\
We obtain a set of second category within all admissible almost complex structures by the union
\[\mathcal{J}_{reg}^{C^-,C^+}:=\bigcup_{\rho>0}\mathcal{J}_{reg}^{C^-,C^+}(\rho).\]
 The countable intersection 
\[\mathcal{J}_{reg}:=\bigcap_{C^\pm\subset\, crit(\mathcal{A}^H)} \mathcal{J}_{reg}^{C^-,C^+}\]
 is still a set of second category in the set of all admissible almost complex structures and has the property that for all $J\in \mathcal{J}_{reg}$ and all connected components \linebreak $C^-,C^+\subset\crita$ holds that $\widehat{\mathcal{M}}(C^-,C^+)$ is a finite dimensional manifold. \qedhere
 \end{enumerate}
\end{proof}
Now we can give the proof of the Global Transversality Theorem \ref{theotrans}.
\begin{proof}~
 \begin{enumerate}
  \item[0.] If there are 0 cascades, the theorem follows from ordinary Morse homology, as $(h,g)$ is a Morse-Smale pair. Hence we may assume that $m\geq 1$.
  \item Let $(x,t)\in\widehat{\mathcal{M}}_s(c^-,c^+,m)$ be an arbitrary stable trajectory with $m$ cascades \linebreak $x_k=(v_k,\eta_k)$ passing through connected components $C_k\subset \crita, 0\leq k\leq m$, i.e.\
  \[\begin{aligned}\lim_{s\rightarrow -\infty} x_k&=x_k^-=(v_k^-,\eta_k^-)\in C_{k-1}\\ \lim_{s\rightarrow +\infty} x_k &= x_k^+=(v_k^+,\eta_k^+)\in C_k\end{aligned}\qquad 1\leq k\leq m.\]
  In terms of Theorem \ref{theoloctrans}, we may write this as $x_k=(v_k,\eta_k)\in\widehat{\mathcal{M}}(C_{k-1},C_k)$. Let $\mathcal{U}^\ell(C_{k-1},C_k)$ denote the universal moduli space of $\mathcal{A}^H$-gradient trajectories between $C_{k-1}$ and $C_k$ from the proof of Theorem \ref{theoloctrans}, i.e.\
  \begin{align*}
   &&\mathcal{U}^\ell(C_{k-1},C_k)&:=\left\{(J,(v,\eta))\in \mathcal{J}^\ell\times\mathcal{B}(C_{k-1},C_k)\,\left|\,\mathcal{F}(J,(v,\eta))=0\Big.\right.\right\}
  \end{align*}
  and consider the $C^\ell$-Banach manifold
  \[\widetilde{\mathcal{U}}:=\mathcal{U}^\ell(C_0,C_1)\times...\times\mathcal{U}^\ell(C_{m-1},C_m)\times\big(\mathbb{R}^+\big)^{m-1}.\]
  We define the universal moduli space $\mathcal{U}^\ell_m:=\mathcal{U}^\ell_m(C_0,...\,,C_m)$ for trajectories with $m$ cascades passing through $C_0,...\,,C_m$ to be the sub-Banach manifold of $\widetilde{\mathcal{U}}$, where each factor has the same almost complex structure $J$, i.e.\ $\mathcal{U}^\ell_m$ consists of tuples
  \[\bigg((J,x_1),...\,,(J,x_m),t_1,...\,,t_{m-1}\bigg),\qquad (J,x_k)\in\mathcal{U}(C_{k-1},C_k),\;t_k\in\mathbb{R}^+,\]
  with $J$ fixed for $1\leq k\leq m$.\pagebreak\\
  Recall that we have for each $k$ the evaluation maps
  \begin{align*}
   ev^-: \widehat{\mathcal{M}}(C_{k-1},C_k)&\rightarrow C_{k-1},& x_k&\mapsto x_k^-\\
   ev^+: \widehat{\mathcal{M}}(C_{k-1},C_k)&\rightarrow C_k,& x_k&\mapsto x_k^+.
  \end{align*}
  Moreover, recall that we denote by $T_h(t_k)$ the time $t_k$ gradient flow of $h$ on $\crita$ and that we have for $(x,t)\in\widehat{\mathcal{M}}_s(c^-,c^+,m)$ the asymptotic and connectedness conditions
  \begin{align*}
    T_h(t_k)(\underbrace{x_k^+}_{ev^+(x_k)})=\underbrace{x_{k+1}^-}_{ev^-(x_{k+1})},\,1\leq k\leq m-1,
   \qquad&\text{and}& x_1^- = ev^-(x_1)&\in W^u(c^-)\\
   &\text{and}& x_m^+ = ev^+(x_m)&\in W^s(c^+).
  \end{align*}
  Here, $W^u(c^-)$ and $W^s(c^+)$ are the unstable/stable manifolds of $c^\pm$ with respect to the Morse-Smale pair $(h,g_h)$ on $C_0$ resp.\ $C_m$. Write for the moment $T_h^{t_k}$ instead of $T_h(t_k)$ and consider the following map 
  \begin{gather*}
   \begin{aligned}\psi:\qquad\quad \mathcal{U}^\ell_m(C_0,C_1,...\,,C_m)&\rightarrow C_0\times (C_1)^2\times...\times (C_{m-1})^2\times C_m\\
   \big(J,x_1,...\,,x_m,t_1,...\,,t_{m-1}\big)&\mapsto\end{aligned}\\
   \Big(ev^-(x_1),T_h^{t_1}\big(ev^+(x_1)\big),...\,,ev^-(x_{m-1}),T_h^{t_{m-1}}\big(ev^+(x_{m-1})\big),ev^-(x_m),ev^+(x_m)\Big).
  \end{gather*}
  Consider in the target space the submanifold $A_m(c^-,c^+)$ consisting of tuples 
  \begin{align*}
   (q_0,p_1,q_1,...\,,p_{m-1},q_{m-1},p_m)\quad\text{ such that }&\quad p_k=q_k\in C_k,\\ \text{ and }&\quad q_0\in W^u(c^-),\,p_m\in W^s(c^+).
  \end{align*}
   We show below that $\psi$ is transverse to $A_m(c^-,c^+)$. This implies by the implicit function theorem that $\psi^{-1}\big(A_m(c^-,c^+)\big)$ is a submanifold of $\mathcal{U}^\ell_m$ of codimension
   \begin{align*}
    \text{codim}\,\psi^{-1}\big(A_m(c^-,c^+)\big)&=\text{codim}\,W^u(c^-)+\sum_{k=1}^{m-1} \dim C_k + \text{codim}\,W^s(c^+)\\
    &=\sum_{k=1}^{m-1}\dim C_k + ind_h(c^-)+\big(\dim C_m-ind_h(c^+)\big).
   \end{align*}
  Note that $\dim W^u(c^-)=\dim C^--ind_h(c^-)$, as we use the positive gradient flow on $C^\pm$ (see Definition \ref{Morsetraj}). Having this result, we consider again the projection
  \[\pi:\psi^{-1}\big(A_m(c^-,c^+)\big)\rightarrow\mathcal{J}^\ell.\]
  This is now a Fredholm map of exactly the required index, as
  \begin{align*}
   ind\, \pi &= \phantom{+}\sum_{k=1}^m ind\, D_{v_k} - \text{codim}\,\psi^{-1}\big(A_m(c^-,c^+)\big)+\dim \big(\mathbb{R}^+\big)^{m-1}\\
   &= \phantom{+}\mu_{CZ}(c^+,\bar{c}^+)-\mu_{CZ}(c^-,\bar{c}^-) + \sum_{k=1}^m 2c_1(\bar{v}_k^- \#v_k\#\bar{v}_k^+)+m-1\\
   &\phantom{=\,}+\frac{1}{2}\dim C_m +\frac{1}{2}\dim C_0 - \Big(\dim C_m - ind_h(c^+)\Big) - ind_h(c^-).
  \end{align*}
Using the Sard-Smale theorem, we find for $\ell$ sufficiently large that the set \linebreak $\mathcal{J}^\ell_{reg}(C_0,...\,,C_m)$ of regular values of $\pi$ is a set of second category in $\mathcal{J}^\ell$. Using Taubes trick, taking the union over $\rho>0$ and countable intersections over tuples $(C_0,...\,,C_m)\subset\crita^m$ and $m\in\mathbb{N}$, we get a set of second category $\mathcal{J}_{reg}$ such that for all $J\in\mathcal{J}_{reg}$, all $m\in\mathbb{N}$ and all tuples $(C_0,...\,,C_m)$ holds that $\widehat{\mathcal{M}}_s(c^-,c^+,m)=\pi^{-1}(J)\subset\psi^{-1}\big(A_m(c^-,c^+)\big)$ is a finite dimensional manifold of dimension $ind\,\pi$ which proves the theorem.
  \item To show that $\psi$ is transverse to $A_m(c^-,c^+,J)$ it suffices to show for any $k$ that $0\times TC_{k-1}\times TC_k\times 0$ lies in the image of
  \[d\psi = d(ev^-_1)\times\big(dT_h(t_k)\circ d(ev_1^+)\big)\times...\times d(ev_m^+).\]
  As $dT_h(t_k)$ is an isomorphism and as $\mathcal{U}^\ell(C_{k-1},C_k)$ embeds naturally into $\mathcal{U}^\ell_m$, it suffices to show that $ev^-\times ev^+: \mathcal{U}^\ell(C^-,C^+)\rightarrow C^-\times C^+$ are submersions for each pair of connected components $C^\pm\subset\crita$. Let $(J,(v,\eta))\in\mathcal{U}^\ell(C^-,C^+)$ and let $\zeta\in T_{ev^-(v,\eta)} C^-\times T_{ev^+(v,\eta)} C^+$ be arbitrary. We have to show that there exists $(Y,\mathfrak{v},\hat{\eta})\in\ker D_{J,(v,\eta)}$ such that
  \[d(ev^-)\times d(ev^+)\big(Y,\mathfrak{v},\hat{\eta}\big) = \zeta.\]
  Surely, $ev^-\times ev^+:\mathcal{B}(C^-,C^+)\rightarrow C^-\times C^+$ is a submersion. Hence, we may choose some arbitrary $(\mathfrak{v}_0,\hat{\eta}_0)\in  T_{(v,\eta)}\mathcal{B}(C^-,C^+)$ such that
  \[d(ev^-)\times d(ev^+)\big(0,\mathfrak{v}_0,\hat{\eta}_0\big) = \zeta.\]
  In the proof of Theorem \ref{theoloctrans}, (2.), we showed that $D_{J,(v,\eta)}$ as an operator with domain 
  \[T_J\mathcal{J}^\ell\oplus W^{1,p}_\delta(\mathbb{R}\times S^1,v^\ast TV)\oplus W^{1,p}_\delta(\mathbb{R},\mathbb{R})=\ker d(ev^-)\cap \ker d(ev^+)\] 
  is surjective, i.e.\ we did not use the $T_{(v^-,\eta^-)}C^-\oplus T_{(v^+,\eta^+)}C^+$ part of $T_{(v,\eta)}\mathcal{B}(C^-,C^+)$. Hence there exists $(Y_1,\mathfrak{v}_1,\hat{\eta}_1)\in T_J\mathcal{J}^\ell\oplus T_{(v,\eta)}\mathcal{B}(C^-,C^+)$ such that
  \[D_{J,(v,\eta)}(Y_1,\mathfrak{v}_1,\hat{\eta}_1)=D_{J,(v,\eta)}(0,\mathfrak{v}_0,\hat{\eta}_0)\qquad\text{and}\qquad d(ev^\pm)_{J,(v,\eta)}(Y_1,\xi_1,\hat{\eta}_1)=0.\]
  Now set $(Y,\mathfrak{v},\hat{\eta}):=(0-Y_1,\mathfrak{v}_0-\mathfrak{v}_1,\hat{\eta}_0-\hat{\eta}_1)$. Then, $(Y,\mathfrak{v},\hat{\eta})$ lies in the kernel of $D_{J,(v,\eta)}$ and satisfies
  \[d(ev^-)\times d(ev^+)_{J,(v,\eta)}\big(Y,\mathfrak{v},\hat{\eta}\big)=\zeta.\qedhere\]
 \end{enumerate}
\end{proof}
\begin{rem}~
 \begin{itemize}
  \item There is a similar version of the Global Transversality Theorem if we consider homotopies, i.e.\ if $H, J, h$ and $g_h$ depend on $s$. In this case however, we can no longer assume that all $\mathcal{A}^H$-gradient trajectories are non-constant. That transversality along constant trajectories holds automatically is shown in Appendix \ref{auttrans}.
  \item To show that $\widehat{\mathcal{M}}(c^-,c^+,m)$ is a manifold, we still have to deal with the unstable trajectories, i.e.\ whose where $t_k=0$ for at least one $k$. Here, one uses a gluing argument that gives $\widehat{\mathcal{M}}(c^-,c^+,m)$ the structure of a manifold with corners (Theorem \ref{theocompactness2}). For details, see \cite{Fra}, appendix A, discussion after Cor.\ A.15.
 \end{itemize}
\end{rem}
Next, we want to extend the Transversality Theorems \ref{theotrans} and \ref{theoloctrans} to situations where we have a smooth symplectic symmetry $\sigma$ on $V$ such that $\sigma^k=Id$ for some finite $k$. Recall that $\sigma$ is a symplectic symmetry if it is a diffeomorphism of $V$ such that $\sigma^\ast\lambda=\lambda$. We assume that the Hamiltonian $H$, the Morse function $h$ and the metric $g$ on $\crita$ are chosen $\sigma$-invariant, i.e.\ $H(\sigma(p))=H(p)$ for all $p\in V$, $h(\sigma(p))=h(p)$ for all $p\in\crita$ and $\sigma^\ast g=g$. Moreover, we assume that the following set is a symplectic submanifold of $V$:
\[V_{fix}:=\Big\{p\in V\,\Big|\,\sigma^l(p)=p \text{ for some }l, 1\leq l\leq k-1\Big\}=\bigcup_{1\leq l\leq k-1} V_{fix}(\sigma^l).\]
An admissible almost complex structure $J$ is called $\sigma$-symmetric if $\sigma^\ast J=J$. We denote the space of smooth $\sigma$-symmetric admissible almost complex structures by $\mathcal{J}^{sym}$.
\begin{theo}[Local Transversality Theorem with symmetry]\label{theosymtransloc}~\\
 Assume that there exists a set of second category $\mathcal{J}^{fix}_{reg}$ of admissable almost complex structures on $V_{fix}$ such that for every $J\in\mathcal{J}^{fix}_{reg}$ the moduli space $\widehat{\mathcal{M}}(C^-,C^+)|_{V_{fix}}$ of solutions of the Rabinowitz-Floer equation (\ref{eq3}) on $V_{fix}$ is empty. Then there exists a set of second category $\mathcal{J}^{sym}_{reg}\subset\mathcal{J}^{sym}$ such that for all $J\in\mathcal{J}^{sym}_{reg}$ and all $C^\pm\subset\crita$ holds that $\widehat{\mathcal{M}}(C^-,C^+)$ is a manifold of the same dimension as in Theorem \ref{theoloctrans}.
\end{theo}
The original idea of the following proof is due to Peter Uebele, who showed a similar result for symplectic homology in \cite{Ueb}.
\begin{proof}
 We proceed as in the proof of Theorem \ref{theoloctrans}. Part (1.) remains completely unchanged, as the proof that $D_{(v,\eta)}$ is a Fredholm operator of the required index does not depend on the chosen almost complex structure.\\
 In Part (2.), we now want to prove that the following universal moduli space is a $C^\ell$-Banach manifold:
 \[\mathcal{U}^\ell_\sigma:=\left\{\big(J,(v,\eta)\big)\in\mathcal{J}^{\ell,sym}\times \mathcal{B}\,\left|\,\mathcal{F}\big(J,(v,\eta)\big)=0, im(v)\not\subset V_{fix}\Big.\right.\right\}.\]
 Note that $\mathcal{U}^\ell_\sigma$ is in general not $\mathcal{F}^{-1}(0)$, as a priori we cannot exclude $im(v)\subset V_{fix}$. However, we clearly have that $\mathcal{U}^\ell_\sigma$ is an open subset of $\mathcal{F}^{-1}(0)$. Hence it suffices again to prove that $D_{J,(v,\eta)}$ is onto for every $(J,(v,\eta))\in \mathcal{U}^\ell_\sigma$. As $D_{(v,\eta)}$ is Fredholm, $D_{J,(v,\eta)}$ has a closed range and it still suffices to show that its range is dense.\\
 Now, the tangent space to $\mathcal{J}^{\ell,sym}$ consists of matrix valued functions $Y:S^1\times \mathbb{R}\rightarrow End(TV)$ such that $Y\in T_J\mathcal{J}^\ell$ and $Y$ is symmetric, i.e.\ $\sigma^\ast Y = Y$. Moreover, recall from the end of Section \ref{secholprop} the definition of regular and symmetric regular points $\mathcal{S}(U)$ of a solution $U=(v,\eta)$ of (\ref{eq3}) with $im(v)\not\subset V_{fix}$:\pagebreak
 \[\mathcal{S}(U)=\left\{(s,t)\in\mathbb{R}\times S^1\;\left|\;\begin{aligned}(\partial_s v(s,t),\partial_s \eta(s))&\neq (0,0)&&\\
 (v(s,t),\eta(s))&\not\in \big(\sigma^l\circ v(\overline{\mathbb{R}},t),\eta(\overline{\mathbb{R}})\big),& l&=1,...\,,k\\
 \text{except}\;\, (v(s,t),\eta(s))&=(\sigma^k\circ v(s,t),\eta(s))&&\end{aligned}\right.\right\}.\]
 This set is open and dense in the open and non-empty set $\{(s,t)\in\mathbb{R}\times S^1\,|\,\partial_s v\neq 0\}$ by Proposition \ref{regpts}. Moreover, by Lemma \ref{notfixdense}, we know that the set $F(U)$ of points $(s,t)$ with $(v,\eta)(s,t)\not\in V_{fix}$ is also open and dense in $\mathbb{R}\times S^1$. Thus we find that the set
 \[\Omega := \left\{(s,t)\in S_\sigma(U)\cap F(U)\,\left|\,v(s,t)\not\in crit(H)\Big.\right.\right\}\]
 is non-empty and open in $\mathbb{R}\times S^1$. As before, we show that every $(\zeta,\hat{\mu})\in\big(\mathcal{E}_{(v,\eta)}\big)^\ast$ with
 \[\int_{\mathbb{R}\times S^1}\left\langle D_{(v,\eta)}(\xi,\hat{\eta}),(\zeta,\hat{\mu})\Big.\right\rangle ds\,dt=0 \;\text{ and }\;\int_{\mathbb{R}\times S^1}\left\langle Y(\partial_t v-\eta X_H),\zeta\Big.\right\rangle ds\,dt =0\]
 for all $(Y,\xi,\hat{\eta})\in T_J\mathcal{J}^{l,sym}\times T_{(v,\eta)}\mathcal{B}$ has to be zero on $\Omega$ and therefore everywhere by unique continuation. Assume that there exists $(s_0,t_0)\in \Omega$ with $\zeta(s_0,t_0)\neq 0$. Set \linebreak $p:=v(s_0,t_0)\in V$ and $n:=\eta(s_0)\in\mathbb{R}$ and choose as before a linear map $Y_p:T_pM\rightarrow T_pM$ such that $Y_p\in T_J(\mathcal{J}^\ell)_{t_0}(p,n)$ and $\big\langle Y_pJ\partial_s v(s_0,t_0),\zeta(s_0,t_0)\big\rangle>0$. Using a cutoff function $\beta: S^1\times V\times\mathbb{R}\rightarrow \mathbb{R}$ supported near $(t_0,p,n)$ construct $\widetilde{Y}\in T_J \mathcal{J}^\ell$ supported near $(t_0,p,n)$ such that 
 \[\int_{\mathbb{R}\times S^1} \Big\langle \widetilde{Y}J\partial_s v,\zeta\Big\rangle ds\, dt>0.\tag{$\ast$}\]
 Note that $\widetilde{Y}\not\in T_J\mathcal{J}^{l,sym}$ in general, as it is not symmetric. Thus, we set
 \[Y:=\widetilde{Y}+\sigma^\ast \widetilde{Y}+(\sigma^2)^\ast\widetilde{Y}+...+(\sigma^{k-1})^\ast\widetilde{Y}.\]
 Note that $Y$ is supported near the $k$ points $(t_0,\sigma^l(p),n),\, 0\leq l\leq k-1$. For each $1\leq l\leq k-1$ we have two cases:
 \begin{itemize}
  \item $\sigma^l(p)\not\in im(v)$. Then $(\ast)$ still holds with $Y$ instead of $\widetilde{Y}$ provided that $supp\,\beta$ is small enough in the $V$-direction.
  \item $\sigma^l(p)\in im(v)$. Then $(\sigma^l\circ v)(s_0,t_0)=v(s_1,t_1)$ for some $(s_1,t_1)\in\mathbb{R}\times S^1$. As $(s_0,t_0)\in \mathcal{S}(U)$ this implies either $\eta(s_1)\neq \eta(s_0)=n$ or that $t_1\neq t_0$. Hence $(\ast)$ still holds for $Y$ instead of $\widetilde{Y}$ provided that $supp\,\beta$ is small enough in the $\mathbb{R}$- resp.\ $S^1$-direction.
 \end{itemize}
 This shows that $\zeta\equiv 0$ on $\Omega$. The proof that also $\hat{\mu}\equiv 0$ stays unchanged as it involves no perturbations on $J$. This shows that $\mathcal{U}^\ell_\sigma$ is a $C^\ell$-Banach manifold.\\
Part (3.) of the proof works again mostly unchanged: Considering as before the projection $\pi_\sigma: \mathcal{U}^\ell_\sigma\rightarrow \mathcal{J}^{\ell,sym}$ we show again that $d\pi_\sigma$ is Fredholm of the same index as $D_{(v,\eta)}$ and that for $\ell$ large enough the set of regular values $\mathcal{J}^{\ell,sym}_{reg}\subset\mathcal{J}^{\ell,sym}$ of $\pi$ is of second category. Using Taubes trick, we then obtain that the set $\widetilde{\mathcal{J}}^{sym}_{reg}:=\mathcal{J}^{\ell,sym}_{reg}\cap\mathcal{J}^{\infty,sym}$ is of second category in $\mathcal{J}^{sym}$.\pagebreak\\
Attention: Note that even though $\pi^{-1}(J)$ is a manifold for $J\in\widetilde{\mathcal{J}}^{sym}_{reg}$ of the required dimension it might not be true that $\pi^{-1}(J)=\widehat{\mathcal{M}}(C^-,C^+)$ as $\mathcal{U}^\ell_\sigma$ only contains $(v,\eta)$ with $im(v)\not\subset V_{fix}$. However, we assumed for all $J\in\mathcal{J}^{fix}_{reg}$ that the space $\widehat{\mathcal{M}}(C^-,C^+)|_{V_{fix}}$ is empty. Let $\mathcal{J}^{fix}$ denote the space of all smooth admissible almost complex structures on $V_{fix}$ and recall that $\mathcal{J}^{fix}_{reg}\subset\mathcal{J}^{fix}$ is of second category. Note that the restriction of any $J\in\mathcal{J}^{sym}$ to $V_{fix}$ gives an element in $\mathcal{J}^{fix}$. Hence we may consider the map
\[\varphi: \mathcal{J}^{sym}\rightarrow\mathcal{J}^{fix},\qquad J\mapsto J|_{V_{fix}}.\]
Note that $\varphi$ is continuous and open, as we can construct locally around any $J\in\mathcal{J}^{sym}$ a continuous map $\psi:\mathcal{J}^{fix}\rightarrow \mathcal{J}^{sym}$ such that $\varphi\circ\psi=Id$ and $\psi(J|_{V_{fix}})=J$. The Lemma \ref{lemZZ} shows then that $\varphi^{-1}\big(\mathcal{J}^{fix}_{reg}\big)\subset \mathcal{J}^{sym}$ is also a set of second category in $\mathcal{J}^{sym}$. Thus
\[\mathcal{J}^{sym}_{reg}:=\widetilde{\mathcal{J}}^{sym}_{reg}\cap\varphi^{-1}\big(\mathcal{J}^{fix}_{reg}\big)\]
is of second category and we conclude that for $J\in\mathcal{J}^{sym}_{reg}$ there is no solution $(v,\eta)$ of (\ref{eq3}) lying entirely in $V_{fix}$ and hence that $\widehat{\mathcal{M}}(C^-,C^+)=\pi^{-1}(J)$.
\end{proof}
\begin{lemme}\label{lemZZ}
 Let $\varphi: X\rightarrow Y$ be an open, continuous map between Baire spaces $X,Y$ and let $A=\underset{k\in\mathbb{N}}{\bigcap} U_k$ be a countable intersection of open and dense sets $U_k\subset Y$, i.e.\ a set of second category in $Y$. Then $\varphi^{-1}(A)$ is of second category in $X$.
\end{lemme}
\begin{proof}
 As $\varphi^{-1}\big(\bigcap U_k\big) = \bigcap \varphi^{-1}\big(U_k\big)$, it suffices to show that the preimage $\varphi^{-1}(U)$ of an open and dense set $U\subset Y$ is open and dense in $X$. That $\varphi^{-1}(U)$ is open follows immediately from the fact that $\varphi$ is continuous.\\
 In order to show that $\varphi^{-1}(U)$ is dense, it suffices to show that for every open $V\subset X$ holds that $V\cap\varphi^{-1}(U)\neq\emptyset$. As $\varphi$ is an open map, $\varphi(V)$ is open in $Y$. By density of $U$, we have that $U\cap\varphi(V)\neq\emptyset$ and hence that $\varphi^{-1}(U)\cap V \neq\emptyset$.
\end{proof}\bigskip
In order to give an appropriate version of the Global Transversality Theorem in a symmetric setup, let us make the following definition. Let $C^\pm\subset \crita$ be any two connected components and let $c^\pm\in C^\pm \cap crit(h)$. We denote by $\widehat{\mathcal{M}}(C^-,c^+)$ the space of all $\mathcal{A}^H$ gradient trajectories $(v,\eta)$ such that
\[\lim_{s\rightarrow -\infty}(v,\eta)=(v^-,\eta^-)\in C^-\quad\text{ and }\quad \lim_{s\rightarrow +\infty}(v,\eta)=(v^+,\eta^+)\in W^s(c^+)\subset C^+.\]
Similarly, we define $\widehat{\mathcal{M}}(c^-,C^+)$ as the space of $\mathcal{A}^H$-gradient trajectories $(v,\eta)$ such that
\[\lim_{s\rightarrow -\infty}(v,\eta)=(v^-,\eta^-)\in W^u(c^-)\subset C^-\quad\text{ and }\quad \lim_{s\rightarrow +\infty}(v,\eta)=(v^+,\eta^+)\in C^+.\]
Note that both spaces are for all $J\in\mathcal{J}^{reg}$ manifolds of local dimensions
\begin{align}
 \dim_{(v,\eta)} \widehat{\mathcal{M}}(c^-,C^+) &= ind\, D_{v} - \dim W^u(c^-) \label{dimcN}\\
 &= \mu_{CZ}(v^+,\bar{v}^+)-\mu_{CZ}(v^-,\bar{v}^-) &&+ \frac{\dim C^+ + \dim C^-}{2} - ind_h(c^-)\notag
 \end{align}
 \begin{align}
 \dim_{(v,\eta)} \widehat{\mathcal{M}}(C^-,c^+) &= ind\, D_{v} - \dim W^s(c^+) \label{dimNc}\\
 &= \mu_{CZ}(v^+,\bar{v}^+)-\mu_{CZ}(v^-,\bar{v}^-) &&+ \frac{\dim C^+ + \dim C^-}{2} \notag\\
 &&&- \Big.\big(\dim C^+ - ind_h(c^+)\big)\notag\\
  &=\mu_{CZ}(v^+,\bar{v}^+)-\mu_{CZ}(v^-,\bar{v}^-) &&+ \frac{\dim C^- - \dim C^+}{2} + ind_h(c^+)\notag.
\end{align}
This follows as in the proof of Theorem \ref{theotrans}. Indeed, consider the map
\[\psi: \mathcal{U}^\ell(C^-,C^+)\rightarrow C^-\times C^+,\qquad \big(J,(v,\eta)\big)\mapsto\big(ev^-(v,\eta),ev^+(v,\eta)\big).\]
Then we see that $\psi^{-1}\big(C^-\times W^s(c^+)\big)$ resp.\
$\psi^{-1}\big(W^u(c^-)\times C^+\big)$ are submanifolds of $\mathcal{U}^\ell(C^-,C^+)$. Restricting the projection $\pi:\mathcal{U}^\ell(C^-,C^+)\rightarrow \mathcal{J}^\ell$ to these submanifolds, we then get that $\widehat{\mathcal{M}}(C^-,c^+)$ resp.\ $\widehat{\mathcal{M}}(c^-,C^+)$ are preimages of regular values of $\pi$.
\begin{theo}[Global Transversality Theorem with symmetry]\label{theotranssym}~\\
 Let $c^\pm = (v^\pm,\eta^\pm)\in crit(h)$. Assume that there exists a set of second category $\mathcal{J}^{fix}_{reg}$ of admissable almost complex structures on $V_{fix}$ such that for all $J\in J^{fix}_{reg}$ and all  \linebreak $\eta^-<\eta_0<\eta_1<\eta^+$ holds that all moduli spaces $\widehat{\mathcal{M}}(\mathcal{N}^{\eta_0},\mathcal{N}^{\eta_1})|_{V_{fix}}$ are empty. If $c^-$ or $c^+$ lie in $V_{fix}$, assume that for all $J\in J^{fix}_{reg}$ and all $\eta^-<\eta_0<\eta^+$ holds that $\widehat{\mathcal{M}}(c^-,\mathcal{N}^{\eta_0})|_{V_{fix}}$ resp.\ $\widehat{\mathcal{M}}(\mathcal{N}^{\eta_0},c^+)|_{V_{fix}}$ are empty. If $c^-$ and $c^+$ are both in $V_{fix}$ assume that $\widehat{\mathcal{M}}(c^-,c^+,1)|_{V_{fix}}$ is empty for all $J\in J^{fix}_{reg}$. Then there exists a set of second category $\mathcal{J}^{sym}_{reg}$ of admissable almost complex structures on $V$ such that for all $J\in\mathcal{J}^{sym}_{reg}$ and all $c^\pm\in crit(h),\,m\in\mathbb{N}$ holds that $\widehat{\mathcal{M}}(c^-,c^+,m)$ is a manifold of the same dimension as in Theorem \ref{theotrans}.
\end{theo}
\begin{proof}
 Using the same arguments as in the proof of Theorem \ref{theotrans}, it is obvious that we can find a set of second category $\mathcal{J}^{sym}_{reg}$ such that for every $J\in\mathcal{J}^{sym}_{reg}$ the space of all trajectories with cascades from $c^-$ to $c^+$, where all cascades are not contained in $V_{fix}$, is a manifold of the required dimension. Simply replace in the proof $\mathcal{U}^\ell(C_{k-1},C_k)$ by $\mathcal{U}^\ell_\sigma(C_{k-1},C_k)$, use Theorem \ref{theosymtransloc} instead of \ref{theoloctrans} and proceed otherwise in the same way.\\
 It remains to verify that these trajectories with cascades are the only ones, i.e.\ we have to make sure that no cascade that is part of a trajectory between $c^-$ and $c^+$ is contained in $V_{fix}$. It clearly suffice to assume for every pair $\eta_0,\eta_1\in[\eta^-,\eta^+]$ that the space $\widehat{\mathcal{M}}(\mathcal{N}^{\eta_0},\mathcal{N}^{\eta_1})|_{V_{fix}}$ of all gradient trajectories between $\mathcal{N}^{\eta_0}$ and $\mathcal{N}^{\eta_1}$ lying entirely in $V_{fix}$ is empty. At the two ends however, we can weaken this assumption:
 \begin{itemize}
  \item If $c^-$ or $c^+$ is not in $V_{fix}$, then it follows from the $\sigma$-symmetry of $g$ that \linebreak $W^u(c^-)\cap V_{fix}=\emptyset$ resp.\ $W^s(c^+)\cap V_{fix}=\emptyset$. Indeed, a point $p\in W^{s/u}(c^\pm)\cap V_{fix}$ would imply that $T_h(t)(p)\in V_{fix}$ for all $t$ and hence that $\displaystyle \lim_{s\rightarrow \pm\infty} T_h(t)(p)=c^\pm\in V_{fix}$. We know therefore that the first/last cascade cannot be  contained in $V_{fix}$ as one of its ends is not in $V_{fix}$.
  \item If $c^-$ or $c^+$ lie in $V_{fix}$ as well as the first/last cascade $x_1$ resp.\ $x_m$ then we get from the symmetry of $g$ that $x_1\in\widehat{\mathcal{M}}(c^-,\mathcal{N}^{\eta_0})|_{V_{fix}}$ resp.\ $x_m\in\widehat{\mathcal{M}}(\mathcal{N}^{\eta_0},c^+)|_{V_{fix}}$ as
  \[T_h(t)\big(ev^-(x_1)\big)\in V_{fix}\quad \text{ resp.\ } \quad T_h(t)\big(ev^+(x_m)\big)\in V_{fix} \quad\text{ for all $t$}.\]
  However, our assumptions guarantee that this cannot happen.
  \item If $c^-$ and $c^+$ lie in $V_{fix}$ and we consider only trajectories with one cascade, which lies in $V_{fix}$, then the same arguments show that we actually consider the space $\widehat{\mathcal{M}}(c^-,c^+,1)|_{V_{fix}}$ which we also assume to be empty.\qedhere
 \end{itemize}
\end{proof}
\begin{cor}\label{symhom}
 Let $\sigma : V\rightarrow V$ be a symplectic symmetry of order $k$ such that the set $V_{fix}=V\setminus\{p\,|\, \sigma^l(p)\neq p, 0<l<k\}$ is a symplectic submanifold. Let $H, h$ and $g_h$ be $\sigma$-symmetric and suppose that $c^\pm=(v^\pm,\eta^\pm)\in crit (h)\cap V_{fix}$. Moreover assume that the assumptions of Theorem \ref{theotranssym} are satisfied. Then it holds for all $J\in \mathcal{J}^{sym}_{reg}$ that the cardinality of the zero-dimensional component of $\mathcal{M}(c^-,c^+)$ is divisible by $k$.
\end{cor}
\begin{proof}
 By Theorem \ref{theotranssym}, we know for $J\in\mathcal{J}^{sym}_{reg}$ that $\mathcal{M}(c^-,c^+)$ is a manifold. Its zero-dimensional component is by Theorem \ref{theo3.6} compact and hence a finite set. Since $H,J,h$ and $g_h$ are $\sigma$-symmetric and $c^\pm\in V_{fix}$, it follows for any flow line with cascades $(x,t)$ from $c^-$ to $c^+$ that $(\sigma\circ x,t)$ is also a flow line with cascades from $c^-$ to $c^+$. Since no flow line lies in $V_{fix}$, all following flow lines with cascades are pairwise different:
 \[(\sigma^0\circ x, t),\,(\sigma^1\circ x, t),...\,,(\sigma^{k-1}\circ x, t).\]
 This implies that $\#\mathcal{M}(c^-,c^+)$ is divisible by $k$.
\end{proof}

\newpage
\phantom{..}
\newpage

\section{Compactness and additional properties}
\subsection{Compactness}\label{sec3.1}
This subsection is (with a few changes) taken from \cite{FraCie}. We repeat it here for completeness and to give a generalization of Corollary 3.8, \cite{FraCie}, at the end of this subsection. The main object is to show that the moduli spaces of $\mathcal{A}^{H_s}$-gradient trajectories for Hamiltonians $H$ or homotopies $H_s$ are compact. As the latter case includes the first, all proofs will be given for homotopies.\\
We still assume that $(V,\lambda)$ is the completion of a compact Liouville domain $\tilde{V}$ with contact boundary $M$, that $\Sigma\subset V$ is a contact hypersurface bounding a compact Liouville domain $W$ and that $H$ is a defining Hamiltonian for $\Sigma$. We denote by $M\times\mathbb{R}$ the symplectization of $M$ lying in $V$.\bigskip\\
A homotopy of defining Hamiltonians is for us a smooth family $(H_s)\subset C^\infty(V)$ with $-\infty<s_-\leq s\leq s_+<+\infty$, where all $H_s$ are defining Hamiltonians of exact contact hypersurfaces $\Sigma_s:=H_s^{-1}(0)$. We fix once and for all a smooth monotone cutoff function $\beta$ satisfying $\beta(s)=s_-$ for $s\leq s_-$ and $\beta(s)=s_+$ for $s\geq s_+$. Using $\beta$, we extend the homotopy to $\mathbb{R}$ as
\[s\mapsto H_{\beta(s)},\quad s\in\mathbb{R},\]
which is everywhere smooth and constant outside $[s_-,s_+]$. We set
\[H_-:=H_s, \quad s\leq s_-\qquad\text{ and }\qquad H_+:=H_s,\quad s\geq s_+.\]
To such a homotopy, we associate the following non-negative quantities:
\[||H||_\infty :=\max_{x\in V,\, s\in[s_-,s_+]} |H_s(x)|,\qquad ||\dot{H}||_1 =\max_{x\in V} \int^{s_+}_{s_-}\left|{\textstyle\frac{d}{ds}}H_s(x)\right|ds.\]
We define analog quantities for $H_{\beta(s)},\, s\in\mathbb{R}$, which we can estimate by
\begin{equation}
 ||H_\beta||_\infty = ||H||_\infty,\qquad ||\dot{H}_\beta||_1\leq ||\beta'||_\infty\cdot||\dot{H}||_1.
\end{equation}
By abuse of language, we say that an $\mathcal{A}^{H_s}$-gradient trajectory for a homotopy $H_s$ is a solution of the $s$-dependent Rabinowitz-Floer equation
\begin{equation} \label{eqRabFlHom}
 \begin{split}
  \partial_s v +J_t (v,\eta)\bigl(\partial_t v-\eta X_{H_s}(v)\bigr) &=0 \bigg.\\
  \partial_s\eta \;+\; \int^1_0 H_{\beta(s)}\bigl(v(s,t)\bigr) dt &=0,
 \end{split}
\end{equation}
where $X_{H_s}$ is the Hamiltonian vector field of $H_{\beta(s)}$. Note that this equation is not the gradient flow equation of $\mathcal{A}^{H_s}$! However, we continue to write by abuse of notation
\[\nabla \mathcal{A}^{H_s}(v,\eta):=\begin{pmatrix}-J_t(v,\eta)\big(\partial_t v-\eta X_{H_s}(v)\big)\\-\int^1_0 H_{\beta(s)}\big(v(s,t)\big) dt\end{pmatrix}.\]
The following theorem implies that the moduli space of such trajectories connecting critical points  $(v^\pm,\eta^\pm)\in crit\mspace{-4mu}\left(\mathcal{A}^{H_\pm}\right)$ is compact for sufficiently slow homotopies $H_s$.
\begin{theo}\label{theo3.6}
 Let $c$ and $\veps$ be given by the Proposition \ref{prop3.4} below and suppose that $H_s, s\in[s_-,s_+],$ is a smooth family of defining Hamiltonians satisfying the inequality
 \begin{equation}\label{eq16}
  \left(c+\frac{||H||_\infty}{\veps}\right)||\beta'||_\infty||\dot{H}||_1=:d<1.
 \end{equation}
Assume furthermore that $w^\nu=(v^\nu,\eta^\nu)\in C^\infty(\mathbb{R}\times S^1,V)\times C^\infty(\mathbb{R},\mathbb{R})$ is a sequence of $\mathcal{A}^{H_s}$-gradient trajectories for which there exist  $a,b\in\mathbb{R}$ such that
\begin{equation}\label{eq16a}
 \lim_{s\rightarrow -\infty}\mathcal{A}^{H_s}\big(w^\nu(s)\big)\geq a\quad\text{ and }\quad \lim_{s\rightarrow+\infty}\mathcal{A}^{H_s}\big(w^\nu(s)\big)\leq b\qquad \forall \;\nu\in\mathbb{N}.
\end{equation}
Then there exists a subsequence $(w^{\nu_j})$ of $(w^\nu)$ and a  $\mathcal{A}^{H_s}$-gradient trajectory $w$ such that $w^{\nu_j}$ converges in the $C^\infty_{loc}(\mathbb{R}\times S^1,V)\times C^\infty_{loc}(\mathbb{R},\mathbb{R})$-topology to $w$.
\end{theo}
\begin{rem}~
\begin{itemize}
 \item If $H_s=H$ is a constant homotopy, then condition (\ref{eq16}) is empty, as $||\dot{H}||_1=0$. So Theorem \ref{theo3.6} is applicable to any Hamiltonian $H$.
 \item If the sequence $w^\nu$ has fixed asymptotics $\lim_{s\rightarrow \pm\infty}w^\nu(s)=(v^\pm,\eta^\pm)\in crit\mspace{-4mu}\left(\mathcal{A}^{H_\pm}\right)$, then (\ref{eq16a}) holds with $a=\mathcal{A}^{H_-}(v^-,\eta^-)=\eta^-$ and $b=\mathcal{A}^{H_+}(v^+,\eta^+)=\eta^+$.
\end{itemize}
\end{rem}
\begin{proof}
 The theorem follows from the standard Gromov-compactness result, as soon as we have the following uniform bounds:
 \begin{itemize}
  \item an $L^\infty$-bound on the loops $v^\nu\in\mathscr{L}$ (so that $v^\nu$ stays in a compact region in $V$),
  \item an $L^\infty$-bound on the Lagrange multiplier $\eta^\nu\in\mathbb{R}$ (so that $\eta^\nu$ stays in a bounded region in $\mathbb{R}$),
  \item an $L^\infty$-bound on the derivatives of the loops $v^\nu$ (i.e.\ excluding bubbling).
 \end{itemize}
The support of $X_{H_s}$ lies inside $V\setminus\big(M\times[R,\infty)\big)$ for some large $R$, as $H_s$ is constant outside a compact set, independent from $s$. So, the first component of any $\mathcal{A}^{H_s}$-gradient trajectory $(v,\eta)$ which enters $M\times[R,\infty)$ satisfies due to (\ref{eqRabFlHom}) the holomorphic curve equation.\\
With our choice of almost complex structures, we conclude that $v$ cannot touch any level set $M\times\{r\},\;r>R,$ from inside (see \cite{Duff}, Lem.\ 2.4). As its asymptotics lie outside of $M\times[R,\infty)$, it has to remain in the compact set $V\setminus M\times[R,\infty)$ for all time. Alternatively, this result follows also from the Maximum Principle (see Lemma \ref{maxprinc}). This gives the $L^\infty$-bound on $v$.\\
The bound on $\eta$ is shown in the remainder of this section (see Corollary \ref{cor3.7}). It is here, where condition (\ref{eq16}) is needed. As the symplectic form $\omega$ is exact, there are no non-constant $J$-holomorphic spheres in $V$. This excludes bubbling and hence the derivatives of $v$ can be controlled (see \cite{DuffSal}).
\end{proof}
Before giving the $L^\infty$-bound on $\eta$, let us quickly state the most important consequences of Theorem \ref{theo3.6}.
\begin{cor}[Gromov-Floer compactness]\label{theocompactness}
 Let $w^\nu$ be a sequence of $\mathcal{A}^{H_s}$-gradient trajectories with $\displaystyle\lim_{s\rightarrow \pm\infty}\mathcal{A}^{H_s}\big(w^\nu(s)\big)=\eta^\pm$ fixed. Then there exists a subsequence $w^{\nu_j}$ and $\mathcal{A}^{H_s}$-gradient trajectories $(w_k)_{1\leq k\leq l}$ and sequences of real numbers $s_k^{\nu_j}$ such that \[w^{\nu_j}(\cdot+s_k^{\nu_j})\rightarrow w_k\quad\text{in the $C^\infty_{loc}(\mathbb{R}\times S^1,V)\times C^\infty_{loc}(\mathbb{R},\mathbb{R})$-topology and}\]
 \begin{itemize}
  \item $\displaystyle \lim_{s\rightarrow-\infty}\mathcal{A}^{H_s}(w_1)=\eta^-$ and $\displaystyle \lim_{s\rightarrow +\infty}\mathcal{A}^{H_s}(w_l)=\eta^+$
  \item $\displaystyle \lim_{s\rightarrow+\infty}w_k(s,t)=\lim_{s\rightarrow-\infty}w_{k+1}(s,t)\qquad\forall\;1\leq k\leq l-1.$
 \end{itemize}
\end{cor}
For the proof of Corollary \ref{theocompactness}, see \cite{Fra}, proof of Thm. A.11 (ff. 69), or \cite{Sal2}, Prop. 4.2. Using a glueing argument and this corollary, one then can prove the following theorem.
\begin{theo}\label{theocompactness2}
Let $H_s$ be a homotopy between defining Hamiltonians $H_-$ and $H_+$, let $h^\pm$ be Morse functions on $crit\big(\mspace{-2mu}\mathcal{A}^{H_\pm}\mspace{-4mu}\big)$ and let $c^\pm\in crit(h^\pm)$. The moduli space
\[\mathcal{M}(c^-,c^+):=\bigcup_{m\in\mathbb{N}}\mathcal{M}(c^-,c^+,m)\]
of all trajectories from $c^-$ to $c^+$ carries the structure of a manifold without boundaries.
\end{theo}
\begin{rem}
 Note that we allow here that the times $t_k$ that we stay on the critical manifold $C_k$ may be zero.
\end{rem}
\begin{defn}[\cite{Fra}, A.8]\label{defnbrokentraj}
 Let $H_s$ be a homotopy between defining Hamiltonians $H_-$ and $H_+$, let $h^\pm$ be Morse functions on $crit\big(\mspace{-2mu}\mathcal{A}^{H_\pm}\mspace{-4mu}\big)$ and let $c^\pm\in crit(h^\pm)$. A \textbf{\textit{broken trajectory with cascades}} from $c^-$ to $c^+$
 \[\textbf{w}=(w_j)_{1\leq j\leq l^-+l^+=l}, \; l^-,l^+,l\in\mathbb{N},\]
 consists of trajectories with cascades $w_j$ from $c_{j-1}$ to $c_j$ for $0\leq j\leq l$ such that $c_0=c^-$ and $c_l=c^+$ and $c_j\in crit(h^-)$ for $0\leq j\leq l^-$ and $c_j\in crit(h^+)$ for $l^-+1\leq j\leq l$.
\end{defn}
\begin{rem}
 By Definition \ref{defnbrokentraj} every (unbroken) trajectory with cascades $w$ from $c^-$ to $c^+$ is also a broken trajectory with cascades via $\textbf{\textit{w}}=(w)$.
\end{rem}
The following theorem is an easy consequence of Theorem \ref{theocompactness2} and the usual compactness result in ordinary Morse theory.
\begin{theo}[Compactness Theorem]\label{theoglue}~\\
 The space $\overline{\mathcal{M}}(c^-,c^+)$ of all broken trajectories with cascades from $c^-$ to $c^+$ carries the structure of a manifold with corners. Its interior is exactly given by the unbroken trajectories with cascades.
\end{theo}
\begin{rem}
  It follows from Theorem \ref{theoglue} that for the boundary of the 1-dimensional component $\overline{\mathcal{M}}^1(c^-,c^+)$ of $\overline{\mathcal{M}}(c^-,c^+)$ holds that
  \[\textstyle\partial \overline{\mathcal{M}}^1(c^-,c^+)=\bigcup_{c\in crit(h)}\mathcal{M}^0(c^-,c)\times\mathcal{M}^0(c,c^+).\]
  This statement implies by a standard argument in Floer theory that $\partial^F\circ\partial^F=0$.
\end{rem}
The topology of the manifolds in Theorem \ref{theocompactness2} and \ref{theoglue} is given by the so called Floer-Gromov convergence that we define now.
\begin{defn}[\cite{Fra}, A.9]\label{defnGromovFloerConv}
 Let $H_s$, $H_\pm$, $h^\pm$ and $c^\pm$ be as in Definition \ref{defnbrokentraj}. Suppose that $w^\nu, \nu\in\mathbb{N}$, is a sequence of trajectories with cascades all from $c^-$ to $c^+$. We say that $w^\nu$ \textbf{\textit{Floer-Gromov converges}} to a broken trajectory with cascades from $c^-$ to $c^+$
 \[\textbf{w}=(w_j)_{1\leq j\leq l},\qquad w_j=\big((x^\nu_k)_{1\leq k\leq m^\nu},(t^\nu_k)_{1\leq k\leq m^\nu -1}\big)\]
  if one of the following conditions holds:
 \begin{enumerate}
  \item If $H_s$ is constant and $\mathcal{A}^H(c^-)=\mathcal{A}^H(c^+)$, then all $w^\nu$ and $w_j$ are trajectories with zero cascades, i.e.\ ordinary $h$-Morse flow lines. Here, we require that there exist real numbers $s_j^\nu$ such that $w^\nu(\cdot+s_j^\nu)$ converges in the $C^\infty_{loc}$-topology to $w_j$.
  \item If $H_s$ is non-constant or $\mathcal{A}^H(c^-)<\mathcal{A}^H(c^+)$, then all $w^\nu$ have at least one cascade. Here, we require:
  \begin{enumerate}
   \item If $w_j\in C^\infty(\mathbb{R},\crita)$ is a trajectory with zero cascades, then there exists a sequence of $h$-Morse flow lines $y_j^\nu\in C^\infty(\mathbb{R},\crita)$ converging in $C^\infty_{loc}$ to $w_j$, a sequence of real numbers $s^\nu_j$ and a sequence of integers $k^\nu\in[1,m^\nu]$ such that either $\displaystyle \lim_{s\rightarrow -\infty} x^\nu_{k^\nu}(s)=y_j^\nu(s_j^\nu)$ or $\displaystyle\lim_{s\rightarrow \infty} x^\nu_{k^\nu}(s)=y_j^\nu(s_j^\nu)$.
   \item If $w_j$ is a trajectory with at least one cascade, we write
   \[ w_j=\big((x_{i,j})_{1\leq i\leq m_j},(t_{i,j})_{1\leq i\leq m_j-1}\big),\; m_j\geq 1.\]
   We require that there exist surjective maps $\gamma^\nu:\big[1,\sum_{p=1}^l m_p\big]\rightarrow[1,m^\nu]$, which are monotone increasing, i.e.\ $\gamma^\nu(\lambda_1)\leq \gamma^\nu(\lambda_2)$ for $\lambda_1\leq \lambda_2$, and real numbers $s^\nu_\lambda$ for every $\lambda\in\big[1,\sum^l_{p=1} m_p\big]$, such that
   \[x^\nu_{\gamma^\nu(\lambda)}(\cdot+s^\nu_\lambda)\rightarrow x_\lambda\quad\text{ in }\quad C^\infty_{loc},\]
   where $x_\lambda = x_{i,j}$ is such that $\lambda=\sum^j_{p=1}m_p+i$. For $\lambda\in\big[1,\sum_{p=1}^l m_p-1\big]$, set
   \begin{align*}
    && \tau_\lambda &=\begin{cases}t_{i,j},&\text{if }\lambda=\textstyle \sum_{p=1}^j m_p+i, \;0<i<m_j+1\\\infty, &\text{if }\lambda=\sum_{p=1}^j m_p \end{cases}\\
    &\text{and}\qquad& \tau_\lambda^\nu &=\begin{cases}t^\nu_{\gamma^\nu(\lambda)},&\text{if }\lambda=\textstyle \max\big\{\lambda'\in\big[1,\sum_{p=1}^l m_p-1\big] :\gamma^\nu(\lambda')=\gamma^\nu(\lambda)\big\}\\0 &\text{otherwise}. \end{cases}
   \end{align*}
   Now, we require that $\displaystyle \lim_{\nu\rightarrow\infty}\tau^\nu_\lambda =\tau_\lambda.$
  \end{enumerate}
 \end{enumerate}
\end{defn}
In the remainder of this subsection, we prove the $L^\infty$-bound on $\eta$ for a solution of the Rabinowitz-Floer equation (\ref{eq3}) satisfying (\ref{eq16}).

\begin{prop}\label{prop3.4}
 Assume that $H_s,\, s\in[s_-, s_+],$ is a smooth family of defining Hamiltonians for exact contact hypersurfaces $\Sigma_s\subset V$. Then there exist constants $\veps>0$ and $c<\infty$ depending only on the set $\{H_s\,|\,s_-\leq s\leq s_+\}$ and not on the particular parametrization, such that for all solutions $(v,\eta)$ of (\ref{eqRabFlHom}) the following implication holds:
 \[||\nabla\mathcal{A}^{H_s}(v,\eta)||\leq \veps \qquad \Rightarrow\qquad |\eta|\leq c\cdot\left(|\mathcal{A}^{H_s}(v,\eta)|+1\right).\]
\end{prop}
\begin{proof}
 The proof is organized in three steps.\bigskip\\
 \underline{Step 1}\hspace{0.5cm}\textit{There exist $\delta>0$ and a constant $c_\delta<\infty$ with the following property: For every $(v,\eta)\in\mathscr{L}\times\mathbb{R}$ with $v(t)\in U^s_\delta:=H^{-1}_s\big((-\delta,\delta)\big)$ for all $t\in S^1$ holds}
 \[|\eta|\leq 2\cdot|\mathcal{A}^{H_s}(v,\eta)|+c_\delta\cdot||\nabla\mathcal{A}^{H_s}(v,\eta)||.\]
 We start by choosing $\delta>0$ so small, such that for all $s_-\leq s\leq s_+$  and all $x\in U^s_\delta$ holds
 \[\lambda(X_{H_s}(x))\geq \textstyle\frac{1}{2}+\delta.\tag{possible, as $\lambda\big(X_{H_s}(x)\big)=1$ for $x\in H^{-1}_s(0)$}\]
 We set $\displaystyle c_\delta:=\max_{s_-\leq s\leq s_+} 2\cdot||\lambda|_{U^s_\delta}||_\infty$ and calculate for $v$ with $im(v)\subset U^s_\delta$ that
 \begin{align*}
  \left|\mathcal{A}^{H_s}(v,\eta)\right| &= \left|\int^1_0\lambda(\dot{v})-\eta H_s(v)dt\right|\\
   &=\left|\int^1_0\lambda(\eta X_{H_s})dt + \int^1_0\lambda(\dot{v}-\eta X_{H_s})dt-\int^1_0\eta H_s(v)dt\right|\\
  &\geq \left|\eta\int^1_0\lambda(X_{H_s})dt\right|-\left|\int^1_0\lambda(\dot{v}-\eta X_{H_s})dt\right|-\left|\int^1_0\eta H_s(v)dt\right|\\
  &\geq |\eta|\cdot\left(\frac{1}{2}+\delta\right)-\frac{c_\delta}{2}||\dot{v}-\eta X_{H_s}||-|\eta|\cdot\delta\\
  &\geq \frac{|\eta|}{2}-\frac{c_\delta}{2}\left|\left|\nabla\mathcal{A}^{H_s}(v,\eta)\right|\right|.
 \end{align*}
\underline{Step 2}\hspace{0.5cm}\textit{For each $\delta>0$ there exists $\veps=\veps(\delta)>0$ such that}
\[\left|\left|\nabla\mathcal{A}^{H_s}(v,\eta)\right|\right|\leq \veps\qquad \Rightarrow\qquad v(t)\in U^s_\delta \quad \forall\;t\in S^1.\]
First assume that $v\in\mathscr{L}$ has the property that there exist $t_0, t_1\in S^1$ such that $|H_s(v(t_0))|\geq\delta$ and $|H_s(v(t_1))|\leq\frac{\delta}{2}$. We claim that
\begin{equation}\label{eq10}
 \left|\left|\nabla\mathcal{A}^{H_s}(v,\eta)\right|\right|\geq \frac{\delta}{2\kappa}\qquad\forall\eta\in\mathbb{R}
\end{equation}
\[\text{where}\qquad\qquad \kappa:=\max_{s_-\leq s\leq s_+}\max_{x\in\bar{U}^s_\delta, t \in S^1} ||\nabla_t H_s(x)||,\qquad\qquad\qquad\phantom{?}\]
with $\nabla_t$ being the gradient of $H_s$ with respect to $g_t=\omega(\cdot,J_t\cdot)$, i.e.\ $\nabla_t H_s=J_t X_{H_s}$. \pagebreak\\
We prove (\ref{eq10}) with the following estimate:
\begin{align*}
 \left|\left|\nabla\mathcal{A}^{H_s}(v,\eta)\right|\right|\geq\sqrt{\int^1_0||\dot{v}-\eta X_{H_s}(v)||^2dt}&\geq\int^1_0||\dot{v}-\eta X_{H_s}(v)||dt\\
 &\geq\int^{t_1}_{t_0}||\dot{v}-\eta X_{H_s}(v)||dt\\
 &\geq\frac{1}{\kappa}\int^{t_1}_{t_0}||\nabla_tH_s(v)||\cdot||\dot{v}-\eta X_{H_s}(v)||dt\\
 &\geq\frac{1}{\kappa}\int^{t_1}_{t_0}\left|g(\nabla_t H_s(v),\dot{v}-\eta X_{H_s}(v)\right|dt\\
 &=\frac{1}{\kappa}\int^{t_1}_{t_0}\left|g(\nabla_tH_s(v),\dot{v})\right|dt\\
 &=\frac{1}{\kappa}\int^{t_1}_{t_0}\left|dH_s(\dot{v})\right|dt\\
 &=\frac{1}{\kappa}\int^{t_1}_{t_0}\left|\partial_t H_s(v)\right|dt\\
 &\geq\frac{1}{\kappa}\left|\int^{t_1}_{t_0}\partial_t H_s(v)dt\right|\\
 &=\frac{1}{\kappa}\left|H_s(v(t_1))-H_s(v(t_0))\right|\\
 &\geq\frac{1}{\kappa}\left(|H_s(v(t_1))|-|H_s(v(t_0))|\right)\\
 &\geq\frac{\delta}{2\kappa}.
\end{align*}
Now assume that for $v$ holds $v(t)\in V\setminus U^s_{\delta/2}$ for all $t\in S^1$. Then, we estimate
\begin{equation}\label{eq11}
 \left|\left|\nabla\mathcal{A}^{H_s}(v,\eta)\right|\right|\geq\left|\int^1_0H_s(v)\,dt\right|\geq \frac{\delta}{2}\qquad\forall\;\eta\in\mathbb{R}.
\end{equation}
From (\ref{eq10}) and (\ref{eq11}), step 2 follows with $\qquad\displaystyle \veps:=\frac{\delta}{2\cdot\max\{1,\kappa\}}$.\medskip\\
\underline{Step 3}\hspace{0.5cm}\textit{Proof of the proposition}\medskip\\
Choose $\delta$ as in step 1, $\veps=\veps(\delta)$ as in step 2 and $c=\max\{2,c_\delta\cdot\veps\}$. Assume that $||\nabla\mathcal{A}^{H_s}(v,\eta)||\leq\veps$. Step 1 and 2 then imply that
\[|\eta|\leq 2|\mathcal{A}^{H_s}(v,\eta)|+c_\delta||\nabla\mathcal{A}^{H_s}(v,\eta)||\leq c\left(|\mathcal{A}^{H_s}(v,\eta)|+1\right).\qedhere\]
\end{proof}
\begin{rem}
 The constants $c$ and $\veps$ are precisely the constants in Theorem \ref{theo3.6}. The stated independence from the parametrization is easily to be seen by the definitions of $\delta$, $c_\delta$ and $\kappa$ in the proof, which do involve only the fixed $H_s$ and in particular no derivatives in $s$. This implies that for an arbitrary path $p:\mathbb{R}\rightarrow[s_-,s_+]$, we may choose for the homotopy $H_{p(s)}$ the same constants $\veps$, $c$ as for $H_s$.
\end{rem}
\begin{cor}\label{cor3.7}
 Let $H_s, s\in[s_-, s_+],$ be a homotopy of defining Hamiltonians, such that
 \[\left(c+\frac{||H||_\infty}{\veps}\right)||\beta'||_\infty||\dot{H}||_1=d<1.\tag{\ref{eq16}}\]
 Let $\omega=(v,\eta)$ be an $\mathcal{A}^{H_s}$-gradient trajectory such that for $a,b\in\mathbb{R}$ holds
\[\lim_{s\rightarrow -\infty}\mathcal{A}^{H_s}\big(w(s)\big)= a,\qquad \lim_{s\rightarrow+\infty}\mathcal{A}^{H_s}\big(w(s)\big) = b\qquad \forall \;n\in\mathbb{N}.\]
Let $\veps$ and $c$ be as in Proposition \ref{prop3.4} and set $M:=\frac{|a|+|b|+b-a}{2}$. Then, $\eta$ is uniformly bounded by 
\begin{equation}\label{eq20}
 ||\eta||_\infty\leq \frac{1}{1-d}\left(c\cdot(M+1)+\frac{||H||_\infty(b-a)}{\veps^2}\right).
\end{equation}
\end{cor}
\begin{proof}~\\
 \underline{Step 1}\textit{ We estimate the action $\mathcal{A}^{H_s}(w)$ and the energy $E(w)$ by
 \begin{align}
  E(w)&\leq b-a+||\eta||_\infty||\beta'||_\infty||\dot{H}||_1\label{eq19},\\
  \left|\mathcal{A}^{H_s}\big(w(s)\big)\right|&\leq \mspace{10mu}M \mspace{10mu}+||\eta||_\infty||\beta'||_\infty||\dot{H}||_1.\label{eq18}
 \end{align}}
At first, we calculate the $s$-derivative of the action $\mathcal{A}^{H_s}$ as
\begin{equation}\label{eq17}
 \frac{d}{ds}\mathcal{A}^{H_{\beta(s)}}(v,\eta)(s)=\left|\left|\nabla\mathcal{A}^{H_s}(v,\eta)(s)\right|\right|^2-\eta(s)\cdot\beta'(s)\int^1_0\left(\frac{d}{ds}H\right)_{\beta(s)}(v)\,dt.
\end{equation}
Then we have
\begin{align*}
 E(w)&=\int^\infty_{-\infty}\left|\left|\nabla\mathcal{A}^{H_s}(v,\eta)(s)\right|\right|^2ds\\
 &\mspace{-5mu}\overset{(\ref{eq17})}{=}\int^\infty_{-\infty}\frac{d}{ds}\mathcal{A}^{H_{\beta(s)}}(v,\eta)(s)ds+\int^\infty_{-\infty}\eta(s)\cdot\beta'(s)\int^1_0\left(\frac{d}{ds}H\right)_{\beta(s)}(v)\,dt\,ds\\
 &\leq b-a + \int^1_0\int^{s_+}_{s_-}\eta(s)\cdot\beta'(s)\cdot\left(\frac{d}{ds}H\right)_{\beta(s)}(v) \,ds\, dt\\
 &\leq b-a + ||\eta||_\infty||\beta'||_\infty||\dot{H}||_1,
\end{align*}
where we used that $\frac{d}{ds}H_s\neq 0$ only for $s_-\leq s\leq s_+$. Next, we estimate the action by
\begin{align*}
 &(a)&\left|\mathcal{A}^{H_s}(w(\sigma))\right|&=\left|\int^\sigma_{-\infty}\frac{d}{ds}\mathcal{A}^{H_s}(w(s))\,ds + \lim_{s\rightarrow -\infty}\mathcal{A}^{H_s}\big(w(s)\big)\right|\\
 &&&\mspace{-5mu}\overset{(\ref{eq17})}{\leq}|a|+\int^\sigma_{-\infty}\left|\left|\nabla\mathcal{A}^{H_s}(v,\eta)(s)\right|\right|^2ds+\int^\sigma_{-\infty}||\eta||_\infty||\beta'||_\infty\int^1_0\left|\frac{d}{ds} H\right| \mspace{-1mu}dt\mspace{1mu} ds\\
 &(b)& \left|\mathcal{A}^{H_s}(w(\sigma))\right|&=\left|\int^\infty_\sigma\frac{d}{ds}\mathcal{A}^{H_s}(w(s))-\lim_{s\rightarrow\infty}\mathcal{A}^{H_s}\big(w(s)\big)\right|\\
 &&&\leq|b|+\int^\infty_\sigma\left|\left|\nabla\mathcal{A}^{H_s}(v,\eta)(s)\right|\right|^2ds+\int^\infty_\sigma||\eta||_\infty||\beta'||_\infty\int^1_0\left|\frac{d}{ds} H\right| \mspace{-1mu}dt\mspace{1mu} ds\\
 &\overset{(a,b)}{\Rightarrow}& \left|\mathcal{A}^{H_s}(w(\sigma))\right|&\leq\frac{1}{2}\Big(|a|+|b|+E(w)+||\eta||_\infty||\beta'||_\infty||\dot{H}||_1\Big)\\
 &\overset{(\ref{eq19})}{\Rightarrow}&\left|\mathcal{A}^{H_s}(w(\sigma))\right|&\leq M+||\eta||_\infty||\beta'||_\infty||\dot{H}||_1.
\end{align*}
\underline{Step 2} \textit{ We prove that $\eta$ is bounded.}\medskip\\
For $s\in\mathbb{R}$ and $\veps$ as in Proposition \ref{prop3.4}, let $\tau(s)\geq0$ be defined by 
\[\tau(s):=\inf\left\{\tau\geq 0\;\;\Big|\;\; \big|\big|\nabla\mathcal{A}^{H_s}(w(s+\tau))\big|\big|<\veps\right.\Big\}.\]
Then we have
\begin{equation}\label{eq12}
 \tau(\sigma)=\int^{\sigma+\tau(\sigma)}_\sigma 1\, ds\leq\int^{\sigma+\tau(\sigma)}_\sigma\frac{1}{\veps^2}\left|\left|\nabla\mathcal{A}^{H_s}(w(s+\tau))\right|\right|^2ds\leq \frac{E(w)}{\veps^2}.
\end{equation}
Using Proposition \ref{prop3.4}, (\ref{eq19}),(\ref{eq18}) and (\ref{eq12}), we can now calculate
\begin{align*}
 |\eta(\sigma)|&\leq|\eta(\sigma+\tau(\sigma))|+\int^{\sigma+\tau(\sigma)}_\sigma\left|\frac{d}{ds}\eta(s)\right|ds\\
 &\leq c\cdot\left(\left|\mathcal{A}^{H_{\sigma+\tau(\sigma)}}(v,\eta)\right|+1\right)+\tau(\sigma)\cdot||H||_\infty\\
 &\leq c\cdot\left(M+1\right)+c\cdot||\eta||_\infty||\beta'||_\infty||\dot{H}||_1+\frac{E(w)\cdot||H||_\infty}{\veps^2}\\
 &\leq c\cdot(M+1)+\frac{||H||_\infty(b-a)}{\veps^2}+\left(c+\frac{||H||_\infty}{\veps^2}\right)||\eta||_\infty||\beta'||_\infty||\dot{H}||_1.
\end{align*}
As $\sigma$ is arbitrary, we obtain
\[||\eta||_\infty\leq c\cdot(M+1)+\frac{||H||_\infty(b-a)}{\veps^2}+\left(c+\frac{||H||_\infty}{\veps^2}\right)||\eta||_\infty||\beta'||_\infty||\dot{H}||_1.\]
Using the assumption (\ref{eq16}), we conclude that
\[ ||\eta||_\infty\leq \frac{1}{1-d}\left(c\cdot(M+1)+\frac{||H||_\infty(b-a)}{\veps^2}\right).\qedhere\]
\end{proof}
\begin{cor}\label{cor3.8}
 Fix $k>1$ and assume that the constant $d$ from (\ref{eq16}) is so small that $d(k+1)<1$. Let $(v^\pm,\eta^\pm)$ be critical points of $\mathcal{A}^{H_\pm}$ and suppose that there exists an $\mathcal{A}^{H_s}$-gradient trajectory $(v,\eta)$ with $\displaystyle\lim_{s\rightarrow\pm\infty}(v,\eta)=(v^\pm,\eta^\pm)$. Then:
 \begin{itemize}
  \item[(a)]
  \begin{align}
   &\text{If for $(v^-,\eta^-)$ holds }& \mathcal{A}^{H_-}(v^-,\eta^-)&\geq \frac{dk}{1-d(k+1)}>0,\label{eq21}\\
   &\text{then it holds for $(v^+,\eta^+)$ that }& \quad\mathcal{A}^{H_+}(v^+,\eta^+) &\geq \left(1- \frac{1}{k}\right)\cdot\mathcal{A}^{H_-}(v^-,\eta^-)>0.\notag
  \end{align} 
  \item[(b)] 
  \begin{align}
   &\text{If for $(v^+,\eta^+)$ holds }& \mathcal{A}^{H_+}(v^+,\eta^+)&\leq \frac{-dk}{1-d(k+1)}<0,\label{eq22}\\
   &\text{then it holds for $(v^-,\eta^-)$ that }& \mathcal{A}^{H_-}(v^-,\eta^-)&\leq \left(1-\frac{1}{k}\right)\cdot\mathcal{A}^{H_+}(v^+,\eta^+)<0.\notag
  \end{align}
 \end{itemize}
\end{cor}
\begin{proof}
 We only prove $(a)$, case $(b)$ being completely analog. We first assume that the absolute value of the action at the positive asymptotic satisfies the inequality
 \begin{equation}\label{eq23}
  \left|\mathcal{A}^{H_+}(v^+,\eta^+)\right|\leq \mathcal{A}^{H_-}(v^-,\eta^-).
 \end{equation}
Using the notation from the proof of Corollary \ref{cor3.7}, we write $a:=\mathcal{A}^{H_-}(v^-,\eta^-)$ and $b:=\mathcal{A}^{H_+}(v^+,\eta^+)$ and deduce 
\begin{align*}
 &&b-a&=\mathcal{A}^{H_+}(v^+,\eta^+)-\mathcal{A}^{H_-}(v^-,\eta^-)&\leq&\, 0 \\
 &\text{ and }\phantom{hihihihi}& M&=\frac{|a|+|b|+b-a}{2}&\leq&\, |b|\leq\mathcal{A}^{H_-}(v^-,\eta^-).\phantom{hihihihi}
\end{align*}
Hence, we get from (\ref{eq20}) and (\ref{eq21}) the inequality
\begin{align*}
 ||\eta||_\infty\leq\frac{c}{1-d}(M+1)&\leq\frac{c}{1-d}\left(\mathcal{A}^{H_-}(v^-,\eta^-)+1\right)\\
 &\leq\frac{c}{1-d}\left(1+\frac{1-d(k+1)}{dk}\right)\mathcal{A}^{H_-}(v^-,\eta^-)=\frac{c}{kd}\cdot\mathcal{A}^{H_-}(v^-,\eta^-).
\end{align*}
This implies together with (\ref{eq16}) and (\ref{eq19}):
\begin{align*}
 \mathcal{A}^{H_+}(v^+,\eta^+)&\overset{(\ref{eq19})}{\geq}\mathcal{A}^{H_-}(v^-,\eta^-)-||\eta||_\infty||\beta'||_\infty||\dot{H}||_1\\
 &\;\geq\left(1-\frac{c}{kd}\cdot||\beta'||_\infty||\dot{H}||_1\right)\mathcal{A}^{H_-}(v^-,\eta^-)\overset{(\ref{eq16})}{\geq}\left(1-\frac{1}{k}\right)\mathcal{A}^{H_-}(v^-,\eta^-).
\end{align*}
This is the statement of the corollary under the additional assumption (\ref{eq23}). Now assume that (\ref{eq23}) does not hold. To prove the corollary, it suffices to exclude the case
\begin{equation}\label{eq24}
 b=\mathcal{A}^{H_+}(v^+,\eta^+)<-\mathcal{A}^{H_-}(v^-,\eta^-)=-a<0.
\end{equation}
We assume by contradiction that (\ref{eq24}) holds. Then we obtain:
\[b-a<0\qquad\text{ and }\qquad M=0.\]
In particular, we get from (\ref{eq20}) that $\displaystyle ||\eta||_\infty\leq\frac{c}{1-d}$. Hence, we can estimate
\begin{align*}
 \mathcal{A}^{H_-}(v^-,\eta^-)&\overset{(\ref{eq19})}{\leq}\mathcal{A}^{H_+}(v^+,\eta^+)+||\eta||_\infty||\beta'||_\infty||\dot{H}||_1\\
 &\overset{(\ref{eq24})}{\leq}-\mathcal{A}^{H_-}(v^-,\eta^-)+\frac{c}{1-d}\cdot||\beta'||_\infty||\dot{H}||_1\\
 &\overset{(\ref{eq16})}{\leq}-\mathcal{A}^{H_-}(v^-,\eta^-)+\frac{d}{1-d}\\
 &\overset{k>1}{\leq}-\mathcal{A}^{H_-}(v^-,\eta^-)+\frac{dk}{1-d(k+1)}\\
 &\overset{(\ref{eq21})}{\leq}0.
\end{align*}
But this implies that $\mathcal{A}^{H_-}(v^-,\eta^-)\leq0$, which contradicts assumption (\ref{eq21}). Hence, (\ref{eq24}) has to be wrong and Corollary \ref{cor3.8} follows.
\end{proof}
\newpage
\subsection{Invariance and action filtration}\label{sec.inv}
The following theorems in this section will show that the Rabinowitz-Floer homology does not depend on any auxiliary structures. To be more precise: It will turn out, that $RFH(H,h)$ only depends on the exact contact filling $W$ of $(\Sigma,\xi)$. As mentioned in the introduction, we can hence write $RFH(W,\Sigma)$ instead.\\
Additionally, we define the Rabinowitz-Floer homology $RFH^{(a,b)}(W,\Sigma)$ for an action window $(a,b)$. While $RFH^{(a,b)}(W,\Sigma)$ does depend on the contact form on $\Sigma$, we will show that $RFH^{(0^\pm,\infty)}(W,\Sigma)$ and $RFH^{(-\infty,0^\pm)}(W,\Sigma)$ are in fact invariant under Liouville isomorphisms and do therefore only depend on the exact contact filling $W$. The same holds true for the growth-rates $\Gamma^\pm(W,\Sigma)$, defined later in this section.\medskip\\ We start with the most basic invariance theorem:
\begin{theo}[\textbf{Frauenfelder \& Cieliebak,\cite{FraCie}}]\label{theo1}~\\
 $RFH(H,h)$ is independent of the almost complex structure $J$, the Morse function $h$ and the metric $g_h$. Moreover, if $H_s, s\in[s_-, s_+],$ is a homotopy of defining Hamiltonians of contact hypersurfaces $\Sigma_s\subset V$, then $RFH(H_-,h_-)$ and $RFH(H_+,h_+)$ are canonically isomorphic.
\end{theo}
\begin{proof}
 The proof uses the usual arguments in Floer theory (see \cite{Fra}, Thm.\ A17). We show only the invariance of $RFH(H_s)$ under homotopies of defining Hamiltonians, as it involves some non-standard technical difficulties due to compactness. However, our proof should enable the reader to prove the invariance for $J,\,h$ and $g_h$ in the same manner by considering generic homotopies $J_s$, $h_s$ and $g_{h_s}$.\medskip\\
 \underline{Step 0}\\
 Without loss of generality, we assume that $s_-=0$ and $s_+=1$. If not, replace $H_s$ by $\tilde{H}_s:=H_{s_-+(s_+-s_-)s)}$.\medskip\\
 \underline{Step 1}\\
 Let $\veps>0$ and $c>0$ be the constants for the homotopy $H_s$ given by Proposition \ref{prop3.4}. At first assume that $H_s$ satisfies the inequality
 \begin{equation}\label{eq27}
  \left( c+\frac{||H||_\infty}{\veps}\right)\cdot||\beta'||_\infty||\dot{H}||_\infty\leq\frac{1}{8}.
 \end{equation}
Here, $H_s$ is extended to $\mathbb{R}$ as $H_s:=H_{\beta(s)}$ as discussed before. The norm $||\dot{H}||_\infty$ is defined as $||\dot{H}||_\infty:=\max |\frac{d}{ds}H_s(x)|$. Note that $||\dot{H}||_1\leq ||\dot{H}||_\infty\cdot(s_+-s_-)=||\dot{H}||_\infty$ here, so that (\ref{eq27}) implies (\ref{eq16}) with $d\leq\frac{1}{8}$. Write again $\displaystyle\lim_{s\rightarrow -\infty} H_s=H_0=:H_-$ and $\displaystyle\lim_{s\rightarrow +\infty} H_s=H_1=:H_+$ and pick Morse-functions $h_\pm$ on the critical manifolds $crit\mspace{-4mu}\left(\mathcal{A}^{H_\pm}\right)$.\\
For two critical points $c^\pm\in crit(h_\pm)$ we consider the moduli space $\widehat{\mathcal{M}}(c^-,c^+)$ of trajectories with cascades, where exactly one cascade consists of an $\mathcal{A}^{H_s}$-gradient trajectory, while all other cascades are either $\mathcal{A}^{H_-}$- or $\mathcal{A}^{H_+}$-gradient trajectories. It follows from a parametric version of the Global Transversality Theorem \ref{theotrans} that $\widehat{\mathcal{M}}(c^-,c^+)$ is a manifold and it follows from the Compactness Theorem \ref{theo3.6} that its zero-dimensional component $\widehat{\mathcal{M}}^0(c^-,c^+)$ is a finite set.\pagebreak\\
We define a linear map $\phi:RFC(H_+,h_+)\rightarrow RFC(H_-,h_-)$ as the linear extension of
\begin{equation}\label{connect}
 \phi(c^+):=\sum_{c^-\in crit(h_-)} \#_2 \widehat{\mathcal{M}}^0(c^-,c^+)\cdot c^-,\qquad c^+\in crit(h_+).
\end{equation}
\underline{\textit{Claim}}: The sum on the right hand side satisfies the finiteness condition (\ref{finite}).\\
\underline{\textit{Proof}}: Condition (\ref{eq27}) on $H_s$, i.e. $d=\frac{1}{8}$, guarantees that Corollary \ref{cor3.8} can be applied with $k=2$. Hence $ \widehat{\mathcal{M}}^0(c^-,c^+)\neq\emptyset$ implies 
\begin{itemize}
 \item $\mathcal{A}^{H_-}(c^-)\leq\max\big\{2\mathcal{A}^{H_+}(c^+),\frac{2d}{1-3d}\big\} $ if $\mathcal{A}^{H_+}(c^+)>-\frac{2d}{1-3d}$
 \item $\mathcal{A}^{H_-}(c^-)\leq\frac{1}{2}\mathcal{A}^{H_+}(c^+)$ \hspace{2.2cm} otherwise.
\end{itemize}
 The action of all $c^-$ on the right hand side is therefore bounded from above in terms of $\mathcal{A}^{H_+}(c^+)$. Condition (\ref{finite}) follows therefore from the fact that the action spectrum of $\mathcal{A}^{H_-}$ is closed and discrete (Theorem \ref{theo17}).\hfill $\square$\medskip\\
Using again the Compactness Theorem \ref{theo3.6}, it follows by standard arguments in Floer theory (i.e.\ glueing, see \cite{Sal} and \cite{Fra}) that
\[\partial^-\circ\phi=\phi\circ\partial^+\]
so that $\phi$ induces a well-defined homomorphism on Floer homologies
\[\Phi: RFH(H_+,h_+)\rightarrow RFH(H_-,h_-).\]
The inverse homotopy $\bar{H}_s:=H_{1-s}$ yields a homomorphism
\[\Psi: RFH(H_-,h_-)\rightarrow RFH(H_+,h_+).\]
For $R\geq 1$, we define the concatenation of $H_s$ and $\bar{H}_s$ by the formula
\[ K_s:=H_s\#_R\bar{H}_s=\begin{cases}H_{s+R}\quad s\leq0\\\bar{H}_{s-R}\quad s\geq 0 \end{cases}\]
which yields a homotopy $K_s$ from $H_-$ via $H_+$ back to $H_-$. Note that $K_s$ is non-constant only for $-R\leq s\leq 1-R$ and $R-1\leq s\leq R$. From (\ref{eq27}) and Proposition \ref{prop3.4} follows 
\[\left(c+\frac{||K||_\infty}{\veps^2}\right)\cdot||\beta'||_\infty||\dot{K}||_1<\frac{1}{4}.\]
Note that we can use the same constants $\veps$ and $c$ as for $H_s$ and $\bar{H}_s$ as they depend by Proposition \ref{prop3.4} only on the set $\{H_s\}=\{K_s\}$. Using again the Corollaries \ref{cor3.7} and \ref{cor3.8} and the standard gluing argument, we see that the composition
\[\Phi\circ\Psi: RFH(H_-,h_-)\rightarrow RFH(H_-,h_-)\]
is given by counting gradient flow lines of $\mathcal{A}^{K_s}$ (see again \cite{Sal}).\pagebreak\\
Now, for $r\in[0,1]$ consider the homotopies of homotopies
\[H^r_s:=H_{r\cdot s}\qquad\text{ and }\qquad \bar{H}^r_s:=\bar{H}_{r\cdot s}\qquad\text{ and }\qquad K^r_s=H^r_s\#_r\bar{H}^r_s.\]
Then for each $r\in[0,1]$, the following estimate still continues to hold:
\[\left(c+\frac{||K^r||_\infty}{\veps^2}\right)\cdot||\beta'||_\infty||\dot{K^r}||_1<\frac{1}{4}.\]
Moreover, we have that $K^0_s=H_-$ does not depend on $s$ any more and therefore induces the identity on $RFH(H_-,h_-)$. It follows that
\begin{align*}
 && \Phi\circ\Psi &= \text{identity on }\; RFH(H_-,h_-)\\
 &\text{and similarly}& \Psi\circ\Phi &= \text{identity on }\; RFH(H_+,h_+).
\end{align*}
Thus, $\Phi$ is an isomorphism between $RFH(H_+,h_+)$ and $RFH(H_-,h_-)$. This finishes the proof under the additional assumption (\ref{eq27}).\bigskip\\
\underline{Step 2}\\
Now consider a general homotopy $H_s$, so that (\ref{eq27}) is not necessarily satisfied. Then we define for any $N\in\mathbb{N}$ and $0\leq j\leq N-1$ the slower homotopies
\begin{equation}\label{eq28}
 H_s^{N,j}:=H_{(j+s)/N}\quad\text{ for }\; 0\leq s\leq 1\qquad\text{ and }\qquad H_s^{N,j}:=H_{\beta(s)}^{N,j}\quad\text{ for }\; s\in\mathbb{R}.
\end{equation}
We see that $||H^{N,j}||_\infty\leq ||H||_\infty$ and $||\dot{H}^{N,j}||_\infty\leq\frac{1}{N}||\dot{H}||_\infty$. Hence, we may choose $N\in\mathbb{N}$ so large, such that for each $0\leq j\leq N-1$ holds
\begin{equation*}
  \left( c+\frac{||H^{N,j}||_\infty}{\veps}\right)\cdot||\beta'||_\infty||\dot{H}^{N,j}||_\infty\leq\frac{1}{8}.
 \end{equation*}
 Write $H_j:=H^{N,j}_0=H^{N,j-1}_1$ for the ends of the slow homotopies and choose Morse functions $h_j$ for $crit\mspace{-4mu}\left(\mathcal{A}^{H_j}\right)$. Then, step 1 yields for $0\leq j\leq N-1$ isomorphisms 
 \[\Phi_j : RFH(H_j,h_j)\rightarrow RFH(H_{j+1},h_{j+1}).\]
 The composition of these isomorphisms then shows that $RFH(H_0,h_0)=RFH(H_-,h_-)$ and $RFH(H_N,h_N)=RFH(H_+,h_+)$ are isomorphic.
\end{proof}
The action $\mathcal{A}^H$ induces an $\mathbb{R}$-filtration on the chain complex $RFC(H,h)$. This allows us to define for $-\infty\leq a\leq b\leq \infty,\,a,b\not\in spec(\Sigma,\alpha)$ the following truncated chain complexes\footnote{For variations on this definition, including $a,b\in spec(\Sigma,\alpha)$ see Section \ref{sectrunc}.}:
\begin{align*}
  RFC^{<b}(H,h)&:=\bigg\{\underset{crit(h)}{\textstyle\sum} \xi_c\cdot c\;\bigg|\;\xi_c=0 \text{ if }\mathcal{A}^H(c)\geq b\bigg\},\\
  RFC^{(a,b)}(H,h)&:=\raisebox{.2em}{$RFC^{<b}(H,h)$}\left/\raisebox{-.2em}{$RFC^{<a}(H,h)$}\right..
\end{align*}
Note that $RFC^{<\infty}(H,h)$ is the original chain complex $RFC(H,h)$.\pagebreak\\
For $a\leq b\leq c$, we have the following natural short exact sequence of chain complexes:
\begin{align}\label{eq23a}
 0\rightarrow RFC^{(a,b)}(H,h)\overset{i}{\hookrightarrow} RFC^{(a,c)}(H,h)\overset{\pi}{\twoheadrightarrow} RFC^{(b,c)}(H,h)\rightarrow 0.
\end{align}
As $\partial^F$ reduces the action, it descends to the truncated chain complexes, yielding the truncated homology groups
\[RFH^{<b}(H,h)\qquad\text{ and }\qquad RFH^{(a,b)}(H,h).\]
For $\veps>0$ smaller then the smallest period of a closed Reeb orbit on $\Sigma$ we define
\begin{align*}
 RFH^{(0^\pm,\infty)}(H,h)&:=RFH^{(\pm\veps,\infty)}(H,h)\\
 RFH^{(-\infty,0^\pm)}(H,h)&:=RFH^{(-\infty,\pm\veps)}(H,h).
\end{align*}
The short exact sequence (\ref{eq23a}) gives the following long exact sequence in homology:
\begin{align}\label{eq25a}
 \rightarrow RFH^{(b,c)}(H,h)\rightarrow RFH^{(a,b)}(H,h)\overset{i_\ast}{\rightarrow} RFH^{(a,c)}(H,h)\overset{\pi_\ast}{\rightarrow} RFH^{(b,c)}(H,h)\rightarrow
\end{align}
If there is no closed Reeb orbit with period in $(a,b)$, then this long exact sequence shows that $RFH^{(b,c)}(H,h)$ and $RFH^{(a,c)}(H,h)$ are isomorphic. This proves that the above definition of $RFH^{(0^\pm,\infty)}(H,h)$ and $RFH^{(-\infty,0^\pm)}(H,h)$ is independent from $\veps$, if $\veps$ is small enough.\\
The maps $\pi_\ast$ and $i_\ast$ give $RFH^{(a,b)}$ the structure of a bidirected system. In Theorem \ref{theolimits}, we show that 
\begin{align*}
 && RFH^{<b}(H,h)&\cong\lim_{\underset{a}{\longleftarrow}} RFH^{(a,b)}(H,h) \text{ as }a\rightarrow-\infty\\
 &\text{and}& RFH(H,h)&\cong\lim_{\underset{b}{\longrightarrow}}RFH^{<b}(H,h)\;\;\text{ as }b\rightarrow \infty.
\end{align*}
The truncated homology groups are independent of $J$, $h$ and $g_h$. However, they do depend on the chosen contact form and are therefore in general not invariant under homotopies $H_s$, except for the groups $RFH^{(0^\pm,\infty)}(H,h)$ and $RFH^{(-\infty,0^\pm)}(H,h)$, as we shall see below. For an arbitrary action window we have only the following result.
\begin{cor}\label{cormain1}
 Let $a,b\in(\mathbb{R}\cup\{-\infty,\infty\})\setminus\text{spec }(\Sigma,\lambda)$. If $H_+$ and $H_-$ are two defining Hamiltonians for $\Sigma$, then $RFH^{(a,b)}(H_+,h_+)$ and $RFH^{(a,b)}(H_-,h_-)$ are isomorphic.
\end{cor}
\begin{proof}
The proof is based on the following idea. If the action spectrum $\mathcal{A}^{H_s}$ is fixed, we can split the homotopy $H_s$ in slower homotopies $H_s^{N,j}$ which allow us to deduce from Corollary \ref{cor3.8} that no $\mathcal{A}^{H_s^{N,j}}$-gradient trajectory can cross the action boundaries $a$ and $b$.\medskip\\
 To start, recall that the space of defining Hamiltonians for $\Sigma$ is convex (Proposition \ref{cor2}). Hence, we can find a homotopy $H_s,\; 0\leq s\leq 1,$ between $H_-$ and $H_+$, where all $H_s$ are defining Hamiltonians for $\Sigma$. For $N\in\mathbb{N}$, we split $H_s$ as in (\ref{eq28}) into the slower homotopies
 \begin{align*}
 H_s^{N,j}&:=H_{(j+s)/N}\quad\text{ for }\; 0\leq s\leq 1&&\text{ and }\qquad H_s^{N,j}:=H_{\beta(s)}^{N,j}\quad\text{ for }\; s\in\mathbb{R}.
\end{align*}
Write again $H_j=H^{N,j}_0=H^{N,j-1}_1$ for the ends of these homotopies.
Note that $crit\mspace{-4mu}\left(\mathcal{A}^{H_j}\right)$ does not depend on $H_j$ but only on $\Sigma$ as it is for every $j$ the set of all closed Reeb-orbits on $\Sigma$. It follows that the action spectrum is $spec\,(\Sigma,\lambda)$ for all $j$. \pagebreak\\ As $a,b\not\in spec\,(\Sigma,\lambda)$, we can choose $k>1$ so large such that for any $w\in crit\mspace{-4mu}\left(\mathcal{A}^{H_j}\right)=\crita$ with $a<\mathcal{A}^{H_j}(w)<b$ holds
\[a<\frac{k-1}{k}\cdot\mathcal{A}^{H_j}(w)<b\qquad\text{ and }\qquad a<\frac{k}{k-1}\cdot\mathcal{A}^{H_j}(w)<b\qquad\forall\;\, 0\leq j\leq N.\tag{$\ast$}\]
Then we may choose $N$ so large, such that 
\[d:=\left(c+\frac{||H^{N,j}||_\infty}{\veps^2}\right)\cdot||\beta'||_\infty\cdot||\dot{H}^{N,j}||_\infty\]
is so small, that $kd/(1-kd-d)$ is smaller then the minimal period of a closed Reeb orbit on $\Sigma$. Consider $(v^+,\eta^+)\in RFC(H_{j+1}),\;\,(v^-,\eta^-)\in RFC(H_j)$ and assume that there exists an $\mathcal{A}^{H^{N,j}_s}$-gradient trajectory connecting them. Assume further that $\mathcal{A}^{H_{j+1}}(v^+,\eta^+)<b$. Corollary \ref{cor3.8} then implies that
\begin{itemize}
 \item if $\mathcal{A}^{H_{j+1}}(v^+,\eta^+)>0$, then $\mathcal{A}^{H_{j}}(v^-,\eta^-)\leq \frac{k}{k-1}\cdot\mathcal{A}^{H_{j+1}}(v^+,\eta^+)$ and due to $(\ast)$ therefore $\mathcal{A}_{H_{j}}(v^-,\eta^-)<b$,
 \item if $\mathcal{A}^{H_{j+1}}(v^+,\eta^+)<0$, then $\mathcal{A}^{H_j}(v^-,\eta^-)\leq \frac{k-1}{k}\cdot\mathcal{A}^{H_{j+1}}(v^+,\eta^+)$ and again with $(\ast)$ that $\mathcal{A}^{H_{j}}(v^-,\eta^-)<b$,
 \item if $\mathcal{A}^{H_{j+1}}(v^+,\eta^+)=0$, then $\mathcal{A}^{H_j}(v^-,\eta^-)\leq0= \mathcal{A}^{H_{j+1}}(v^+,\eta^+)<b$, as otherwise Corollary \ref{cor3.8}(a) would imply that $\mathcal{A}^{H_{j+1}}(v^+,\eta^+)$ is positive.
\end{itemize}
We obtain an analog result for $a$ instead of $b$. Let $\phi^j: RFC(H_{j+1})\rightarrow RFC(H_j)$ be defined as in (\ref{connect}) by counting solutions of the $s$-dependent Rabinowitz-Floer equation. The above estimates then show that the $\phi^j$ descend to well-defined maps
\[\phi^j : RFC^{(a,b)}(H_{j+1})\rightarrow RFC^{(a,b)}(H_j),\]
which induce, as in the untruncated case, isomorphisms in homology
\[\Phi^j: RFH^{(a,b)}(H_{j+1})\rightarrow RFH^{(a,b)}(H_j).\]
The composition of the $\Phi^j$ then yields $RFH^{(a,b)}(H_+)\cong RFH^{(a,b)}(H_-)$.
\end{proof}
In the definition of the Rabinowitz-Floer homology we assumed that the ambient manifold $V$ is the completion of a Liouville domain. For a Liouville domain $W$ with $\partial W=\Sigma$, we can always set $V:=\hat{W}$. The following theorem shows that the Rabinowitz-Floer homology actually only ``sees'' this completion. This means that we get the same homology if we take larger ambient manifolds $V\supset\hat{W}$.
\begin{theo}[\textbf{Cieliebak, Frauenfelder \& Oancea,\cite{FraCieOan}}]\label{theo2}~\\
 The Rabinowitz-Floer homology $RFH(W,\Sigma)$ does not depend on the ambient manifold $(V,\lambda)$, but only on the compact Liouville domain $(W,\lambda|_W)$ bounded by $\Sigma$.
\end{theo}
\begin{proof}
 As $(V,\lambda)$ is the completion of a Liouville domain, the Liouville vector field $X_\lambda$ is complete. Its flow defines therefore a symplectic embedding $i:\Sigma\times\mathbb{R}\hookrightarrow V$ of the symplectization of $(\Sigma,\alpha)$ such that $i^\ast\lambda=e^r\cdot\alpha$ (see Discussion \ref{dis2}).\pagebreak\\ Pick a cylindrical almost complex structure $J_\Sigma$ on $\Sigma\times\mathbb{R}$. By Gromov's Monotonicity Lemma (see \cite{Sik}, Prop.\ 4.3.1 and \cite{Oan2}, Lem.\ 1), there exists an $\veps>0$ such that $J_\Sigma$-holomorphic curves in $\Sigma\times\mathbb{R}$ which meet the level $\Sigma\times\{\log 3\}$ and exit $\Sigma\times[\log 2,\log 4]$ have symplectic area at least $\veps$. Rescaling by $R>1$, it follows that $J_\Sigma$-holomorphic curves which meet the level $\Sigma\times\{\log 3R\}$ and exit the set $\Sigma\times[\log 2R,\log 4R]$ have symplectic area at least $R\veps$.\\
 Now fix a defining Hamiltonian $H$ for $\Sigma$. Then there exists a constant $c>0$ such that $H$ is constant outside $W\cup i\big(\Sigma\times(-\infty,c)\big)$. For any $\log R>c$ pick an admissible almost complex structures $J$ on $(V,\lambda)$ such that $i^\ast J=J_\Sigma$ over $\Sigma\times[\log 2R,\log 4R]$. Assume that $(v,\eta)$ is an $\mathcal{A}^H$-gradient trajectory with asymptotics $(v^\pm,\eta^\pm)\in\crita$ such that $v$ meets the level $i(\Sigma\times\{\log 3R\})$. As the $v^\pm$ are contained in $\Sigma=i(\Sigma\times\{0\})$, $v$ exits the set $i(\Sigma\times[\log 2R,\log 4R])$. Let $U\subset \mathbb{R}\times S^1$ be a connected component of $v^{-1}\big(i(\Sigma\times[\log 2R,\log 4R])\big)$. As $X_H$ vanishes over $i(\Sigma\times[\log 2R,\log 4R])$, it follows that $v$ has over $U$ a symplectic area of at least $R\veps$. This allows us to estimate
 \[\mathcal{A}^H(v^+,\eta^+)-\mathcal{A}^H(v^-,\eta^-)=\int^\infty_{-\infty}\left|\left|\nabla\mathcal{A}^H(v,\eta)(s)\right|\right|^2ds\geq\int_U\left|\frac{d}{ds}v\right|^2dsdt=\int_Uv^\ast d\lambda\geq R\veps.\]
 Thus, $v$ can leave $W\cup i(\sigma\times(-\infty,\log 2R))$ only if the action difference of its asymptotics is greater or equal to $R\veps$. By choosing $R$ large enough, we find hence that the moduli space $\widehat{M}(c^-,c^+)$ involves only $\mathcal{A}^H$-gradient trajectories which are contained in the completion $(\hat{W},\hat{\lambda})$. This shows that $RFH^{(a,b)}(V,\Sigma)$ can be computed using only the completion $(\hat{W},\hat{\lambda})$ and is therefore independent from the ambient manifold.\\
 Since $\displaystyle RFH(V,\Sigma)=\lim_{b\rightarrow \infty}\lim_{ -\infty\leftarrow a}RFH^{(a,b)}(V,\Sigma)$ by Theorem \ref{theolimits}, the independence carries over to the full Rabinowitz-Floer homology.
\end{proof}
\begin{cor}\label{liouvilleinvariance}
 The Rabinowitz-Floer homology $RFH(W,\Sigma)$ is invariant under Liouville isomorphisms. It is thus an invariant of the exact contact filling $(W,\Sigma,\xi)$.
 \end{cor}
\begin{proof}
 At first, we consider only the trivial Liouville isomorphism of $(W,\lambda)$ with itself. Fix any defining Hamiltonian $H_0$ for $\Sigma$ and a Morse function $h_0$ for $crit\mspace{-4mu}\left(\mathcal{A}^{H_0}\right)$. Let $f\in C^\infty(\Sigma)$ be an arbitrary smooth function and consider the exact contact hypersurface $\Sigma^f:=\{(y,f(y))\,|\,y\in\Sigma \}$ in $\hat{W}$. Fix a defining Hamiltonian $H_1$ for $\Sigma^f$ and a Morse function $h_1$ for $crit\mspace{-4mu}\left(\mathcal{A}^{H_1}\right)$.\\
 Due to Proposition \ref{cor2}, there exists a homotopy of defining Hamiltonians $H_s,\; 0\leq s\leq 1$ between $H_0$ and $H_1$. It follows from Theorem \ref{theo1} that $RFH(H_0,h_0)$ and $RFH(H_1,h_1)$ are isomorphic. Thus, the Rabinowitz-Floer homology is invariant under trivial Liouville isomorphisms.\bigskip\\
 Now consider any Liouville isomorphism $\varphi:(W_0,\Sigma_0,\lambda_0)\rightarrow(W_1,\Sigma_1,\lambda_1)$. Fix any defining Hamiltonians $H_0,H_1$ for $\Sigma_0,\Sigma_1$ and Morse functions $h_0,h_1$ for $crit\mspace{-4mu}\left(\mathcal{A}^{H_0}\right)$ and $crit\mspace{-4mu}\left(\mathcal{A}^{H_1}\right)$. It follows from Proposition \ref{LI} that there exists an $R>0$ such that on $\Sigma_0\times[R,\infty)$ holds $\varphi^\ast\hat{\lambda}_2=\hat{\lambda}_1$ and $\varphi$ is of the form
 \[\varphi(y,r)=\big(\psi(y),r-f(y)\big),\]
 where $\psi:\Sigma_0\rightarrow\Sigma_1$ is a contact isomorphism satisfying $\psi^\ast\alpha_1=e^f\cdot\alpha_0$ for $f\in C^\infty(\Sigma_0)$. The image $\varphi(\Sigma_0\times\{R\})\subset\hat{W}_1$ is hence the hypersurface $\Sigma^{R-f}_1$.\pagebreak\\ Pick a defining Hamiltonian $H$ for $\Sigma^{R-f}_1$ and a Morse function $h$ for $\crita$. It follows from the discussion above, that $RFH(H,h)$ and $RFH(H_1,h_1)$ are isomorphic. Let $\varphi^\ast H:=H\circ\varphi$ and $\varphi^\ast h=h\circ\varphi$ be the pullbacks. As $\varphi$ is a global symplectomorphism, we find that $\varphi^\ast H$ is a defining Hamiltonian for $\Sigma_0\times\{R\}$.\\
 For any generic almost complex structure $J_1$ on $\hat{W}_1$, we choose the almost complex structure $J_0$ on $\hat{W}_0$ to be
 \[J_0=\varphi^\ast J_1:=D\varphi^{-1}\circ J_1 \circ D\varphi.\]
 Analogously, let $g_0 =\varphi^\ast g_1$, be the pullback of a generic metric on $crit\mspace{-4mu}\left(\mathcal{A}^{H_1}\right)$. Then, we find that $\varphi^\ast h$  is a Morse function on $crit\mspace{-4mu}\left(\mathcal{A}^{\varphi^\ast H}\right)$ and all trajectories with cascades of $(\varphi^\ast H,\varphi^\ast h)$ are in one-to-one correspondence to the trajectories with cascades of $(H,h)$. Therefore, we have that $RFH(H,h)$ and $RFH(\varphi^\ast H,\varphi^\ast h)$ are isomorphic.\\
 A discussion similar to the one above shows that $RFH(H_0,h_0)$ and $RFH(\varphi^\ast H,\varphi^\ast h)$ are also isomorphic. Combining the 3 isomorphisms gives $RFH(H_0,h_0)\cong RFH(H_1,h_1)$.
\end{proof}
\begin{defn}\label{defngrowth}
 Let $(W,\Sigma,\lambda)$ be a Liouville domain and let $f:\mathbb{R}\rightarrow\mathbb{R}$ be a strictly increasing function with $\displaystyle\lim_{x\rightarrow \infty} f(x)=\infty$. Hence, $f$ is invertible and $\displaystyle\lim_{x\rightarrow \infty} f^{-1}(x)=\infty$.\linebreak  For $a>0$, let $d^+(\Sigma,a)$ be the $\mathbb{Z}_2$-dimension of the image $i_\ast\big(RFH^{(0,a)}(W,\Sigma)\big)$ in \linebreak $RFH^{(0,\infty)}(W,\Sigma)$ and let $d^-(\Sigma,a)$ be the $\mathbb{Z}_2$-dimension of  $\pi_\ast\big(RFH^{(-\infty,0)}(W,\Sigma)\big)$ in \linebreak $RFH^{(-a,0)}(W,\Sigma)$. Clearly $d^+(\Sigma,a)$ and $d^-(\Sigma,a)$ are increasing functions in $a$. We define the positive/negative \textbf{\textit{growth rates}} of class $f$ of a Liouville domain $(W,\Sigma,\lambda)$ by
 \[\Gamma^\pm(W,\Sigma,f):=\varlimsup_{a\rightarrow \infty}\frac{\big(f^{-1}\circ\log\big)\big(d^\pm(\Sigma,a)\big)}{\log(a)}\in\{-\infty\}\cup[0,\infty].\]
\end{defn}
\begin{rem}
 For $f=id$, we say that $\Gamma^\pm(W,\Sigma,id)$ is the polynomial growth, for $f=\log$ the logarithmic, for $f=e^x$ the exponential growth.
\end{rem}
\begin{cor}\label{XX}
 Let $(W,\Sigma,\lambda)$ be a Liouville domain. The following growth rates and truncated groups are invariant under Liouville isomorphisms
 \[RFH^{(-\infty,0^\pm)}(W,\Sigma),\quad RFH^{(0^\pm,\infty)}(W,\Sigma)\quad\text{ and }\quad\Gamma^\pm(W,\Sigma,f).\]
\end{cor}
\begin{rem}
 As mentioned above, the truncated Rabinowitz-Floer groups $RFH^{(a,b)}(W,\Sigma)$ are in general not invariant under Liouville isomorphisms. This is easy to see, simply rescale $\Sigma$, i.e.\ consider in the symplectization $\Sigma\times\{R\}\hookrightarrow \hat{W}$ for $R\neq 1$.
\end{rem}
\begin{proof}~\\
\underline{Step 1:} Invariance of $RFH^{(-\infty,0^\pm)}(W,\Sigma)$ and $RFH^{(0^\pm,\infty)}(W,\Sigma)$\medskip\\
 Following the same arguments as the previous corollary, it suffices to show invariance under homotopies of defining Hamiltonians. Let $H_s,\;0\leq s\leq 1,$ be a homotopy of Hamiltonians which are defining for exact contact hypersurfaces $\Sigma_s:=H^{-1}_s(0)$. For any $N\in\mathbb{N}$, we may split $H_s$ again as is (\ref{eq28}) into slower homotopies
 \begin{align*}
 H_s^{N,j}&:=H_{(j+s)/N}\quad\text{ for }\; 0\leq s\leq 1&&\text{ and }\qquad H_s^{N,j}:=H_{\beta(s)}^{N,j}\quad\text{ for }\; s\in\mathbb{R}
\end{align*}
and we write as before $ H_j=H^{N,j}_0=H^{N,j-1}_1$ for the ends of the homotopies $H^{N,j}_s$.\pagebreak\\ Fix an $k>1$. Then we can choose $N$ so large, such that
\[d:=\left(c+\frac{||H^{N,j}||_\infty}{\veps^2}\right)\cdot||\beta'||_\infty\cdot||\dot{H}^{N,j}||_\infty\]
becomes so small that $kd/(1-kd-d)$ is smaller then the smallest (positive) minimal period of a closed Reeb orbit on any $\Sigma_s$.\\
Note that we cannot require that closed Reeb orbits stay in a fixed action window as in $(\ast)$ in Corollary \ref{cormain1}, since $spec(\Sigma_s,\lambda)$ is no longer independent from $s$, except for the common spectral value $0$.\\
Let $(v^+,\eta^+)\in RFC(H_{j+1}),\;\,(v^-,\eta^-)\in RFC(H_j)$ and assume that there exists an $\mathcal{A}^{H^{N,j}_s}$-gradient trajectory connecting them. For any $a>0$ we find that our choice of $kd/(1-kd-d)$ together with Corollary \ref{cor3.8} implies that if
\begin{itemize}
 \item[(1)] $\mathcal{A}^{H_+}(v^+,\eta^+)<a$, \hspace{1.55cm} then $\mathcal{A}^{H_-}(v^-,\eta^-)<\frac{k}{k-1}\cdot a$,
 \item[(2)] $\mathcal{A}^{H_+}(v^+,\eta^+)<-\frac{k}{k-1}\cdot a$, \hspace{0.3cm} then $\mathcal{A}^{H_-}(v^-,\eta^-)<-a$,
 \item[(3)] $\mathcal{A}^{H_+}(v^+,\eta^+)\leq 0$,\hspace{1.75cm} then $\mathcal{A}^{H_-}(v^-,\eta^-)\leq 0$.
\end{itemize}
Let $\phi^j: RFC(H_{j+1})\rightarrow RFC(H_j)$ be defined as in (\ref{connect}) by counting solutions of the $s$-dependent Rabinowitz-Floer equation. Abbreviating $C:=\frac{k}{k-1}$, the statements (1)-(3) show that the $\phi^j$ descend to well-defined maps
\begin{align*}
 \mspace{160mu}\phi^j&:& RFC^{(0^\pm,a)}(H_{j+1})&\rightarrow RFC^{(0^\pm,C a)}(H_j)\\
 \phi^j&:& RFC^{(-C a,0^\pm)}(H_{j+1})&\rightarrow RFC^{(-a,0^\pm)}(H_j),\mspace{160mu}
\end{align*}
which then induce maps in homology
\begin{equation}\label{0}\begin{aligned}
 \mspace{100mu}\Phi^j&:& RFH^{(0^\pm,a)}(H_{j+1})&\rightarrow RFH^{(0^\pm,C a)}(H_j)\mspace{100mu}\\
 \Phi^j&:& RFH^{(-C a,0^\pm)}(H_{j+1})&\rightarrow RFH^{(-a,0^\pm)}(H_j).
\end{aligned}
\end{equation}
Considering the inverse homotopy $\overline{H}_s:=H_{1-s}$ yields maps
\begin{equation}\label{00}\begin{aligned}
 \mspace{100mu}\Psi^j&:& RFH^{(0^\pm,a)}(H_j)&\rightarrow RFH^{(0^\pm,C a)}(H_{j+1})\mspace{100mu}\\
 \Psi^j&:& RFH^{(-C a,0^\pm)}(H_j)&\rightarrow RFH^{(-a,0^\pm)}(H_{j+1}).
\end{aligned}
\end{equation}
The compositions maps
\begin{equation*}
 \begin{aligned}
  \mspace{110mu}\Psi^j\circ\Phi^j&:& RFH^{(0^\pm,a)}(H_{j+1})&\rightarrow RFH^{(0^\pm,C^2\cdot a)}(H_{j+1})\mspace{110mu}\\
  \Psi^j\circ\Phi^j&:& RFH^{(-C^2a,0^\pm)}(H_{j+1})&\rightarrow RFH^{(-a,0^\pm)}(H_{j+1})
 \end{aligned}
\end{equation*}
are just the truncation maps from the long exact sequence (\ref{eq25a}) induced by the inclusion $RFC^{(0^\pm,a)}\hookrightarrow RFC^{(0^\pm,C^2\cdot a)}$ and the projection $RFC^{(-C^2\cdot a,0^\pm)}\twoheadrightarrow RFC^{(-a,0^\pm)}$. This holds true as the untruncated maps $\Psi^j\circ\Phi^j$ are isomorphisms. With $a=\infty$, we find that
\begin{equation*}\begin{aligned}
 \mspace{100mu}\Phi^j&:& RFH^{(0^\pm,\infty)}(H_{j+1})&\rightarrow RFH^{(0^\pm,\infty)}(H_j)\mspace{100mu}\\
 \Phi^j&:& RFH^{(-\infty,0^\pm)}(H_{j+1})&\rightarrow RFH^{(-\infty,0^\pm)}(H_j)
\end{aligned}
\end{equation*}
are isomorphisms, as $\Phi^j\circ\Psi^j$ and $\Psi^j\circ\Phi^j$ are isomorphisms. Combining all $\Phi^j$ yields $RFH^{(0^\pm,\infty)}(H_0)\cong RFH^{(0^\pm,\infty)}(H_1)$ and $RFH^{(-\infty,0^\pm)}(H_0)\cong RFH^{(-\infty,0^\pm)}(H_1)$.\pagebreak\\
\underline{Step 2:} Invariance of $\Gamma^\pm(W,\Sigma,f)$\medskip\\
For any $a<\infty$, we find that the composition of all $\Phi^j$ resp.\ all $\Psi^j$ yields maps
\begin{align*}
 \mspace{170mu}\Phi^+ &:& RFH^{(0,a)}(H_1)&\rightarrow RFH^{(0,D\cdot a)}(H_0)\mspace{170mu}\\
 \Phi^- &:& RFH^{(-D\cdot a,0)}(H_1)&\rightarrow RFH^{(-a,0)}(H_0)\\
 \Psi^+ &:& RFH^{(0,a)}(H_0)&\rightarrow RFH^{(0,D\cdot a)}(H_1)\\
 \Psi^- &:& RFH^{(-D\cdot a,0)}(H_0)&\rightarrow RFH^{(-a,0)}(H_1),
\end{align*}
with $D:=C^N=(\frac{k}{k-1})^N$. Their compositions yield again natural truncation maps:
\begin{align*}
 \mspace{120mu}\Psi^+\circ\Phi^+&:& RFH^{(0,a)}(H_1)&\rightarrow RFH^{(0,D^2\cdot a)}(H_1),\mspace{120mu}\\
 \Psi^-\circ\Phi^-&:& RFH^{(-D^2\cdot a,0)}(H_1)&\rightarrow RFH^{(-a,0)}(H_1).
\end{align*}
 Therefore, we get the following ladder-shaped diagrams:
\begin{align*}
 \begin{xy} \xymatrix{
  \dots &\dots\\
  RFH^{(0,D^4\cdot a)}(H_0)\ar[u]\ar[r]& RFH^{(0,D^5\cdot a)}(H_1)\ar[ul]\ar[u]\\
  RFH^{(0,D^2\cdot a)}(H_0)\ar[u]\ar[r]& RFH^{(0,D^3\cdot a)}(H_1)\ar[ul]\ar[u]\\
  RFH^{(0,\,a)}(H_0)\ar[u]\ar[r]& RFH^{(0,D a)}(H_1)\ar[ul]\ar[u]}
  \end{xy}\mspace{30mu}
  \begin{xy}
  \xymatrix{
  \dots\ar[d]\ar[dr] &\dots\ar[d]\\
  RFH^{(-D^4\cdot a,0)}(H_0)\ar[d]\ar[dr]& \ar[l]RFH^{(-D^5\cdot a,0)}(H_1)\ar[d]\\
  RFH^{(-D^2\cdot a,0)}(H_0)\ar[d]\ar[dr]&\ar[l] RFH^{(-D^3\cdot a,0)}(H_1)\ar[d]\\
  RFH^{(-a,0)}(H_0)& \ar[l]RFH^{(-D a,0)}(H_1).}
 \end{xy}
\end{align*}
They imply the following chain of inequalities:
\[d^\pm(H_0,a)\leq d^\pm(H_1,Da)\leq d^\pm(H_0,D^2 a)\leq\dots\]
Now assume that $RFH^{(0,\infty)}(H_0)$ and $RFH^{(-\infty,0)}(H_0)$ are infinite dimensional. As these are obtained by direct/inverse limits, this implies that $d^\pm(H_0,a)\rightarrow \infty$ as $a\rightarrow\infty$, which yields in particular
\[\lim_{a\rightarrow\infty}\frac{\log(D)}{f^{-1}(\log(d^\pm(H_0,a)))}=0.\]
With this result, we obtain
\begin{align*}
 \frac{1}{\Gamma^\pm(H_0,f)}=\varliminf_{a\rightarrow\infty}\frac{\log(a)}{f^{-1}(\log(d^\pm(H_0,a)))}&=\varliminf_{a\rightarrow\infty}\frac{\log(D a)}{f^{-1}(\log(d^\pm(H_0,a)))}\\
 &\geq \varliminf_{a\rightarrow\infty}\frac{\log(D a)}{f^{-1}(\log(d^\pm(H_1,Da)))} =\frac{1}{\Gamma^\pm(H_1,f)}.
\end{align*}
The same argument works in the opposite direction, so that we get altogether \linebreak $\Gamma^\pm(H_0,f)=\Gamma^\pm(H_1,f)$. In the remaining case, where $RFH^{(0,\infty)}(H_0)$ or $RFH^{(-\infty,0)}(H_0)$ are finite dimensional, the growth rate is either zero, if $0< \dim RFH^{(0,\infty)}(H_0)<\infty$, or $-\infty$, if $0= \dim RFH^{(0,\infty)}(H_0)$.
\end{proof}

\subsection{The Conley-Zehnder index}\label{1.4}
To obtain more information about its structure, we endow the Rabinowitz-Floer homology in the next section with a $\mathbb{Z}$-grading via the Conley-Zehnder index $\mu_{CZ}$, just as in regular Floer homology. To define $\mu_{CZ}$, let $Sp(2n)$ denote the group of $2n\times 2n$ symplectic matrices. In \cite{CoZeh}, Conley and Zehnder introduced a Maslov type index for paths $\Psi:[0,1]\rightarrow Sp(2n)$. Their index assigns an integer $\mu_{CZ}(\Psi)$ to every  path $\Psi$, provided that $\Psi(0)=\mathbbm{1}$ and $\det(\mathbbm{1}-\Psi(1))\neq 0$. Later, Robbin and Salamon gave in \cite{RoSa} a different definition, thus extending $\mu_{CZ}$ to arbitrary paths. It goes as follows:\\
Any smooth path $\Psi:[a,b]\rightarrow Sp(2n)$ can be expressed as a solution of an ordinary differential equation
\[\dot{\Psi}(t)=J_0 S(t)\Psi(t),\qquad \Psi(a)\in Sp(2n),\]
where $t\mapsto S(t)=S(t)^T$ is a smooth path of symmetric matrices and $J_0$ the standard almost complex structure. A time $t\in[a,b]$ is called a crossing if $\det(\mathbbm{1}-\Psi(t))=0$. The crossing form at a crossing $t$ is a quadratic form $\Gamma(\Psi,t)$ defined for $\xi_0\in\ker(\mathbbm{1}-\Psi(t))$ by the  formula\footnote{Actually, $(\ast)$ cannot be found in \cite{RoSa}. However, it follows from their Rem.\ 5.4 together with their Thm.\ 1.1(2) applied to the Lagrangian frame $(Id, \Psi)^T$ for $Graph(\Psi)$.}
\[\Gamma(\Psi,t)\xi_0=\langle\xi_0,S(t)\xi_0\rangle\tag{$\ast$}.\]
A crossing $t$ is called regular, if $\Gamma(\Psi,t)$ is non-degenerate. Regular crossings are isolated. For a path $\Psi$ with only regular crossings, the Conley-Zehnder index is defined by
\begin{align}\label{muCZdef}
 \mu_{CZ}(\Psi;a,b):=\frac{1}{2}\text{sign}\,\Gamma(\Psi,a)+\sum_{a<t<b}\text{sign}\,\Gamma(\Psi,t)+\frac{1}{2}\text{sign}\,\Gamma(\Psi,b),
\end{align}
where the sum runs over all crossings $t\in(a,b)$. Here, $\text{sign}\,A$ denotes the signature of $A$, i.e.\ the number of positive eigenvalues minus the number of negative eigenvalues. Note that $\Gamma(\Psi,a)=0$ or $\Gamma(\Psi,b)=0$ if $a$ or $b$ are not crossings. To ease notation, we will often omit one or both boundaries if they are clear from the context.\\
The index $\mu_{CZ}$ has (among others) the following properties:
\[\renewcommand{\arraystretch}{1.5}\begin{tabular}{lp{12cm}}
 \textbf{(Naturality)} & For any path $\Phi:[a,b]\rightarrow Sp(2n)$ holds $\mu_{CZ}(\Phi\Psi\Phi^{-1})=\mu_{CZ}(\Psi)$\\
 \textbf{(Homotopy)}& $\mu_{CZ}(\Psi_s)$ is constant in $s$ for any homotopy $\Psi_s$ with fixed endpoints\\
 \textbf{(Product)}& If $Sp(2n)\oplus Sp(2n')$ is identified with a subgroup of $Sp(2(n+n'))$ in the obvious way, then $\mu_{CZ}(\Psi\oplus\Psi')=\mu_{CZ}(\Psi)+\mu_{CZ}(\Psi').$
\end{tabular}\]
The homotopy property allows us to define $\mu_{CZ}(\Psi;a,b)$ also for paths with non-regular crossings, provided its ends $a$ and $b$ are regular crossings or no crossings. Just perturb $\Psi$ through a homotopy $\Psi_s,\, 0\leq s\leq 1$ to a path $\Psi_1$ with only regular crossings and set $\mu_{CZ}(\Psi;a,b):=\mu_{CZ}(\Psi_1;a,b)$.\\
Next, we calculate the indices $\mu_{CZ}(\Psi)$ of some explicit paths.
\begin{lemme}\label{expliciteCZ}
 Let $\Psi_1,\Psi_2,\Psi_3:[0,T]\rightarrow Sp(2)$ be the following paths:
 \[\Psi_1(t)=e^{it},\quad \Psi_2(t)=e^{-it},\quad \Psi_3(t)=\begin{pmatrix}e^{\rho(t)}&0\\0&e^{-\rho(t)}\end{pmatrix},\;\rho\in C^1(\mathbb{R}).\]
 Then, their Conley-Zehnder indices are given as follows:
 \begin{align*}
  \mu_{CZ}(\Psi_1)&=\left\lfloor\frac{T}{2\pi}\right\rfloor+\left\lceil\frac{T}{2\pi}\right\rceil\;\;\;=\begin{cases}\frac{T}{\pi} & \text{if }T\in2\pi\mathbb{Z}\\ 2\left\lfloor\frac{T}{2\pi}\right\rfloor+1 & \text{otherwise,}\end{cases}\\
  \mu_{CZ}(\Psi_2)&=\left\lfloor\frac{-T}{2\pi}\right\rfloor+\left\lceil\frac{-T}{2\pi}\right\rceil=-\mu_{CZ}(\Psi_1),\\
  \mu_{CZ}(\Psi_3)&=0.
 \end{align*}
\end{lemme}
\begin{proof}
 We note that $\Psi_1(t)$ is represented by the real matrix
 \[\Psi_1(t)=\begin{pmatrix}\cos t& -\sin t\\ \sin t&\phantom{-}\cos t\end{pmatrix}\quad\text{ with }\quad i\, \widehat{=}\begin{pmatrix}0&-1\\1&\phantom{-}0\end{pmatrix}=J_0.\]
 Hence we calculate that
 \[\dot{\Psi}_1(t)=\begin{pmatrix}-\sin t & -\cos t\\\phantom{-}\cos t& -\sin t\end{pmatrix}=J_0\circ S_1(t)\circ \Psi_1(t),\]
 where $S_1(t)=\binom{1 \; 0}{0\;1}$. The crossings are exactly the set $2\pi\mathbb{Z}\cap[0,T]$ and the crossing form is always $\Gamma(\Psi_1,t)=\binom{1\;0}{0\;1}$ having signature 2. Hence
 \[\mu_{CZ}(\Psi_1)=\begin{cases}1+2\left\lfloor\frac{T}{2\pi}\right\rfloor&\text{if } T\not\in2\pi\mathbb{Z}\\1+2\left\lfloor\frac{T}{2\pi}\right\rfloor-1=\frac{T}{\pi}&\text{if } T\in2\pi\mathbb{Z}.\end{cases}\]
 The formula for $\Psi_2$ is completely analog with $S_2(t)=-\binom{1\;0}{0\;1}$. For $\Psi_3$ we calculate
 \[\dot{\Psi}_3(t)=\begin{pmatrix}\dot{\rho}(t)e^{\rho(t)} &0\\0&-\dot{\rho}(t)e^{-\rho(t)}\end{pmatrix}=\begin{pmatrix}0&-1\\1&0\end{pmatrix}\begin{pmatrix}0&-\dot{\rho}(t)\\-\dot{\rho}(t)&0\end{pmatrix}\begin{pmatrix}e^{\rho(t)}&0\\0&e^{-\rho(t)}\end{pmatrix}\]
 and hence $S_3(t)=-\big(\begin{smallmatrix}0&\dot{\rho}(t)\\\dot{\rho}(t)&0\end{smallmatrix}\big)$. But this matrix has signature 0 as its eigenvalues are $\pm\dot{\rho}(t)$. It follows that $\text{sign}\,\Gamma(\Psi_3,t)=0$ for every crossing $t$ and thus $\mu_{CZ}(\Psi_3)=0$.
\end{proof}
In \cite{SaZeh}, Salamon and Zehnder introduced yet another approach to $\mu_{CZ}$. They showed that there is a continuous extension $\rho:Sp(2n)\rightarrow S^1$ of the determinant map \linebreak $\det:U(n)=Sp(2n)\cap O(2n)\rightarrow S^1$, which is unique when one requires some additional properties. Moreover, they showed that the space $Sp(2n)^\ast$ of symplectic matrices not having 1 as eigenvalue has two connected components which are semi-simple connected in $Sp(2n)$.\\
Now, any path $\Psi:[0,T]\rightarrow Sp(2n)$ with $\Psi(0)=\mathbbm{1}$ and $\Psi(T)\in Sp(2n)^\ast$ admits a homotopic unique extension $\Psi:[0,T+1]\rightarrow Sp(2n)$ such that $\Psi|_{[T,T+1]}$ connects $\Psi(T)$ in $Sp(2n)^\ast$ to one of the matrices $W^+=-\mathbbm{1}$ or $W^-=diag(2,-1,...\,,-1,1/2,-1,...\,,-1)$. To define $\mu_{CZ}(\Psi)$, choose a lift $\alpha:[0,T+1]\rightarrow \mathbb{R}$ of $\rho\circ\Psi:[0,T+1]\rightarrow S^1$ and set
\[\mu_{CZ}(\Psi;0,T)=\mu_{CZ}(\Psi,T):=\frac{\alpha(T+1)-\alpha(0)}{\pi}.\]
That both definitions of $\mu_{CZ}(\Psi)$ coincide for paths $\Psi$ with $\Psi(0)=\mathbbm{1}$ and $\Psi(T)\in Sp(2n)^\ast$ is shown in \cite{RoSa}. The advantage of the second approach is that it allows us to define the mean index $\Delta(\Psi,T)$ of a path $\Psi:[0,T]\rightarrow Sp(2n)$ as
\[\Delta(\Psi,T):=\frac{\alpha(T)-\alpha(0)}{\pi}.\]
Note that $\Delta(\Psi,T)$ is in general not an integer and that the definition of $\Delta(\Psi,T)$ does not require that $\Psi(T)\in Sp(2n)^\ast$.\\
The mean index allows us to estimate $\mu_{CZ}$ for iterated paths.
\begin{lemme}[\textbf{Iterations formula}]\label{lemitformula}~ Assume that $\Psi:[0,T]\rightarrow Sp(2n),\;\Psi(0)=\mathbbm{1}$ is an iterated path, i.e.\ it holds that $S(t+\tau)=S(t)$ for some $\tau\in\mathbb{R}$ and $\dot{\Psi}(t)=J_0 S(t)\Psi(t)$ as above. Equivalently, we could require that
 \[\Psi(k\tau+t)=\Psi(t)\Psi(\tau)^k,\qquad\text{ for any }k\in\mathbb{Z}.\]
 Under these conditions, we have that
 \[\mu_{CZ}(\Psi,k\tau)=k\cdot\Delta(\Psi,\tau)+R\qquad\text{ with }|R|\leq 2n.\]
\end{lemme}
\begin{proof}
 Without loss of generality, we assume that $\Psi$ has only regular crossings and that $(k\tau,k\tau+t]$ contains no crossings for $t$ small enough. We define for $\Psi(T)\in Sp(2n)^\ast$
 \[r(\Psi,T):=\mu_{CZ}(\Psi,T)-\Delta(\Psi,T).\]
 Note that $r$ is continuous on $(k\tau,k\tau+t]$, as $\mu_{CZ}(\Psi,T)$ is constant for $T\in(k\tau,k\tau+t]$ and $\Delta(\Psi,T)$ is continuous in $T$. Moreover, it follows from (\ref{muCZdef}) that
 \begin{align*}
  \mu_{CZ}(\Psi,k\tau)&=\mu_{CZ}(\Psi,k\tau+t)-\frac{1}{2}\Gamma(\Psi,k\tau)\\
  &=\Delta(\Psi,k\tau+t)+r(\Psi,k\tau+t)-\frac{1}{2}\Gamma(\Psi,k\tau)\\
  &=\Delta(\Psi,k\tau)+\lim_{t\searrow 0} r(\Psi,k\tau+t)-\frac{1}{2}\Gamma(\Psi,k\tau).
 \end{align*}
It is shown in \cite{SaZeh}, Lem.\ 3.4, that $|r(\Psi,k\tau+t)|<n$. That $\Delta(\Psi,k\tau)=k\cdot\Delta(\Psi,\tau)$ follows from the fact that $\Delta(\Psi,k\tau)$ is the winding number of the path $\rho\circ\Psi|_{[0,k\tau]}$, which is the $k$-fold iteration of the path $\rho\circ\Psi|_{[0,\tau]}$. As $\Gamma(\Psi,k\tau)$ is the signature of an $m\times m$ matrix with $m\leq 2n$, we have also $|\Gamma(\Psi,k\tau)|\leq 2n$. Hence, it follows that
\begin{align*}
 \mu_{CZ}(\Psi,k\tau)&=k\cdot\Delta(\Psi,\tau)+R,\\
 \text{where }\qquad\qquad|R|&=\left|\lim_{t\searrow 0}r(\Psi,k\tau+t)-\frac{1}{2}\Gamma(\Psi,k\tau)\right|\leq n+n=2n.\qedhere
\end{align*}
\end{proof}

\subsection{A $\mathbb{Z}$-grading for $RFH$}\label{1.5}
Now, we associate to each contractible periodic Reeb trajectory $v$ a Conley-Zehnder index $\mu_{CZ}(v)$, which allows us to define a $\mathbb{Z}$-grading for Rabinowitz-Floer homology. For simplicity, let us make the following assumptions:
\begin{enumerate}
 \item[(A)] \textit{The map $i_\ast : \pi_1(\Sigma)\rightarrow\pi_1(W)$ induced by the inclusion is injective.}
 \item[(B)] \textit{The integral $I_{c_1}: \pi_2(W)\rightarrow \mathbb{Z}$ of the first Chern class $c_1(TW)$ vanishes on spheres.}
\end{enumerate}
\begin{rem}~
 \begin{itemize}
  \item Assumption (A) is automatically satisfied if $\Sigma$ is simply connected.
  \item One can associate a Conley-Zehnder index to every closed Reeb trajectory $v$, but it will depend on the trivialization of $v^\ast \xi$. In order to fix the trivialization, we restrict ourself to contractible trajectories. These have, due to (B), a unique Conley-Zehnder index (see below). Note that the two ends of a solution of the Rabinowitz-Floer equation are either both contractible or not. Hence, the contractible closed Reeb trajectories generate a subcomplex of $RFC(W,\Sigma)$. Its homology is by abuse of notation also denoted by $RFH(W,\Sigma)$. If $\Sigma$ is simply connected this version of Rabinowitz-Floer homology coincides with the original one.
 \end{itemize}
\end{rem}
To obtain the (transversal) Conley-Zehnder index of a closed contractible Reeb trajectory $v$ choose a map $u$ from the unit disc $D\subset \mathbb{C}$ to $\Sigma$ such that $u(e^{2\pi it})=v(t)$. The existence of such maps is guaranteed by assumption (A). Now choose a symplectic trivialization $\Phi: D\times\mathbb{R}^{2n-2}\rightarrow u^\ast\xi$ of the pullback bundle $(u^\ast\xi,u^\ast d\alpha)$. Such trivializations exist and are homotopically unique as $D$ is contractible. The linearization of the Reeb flow $\psi^t$ along $v$ with respect to $\Phi$ defines a path $\Psi$ in the group $Sp(2n-2)$ starting at $\mathbbm{1}$ by
\[\Psi(t):=\Phi(v(t))^{-1}\circ d\psi^t(v(0))\circ\Phi(v(0)).\]
 The Conley-Zehnder index of this path is the (transverse) Conley-Zehnder index $\mu_{CZ}(v)$. It is independent from the choice of $u$ due to assumption (B). Indeed, if we choose another disc $u': D\rightarrow \Sigma$ with $u'|_{\partial D}=v$ and another trivialization $\Phi'$ of $u'^\ast\xi$, then we can glue $u$ and $u'$ together to a map $w:S^2\rightarrow \Sigma$.\\
 Now, $(w^\ast\xi,w^\ast d\alpha)$ is trivial. Indeed, $\int w^\ast c_1(\xi)=\int w^\ast c_1(TV)=0$ and hence $[w^\ast c_1(\xi)]=0\in H^2(S^2)$, as the complement of $\xi$ in $TV$ is trivialized by the Reeb and Liouville vector field. Therefore, there exists a trivialization $\Phi''$ of $w^\ast\xi$ and both $\Phi$ and $\Phi'$ are homotopic to $\Phi''$ along $v$ (see \cite[Lem.5.2]{SaZeh}). \medskip\\
 The transversal Conley-Zehnder index $\mu_{CZ}$ allows us to grade $RFH$ as follows.
\begin{prop}[\textbf{Frauenfelder \& Cieliebak, \cite{FraCie}}] \label{Prop13}
 If $\pi_1(\Sigma)\rightarrow\pi_1(V)$ is injective and $I_{c_1}$ vanishes, then we have a $\mathbb{Z}$-grading of the Rabinowitz-Floer homology $RFH(\Sigma,V)$, which is independent of $V$ and given by the index
\[\mu(c) :=\mu_{CZ}(c)+ind_h(c)-\frac{1}{2}\dim_c\left(crit(\mathcal{A}^H)\right)+\frac{1}{2},\]
where $c\in crit(h)$ and $\dim_c\left(\crita \right)$ is the real dimension of the connected component of $\crita$ which contains $c$.\pagebreak
\end{prop}
\begin{proof}
 Recall that the boundary operator $\partial^F$ counted points in the zero-dimensional manifolds $\mathcal{M}(c^-,c^+,m)$, which where given as quotients $\raisebox{.2em}{$\widehat{\mathcal{M}}(c^-,c^+,m)$}\left/\raisebox{-.2em}{$\mathbb{R}^m$}\right.$ if $m\neq0$ and $\raisebox{.2em}{$\widehat{\mathcal{M}}(c^-,c^+,0)$}\left/\raisebox{-.2em}{$\mathbb{R}$}\right.$ if $m=0$. The Global Transversality Theorem \ref{theotrans} gave the following dimension formula near a flow line $(v,\eta)$ with $m$ cascades by
 \begin{align*}
 \dim_{(v,t)}\,\widehat{\mathcal{M}}(c^-,c^+,m) \overset{m\neq 0}{=}&\phantom{\,+\,}\Big(\mu_{CZ}(c^+,\bar{c}^+)+ind_h(c^+)-\frac{1}{2}\dim_{c^+}(\crita)\Big)\\
 &-\Big(\mu_{CZ}(c^-,\bar{c}^-)+ind_h(c^-)-\frac{1}{2}\dim_{c^-}(\crita)\Big)\\
 &+m-1+\sum_{k=1}^m 2c_1(\bar{v}_k^-\#v_k\#\bar{v}_k^+)\\
 \dim_{(v,t)}\,\widehat{\mathcal{M}}(c^-,c^+,m) \overset{m= 0}{=}&\phantom{\,+\,}ind_h(c^+)-ind_h(c^-).
\end{align*}
Using condition (B), we get a dimension formula for the moduli space by
\begin{align*}
 \dim_{(v,t)}\,\mathcal{M}(c^-,c^+,m) =&\phantom{\,+\,}\Big(\mu_{CZ}(c^+,\bar{c}^+)+ind_h(c^+)-\frac{1}{2}\dim_{c^+}(\crita)\Big)\\
 &-\Big(\mu_{CZ}(c^-,\bar{c}^-)+ind_h(c^-)-\frac{1}{2}\dim_{c^-}(\crita)\Big)-1\\
 =&\;\mu(c^+)-\mu(c^-)-1.
\end{align*}
Note that the formula holds also for $m=0$, as then $\mu_{CZ}(c^+,\bar{c}^+)=\mu_{CZ}(c^-,\bar{c}^-)$ and $\dim_{c^+}\left(\crita \right)=\dim_{c^-}\left(\crita \right)$, since $(v,\eta)$ is simply a Morse trajectory on one connected component of $\crita$. The moduli space is hence zero-dimensional if and only if $\mu(c^+)-\mu(c^-)=1$. It follows that $\partial^F$ reduces the index $\mu$ exactly by 1 so that $\mu$ provides a well-defined grading for the homology.
\end{proof}
\begin{dis}\label{grading}
 The term $\frac{1}{2}$ in the definition of $\mu$, which does not appear in \cite{FraCie}, has no influence on the relative grading given by the other terms. It normalizes $\mu$ such that it takes values in $\mathbb{Z}$ and fits with the grading of symplectic (co)homology (see \cite{FraCieOan}). Note that this convention differs also from the one used previously by the author in \cite{Fauck1}, where $\frac{1}{2} \dim \Sigma$ was added instead of $\frac{1}{2}$. So all indices in this work are shifted by $-\frac{1}{2}\dim \Sigma+\frac{1}{2}=-n+1=-(n-2)-1$ in comparison to the indices in \cite{Fauck1}.
\end{dis}
The grading $\mu$ allows us to define more refined invariants for contact structure:
\begin{defn}
 Let $c_k(W):=\dim_2 RFH_k(W,\Sigma)$ denote the $k^{\text{th}}$ Betti-number of the \linebreak Rabinowitz-Floer homology of the filling $W$ of $\Sigma$. We define
 \[C_k(\Sigma):=\{c_k(W) \,|\, W \text{ Liouville domain, } \partial W=\Sigma\}\subset[0,\infty]\]
 to be the set of all Betti numbers $c_k(W)$ for varying fillings $W$ of $\Sigma$.
\end{defn}
 Observe that $C_k(\Sigma)$ is (trivially) independent from $W$. It is therefore an invariant of $(\Sigma,\xi)$, as $RFH_k(W,\Sigma)$ was apart from $\xi$ only dependent on $W$. The most simplest case is, when $|C_k(\Sigma)|=1$, i.e.\ if the group $RFH_k(W,\Sigma)$ does not depend on $W$ at all and is itself an invariant of the contact structure.
\begin{prop} \label{prop16}
 Assume that $(W,\Sigma)$ is a Liouville domain satisfying (A) and (B). Then $RFH_k(W,\Sigma)$ is independent of $W$, if $\Sigma$ admits a contact form for which the closed contractible Reeb orbits are Morse-Bott (MB) and for all $c_\ast=(v_\ast,\eta_\ast)\in RFC_\ast(W,\Sigma)$ with $\ast\in\{k-1,k,k+1\}$ holds that $\eta_{k-1}\geq\eta_k\geq\eta_{k+1}$.
\end{prop}
\begin{proof}
 At first, observe that $RFH(W,\Sigma)$ does not depend on the particular contact form. Moreover, the grading $\mu$ of the chain complex and the chain complex itself do not depend on $W$. The chain groups $RFC_k(W,\Sigma)$ are hence independent of $W$. Now consider $(v_\ast,\eta_\ast)\in RFC_\ast(W,\Sigma)$ with $\ast\in\{k-1,k,k+1\}$.\\
If $\eta_{k-1}\geq\eta_k\geq\eta_{k+1}$, the Lemmas \ref{lem2a} and \ref{lem2} tell us that if there are flow lines with cascades from $c_{k-1}$ to $c_k$ or from $c_k$ to $c_{k+1}$ then they must have zero cascades, i.e.\ they are $h$-Morse flow lines on $C_k$, as each cascade would reduce the action. Since $h$-Morse flow lines are independent of the filling $W$, so are the boundary operators 
\[\partial^F_\ast: RFC_{\ast+1}(W,\Sigma)\rightarrow RFC_{\ast}(W,\Sigma),\quad\ast\in\{k-1,k\}\]
and hence the quotient
\[RFH_k(W,\Sigma)=\frac{\ker \partial^F_{k-1}}{\text{im}\;\partial^F_{k\phantom{\mp}}\;}.\qedhere\]
\end{proof}
In \cite[Cor.1.15]{FraCieOan}, Cieliebak, Frauenfelder and Oancea proved another criterion under which $RFH_k(W,\Sigma)$ does not depend on $W$, i.e.\ where $|C_k(\Sigma)|=1$ for all $k$.
\begin{theo}\label{theoinvar}
 Let $(W,\Sigma)$ be a Liouville domain, $\dim W=2n$, satisfying assumption (A) and (B). $RFH_k(W,\Sigma)$ is independent of $W$ for all $k$ if $\Sigma$ admits a contact form for which all contractible closed Reeb orbits $v$ are Morse-Bott (MB) and satisfy
 \[\mu_{CZ}(v)>3-n.\]
\end{theo}
\begin{proof}\textit{(Sketch)}\\
 The theorem is shown by proving that all $\mathcal{A}^H$-gradient trajectories lie entirely in the symplectization of $\Sigma$ and are hence independent of the filling. This follows from a Gromov compactness argument, as trajectories leaving the symplectization would lead to the bubbling-off of holomorphic planes. The latter cannot happen due to the condition $\mu_{CZ}(v)>3-n$.
\end{proof}

We finish the section with a short discussion of what happens with (A) and (B) under handle attachment. Let $(W,\Sigma)$ be a Liouville domain satisfying (A) and (B). Assume that $(W',\Sigma')$ is obtained from $(W,\Sigma)$ by attaching a $k$-handle $H^{2n}_k$ as described in Section \ref{secsur}. In general, one cannot (topologically) expect that (A) still holds for $(W',\Sigma')$ as is shown by attaching a 1-handle to the unit ball in $\mathbb{R}^3$. The result is diffeomorphic to the full 2-torus, which violates (A). However, we have the following lemma.
\begin{lemme}\label{lembehaviourA} Assume that $(W,\Sigma)$ satisfies (A), $\dim W=2n\geq 4$ and that $(W',\Sigma')$ is obtained by attaching a $k$-handle $H^{2n}_k$, $k\leq n-2$. Then $(W',\Sigma')$ satisfies (A) if
 \begin{enumerate}
  \item $k\geq 3$ or
  \item $k=1$, $W$ has 2 components $(W_1,\Sigma_1),\;(W_2,\Sigma_2)$ and $H^{2n}_k$ is glued to $W$ so that $\Sigma'$ is the connected sum $\Sigma_1\#\Sigma_2$ resp. $W'$ is the boundary connected sum $W_1\#W_2$.\pagebreak
 \end{enumerate}
\end{lemme}
\begin{proof}~
  \begin{enumerate}
   \item[@1.] The first part of this proof follows Milnor, \cite[Lem.2]{Mil3}. Set $l=2n-k$ and note that $k<l$. Let $X$ denote the space which is obtained by gluing the handle $H^{2n}_k\cong D^k\times D^l$ to $\Sigma$. It is formed from the topological sum $\Sigma+(D^k\times D^l)$ by identifying $S^{k-1}\times D^l$ with the attaching region in $\Sigma$. The subset $\Sigma\cup(D^k\times 0)$ is a deformation retract of $X$. This subset is formed from $\Sigma$ by attaching a $k$-cell. It follows thus that the map $\pi_1(\Sigma)\rightarrow \pi_1(X)$ induced by inclusion is an isomorphism as $k\geq 3$. But $\Sigma'$ is also embedded topologically in $X$ and similar arguments show that $\pi_1(\Sigma')\rightarrow\pi_1(X)$ is also an isomorphism as $l> k\geq 3$.\\
   Analogously, we consider the space $Y$ which is obtained by gluing the handle $H^{2n}_k$ to $W$. Note that $Y=W'$ and the same reasonings as for $\Sigma$ shows that $\pi_1(W)\rightarrow \pi_1(W')$ is also an isomorphism. Combining all these spaces, we obtain the following diagram:
   \begin{align*}
    \begin{xy}
     \xymatrix{\pi_1(\Sigma)\ar[d]^{(a)}\ar[r] & \pi_1(X) \ar[d]^{(b)} & \pi_1(\Sigma')\ar[d]^{(c)}\ar[l]\\
     \pi_1(W)\ar[r] & \pi_1(W')\ar@{=}[r] & \pi_1(W')}
    \end{xy}.
   \end{align*}
   As all the maps are induced by inclusions, the above diagram commutes. Moreover, all horizontal maps are isomorphisms. As $(a)$ is injective by assumption, it follows from the commutativity of the left square that $(b)$ is also injective. Then it follows from the commutativity of the right square that $(c)$ is also injective.
   \item[@2.] Without loss of generality we assume that $W=W_1\cup W_2$ and $\Sigma=\Sigma_1\cup \Sigma_2$, such that $W'=W_1\#W_2$ and $\Sigma'=\Sigma_1\#\Sigma_2$. Now we apply the van Kampen Theorem in $\Sigma'$ to the two sets $\Sigma_1'\cong\Sigma_1\setminus\{pt.\}$ and $\Sigma_2'\cong\Sigma_2\setminus\{pt.\}$ coming from $\Sigma_1$ and $\Sigma_2$ plus the connecting cylinder. This gives a map $\pi_1(\Sigma_1\setminus\{pt.\})\ast\pi_1(\Sigma_2\setminus\{pt.\})\rightarrow \pi_1(\Sigma_1\#\Sigma_2)$ induced by inclusion.\\ Analogously, we obtain a map $\pi_1(W_1\setminus\{pt.\})\ast\pi_1(W_2\setminus\{pt.\})\rightarrow \pi_1(W_1\#W_2)$. As the intersections $\Sigma_1'\cap\Sigma_2'\cong D^1\times S^{2n-2}$ and $W'_1\cap W_2'\cong D^1\times D^{2n-1}$ are simply connected (as $n\geq 2$), both maps are isomorphisms. Moreover, they fit into the following diagram:
   \begin{align*}
    \begin{xy}
     \xymatrix{\pi_1(\Sigma_1)\ast\pi_1(\Sigma_2)\ar[d]^{(a)} & \pi_1(\Sigma_1\setminus\{pt.\})\ast\pi_1(\Sigma_2\setminus\{pt.\})\ar[d]^{(b)}\ar[l]\ar[r] &\pi_1(\Sigma_1\#\Sigma_2)\ar[d]^{(c)}\\
     \pi_1(W_1)\ast\pi_1(W_2) & \pi_1(W_1\setminus\{pt.\})\ast\pi_1(W_2\setminus\{pt.\})\ar[l]\ar[r] &\pi(W_1\# W_2)}
    \end{xy}.
   \end{align*}
   Again all the maps are induced by inclusion, so this diagram also commutes. The two horizontal maps on the left are isomorphisms as $\dim\Sigma,\dim W\geq 3$. Therefore, all horizontal maps are isomorphisms. As $(a)$ is injective by assumption, it follows as above that $(c)$ is also injective.\qedhere
  \end{enumerate}
\end{proof}
A similar result holds for the behaviour of (B) under handle attachment.
\begin{lemme}\label{lembehaviourB}
 Let $(W,\Sigma)$ be as above, satisfying (B). Let $(W',\Sigma')$ be obtained by attaching a $k$-handle. Then $(W',\Sigma')$ also satisfies (B) if $k=1$ or $k\geq 3$.
\end{lemme}
\begin{proof}
 Assume that $I_{c_1}:\pi_2(W)\rightarrow\mathbb{Z}$ vanishes. As $W'$ is obtained from $W$ by attaching a handle, $W\subset W'$ is an open subset and $c_1(TW')|_V=c_1(TW)$. As in the proof above, there is a retraction $\rho:W'=W\cup H^{2n}_k \rightarrow W\cup D^k$, where $\partial D^k$ is identified with an embedded sphere $S^{k-1}$ in $\Sigma$. Let $w:S^2\rightarrow W'$ be any smooth sphere.
 \begin{itemize}
  \item If $k\geq 3$, we may assume that there is a point $q$ in the interior of the attached disc $D^k$ such that $q\not\in im(\rho\circ w)$. Note that there exists a retraction $R:W'\setminus\{q\}\rightarrow W$ and that $R\circ w$ is well-defined. It follows that $w$ and $R\circ w$ are homotopic and hence
 \[\int_{S^2} w^\ast c_1(TW')=\int_{S^2} (R\circ w)^\ast c_1(TW')=\int_{S^2} (R\circ w)^\ast c_1(TW)=0.\]
  \item If $k=1$, then $D^1$ is a line. Hence we may split $\rho\circ w$ into a finite number of spheres $\tilde{w}_i$ such that no $\tilde{w}_i$ passes over the line. It may well hit its centre but $D^1\nsubseteq image(\tilde{w}_i)$. Now, we may homotope each $\tilde{w}_i$ into $W$ to obtain spheres $w_i: S^2\rightarrow W$. Then
 \begin{align*}
  \int_{S^2} w^\ast c_1(TW')=\int_{S^2} (\rho\circ w)^\ast c_1(TW')&=\sum_i \int_{S^2} \tilde{w}_i^\ast c_1(TW')\\
  &=\sum_i \int_{S^2} w_i^\ast c_1(TW)=0.\qedhere
 \end{align*}
 \end{itemize}
\end{proof}

\newpage
\section{Some algebra for Floer theory}\label{secalg}
\subsection{Direct and inverse limits}
In the previous sections, we gave a description of Rabinowitz-Floer homology with the help of direct and inverse limits. Moreover, the construction of symplectic (co)homology in Section \ref{secsymhom} will also need these limits. In this section, we recall the definition of these algebraic limits and prove some of their properties. Our discourse is based on the fundamental books \cite{bourbaki2} and \cite{eilenberg}.
Throughout this section let $R$ denote a unitary ring.
\begin{defn}
 A relation $\alpha\leq\beta$ on a set $M$ is called a \textbf{quasi order} if it is reflexive and transitive. A \textbf{directed set} $(M,\leq)$ is a quasi ordered set such that for each pair $\alpha,\beta\in M$ there exists a $\gamma\in M$ with $a\leq\gamma$ and $\beta\leq \gamma$. A directed set $M'$ is a \textbf{subset} of a directed set $(M,\leq)$ if $M'\subset M$ and the quasi order on $M'$ is the restriction of $\leq$ to $M'\times M'$. A subset $M'$ is \textbf{cofinal} in $M$ if for each $\alpha\in M$ exists a $\beta\in M'$ such that $\alpha\leq \beta$.
\end{defn}
\begin{defn}~
 \begin{itemize}
  \item A \textbf{direct system} $(X,\iota)$ of $R$-modules  over a directed set $M$ is a function which attaches to each $\alpha\in M$ an $R$-module $X^\alpha$ and to each pair $\alpha\leq\beta$ an $R$-linear map
  \[\iota^{\beta\alpha}:X^\alpha\rightarrow X^\beta\]
  such that $\qquad\qquad\displaystyle \iota^{\alpha\alpha}=id \qquad\qquad\text{and}\qquad\qquad \iota^{\gamma\alpha}=\iota^{\gamma\beta}\iota^{\beta\alpha}\qquad \forall \alpha\leq \beta\leq \gamma.$
  \item An \textbf{inverse system} $(X,\pi)$ of $R$-modules over a directed set $M$ is a function which attaches to each $\alpha\in M$ an $R$-module $X_\alpha$ and to each pair $\alpha\leq \beta$ an $R$-linear map
  \[\pi_{\alpha\beta}:X_\beta\rightarrow X_\alpha\]
  such that $\displaystyle\qquad\qquad\pi_{\alpha\alpha}=id\quad\qquad\text{and}\quad\qquad \pi_{\alpha\gamma}=\pi_{\alpha\beta}\pi_{\beta\gamma}\qquad\forall \alpha\leq\beta\leq\gamma.$
 \end{itemize}
\end{defn}
\begin{defn}~
 \begin{itemize}
  \item Let $(X,\iota)$ and $(Y,j)$ be direct systems over $M$ resp.\ $N$. An \textbf{$R$-homomorphism} $\Phi:(X,\iota)\rightarrow(Y,j)$ consists of an order preserving map $\phi:M\rightarrow N,\,\phi(\alpha)=\alpha'$ and $R$-linear maps $\phi^\alpha:X^\alpha\rightarrow Y^{\alpha'}$ for each $\alpha\in M$ such that for $\alpha\leq \beta$ the following diagram commutes
  \[
  \begin{xy}
  \xymatrix{
  X^\alpha \ar[r]^{\iota^{\beta\alpha}} \ar[d]_{\phi^\alpha} & X^\beta \ar[d]^{\phi^\beta}\\
  Y^{\alpha'}\ar[r]^{j^{\beta'\alpha'}} & Y^{\beta'}}
 \end{xy}.
 \]
 \item Let $(X,\pi)$ and $(Y,p)$ be inverse systems over $M$ resp.\ $N$. An \textbf{$R$-homomorphism} $\Phi:(X,\pi)\rightarrow(Y,p)$ consists of an order preserving map $\phi: N\rightarrow M,\, \phi(\alpha')=\alpha$ and $R$-linear maps $\phi_{\alpha'}: X_\alpha\rightarrow Y_{\alpha'}$ for each $\alpha'\in N$ such that for $\alpha'\leq \beta'$ the following diagram commutes
 \[
  \begin{xy}
  \xymatrix{
  X_\alpha \ar[d]_{\phi_{\alpha'}} & X_\beta \ar[l]_{\pi_{\alpha\beta}} \ar[d]^{\phi_{\beta'}}\\
  Y_{\alpha'} & Y_{\beta'} \ar[l]_{\pi_{\alpha'\beta'}}}
 \end{xy}.
 \]
 \end{itemize}
\end{defn}
\begin{defn} Let $(X,\pi)$ be an inverse system of $R$-modules. The \textbf{inverse limit} $X_\infty$ of $(X,\pi)$ is the following sub-$R$-module of the product $\Pi X_\alpha$:
  \[X_\infty:=\left\{x=(x_\alpha)\in\Pi X_\alpha\;\Big|\,\pi_{\alpha\beta}(x_\beta)=x_\alpha\;\forall\,\alpha\leq\beta\right\}.\]
 \end{defn}
Note that we have for each $\alpha\in M$ a projection 
\[\pi_\alpha: X_\infty\rightarrow X_\alpha,\, \pi_\alpha(x)=x_\alpha \qquad \text{ satisfying }\qquad\pi_\alpha=\pi_{\alpha\beta}\pi_\beta\quad \forall \alpha\leq \beta.\]
These maps allow us to describe $X_\infty$ alternatively by the following universal property:\\
For each $R$-module $Y$ with a family of $R$-linear maps $\tau_\alpha: Y\rightarrow X_\alpha$ which satisfy $\tau_\alpha=\pi_{\alpha\beta}\tau_\beta$ for all $\alpha\leq \beta$ there exists a unique $R$-linear map $\tau: Y\rightarrow X_\infty$ such that for any $\alpha\in M$ the following diagram commutes:
\[
  \begin{xy}
  \xymatrix{
  Y \ar@{.>}[rr]^\tau \ar[rd]_{\tau_\alpha} && X_\infty \ar[ld]^{\pi_\alpha}\\
  &X_\alpha &}
 \end{xy}.
 \]
 Any $R$-homomorphism $\Phi:(X,\pi)\rightarrow(Y,p)$ between two inverse systems induces an $R$-linear map $\phi_\infty:X_\infty\rightarrow Y_\infty$ via $\phi_\infty((x_\alpha)_{\alpha\in M})=(\phi_{\alpha'}(x_\alpha))_{\alpha'\in N}$. We say that $\phi_\infty$ is the inverse limit of the $\phi_\alpha$.
 \begin{defn}
  Let $(X,\iota)$ be a direct system of $R$-modules. Let $\bigoplus X^\alpha$ denote the direct sum and let $Q\subset\bigoplus X^\alpha$ be the submodule generated by all elements of the from\linebreak $\iota^{\beta\alpha}(x^\alpha)-x^\alpha$ for any $a\leq \beta$. The \textbf{direct limit} of $(X,\iota)$ is the quotient module 
  \[X^\infty:=\raisebox{.2em}{$\bigoplus X^\alpha$}\left/\raisebox{-.2em}{$Q$}\right..\]
 \end{defn}
Note that the inclusions $X^\alpha\subset\bigoplus X^\alpha$ induce for each $\alpha\in M$ an $R$-linear map
\[\iota^\alpha: X^\alpha\rightarrow X^\infty,\,\iota^\alpha(x^\alpha)=[x^\alpha]\quad\text{ satisfying }\quad \iota^\alpha=\iota^\beta\iota^{\beta\alpha}\quad\forall\;\alpha\leq\beta.\]
Again, there is an alternative description of $X^\infty$ by a universal property: For each $R$-module $Y$ with a family of $R$-linear maps $\tau^\alpha: X^\alpha\rightarrow Y$ satisfying $\tau^\alpha=\tau^\beta\iota^{\beta\alpha}$ for all $\alpha\leq\beta$ there exists a unique $R$-linear map $\tau: X^\infty\rightarrow Y$ such that for any $\alpha\in M$ the following diagram commutes 
\[
  \begin{xy}
  \xymatrix{
  Y   && X^\infty \ar@{.>}[ll]_\tau \\
  &X_\alpha \ar[ru]_{\iota^\alpha} \ar[lu]^{\tau^\alpha}&}
 \end{xy}.
 \]
 Any $R$-homomorphism $\Phi:(X,\pi)\rightarrow(Y,p)$ between direct systems induces an $R$-linear map $\phi^\infty: X^\infty\rightarrow Y$ by $\phi^\infty\big(\big[\sum x^\alpha\big]\big)=\sum \phi^\alpha(x^\alpha)$. We call $\phi^\infty$ the direct limit of the $\phi^\alpha$.\pagebreak
 \begin{theo}[\textbf{\cite{eilenberg}, Thm.\ 5.4.; \cite{bourbaki2}, §6, no3, prop.\ 4}]~\label{theoindilimits}
  \begin{itemize}
   \item The direct limit is an exact functor, that is if
   \[(A,\iota)\overset{\phi}{\longrightarrow}(B,\iota)\overset{\psi}{\longrightarrow}(C,\iota)\]
   is an exact sequence of direct systems, the following limit sequence is also exact
   \[A^\infty\overset{\phi^\infty}{\longrightarrow}B^\infty\overset{\psi^\infty}{\longrightarrow}C^\infty.\]
   \item The inverse limit is a left exact functor, that is if
   \[0\longrightarrow(A,\pi)\overset{\phi}{\longrightarrow}(B,\pi)\overset{\psi}{\longrightarrow}(C,\pi)\]
   is an exact sequence of inverse systems, the following limit sequence is also exact
   \[0\longrightarrow A_\infty\overset{\phi_\infty}{\longrightarrow}B_\infty\overset{\psi_\infty}{\longrightarrow}C_\infty.\]
  \end{itemize}
 \end{theo}
\begin{rem}
 The inverse limit is in general not an exact functor (see \cite{bourbaki2} for an example). However, it is exact if $R$ is a field and all $R$-modules $A_\alpha,B_\alpha,C_\alpha$ have finite dimension.
\end{rem}
\begin{theo}[\textbf{\cite{eilenberg}, Thm.\ 4.13/Cor.\ 4.14, Thm 3.15/Cor 3.16}]\label{cofinal}
 Let $(X,\pi),(X,\iota)$ be an inverse/direct system over the directed set $M_X$ and let $M_Y\subset M_X$ be a cofinal subset. Let $(Y,\pi),(Y,\iota)$ be the restricted systems over $M_Y$, i.e.\ $Y_\alpha=X_\alpha$ resp.\ $Y^\alpha=X^\alpha$ for all $\alpha\in M_Y\subset M_X$. Then the limits of the systems coincide, i.e.\
 \[\lim_{\longleftarrow} X_\alpha = \lim_{\longleftarrow} Y_\alpha \qquad\qquad\text{ and }\qquad\qquad \lim_{\longrightarrow} X^\alpha=\lim_{\longrightarrow} Y^\alpha.\]
\end{theo}
\begin{proof}~
 \begin{itemize}
  \item[a)]\underline{inverse limits}: Consider the $R$-linear map  
  \[ \phi: \lim_{\longleftarrow}X_\alpha \rightarrow \lim_{\longleftarrow} Y_\alpha,\qquad(x_\alpha)_{\alpha\in M_X}\mapsto (x_\alpha)_{\alpha\in M_Y},\]
  which deletes in the family $(x_\alpha)$ all elements with $\alpha\not\in M_Y$. Define a map
  \begin{align*}
  \psi:\lim_{\longleftarrow} Y_\alpha&\rightarrow \lim_{\longleftarrow} X_\alpha,\\
   (y_\alpha)_{\alpha\in M_Y}&\mapsto(x_\alpha)_{\alpha\in M_X},\qquad x_\alpha:=\begin{cases}y_\alpha & \text{if }\alpha\in M_Y\\ \pi_{\alpha\beta}(y_\beta) &\text{if }\alpha\not\in M_Y,\,\beta\in M_Y,\,\beta\geq \alpha. \end{cases}
  \end{align*}
  As $M_Y\subset M_X$ is cofinal, there exists for any $\alpha\in M_X$ a $\beta$ such that $\alpha\leq \beta$. As $(Y,\pi)$ is an inverse system over the directed set $M_Y$, the definition of $x_\alpha$ does not depend on the particular choice of $\beta$. The map $\psi$ is hence well-defined. Obviously, $\psi$ is also $R$-linear and satisfies $\phi\circ\psi=id|_{X_\infty}$ and $\psi\circ \phi=id|_{Y_\infty}$. In particular, $\phi$ is an isomorphism and the two limits do therefore coincide.
  \item[b)]\underline{direct limits}: Consider the following composition of inclusion and projection:
  \[\lim_{\longrightarrow} Y^\alpha= \raisebox{.2em}{$\bigoplus Y^\alpha$}\left/\raisebox{-.2em}{$Q_Y$}\right.\hookrightarrow\raisebox{.2em}{$\bigoplus X^\alpha$}\left/\raisebox{-.2em}{$Q_Y$}\right.\twoheadrightarrow \raisebox{.2em}{$\bigoplus X^\alpha$}\left/\raisebox{-.2em}{$Q_X$}\right.=\lim_{\longrightarrow} X^\alpha.\]
  It yields an $R$-linear map ${\displaystyle\phi:\lim_{\longrightarrow} Y^\alpha\rightarrow \lim_{\longrightarrow} X^\alpha,}\;\big[\sum y^\alpha\big]_{Q_Y}\mapsto\big[\sum y^\alpha\big]_{Q_X}$.\\
  Define a map $\displaystyle\psi : \lim_{\longrightarrow} X^\alpha\rightarrow\lim_{\longrightarrow} Y^\alpha$ as follows. Any class $\displaystyle\xi\in\lim_{\longrightarrow} X^\alpha$ has a representative $\big[\sum_{i=1}^n x^{\alpha_i}\big]_{Q_X}$ with $x^{\alpha_i}\in X^{\alpha_i}$. As $M_Y\subset M_X$ is cofinal, we may choose $\beta\in M_Y$ such that $\beta\geq \alpha_i$ for all $i$. Then, $\xi$ is also represented by $\big[\sum_{i=1}^n \iota^{\beta\alpha_i}(x^{\alpha_i})\big]_{Q_X}$. Now we set
  \[\textstyle\psi(\xi)=\psi\left(\big[\sum_{i=1}^n x^{\alpha_i}\big]_{Q_X}\right) =\Big[\sum_{i=1}^n \iota^{\beta\alpha_i}(x^{\alpha_i})\Big]_{Q_Y}.\]
  It is not to difficult to see that this definition does not depend on the particular choice of $\beta$ nor on the choice of the representatives for $\xi$. The map $\psi$ is hence well-defined. Obviously, $\psi$ is also an $R$-linear map and satisfies $\phi\circ\psi=id|_{Y^\infty}$ and $\psi\circ\phi=id|_{X^\infty}$. In particular, $\phi$ is an isomorphism and the limits do coincide.\qedhere
 \end{itemize}
\end{proof}
\begin{rem}
 Note that any direct/inverse limit of $R$-homomorphisms does also only depend on cofinal subsets.
\end{rem}
\begin{defn}
 A \textbf{bidirect system} $(X,\pi,\iota)$ of $R$-modules over two directed sets $M$ and $N$ consists of a function which attaches to each pair $(\alpha,\beta)\in M\times N$ an $R$-module $X^\beta_\alpha$, to each triple $(\alpha_1,\alpha_2,\beta)\in M^2\times N$ with $\alpha_1\leq\alpha_2$ an $R$-homomorphism
 \[\pi^\beta_{\alpha_1\alpha_2}: X^\beta_{\alpha_2}\rightarrow X^\beta_{\alpha_1},\]
 to each triple $(\alpha,\beta_1,\beta_2)\in M\times N^2$ with $\beta_1\leq \beta_2$ an $R$-homomorphism
 \[\iota^{\beta_2\beta_1}_{\alpha}: X^{\beta_1}_\alpha\rightarrow X^{\beta_2}_\alpha\]
 such that for each fixed $\beta$ the tuple $(X^\beta,\pi^\beta)$ is an inverse system over $M$ and for each fixed $\alpha$ the tuple $(X_\alpha,\iota_\alpha)$ is a direct system over $N$. Moreover, for each quadruple $\alpha_1\leq\alpha_2,\,\beta_1\leq\beta_2$ the following diagram has to commute:
 \[\begin{xy}\xymatrix{
    \text{\Large $X$}^{\beta_1}_{\alpha_2} \ar[rr]^{\text{\normalsize $\pi$}^{\beta_1}_{\alpha_1\alpha_2}} \ar[d]_{\text{\normalsize $\iota$}^{\beta_2\beta_1}_{\alpha_2}} && \text{\Large $X$}^{\beta_1}_{\alpha_1} \ar[d]^{\text{\normalsize $\iota$}^{\beta_2\beta_1}_{\alpha_1}}\\
    \text{\Large $X$}^{\beta_2}_{\alpha_2} \ar[rr]^{\text{\normalsize $\pi$}^{\beta_2}_{\alpha_1\alpha_2}} && \text{\Large $X$}^{\beta_2}_{\alpha_1}
   }\end{xy}.\]
\end{defn}
It follows that $(X_\infty,\iota_\infty)$ is a direct system over $N$ and that $(X^\infty,\pi^\infty)$ is an inverse system over $M$. Note that in general the limits of these two systems are not equal. In other words direct and inverse limit do not commute, i.e.\ 
   \[\lim_{\longrightarrow}\lim_{\longleftarrow} X \neq \lim_{\longleftarrow}\lim_{\longrightarrow} X.\]
\subsection{Abstract Floer theory}
In this section, we present an axiomatic model for Morse and Floer theory over possibly non-compact manifolds. We follow \cite{FraCie2}, who claim that it all goes back to a suggestion made by D. Salamon. The Main Theorem \ref{theolimits} states that the two basic approaches in this setup - working over a Novikov completion or taking limits - yields the same homology, i.e.\ that $\displaystyle FH\cong \lim_{\longrightarrow}\lim_{\longleftarrow} FH^{[a,b]}$. Our proof of this crucial theorem is less abstract (compared to \cite{FraCie2}, though perhaps tedious) and completely selfcontained.
\begin{defn}
 A Floer triple over a ring $R$ is a tuple $\mathcal{F}=(C,f,m)$
 consisting of a set $C$, a function $f: C\rightarrow \mathbb{R}$ and a function $m:C\times C\rightarrow R$ such that the following conditions hold.
 \begin{enumerate}
  \item[i.] The set $C^b_a:=\{c\in C\,|\,a\leq f(c)\leq b\}$ is finite for each $-\infty<a\leq b<\infty$.
  \item[ii.] If $m(c_1,c_2)\neq 0$ then $f(c_1)\leq f(c_2)$.
  \item[iii.] For all $c_1,c_3\in C$ holds $\displaystyle\sum_{c_2\in C}m(c_1,c_2)\cdot m(c_2,c_3)=0$.
 \end{enumerate}
\end{defn}
\begin{rem}~
 \begin{itemize}
  \item Assertions \textit{i.} and \textit{ii.} assure that the sum in \textit{iii.} is actually finite.
  \item We call $f$ the action function, call the elements of $C$ critical points of action $f(c)$ and call $C^b_a$ the set of critical points in the action window $[a,b]$. The value $m(c_1,c_2)$ represents the count (with signs) of Floer cylinders with asymptotics $c_1$ and $c_2$.
  \item The condition $f(c_1)\leq f(c_2)$ in assertion \textit{ii,} reflects the fact that we also include Morse-Bott theories. For pure Morse theories replace ``$\leq$'' by ``$<$''.
  \item Assertion \textit{i.} implies that the action spectrum $spec(f):=f(C)\subset \mathbb{R}$ is closed and discrete. We can hence find a monotone injective map from $f(C)$ to $\mathbb{Z}$. With the help of this map we will henceforth assume that in fact
  \[f(C)\subset \mathbb{Z}.\]
  This is solely for notational reasons and does not effect the generality of our theorems. Under this assumption, the first assertion takes the following form
  \textit{
  \begin{itemize}
   \item[i'.] For every $a\in\mathbb{Z}$, the set $C^a_a=f^{-1}(a)=\{c\in C\,|\,f(c)=a\}$ is finite.
  \end{itemize}}
  \item Note the changed notation for the action window in this chapter. In all other chapters we write $C^{(a,b)}$ instead of $C^b_a$. In particular, the lower index denotes the lower end of the action window and not the grading with respect to the boundary operator, which plays only a minor role in this chapter.
  \item Assertion \textit{iii.} guarantees that the formula $dc_0:=\sum_{c\in C} m(c,c_0)\cdot c$ will define a boundary operator.
 \end{itemize}
\end{rem}
\begin{defn}~
 For $a\leq b\in \mathbb{Z}$ let $FC^b_a:=C^b_a\otimes R$ be the free $R$-module generated by $C^b_a$. For any $a\leq c\leq b$ let
 \[\pi^b_{ca}:FC^b_a\longrightarrow FC^b_c\cong \raisebox{.2em}{$FC^b_a$}\left/\raisebox{-.2em}{$FC^{c-1}_a$}\right.\qquad\text{ and }\qquad \iota^{bc}_a: FC^c_a\longrightarrow FC^b_a\cong FC^c_a\oplus FC^b_{c+1}\]
 be the natural projection resp.\ inclusion.
\end{defn}
Obviously, $(FC^b_a,\pi,\iota)$ is a bidirect system over $\mathbb{Z}\times\mathbb{Z}$, where the quasi order in the upper index is ``$\leq$'', while it is ``$\geq$'' on the lower one.
\begin{defn}
 We set $\quad\displaystyle FC^b:=\lim_{\underset{a}{\longleftarrow}} FC^b_a\qquad\text{ and }\qquad FC:=\lim_{\underset{b}{\longrightarrow}} FC^b=\lim_{\underset{b}{\longrightarrow}}\lim_{\underset{a}{\longleftarrow}}FC^b_a.$
\end{defn}
Note that $s\in FC^b$ can be interpreted as a formal sum 
$\quad \displaystyle s=\sum_{a\leq b}s_a,\; s_a\in FC^a_a.$\medskip\\
Then, $s\in FC$ is also interpreted as a formal sum $\displaystyle \quad s=\sum_{a\in\mathbb{Z}}s_a,\; s_a\in FC^a_a$.\medskip\\
As $FC$ is obtained by a direct sum, we find for the second sum that there exists a $b\in\mathbb{Z}$ such that $s_a=0$ for all $a>b$. Therefore, we write the elements of $FC$ also as $s=\sum_{a\leq b}s_a$. Note that the maps $\pi$ and $\iota$ extend to the limits and we have in particular
\[\pi^b_c\,:\, FC^b\rightarrow FC^b_c,\qquad \sum_{a\leq b} s_a\mapsto \sum_{a=c}^b s_a.\]
We call the summands $s_a$ in the representation of $s\in FC$ the \textbf{\textit{$a$-th coefficient}} of $s$ and we use for the $A$-th coefficient of $s$ the notation 
\[K_A(s)=K_A\Big(\sum_{a\leq b}s_a\Big):=s_A.\]
Note that $K_a:FC\rightarrow FC^a_a$ as a map is well-defined and linear. The following important fact is an easy consequence of the definition of $FC$.
\begin{lemme}
 Any sequence of coefficients $(s_a)_{a\in\mathbb{Z}}$ with $s_a=0$ for all $a>b$ with some $b\in\mathbb{Z}$ determines uniquely an element in $FC$.
\end{lemme}
 Now let us take a strictly decreasing sequence $(s^b)_{b\leq b_0}\subset FC,\,s^b=\sum_{a\leq b}s^b_a\in FC^b$. We associate to the formal Laurent series $\sum_{b\leq b_0} s^b$ a value in $FC$ by the well-known formula
\[\sum_{b\leq b_0} s^b:=\sum_{a\leq b_0}\sum_{b=a}^{b_0} s^b_a=\sum_{a\leq b_0}\sum_{b=a}^{b_0} K_a(s^b).\]
Observe that the inner sum is finite thus determining an element $s_a\in FC^a_a$. The value of the infinite sum $\sum s^b$ is then the unique element in $FC$ represented by the sequence of coefficients $(s_a)_{a\leq b_0}$. The well-known features of this infinite sum are summarized in the following lemma.\pagebreak
\begin{lemme}\label{lemX}
 For $b_0\in\mathbb{Z}$ and $b\leq b_0$ let $s^b, t^b \in FC^b$. Then
 \begin{itemize}
  \item[(a)] $\displaystyle\sum_{b\leq b_0} s^b=s^{b_0}+s^{b_0-1}+\dots+s^{b_0-k+1}+\sum_{b\leq b_0-k} s^b$\hfill(associativity)
  \item[(b)] $\displaystyle \sum_{b\leq b_0}(s^b+t^b)=\sum_{b\leq b_0}s^b\;+\sum_{b\leq b_0} t^b$\hfill (countable associativity)
  \item[(c)] $\displaystyle s^{b_0}=\sum_{b\leq b_0}(s^b-s^{b-1})$ \hfill (telescope sum)
  \item[(d)] If $d:FC\rightarrow FC$ is a homomorphism with $d(FC^b)\subset FC^b$, then \[d\Big(\sum_{b\leq b_0} s^b\Big)=\sum_{b\leq b_0}d(s^b).\tag{\textit{countable linearity}}\]
 \end{itemize}
\end{lemme}
\begin{proof} We calculate
\begin{align*}
 \text{(a) }  \mspace{150mu} K_A\Big(\sum_{b\leq b_0} s^b\Big) &= K_A\Big(\sum_{a\leq b_0}\sum_{b=a}^{b_0} s^b_a\Big) = \sum_{b=A}^{b_0}s^b_A\\
 K_A\Big(s^{b_0}+\dots+s^{b_0-k+1}+\sum_{b\leq b_0-k}s^b\Big) &= K_A\big(s^{b_0}\big)+\dots+K_A\big(s^{b_0-k+1}\big)+K_A\Big(\sum_{b\leq b_0-k}s^b\Big)\\
 &=s^{b_0}_A+\dots+s^{b_0-k+1}_A+\sum_{b=A}^{b_0-k}s^b_A=\sum_{b=A}^{b_0}s^b_A
\end{align*}
So all coefficients of the left side of equation (a) coincide with all coefficients of the right side. Hence they represent the same element in $FC$. The proof of the subsequent assertions follows the same scheme.
\begin{align*}
 &\text{(b)}& 
 K_A\Big(\sum_{b\leq b_0}(s^b+t^b)\Big) &= K_A\Big(\sum_{a\leq b_0}\sum_{b=a}^{b_0} (s^b_a+t^b_a)\Big) \qquad= \sum_{b=A}^{b_0}(s^b_A+t^b_A)\\
 && K_A\Big(\sum_{b\leq b_0} s^b\, + \sum_{b\leq b_0} t^b\Big) &= K_A\Big(\sum_{b\leq b_0} s^b\Big) + K_A\Big(\sum_{b\leq b_0} t^b\Big) = \sum_{b=A}^{b_0} s_A^b + \sum_{b=A}^{b_0}t^b_A\\
&\text{(c)}& \Bigg.
 \sum_{b\leq b_0}(s^b-s^{b-1})&\overset{(b)}{=}\sum_{b\leq b_0} s^b-\sum_{b\leq b_0-1}s^b \overset{(a)}{=} s^{b_0}+\sum_{b_0-1}s^b-\sum_{b_0-1}s^b=s^{b_0}\\
 &\text{(d)}&  \Bigg.
 K_A\Big[ d\Big(\sum_{b\leq b_0}s^b\Big)\Big] &\overset{(a)}{=} K_A\Big[d\Big(s^{b_0}+\dots+s^{A}+\sum_{b\leq A-1} s^b\Big)\Big]\\
 &&&=K_A\Big[d(s^{b_0})+\dots+d(s^A)+d\Big(\sum_{b\leq A} s^b\Big)\Big]\\
 &&&\overset{(\ast)}{=}K_A\big(d(s^{b_0})\big)+\dots+K_A\big(d(s^A)\big) \qquad=\sum_{b=A}^{b_0} K_A\big(d(s^b)\big)
\end{align*}
where $(\ast)$ holds as $d\big(FC^{A-1})\subset FC^{A-1}$ and hence $K_A\big(d\big(\sum_{b\leq A-1} s^b\big)\big)=0$.\pagebreak\\
On the other hand $\displaystyle K_A\Big(\sum_{b\leq b_0} d(s^b)\Big) = K_A\Big(\sum_{a\leq b_0}\sum_{b=a}^{b_0} K_a\big(d(s^b)\big)\Big)
 =\sum_{b=A}^{b_0} K_A\big(d(s^b)\big)$.\qedhere 
\end{proof}
 We define a boundary operator $d$ on $FC^b_a$ as the linear extension of
 \[dx:=\sum_{y\in C^b_a}m(y,x)\cdot y, \qquad x\in C^b_a.\]
 Note that assertion iii. implies that $d^2=0$, i.e.\ that $d$ really is a boundary operator. As $d$ does not increase action (due to ii.), we find that $d$ satisfies
\begin{align}\label{eqdpiiota}
 \pi^b_{ca}\circ d=d\circ \pi^b_{ca}\qquad\text{ and }\qquad \iota^{bc}_a\circ d=d\circ \iota^{bc}_a\qquad\forall\; a\leq c\leq b.
\end{align}
Therefore, $d$ induces maps on $FC^b$ and $FC$, which we will also denote with $d$. Explicitly, they are given as the infinite linear extension of
\[dx :=\sum_{y\in C} m(y,x)\cdot y,\qquad x\in C.\]
Note that ii. and iii. imply that $d^2=0$ and that (\ref{eqdpiiota}) still holds, now including $a=-\infty$ and $b=\infty$. Note that we may express the boundary operator on $FC^b_a$ as $\pi^b_a\circ d$, where $d$ is the operator on $FC$. We denote the resulting homologies by
\[ FH^b_a,\quad FH^b\quad\text{ and }\quad FH.\]
The elements of these homologies are denoted by $[s]^b_a,[s]^b$ and $[s]$, where we suppress the upper index whenever it is clear from the context.\\
As a consequence of (\ref{eqdpiiota}), we find that $\pi$ and $\iota$ descend to linear maps in homology, still denoted by $\pi$ and $\iota$:
\[\pi^b_{ca}\,:\, FH^b_a\rightarrow FH^b_c\qquad \iota^{bc}_a\,:\, FH^c_a\rightarrow FH^b_a\qquad\forall\, -\infty\leq a\leq c\leq b\leq \infty.\]
As $(FC^b_a,\pi,\iota)$ resp.\ $(FC^b,\iota)$ are bidirect resp.\ direct systems, it follows that $(FH^b_a,\pi,\iota)$ resp.\ $(FH^b,\iota)$ are bidirect resp.\ direct systems as well. Observe that we have in particular for every fixed $b_0$ and any $a\leq b_0$ the following map
\[\pi^{b_0}_a\,:\, FH^{b_0}\rightarrow FH^{b_0}_a,\qquad [s]\mapsto\Big[\sum_{b=a}^{b_0}s_b\Big]_a, \quad\text{ where }\quad s=\sum_{b\leq b_0}s_b\in FC^{b_0},\]
which satisfies $\pi^{b_0}_c=\pi^{b_0}_{ca}\pi^{b_0}_a$ for every $a\leq c$. Hence, there exists by the universal property a unique $R$-linear map
\[\phi\,:\, FH^{b_0}\rightarrow \lim_{\underset{a}{\longleftarrow}}FH^{b_0}_a,\qquad [s]\mapsto\bigg(\Big[\sum_{b=a}^{b_0}s_b\Big]_a\bigg)_{a\leq b_0}.\]
Analogously, we have for every $b\in\mathbb{Z}$ a map
\[\iota^b\,:\, FH^b\rightarrow FH,\]
which satisfies $\iota^a=\iota^b\iota^{ba}$ for every $a\leq b$. By the universal property, we hence obtain a unique $R$-linear map
\[\psi\,:\, \lim_{\underset{b}{\longrightarrow}} FH^b\rightarrow FH,\qquad \Big[\sum_{i=1}^n\sigma_i\Big]\mapsto\sum_{i=1}^n\iota^{b_i}(\sigma_i),\qquad \sigma_i\in FH^{b_i}.\]
\begin{theo}~\label{theolimits}
 \begin{itemize}
  \item[a)] The map $\phi$ is always surjective. If $R$ is a field, it is injective and yields for every $b\in\mathbb{Z}$ an isomorphism 
  \[FH^b\cong\lim_{\underset{a}{\longleftarrow}} FH^b_a.\]
  \item[b)] The map $\psi$ is an isomorphism for any $R$, i.e.\ $\quad\displaystyle FH\cong\lim_{\underset{b}\longrightarrow} FH^b.$\\
  If $R$ is a field, we have therefore $\quad\displaystyle FH \cong \lim_{\underset{b}\longrightarrow}\lim_{\underset{a}{\longleftarrow}} FH^b_a$.\\
  In the notation of the other sections, this reads as $\quad\displaystyle FH \cong \lim_{\underset{b}\longrightarrow}\lim_{\underset{a}{\longleftarrow}} FH^{(a,b)}$.
 \end{itemize}
\end{theo}
\begin{proof}~
 \begin{itemize}
  \item[(a)]
  \underline{claim}: $\phi$ is surjective.\\
  \underline{proof}: Let $\displaystyle(\sigma_a)_{a\leq b}\in\lim_{\longleftarrow} FH^b_a$ be arbitrary with $\sigma_a\in FH^b_a$ and $\pi_{ca}\sigma_a=\sigma_{c}$ for all $a\leq c\leq b$. We will construct inductively coefficients $s_a\in FC^a_a$ such that for all $A\leq b$ holds $\sigma_A=\big[\sum_{a=A}^b s_a\big]_A$. Then we set $s:=\sum_{a\leq b} s_a\in FC$ and calculate
  \[ds=d\sum_{a\leq b} s_a=\sum_{a\leq b} ds_a=\sum_{A\leq b}\sum_{a=A}^b K_A(ds_a)=\sum_{A\leq b}K_A\Big(d\sum_{a=A}^b s_a\Big)=0\]
  as $\big[\sum_{a=A}^b s_a\big]_A=\sigma_A$ and hence $\pi_A\big(d\sum_{a=A}^b s_a\big)=0$ in $FC^b_A$. Thus, we have that $s$ is a cycle and so that $[s]\in FH^b$ well-defined. We then see that $\phi$ is surjective, as
  \[\phi\big([s]\big)=\bigg(\Big[\sum_{a=A}^b s_{a}\Big]_A\bigg)_{A\leq b}=\big(\sigma_A\big)_{A\leq b}.\]
  For the construction of the $s_a$ first consider the following short exact sequence
  \[0\rightarrow FC^a_a\overset{\iota}{\longrightarrow} FC^b_a\overset{\pi}{\longrightarrow} FC^b_{a+1}\rightarrow 0,\]
  which yields in homology the long exact sequence
  \begin{align}\label{longexseq}
   \dots\rightarrow FH^b_{a+1}\overset{\delta}{\longrightarrow}FH^a_a\overset{\iota}{\longrightarrow}FH^b_a\overset{\pi}{\longrightarrow}FH^b_{a+1}\overset{\delta}{\longrightarrow}FH^a_a\rightarrow\dots
  \end{align}
Here, the connecting homomorphism $\delta$ is given by
\[\delta[s]^b_{a+1} :=\big[\pi_a ds\big]^a_a,\]
which is the class of the $a$-th coefficient of $ds$, as $ds\in FC^a$.\pagebreak\\ Here and in the following we suppress the indices and write $\iota$ or $\pi$ instead of $\iota^{ba}_a$ or $\pi^b_{a+1,a}$, whenever it is clear from the context.\\
Let $s_b\in FC^b_b$ be any representative of $\sigma_b$, i.e.\ $\sigma_b=[s_b]_b$. Now assume that $s_{b-1},...\,,s_{A+1}$ have already been constructed such that $\sigma_{A+1}=\big[\sum_{a=A+1}^b s_a\big]_{A+1}$. We apply (\ref{longexseq}) with $a=A$ and find $\sigma_{A+1}=\pi\sigma_A\in im(\pi)=\ker(\delta)$. Hence there exists a $n_A\in FC^A_A$ such that $\pi_A\big(d(\sum_{a=A+1}^b s_a)-dn_A\big)=0\in FC^b_A$. Thus, we find that $\big[\sum_{a=A+1}^b s_a - n_A\big]_A\in FH^b_A$ is well defined. As $\pi\big[\sum_{a=A+1}^b s_a - n_A\big]_A=\big[\sum_{a=A+1}^b s_a\big]_{A+1}=\sigma_{A+1}=\pi\sigma_A$ we find that $\sigma_A-\big[\sum_{a=A+1}^b s_a - n_A\big]_A\in\ker(\pi)=im(\iota)$. Hence there exists a $\tilde{s}_A\in FC^A_A$ such that $\sigma_A=\big[\sum_{a=A+1}^b s_a - n_A+\tilde{s}_A\big]_A$. With $s_A:=\tilde{s}_A-n_A$ we obtain $\sigma_A=\big[\sum_{a=A}^b s_a]_A$.\bigskip\\
\underline{claim}: $\phi$ is injective if $R$ is a field.\\
\underline{proof}: Assume that $\phi\big([s]\big)=0\in\displaystyle \lim_{\longleftarrow}FH^b_a$ for some $s=\sum_{a\leq b} s_a\in FC^b$. We will construct a sequence of coefficients $n_a\in FC^a_a$ such that for all $A\leq b$ holds
\[\pi_A\Big(s-d\textstyle\sum_{a=A}^b n_a\Big)=0\qquad\text{(in $ FC^b_A$)},\]
which is equivalent  to $s-d\sum_{a=A}^b n_a\in FC^{A-1}$.\\
Then we set $\beta:=\sum_{a\leq b} n_a$ and we find that
\[s-d\beta\in FC^a\;\forall a\leq b\quad\Rightarrow\quad s-d\beta\in \bigcap_{a\leq b}FC^a=\{ 0\}\quad\Rightarrow\quad s=d\beta.\]
It follows that $[s]=0\in FH^b$ and therefore that $\phi$ is injective.\medskip\\
To start, note that $\phi\big([s]\big)=\big(\big[\sum_{a=A}^b s_a\big]_A\big)_{A\leq b}=0$ is equivalent to
\[\begin{aligned}&&\forall A\leq b &:&\,0=\big[\textstyle\sum_{a=A}^b s_a\big]_A&\in FH^b_A\\
&\Leftrightarrow& \forall A\leq b\;\exists\,\beta_A\in FC^b_A &:&\,s-d\beta_A&\in FC^{A-1}.\end{aligned}\tag{$\ast$}\]
The sequence $n_b, n_{b-1},\dots$ we are going to construct will actually satisfy the following slightly stronger condition
\[\forall A\leq b\;\forall k\geq 0\;\exists \beta^k_A\in FC^A\;\;:\;\;s-d\Big[\Big(\sum_{a=A+1}^b n_a\Big) + \beta_A^k\Big]\in FC^{A-k}.\tag{$\ast\ast$}\]
For the empty sequence, condition $(\ast\ast)$ has to hold for $A=b$, which is true by $(\ast)$.\\
Now assume that $n_b,\dots,n_{A+1}$ have already been constructed such that $(\ast\ast)$ holds. Then let $G_k\subset FC^A_A$ denote the set of all $A$-coefficients $n_A$ such that there exists a $\beta_k\in FC^{A-1}$ with $s-d\big[\big(\sum_{a=A+1}^b n_a\big) +n_A+ \beta_k\big]\in FC^{A-1-k}$. Due to assumption $(\ast\ast)$, we find that $G_k\neq\emptyset$. Moreover, $G_k$ has the structure of an affine subspace of $FC^A_A$ and $G_{k+1}\subset G_k$. As $FC^A_A$ is a finite dimensional vector space ($R$ is a field and $C^A_A$ finite), we find that the common intersection is not empty, i.e.\
\[G:=\bigcap_{k\geq 0} G_k\neq\emptyset.\]
Choose any $n_A\in G$. The construction of $G$ shows that $(\ast\ast)$ holds for $n_b,\dots,n_A$.\pagebreak
\item[(b)]
\underline{claim}: $\psi$ is injective.\\
  \underline{proof}: Assume that $\psi\big(\big[\sum_{i=1}^n \sigma^i\,\big]\big)=0\in FH$ for some $\big[\sum_{i=1}^n \sigma^i\,\big]\in\displaystyle\lim_{\longrightarrow} FH^b,\linebreak \sigma^i\in FH^{b_i}$. Let $s^i\in FC^{b_i}$ be such that $\sigma^i=[s^i]^{b_i}$. It follows that $\psi\big(\big[\sum_{i=1}^n \sigma^i\,\big]\big)=\big[\sum_{i=1}^n s^i\big]=0.$
  Hence, there exists a $\beta\in FC$ such that $\sum_{i=1}^n s^i=d\beta$. Then there exists a $b_0\in\mathbb{Z}$ such that $b_0\geq b_i$ for all $i$ and $\beta\in FC^{b_0}\subset FC$. This implies that
  \[\sum_{i=1}^n \iota^{b_0b_i}\sigma^i=\Big[\sum_{i=1}^n s^i\Big]^{b_0}=\big[d\beta\big]^{b_0}=0\in FH^{b_0},\]
  which shows that $\big[\sum_{i=1}^n \sigma^i\,\big]=0\in\displaystyle\lim_{\longrightarrow}FH^b$.\bigskip\\
  \underline{claim}: $\psi$ is surjective.\\
  \underline{proof}: Consider any $[s]\in FH,\, s\in FC$. Then there exists a $b_0\in\mathbb{Z}$ such that $s\in FC^{b_0}$. As $ds=0$, we find that $s$ represents a well-defined class $[s]^{b_0}\in FH^{b_0}$. Let $\big[[s]^{b_0}\big]$ denote its class in $\displaystyle\lim_{\longrightarrow}FH^b$. Then we see that $\psi\big(\big[[s]^{b_0}\big]\big)=[s]$, thus showing that $\psi$ is surjective.\qedhere
 \end{itemize}
\end{proof}
\subsection{Reducing filtered complexes}\label{secRedFil}
 Recall from the previous section that abstract Floer theory is based on a set $C$ which generates (countably linear) the chain complex $FC$ and the homology $FH$. Now we are going to address the question to what extend we can reduce $C$ and still get the same homology $FH$. Under the assumption that $R$ is a field such that all $FH^b_a$ are vector spaces, we will show that we can replace for any $a$ the set $C^a_a$ by any basis of $FH^a_a$.\\
 The motivation behind this is the following. Given a Morse-Bott setup, $C^a_a$ is the set of all critical points of a Morse function on the critical manifold $\mathcal{N}^a$ on the action level $a$, while the groups $FH^a_a$ are the Morse homology of $\mathcal{N}^a$. The fact that we can built $FH$ from $FH^a_a$ then shows that for calculations of $FH$ we may always algebraically pretend that we have a perfect Morse function on $\mathcal{N}^a$, as the critical points of such a function form a basis of $FH^a_a$. This will be used in Section \ref{sec7.1} when we calculate the Rabinowitz-Floer homology for some Brieskorn manifolds, where we know the singular homology of the critical manifolds, but we do not know if we have perfect Morse functions.\\
 To be more explicit, let the reduced chain groups be
 \[F\mathscr{C}:=\lim_{\underset{b}{\longrightarrow}}\lim_{\underset{a}{\longleftarrow}} \textstyle\bigoplus_{c=a}^b FH^c_c\]
 and consider the following variation of the long exact sequence (\ref{longexseq}):
 \begin{equation}\label{SX}
   \dots\rightarrow FH^b_b\overset{\delta}{\longrightarrow}FH^{b-1}\overset{\iota}{\longrightarrow}FH^b\overset{\pi}{\longrightarrow}FH^b_b\overset{\delta}{\longrightarrow}FH^{b-1}\rightarrow\dots,
 \end{equation}
 where the map $\delta$ is induced by the boundary operator $d$ and given by $\delta[n_b]^b_b=[dn_b]^{b-1}$.\\
 We want to use $\delta$ to define a boundary operator on $F\mathscr{C}$. For that, let 
 \[K:=\delta(F\mathscr{C})\subset\lim_{\longrightarrow}\lim_{\longleftarrow}\textstyle \bigoplus_{c=a}^b FH^{c-1}.\]
 Using the sequence (\ref{SX}), we define below a (non-canonical) map $\Phi:K\rightarrow F\mathscr{C}$. Setting $\partial:=\Phi\circ \delta$, we will see that $\partial^2=0$. The Reduction Theorem  \ref{reductiontheo} then tells us that the homology of $\big(F\mathscr{C},\partial\big)$ is isomorphic to $FH$. Here, the isomorphism is given by a map $\Psi:\ker\delta\rightarrow FH$. We will use this map to show in Theorem \ref{theo31} that always (not only over field coefficients) $F\mathscr{C}$ generates $FH$, i.e\ that the whole homology cannot have more elements then all the singular homologies of the critical manifolds togehter.\\
 Note that one could take $\Phi=\pi$. However, the resulting homology $(F\mathscr{C},\partial)$ is then isomorphic to the second page of the spectral sequence of the filtration, which is in general not isomorphic to $FH$. The reason for this discrepancy is that $\pi$ discards $\ker\pi=im (\iota)$. Our $\Phi$ will also use the $(\ker\pi)$-part of $K$.
 \begin{rem}
 Let $\mathscr{C}^a_a$ denote a basis of $FH^a_a$ and set $\mathscr{C}=\bigcup \mathscr{C}^a_a$. Then we find that $\mathscr{C}$ generates $F\mathscr{C}$. The isomorphism between $(F\mathscr{C},\partial)$ and $FH$ then shows that we can replace $C$ by $\mathscr{C}$ and obtain the same homology $FH$. Moreover, as $\delta$ \textbf{\textit{decreases}} the action, so does $\partial$. Thus, we reduce the Morse-Bott situation to a Morse situation. 
\end{rem}
To start, we note that the maps in (\ref{SX}) are explicitly given by
 \begin{align*}
  \mspace{70mu}\iota\big[s\big]^{b-1}&=\big[s\big]^b& &\text{ for }& s&\in FC^{b-1}\mspace{70mu}\\
  \pi\big[s\big]^b&= \big[n_b\big]^b_b&&\text{ for }&s=\textstyle\sum_{a\leq b} n_a&\in FC^b\\
  \delta\big[n_b\big]^b_b&=\big[dn_b\big]^{b-1}&&\text{ for }&n_b&\in FC^b_b.
 \end{align*}
Let us make the following definition: $\bigg.\displaystyle\quad\aleph^b:=\ker(\delta)=im(\pi)=\pi(FH^b)\subset FH^b_b.$\\
In the following, we are interested in bounded infinite sums $\sum_{b\leq b_0} \eta_b$ of classes $\eta^b\in\aleph^b$, i.e.\ in elements of the space
\[\aleph:=\lim_{\underset{b}{\longrightarrow}}\lim_{\underset{a}{\longleftarrow}}\bigoplus_{a\leq c\leq b} \aleph^c=\bigg\{\sum_{b\leq b_0}\eta^b\;\bigg|\;b_0\in\mathbb{Z},\eta^b\in\aleph^b\bigg\}.\]
We want to define a surjective map $\Psi:\aleph\rightarrow FH$, which is $R$-linear if $R$ is a field. For this purpose choose for every $\eta\in\aleph^b$ a preimage under $\pi$, i.e.\ fix a map \begin{align}\label{defnrho}
   \rho:\aleph^b\rightarrow FH^b\qquad\text{ such that }\qquad \pi(\rho(\eta))=\eta\quad\text{ and }\quad\rho(0)=0.                                                                                                                                                                                                                                \end{align}
 Note that in general $\rho$ is neither unique nor does there exist any $R$-linear $\rho$ \textbf{!} However, one can pick an $R$-linear $\rho$ if the sequence splits at $\pi$, in particular if $R$ is a field or semi-simple.
 Now we define $\Psi$ by setting
\begin{equation}\label{defnpsi}
 \Psi\bigg(\sum_{b\leq b_0} \eta^b\bigg):=\bigg[\sum_{b\leq b_0} s^b\bigg],\qquad\text{ where }\quad [s^b]^b=\rho(\eta^b),\;s^b\in FC^b.
\end{equation}
It is easy to see that $\Psi$ is well-defined, as $\sum_{b\leq b_0} s^b$ is a cycle, whose class does not depend on the particular choice of the representative $s^b$ for $\rho(\eta^b)$. Indeed we have:\pagebreak
\begin{align*}
 d\bigg(\sum_{b\leq b_0} s^b\bigg) &= \sum_{b\leq b_0} d s^b =\sum_{b\leq b_0} 0 =0\tag{$s^b\in \ker(d)\;\;\forall b\leq b_0$}\\
 \sum_{b\leq b_0} (s^b+d\beta^b) &= \sum_{b\leq b_0} s^b + d\sum_{b\leq b_0}\beta^b.\tag{$\beta^b\in FC^b$}
\end{align*}
As $\rho(0)=0$, $\Psi$ is also not effected by leading zeros, i.e.\ $\Psi\big(\sum_{b\leq b_0}\eta^b\big)=\Psi\big(\sum_{b\leq b_0-1}\eta^b\big)$ if $\eta^{b_0}=0$. Note that $\Psi$ does depend on the choice of $\rho$ and is therefore in general not an $R$-homomorphism! However, if $\rho$ is $R$-linear, then $\Psi$ becomes $R$-linear as well.
\begin{theo}[\textbf{Representation Theorem}]\label{theo31}~\\
 The homology $FH$ is represented by $\aleph$ that is $\Psi(\aleph)=FH$, i.e.\ $\Psi$ is surjective.
\end{theo}
\begin{rem}
 It follows from this theorem that $FH$ as a set cannot be larger then $\aleph$. Note that if we have a grading of the complex, then this theorem and all following extend  to the graded homology groups and reads as $\Psi(\aleph_k)=FH_k$.
\end{rem}
\begin{proof}~\\
 To show that $\Psi$ is surjective let $\sigma\in FH$ be arbitrary with $[s^{b_0}]=\sigma,\;s^{b_0}\in FC,$ such that the highest non-zero coefficient of $s^{b_0}$ has index $b_0$. So $s^{b_0}\in FC^{b_0}$ and the class $\sigma^{b_0}:=[s^{b_0}]^{b_0}\in FH^{b_0}$ is well-defined. We will construct for each $b\leq b_0$ a quintuple 
 \begin{gather*}
  (\sigma^b,\eta^b,s^b,\mathfrak{s}^b,\beta^b)\in FH^b\times\aleph^b\times FC^b\times FC^b\times FC^b\qquad\text{ where }\\                                                                                                                                                                                                                                                                                                   
  \eta^b=\pi(\sigma^b),\quad [s^b]^b=\sigma^b,\quad [\mathfrak{s}^b]^b=\rho(\eta^b)\quad\text{ and }\qquad \mathfrak{s}^b=s^b-s^{b-1}-d\beta^b.\tag{$\ast$}
  \end{gather*}
Then we see that $\Psi\big(\sum_{b\leq b_0}\eta^b\big)=\sigma$, i.e.\ that $\Psi$ is surjective, as
\begin{align*}
 \Psi\bigg(\sum_{b\leq b_0} \eta^b\bigg)=\bigg[\sum_{b\leq b_0}\mathfrak{s}^b\bigg]&=\bigg[\sum_{b\leq b_0} s^b-s^{b-1}-d\beta^b\bigg]\\
 &=\bigg[\sum_{b\leq b_0}(s^b-s^{b-1})-\sum_{b\leq b_0}d\beta^b\bigg]=\bigg[s^{b_0}-d\sum_{b\leq b_0}\beta^b\bigg]=[s^{b_0}]=\sigma.
\end{align*}
The first element $\sigma^{b_0}=[s^{b_0}]^{b_0}$ is already given. Now let $\sigma^b,\; b\leq b_0$ be constructed and set $\eta^b,s^b$ and $\mathfrak{s}^b$ as in $(\ast)$. Then we have
\[\pi(\sigma^b-\rho(\eta^b))=\pi(\sigma^b)-\pi(\rho(\eta^b))\overset{(\ref{defnrho})}{=}\eta^b-\eta^b=0.\]
Thus we have $\sigma^b-\rho(\eta^b)\in\ker(\pi)=im(\iota)$. Hence we can apply $j$ and define \linebreak $\sigma^{b-1}:=j(\sigma^b-\rho(\eta^b))\in FH^{b-1}$. Let $s^{b-1}$ be a representative of $\sigma^{b-1}$. The definition of $\sigma^{b-1}$ then yields
\[[s^{b-1}]^{b}=\iota\Big([s^{b-1}]^{b-1}\Big)=\iota(\sigma^{b-1})\overset{(\ref{defnj})}{=}\sigma^b-\rho(\eta^b)=[s^b-\mathfrak{s}^b]^b.\]
Hence there exists a $\beta^b\in FC^b$ such that $s^{b-1}+d\beta^b=s^b-\mathfrak{s}^b$. This gives us the quintuple $(\sigma^b,\eta^b,s^b\mathfrak{s}^b,\beta^b)$ and a new $\sigma^{b-1}$.
\end{proof}
In order to obtain more precise information, we are now going to analyze the ``kernel'' of $\Psi$, i.e.\ the preimage $\Psi^{-1}(0)$. To this purpose define
\[K^b:= im(\delta)=\delta\big(FH^{b+1}_{b+1}\big)\cong \raisebox{.2em}{$FH^b_b$}\left/\raisebox{-.2em}{$\pi(FH^b)$}\right.\qquad\text{ and }\qquad K:=\lim_{\longrightarrow}\lim_{\longleftarrow}\bigoplus_{a\leq c\leq b} K^c,\]
where $K$ is the space of bounded infinite sums $\sum_{k\leq b_0} \kappa^b,\,\kappa^b\in K^b$.\\
We want to define a map $\Phi: K\rightarrow \aleph$ with $\Phi(K)=\Psi^{-1}(0)$, such that $\Phi$ is $R$-linear if $R$ is a field. Note that this implies that if $\Phi$ and $\Psi$ are $R$-linear, then $FH\cong \raisebox{.2em}{$\aleph$}\left/\raisebox{-.2em}{$\Phi(K)$}\right.$. This will ensure that the homology of $\partial=\Phi\circ \delta$ is isomorphic to $FH$.\\
To define $\Phi$, choose for every $\sigma\in im(\iota)\subset FH^b$ a preimage under $\iota$, i.e.\ fix a map 
 \begin{align}\label{defnj}
  j:im(\iota)\rightarrow FH^{b-1}\qquad \text{ such that }\qquad\iota(j(\sigma))=\sigma\quad\text{ and }\quad j(0)=0.
 \end{align}
Again, $j$ is in general only a map, but can be chosen $R$-linear if the sequence splits at $\iota$. Then, for $\sum_{b\leq b_0}\kappa^b\in K$ set
\begin{align*}
 \sigma^{b_0}&:=\kappa^{b_0}\in FH^{b_0},& \eta^{b_0}&:=\pi(\sigma^{b_0})\in \aleph^{b_0}\\
 \sigma^{b-1}&:=j\big(\sigma^b-\rho(\eta^b)\big)+\kappa^{b-1}\in FH^{b-1}, &\eta^{b-1}&:=\pi(\sigma^{b-1})\in\aleph^{b-1}.
\end{align*}
Now, we define $\Phi$ by
\[\Phi\bigg(\sum_{b\leq b_0} \kappa^b\bigg):=\sum_{b\leq b_0} \eta^b\in \aleph.\]
As $\Phi$ does depend on $\rho$ and $j$, it is in general not $R$-linear. However, if $j$ and $\rho$ are $R$-linear, the same holds for $\Phi$. As $j(0)=0$ and $\rho(0)=0$, $\Phi$ is not effected by leading zeros, i.e.\ $\Phi\big(\sum_{b\leq b_0} \kappa^b\big)=\Phi\big(\sum_{b\leq b_0-1} \kappa^b\big)$ if $\kappa^{b_0}=0$. 
\begin{lemme}\label{lem32}
 We have indeed $\Psi^{-1}(0)=\Phi(K)$.
\end{lemme}
\begin{proof}~
 \begin{itemize}
  \item $\Psi^{-1}(0)\subset \Phi(K)$\medskip\\
  \underline{Proof:} Let $\sum_{b\leq b_0} \eta^b\in \Psi^{-1}(0)$ be arbitrary with representatives $[s^b]^b=\rho(\eta^b)$ and define $\sigma^B:=\big[\sum_{b\leq B} s^b\big]^B\in FH^B$. As $\Psi\big(\sum_{b\leq b_0} \eta^b\big)=\big[\sum_{b\leq b_0} s^b\big]=0$, there exists a $b_0'\in \mathbb{Z}$ and $\beta\in FC^{b_0'}$ such that $\sum_{b\leq b_0} s^b=d\beta$. By setting $s^b=0$ for $b\geq b_0$ we may without loss of generality assume that $b_0'=b_0+1$ (note that $\Psi\big(\sum_{b\leq b_0}\eta^b\big)=\Psi\big(\sum_{b\leq b_0-1}\eta^b\big)$ if $\eta^{b_0}=0$).\\
  Now, we construct a sequence $\kappa^b\in K^b, b\leq b_0$ such that
  \[\kappa^{b_0}=\sigma^{b_0}\quad\text{ and }\quad \sigma^{b-1}=j\big(\sigma^b-\rho(\eta^b)\big)+\kappa^{b-1}.\]
  By the construction of $\Phi$, it then follows that $\Phi\big(\sum_{b\leq b_0} \kappa^b\big)=\sum_{b\leq b_0}\eta^b$.\medskip\\
  We saw above that $\sum_{b\leq b_0} s^b =d\beta$ with $\beta\in FC^{b_0+1}$. It follows that
  \[\iota(\sigma^{b_0})=\iota\Big(\Big[\sum_{b\leq b_0}s^b\Big]^{b_0}\Big)=\Big[\sum_{b\leq b_0}s^b\Big]^{b_0+1}=0\]
  and hence that $\sigma^{b_0}\in\ker(\iota)=im(\delta)$. The definition $\kappa^{b_0}:=\sigma^{b_0}\in \delta\big(FH^{b_0+1}_{b_0+1}\big)=K^{b_0}$ is therefore correct.\pagebreak\\ For any $B\leq b_0$ we calculate
  \begin{align*}
   \iota\Big(\sigma^{B-1}-j\big(\sigma^B-\rho(\eta^B)\big)\Big)&\overset{(\ref{defnj})}{=}\iota\Big(\big[\sum_{b\leq B-1} s^b\big]^{B-1}\Big)-\big(\sigma^B-\rho(\eta^B)\big)\\
   &=\big[\sum_{b\leq B-1}s^b\big]^B-\big[\sum_{b\leq B}s^b\big]^B+[s^B]^B&&=0.
  \end{align*}
  Therefore, we have $\sigma^{B-1}-j(\sigma^B-\rho(\eta^B))\in\ker(\iota)=im(\delta)$ and hence there exists a $\kappa^{B-1}\in\delta(FH^{B-1}_{B-1})$ such that $\sigma^{B-1}=j\big(\sigma^B-\rho(\eta^B)\big)+\kappa^{B-1}$.
  \item $\Psi^{-1}(0)\supset \Phi(K)$\medskip\\
  \underline{Proof:} Let $\sum_{b\leq b_0} \kappa^b\in K$ be arbitrary and let $\sigma^b,\,\eta^b$ for $b\leq b_0$ be as in the definition of $\Phi\big(\sum_{b\leq b_0} \kappa^b\big)$. Let $[dn^{b+1}]^b=\kappa^b,\;[s^b]^b=\sigma^b$ and $[t^b]^b=\rho(\eta^b)$ be their representatives, where $s^b, t^b\in FC^b$ and $n^{b+1}\in FC^{b+1}_{b+1}$ for $b\leq b_0$. Then \[\Psi\Big(\Phi\bigg(\sum_{b\leq b_0}\kappa^b\bigg)\Big)=\Psi\bigg(\sum_{b\leq b_0}\eta^b\bigg)=\bigg[\sum_{b\leq b_0} t^b\bigg].\]
  It follows from the definition of $\Phi$ that $[s^{b_0}]^{b_0}=\sigma^{b_0}=\kappa^{b_0}=[dn^{b_0+1}]^{b_0}$. Hence we may without loss of generality assume that $s^{b_0}=dn^{b_0+1}$. It also follows from the definition of $\Phi$ that
  \[[s^{b-1}]^b=\iota(\sigma^{b-1})=\sigma^b-\rho(\eta^b)+\iota(\kappa^{b-1})=\big[s^b-t^b+dn^b\big]^b.\]
  Hence there exists a $\beta^b\in FC^b$ such that $s^{b-1}=s^b-t^b+dn^b+d\beta^b$. Then we have
  \[\sum_{b\leq b_0} t^b = \sum_{b\leq b_0}\left(s^b-s^{b-1}+dn^b+d\beta^b\right)=s^{b_0}+d\sum_{b\leq b_0}(n^b+\beta^b) = dn^{b_0+1}+d\sum_{b\leq b_0}(n^b+\beta^b).\]
  Thus, $\sum_{b\leq b_0} t^b$ is a boundary and hence $\Psi\circ\Phi\big(\sum_{b\leq b_0} \kappa^b\big)=\big[\sum_{b\leq b_0} t^b\big]=0\in FH$.
 \end{itemize}
\end{proof}
\begin{lemme}\label{lem33}
 $\Phi^{-1}(0)=\{0\}$, i.e.\ if $\Phi$ is $R$-linear, then it is injective.
\end{lemme}
\begin{proof}
 Let $\sum_{b\leq b_0} \kappa^b\in \Phi^{-1}(0)$ be arbitrary. Let $\sigma^b$ and $\eta^b$ for $b\leq b_0$ be as in the definition of $\Phi$. As $\Phi\big(\sum_{b\leq b_0}\kappa^b\big)=\sum_{b\leq b_0}\eta^b=0$, we find that $\eta^b=0$ for all $b\leq b_0$. Let $[dn^{b+1}]^b=\kappa^b$ and $[s^b]^b=\sigma^b$ be their representatives. Note that we may choose $[0]^b=\rho(\eta^b)$ as representatives, as $\eta^b=0$ and $\rho(0)=0$. The same calculations as in the second part of the proof of Lemma \ref{lem32} now show with $t^b=0$ that
 \[\textstyle 0=dn^{b_0+1}+d\sum_{b\leq b_0}(n^b+\beta^b).\]
 As $\sum_{b\leq b_0}(n^b+\beta^b)\in FC^{b_0}$, we find that $\kappa^{b_0}=[dn^{b_0+1}]^{b_0}=0$. This implies that $0=\Phi\big(\sum_{b\leq b_0} \kappa^b\big)=\Phi\big(\sum_{b\leq b_0-1}\kappa^b\big)$. Repeating the same arguments iteratively then shows that $\kappa^b=0$ for all $b\leq b_0$ and therefore that $\sum_{b\leq b_0} \kappa^b=0$.
\end{proof}
We defined above $\displaystyle F\mathscr{C}:=\lim_{\longrightarrow}\lim_{\longleftarrow} \textstyle\bigoplus_{c=a}^b FH^c_c$. Note that $\aleph=\ker\delta$ is a natural subset of $F\mathscr{C}$. The limit of the maps $\delta:FH^{b+1}_{b+1}\rightarrow im(\delta)\subset FH^b$ is an $R$-linear map $\delta : F\mathscr{C}\rightarrow K$.\\
If $R$ is a field or semi-simple, we can choose $\Phi$ to be $R$-linear. As mentioned above, we then define an $R$-linear operator $\partial$ on $F\mathscr{C}$ by \[\partial := \Phi\circ \delta.\]
Since $im(\Phi)\subset \aleph =\ker(\delta)$, we find that $\partial^2=0$, i.e.\ that $\partial$ is a boundary operator.
\begin{theo}[Reduction Theorem]\label{reductiontheo}~\\
 Let $R$ be a field (or semi-simple). Then, $\Psi$ induces a filtration preserving isomorphism between the homology of $(F\mathscr{C},\partial)$ and $FH$.\\
 As $F\mathscr{C}$ is generated by $FH^a_a,\,a\in\mathbb{Z}$, we can hence in a Morse-Bott setup always algebraically pretend that $FH$ is built from the singular homologies of the critical manifolds.
\end{theo}
\begin{proof}
 As $R$ is a field (or semi-simple), we may choose $\Phi$ and $\Psi$ to be $R$-linear. Then it follows from Lemma \ref{lem33} that $\ker(\Phi)=\Phi^{-1}(0)=\{0\}$. This shows that $\Phi$ is injective and hence $\ker(\partial)=\ker(\Phi\circ \delta)=\ker(\delta)=\aleph$. Since $im(\partial)=\Phi(\delta(F\mathscr{C}))=\Phi(im(\delta))=\Phi(K)$, we have for the homology of $(F\mathscr{C},\partial)$ that
 \[\raisebox{.2em}{$\ker(\partial)$}\left/\raisebox{-.2em}{$im(\partial)$}\right.=\raisebox{.2em}{$\aleph$}\left/\raisebox{-.2em}{$\Phi(K)$}\right..\]
 By Lemma \ref{lem32}, we have $\Phi(K)=\ker(\Psi)$. Hence, $\Psi$ induces a well-defined injective map
 \[\Psi:\raisebox{.2em}{$\ker(\partial)$}\left/\raisebox{-.2em}{$im(\partial)$}\right.=\raisebox{.2em}{$\aleph$}\left/\raisebox{-.2em}{$\Phi(K)$}\right.\longrightarrow FH,\tag{$\ast$}\]
 which we denote by abuse of notation also $\Psi$. As $\Psi$ is by Theorem \ref{theo31} surjective, we find that $(\ast)$ is in fact an isomorphism. The action filtration preserving property is obvious by the construction of $\Psi$ (see (\ref{defnpsi})).
\end{proof}
\begin{ex}
 Let us consider 4 critical points $C=\{a_1,b_1,a_0,b_0\}$, where the index denotes the action, i.e.\ $f(a_k)=k$. We set
 \begin{align*} 
 && m(b_1,a_1)&=2,& m(b_0,a_1)&=1& \text{ and }& \qquad m(b_0,a_0)=2.\\
  &\text{Then }& da_1&=2b_1+b_0,& da_0&=2b_0&  db_1&=db_0=0.
 \end{align*}
In a graded context, one should think of the $a_k$ as having one index higher then the $b_k$. Using $\mathbb{Z}$-coefficients, we get
\begin{align*}
 FH^1_1 &= \raisebox{.2em}{$\langle b_1\rangle$}\left/\raisebox{-.2em}{$\langle 2b_1\rangle$}\right.\cong \mathbb{Z}_2 & FH^0_0 = FH^0 = \raisebox{.2em}{$\langle b_0\rangle$}\left/\raisebox{-.2em}{$\langle 2b_0\rangle$}\right.&\cong \mathbb{Z}_2\\
 FH = FH^1_0 &= \raisebox{.2em}{$\langle b_0,b_1\rangle$}\left/\raisebox{-.2em}{$\langle 2b_0,2b_1+b_0\rangle$}\right.=&\raisebox{.2em}{$\langle 2b_1+b_0,b_1\rangle$}\left/\raisebox{-.2em}{$\langle -4b_1,2b_1+b_0\rangle$}\right.&\cong \mathbb{Z}_4.
\end{align*}
This shows that there is no hope of building a chain complex from $FH_1^1$ and $FH_0^0$ whose homology is $FH$, if we do not use field coefficients.\medskip\\
If we take $\mathbb{Z}_2$-coefficients instead, we get $da_1=b_0$ and $da_0=db_1=db_0=0$. Then
\begin{align*}
 FH^1_1 &= \langle a_1,b_1\rangle\cong \big(\mathbb{Z}_2\big)^2, &
 FH^0_0 = FH^0 &= \langle a_0,b_0\rangle\cong \big(\mathbb{Z}_2\big)^2,\\
 FH = FH^1_0 &= \raisebox{.2em}{$\langle a_0,b_0,b_1\rangle$}\left/\raisebox{-.2em}{$\langle b_0\rangle$}\right.\cong \big(\mathbb{Z}_2\big)^2.
\end{align*}
The only relevant sequence here is
\[FH^0\overset{\iota}{\longrightarrow}FH^1\overset{\pi}{\longrightarrow}FH^1_1\overset{\delta}{\longrightarrow}FH^0,\]
where $\delta$ maps the class of $b_1$ to 0 and the class of $a_1$ to the class of $b_0$, as $da_1=b_0$ and $db_1=0$. As $FH^0=FH^0_0$ no further maps have to be constructed. The boundary operator $\partial$ on the complex $F\mathscr{C}=\big\langle[a_1],[a_0],[b_1],[b_0]\big\rangle$ is thus given by $\partial[a_1]=[b_0]$ and $\partial[a_0]=\partial[b_1]=\partial[b_0]=0$. The homology of this complex is therefore
\[\big(F\mathscr{C},\partial\big) = \raisebox{.2em}{$\big\langle [a_0],[b_0],[b_1]\big\rangle$}\left/\raisebox{-.2em}{$\big\langle [b_0]\big\rangle$}\right.\cong \big(\mathbb{Z}_2\big)^2,\]
which is trivially isomorphic to $FH$.
\end{ex}

\newpage\phantom{.}
\newpage

\section{Contact surgery and handle attaching}\label{secsur}
This section is mostly already included in \cite{Cie}. However, we will redo the line of arguments to fill in some delicate details left open in the original article. Thus, we hope to make the proof that symplectic (co)homology is invariant under subcritical surgery (Theorem \ref{theoinvsur}) more transparent, at least to some readers.\\
First, we describe the general construction for contact surgery, which is done by attaching a symplectic handle $H^{2n}_k$ to the symplectization of a contact manifold. Then, we describe the symplectic standard handle, which is a subset of $\mathbb{R}^{2n}$ defined as the intersection of two sublevel sets $\{\psi <-1\}\cap\{\phi>-1\}$, where $\phi$ and $\psi$ are functions on $\mathbb{R}^{2n}$. While $\phi$ is explicitly given, we describe the construction of a suitable $\psi$ in the Subsection \ref{ExPsi}. The calculation of Conley-Zehnder indices for 1-periodic Reeb orbits on $H^{2n}_k$ concludes this section.

\subsection{Surgery along isotropic spheres}
Let us briefly recall the contact surgery construction due to Weinstein, \cite{Wein}. Consider an isotropic sphere $S^{k-1}$ in a contact manifold $(N^{2n-1},\xi)$. The 2-form $\omega=d\lambda$ for a contact form $\lambda$ (with $\xi=\ker\lambda$) defines a natural conformal symplectic structure on $\xi$. Denote the $\omega$-orthogonal on $\xi$ by $\perp_\omega$. Since $S$ is isotropic, it holds that $TS\subset TS^{\perp_\omega}$. So, the normal bundle of $S$ in $N$ is given by
\[ TN/ TS= TN/ \xi \,\oplus\, \xi/(TS)^{\perp_\omega} \oplus (TS)^{\perp_\omega}/ TS.\]
The Reeb field $R_\lambda$ trivializes $TN/\xi$. The bundle $\xi/(TS)^{\perp_\omega}$ is canonically isomorphic to $T^\ast S$ via $v\mapsto \iota_v\omega$. The conformal symplectic normal bundle $CSN(S):=(TS)^{\perp_\omega}/TS$ carries a natural conformal symplectic structure induced by $\omega$. Since $S$ is a sphere, the embedding $S^{k-1}\subset\mathbb{R}^k$ provides a natural trivialization of the bundle $\mathbb{R}R_\lambda\oplus T^\ast S$. This trivialization together with a conformally symplectic trivialization of $CNS(S)$ specifies a standard framing for $S$ in $N$.\\
Note that we have to assume that $CNS(S)$ is trivializable. This holds certainly true for $S=S^0 =\{N,S\}$ (two points) or $S=S^{n-1}$. In the latter case we have $(TS)^{\perp_\omega}=TS$ and hence $CNS(S)=(0)$. Therefore, taking connected sums and surgery along Legendrian spheres is always possible.\bigskip\\ 
Following Weinstein, we define an isotropic setup as a quintuple $(P,\omega,Y,\Sigma,S)$, where $(P,\omega)$ is a symplectic manifold, $Y$ a Liouville vector field for $\omega$, $\Sigma$ a hypersurface transverse to $Y$ (so $\Sigma$ is contact) and $S$ an isotropic submanifold of $\Sigma$. In \cite{Wein}, Weinstein proves the following variant of his famous neighborhood theorem for isotropic manifolds:
\begin{prop}[\textbf{Weinstein}] \label{X}
 Let $(P_0,\omega_0,Y_0,\Sigma_0,S_0)$ and $(P_1,\omega_1,Y_1,\Sigma_1,S_1)$ be two isotropic setups. Given a diffeomorphism from $S_0$ to $S_1$ covered by an isomorphism of their symplectic subnormal bundles, there exist neighborhoods $U_j$ of $S_j$ in $P_j$ and an isomorphism of isotropic setups
 \[\phi : (U_0,\omega_0,Y_0,\Sigma_0\cap U,S_0) \rightarrow (U_1,\omega_1,Y_1,\Sigma_1\cap U_1,S_1)\]
 which restricts to the given mappings on $S_0$.\pagebreak
\end{prop}
We may now define contact surgery along an isotropic sphere as follows:\\ Let $H^{2n}_k\approx D^k\times D^{2n-k}$ be a symplectic standard handle (see \ref{handle}) and let $S^{k-1}$ be an isotropic sphere in a contact manifold $(N^{2n-1},\xi)$. Then, Proposition \ref{X} allows us to glue the (lower) boundary $S^k\times D^{2n-k}$ of $H^{2n}_k$ to the symplectization $N\times[0,1]$ along the boundary part $U_1\cap N\times[0,1]$ of a tubular neighborhood $U_1$ of $S\times\{1\}$ (see Figure \ref{fig2}). We obtain an exact symplectic manifold $P:=N\times[0,1]\cup_{S}H^{2n}_k$ with a Liouville vector field $Y$ which is on $N\times[0,1]$ simply $\frac{\partial}{\partial t}$, where $t$ denotes the coordinate on $[0,1]$. The field $Y$ is inward pointing along $\partial^-P :=N\times\{0\}$ and outward pointing along the other boundary component $\partial^+P$. Both manifolds are hence contact and $\partial^+P$ is obtained from $N$ by surgery along $S$. Moreover, $P$ is an exact symplectic cobordism between $\partial^-P$ and $\partial^+P$.
\begin{figure}[h]
\centering
 \resizebox{15cm}{!}{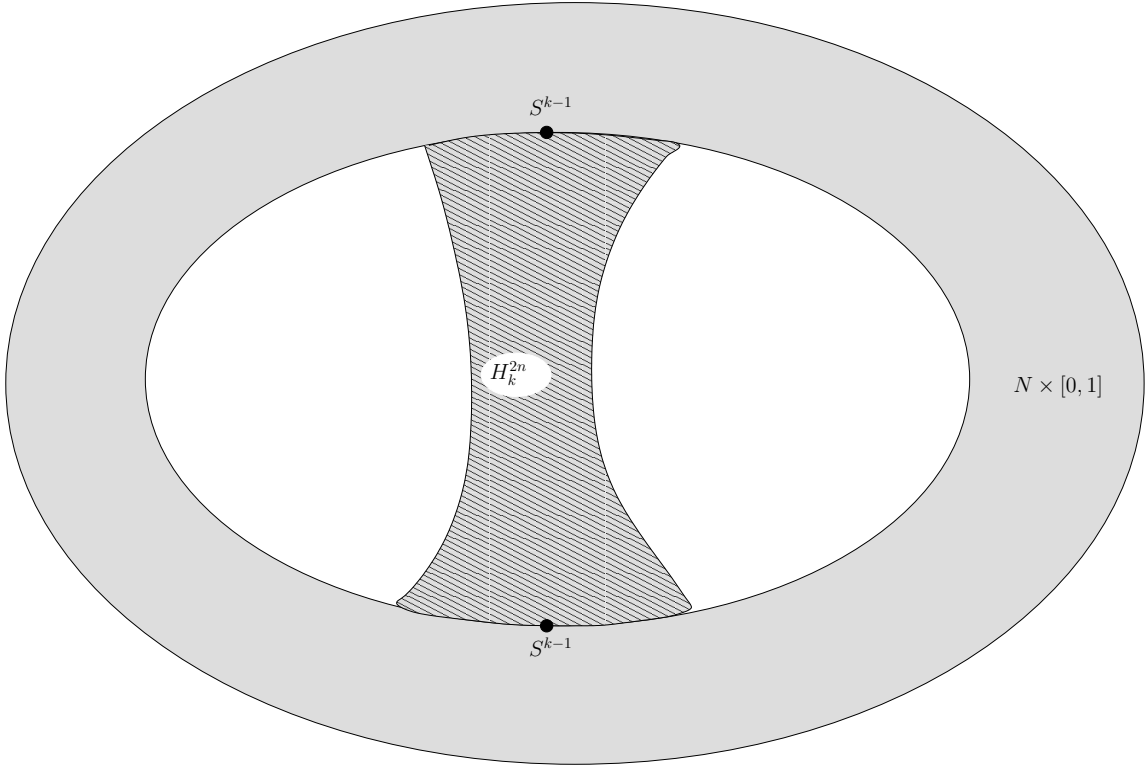}
 \caption{\label{fig2} $N\times[0,1]$ with handle attached}
\end{figure}
\subsection{The handle $H^{2n}_k$}\label{handle}
In order to specify a standard handle $H^{2n}_k$, we consider $\mathbb{R}^{2n}$ with symplectic coordinates $(q,p)=(q_1,p_1,...\,,q_n,p_n)$ and the following Weinstein structure (cf. \cite{Wein}):
\begin{align*}
 \omega &:= \sum_{j=1}^n dq_j\wedge dp_j\\
 Y &:= \sum_{j=1}^k \left(2q_j\frac{\partial}{\partial q_j} - p_j\frac{\partial}{\partial p_j}\right) + \sum_{j=k+1}^n \frac{1}{2}\left(q_j\frac{\partial}{\partial q_j} + p_j\frac{\partial}{\partial p_j}\right)\\
  \phi&:=\sum_{j=1}^k\left( q_j^2-\frac{1}{2}p_j^2\right) + \sum_{j=k+1}^n \frac{A_j}{4}\left(q_j^2+p_j^2\right),  
\end{align*}
with constants $A_j>0$.\\
Observe that the Liouville vector field $Y$ and the Weinstein function $\phi$ satisfy $Y\cdot \phi >0$ away from the origin. Note that $Y$ is in fact a Liouville vector field for $\omega$, as $\mathcal{L}_X\omega = \omega$, and its associated Liouville 1-form $\lambda := \iota_X\omega$ satisfies $d\lambda = \omega$. Explicitly, $\lambda$ is given by
\[\lambda := \sum_{j=1}^k\left(2q_jdp_j + p_jdq_j\right)+\sum_{j=k+1}^n\frac{1}{2}\left(q_jdp_j -p_jdq_j\right).\]
We introduce furthermore the following three quantities:
\[ x:= \sum_{j=1}^k q_j^2 \qquad y:= \sum_{j=1}^k \frac{1}{2} p_j^2 \qquad z:= \sum_{j=k+1}^n \frac{A_j}{4}\left( q_j^2+p_j^2\right),\]
whose Hamiltonian vector fields are given by
\[ X_x = \sum_{j=1}^k 2q_j\frac{\partial}{\partial p_j}\quad
 X_y = \sum_{j=1}^k -p_j\frac{\partial}{\partial q_j}\quad
 X_z = \sum_{j=k+1}^n \frac{A_j}{2}\left(q_j\frac{\partial}{\partial p_j} - p_j\frac{\partial}{\partial q_j}\right).\]
This convention allows us to write $\phi=x-y+z$ and $X_\phi=X_x-X_y+X_z$.\\
Now, consider the level surface $\Sigma^- :=\{\phi=-1\}$ and note that $Y$ is transverse to $\Sigma^-$, as $Y\cdot\phi|_{\Sigma^-}>0$. Hence, $\lambda|_{T\Sigma^-}$ is a contact form. The set $S:=\{x=z=0,\;y=+1\}$ is an isotropic sphere in $\Sigma^-$ and the quintuple $(\mathbb{R}^{2n},\omega,Y,\Sigma^-,S)$ will be the isotropic setup where we glue $H^{2n}_k$ to a contact manifold.
To specify a handle $H^{2n}_k$, we choose a different Weinstein function $\psi(q,p)=\psi(x,y,z)$ on $\mathbb{R}^{2n}$ such that the following holds:
\begin{itemize}
 \item[$(\psi 1)$] $\displaystyle \frac{\partial \psi}{\partial x},\frac{\partial \psi}{\partial z} \geq 0,\quad \frac{\partial \psi}{\partial y}\leq 0,\quad$ and $\displaystyle\quad \frac{\partial \psi}{\partial x}=\frac{\partial \psi}{\partial y}=\frac{\partial \psi}{\partial z}=0\quad$ only at the origin.
 \item[$(\psi 2)$] $\psi=\phi\quad\text{ for }\quad y>1+\veps\quad $ with $\veps$ arbitrarily small.
 \item[$(\psi 3)$] The set $\big\{\psi< -1\big\}\cap\big\{\phi> -1\big\}$ is diffeomorphic to $D^k\times D^{2n-k}$.
\end{itemize}
\begin{figure}[ht]
\centering
 \resizebox{10cm}{!}{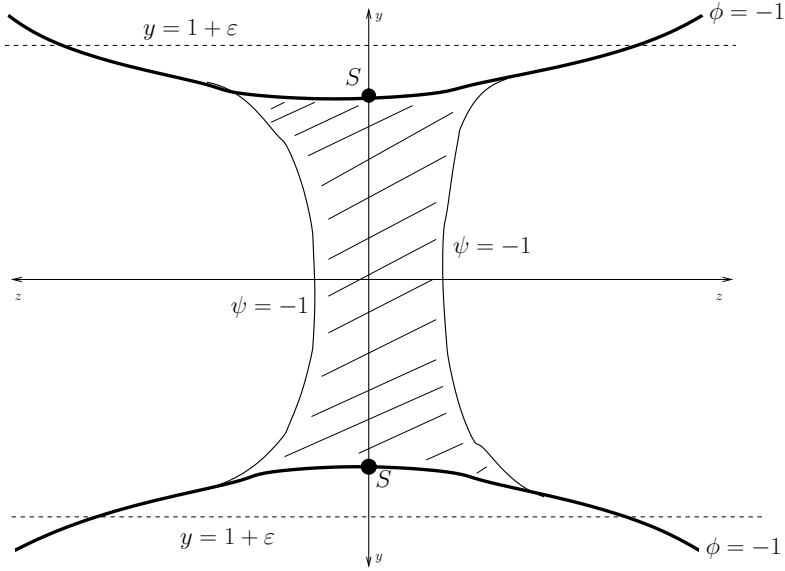}
 \\\caption{\label{fig3} The handle $H^{2n}_k$}
\end{figure}
The handle is defined to be the closure $\displaystyle\quad H^{2n}_k:=\overline{\big.\{\psi< -1\}\cap\{\phi>-1\}}\;\,$ (see Figure \ref{fig3}).
\begin{proper} ~
 \begin{itemize}
  \item It follows from $(\psi 1)$ that the level sets $\Sigma^+:=\{\psi=-1\}$ and $\Sigma^-=\{\phi=-1\}$ are both contact hypersurfaces, as also $X\cdot \psi > 0$. They coincide for $y\geq 1+\veps$ due to $(\psi 2)$ and they contain the boundary of $H^{2n}_k$. Condition $(\psi 3)$ on the other hand assures that $\Sigma^+$ is obtained from $\Sigma^-$ by surgery along $S$.
  \item It follows also from $(\psi 2)$ that reducing $\veps$ lets the handle become thinner. By choosing $\veps $ sufficiently small, we can make the handle so thin that its ``lower'' boundary $\{\phi=-1\}\cap H^{2n}_k$ lies inside any prescribed neighborhood of $S$.
  \item The handle stays unchanged if we take $\phi'=\alpha\cdot \phi + \beta$ and $\psi'=\alpha\cdot\psi+\beta$, $\alpha\neq 0$, provided that we set $\bigg.\quad H^{2n}_k=\overline{\big.\{\psi'< -\alpha+\beta\}\cap\{\phi'> -\alpha+\beta\}}$.
 \end{itemize}
\end{proper}

The Hamiltonian vector field $X_{\psi'}$ of $\psi'=\alpha\cdot\psi+\beta$ is given by
\begin{align} \label{HamEq}
 &\phantom{=}\;\;X_{\psi'} = \alpha\cdot X_\psi= \alpha\cdot\left(\frac{\partial \psi}{\partial x} X_x + \frac{\partial \psi}{\partial y}X_y + \frac{\partial \psi}{\partial z}X_z\right)\notag\\
 &= \alpha\sum_{j=1}^k \left(2\frac{\partial\psi}{\partial x} q_j\frac{\partial}{\partial p_j}-\frac{\partial\psi}{\partial y} p_j\frac{\partial}{\partial q_j}\right) + \alpha\sum_{j=k+1}^n\frac{\partial \psi}{\partial z}\cdot\frac{A_j}{2}\left(q_j\frac{\partial}{\partial p_j}-p_j\frac{\partial}{\partial q_j}\right).
\end{align}
For our symplectic setting, we consider the following Lyapunov function $f:=\sum_{j=1}^k q_jp_j$. Note that $f$ satisfies $X_{\psi'}\cdot f>0$ away from the critical points of $f$. This shows that all periodic orbits of $X_{\psi'}$ are contained in the set
\[\{x=y=0\} = \{q_1=p_1=...=q_k=p_k=0\}.\]
\subsection{An explicit $\psi$ and its extension to a neighborhood of $H^{2n}_k$}\label{ExPsi}
It is not difficult to find a Weinstein function $\psi$ which satisfies $(\psi 1)$--$(\psi 3)$. Fix $\veps>0$ and choose a smooth, monotone function $g:\mathbb{R}\rightarrow [0,1]$ such that
\begin{align*}
 g(t) &= \begin{cases}0\quad \text{ for }\quad t\leq 0\\ 1\quad\text{ for }\quad t\geq 1+ \veps\end{cases}\\
 \text{and }\qquad 0\leq g'(t) &\leq \frac{1}{1+\veps}+\delta,\qquad\delta>0\quad\text{small},\quad\forall\, t.
\end{align*}
\begin{equation}\label{eqXX}
 \text{Then set }\qquad\qquad\psi:=x-y+z-(1+\veps/2)+(1+\veps/2)\cdot g(y).
\end{equation}
This satisfies $(\psi 1)$--$(\psi 3)$, provided that $\delta$ is small enough.\bigskip\\
For symplectic (co)homology, we need $\psi$ not only on $H^{2n}_k$, but we have to arrange that $\psi$ extends to a neighborhood of $H^{2n}_k$ in a specific way. In order to describe what we mean by that, we make the following observation:\\
Let $\Sigma^-=\{\phi=-1\}$ and $\Sigma^+=\{\psi =-1\}$ be as above. As both are hypersurfaces transversal to the Liouville vector field $Y$, the flow $\varphi^t$ of $Y$ provides symplectic embeddings of the symplectizations of $\Sigma^+$ resp.\ $\Sigma^-$ into $\mathbb{R}^{2n}$:
\[
 \Phi^\pm : \Sigma^\pm \times (-\infty,\infty)\rightarrow \mathbb{R}^{2n},\qquad \Phi^\pm(y,t) = \varphi^t(y) \]
On any symplectization $(\Sigma\times(-\infty,\infty),d(e^t\lambda))$ of a contact manifold and for any $\alpha,\beta\in\mathbb{R}$, we define a function $h_\Sigma$ by $\qquad\Big.h_\Sigma(s,t) := \alpha\cdot e^t+\beta.$\\
We call such a function linear on $\Sigma\times\mathbb{R}$, and in fact it is linear when using the coordinate $r:=e^t, \,r\in(0,\infty),$ instead of $t$. Observe that the Hamiltonian vector field of $h_\Sigma$ is given by $X_{h_{\Sigma}}(s,t)=\alpha\cdot R_\lambda(s)$, where $R_\lambda$ is the Reeb vector field of $\lambda$, the contact form on $\Sigma$. Let $\tilde{h}^\pm_\Sigma$ be functions of this form for $\Sigma^\pm$ with $\alpha=1,\beta=-2$, such that $\tilde{h}_{\Sigma^\pm}(\Sigma^\pm)=-1$, and let $h_\Sigma^\pm := \tilde{h}_{\Sigma^\pm}\circ (\Phi^\pm)^{-1}$ be their pushforward onto the image of $\Phi^\pm$ in $\mathbb{R}^{2n}$. Note that $h_\Sigma^+$ and $h_\Sigma^-$  coincide on $\Phi^\pm((\Sigma^+\cap\Sigma^-)\times(-\infty,\infty))$, as $\Phi^-=\Phi^+$ on $(\Sigma^-\cap\Sigma^+)\times(-\infty,\infty)$.\\
In order to compare the symplectic (co)homologies of $\Sigma^-$ and $\Sigma^+$, we need a Hamiltonian that is linear on the negative symplectization of $\Sigma^-$ and the positive symplectization of $\Sigma^+$. As $\psi$ will serve as such a Hamiltonian, we require that $\psi=h_\Sigma^+$ on $\{\psi\geq -1\}$ and $\psi=h_\Sigma^-$ on $\{\phi\leq -1\}\setminus U$, where $U$ is a compact neighborhood of $S=\{x=z=0,\,y=1\}$ (see Figure \ref{fig4}).
\newpage
\begin{figure}[h]
\centering
 \resizebox{9cm}{!}{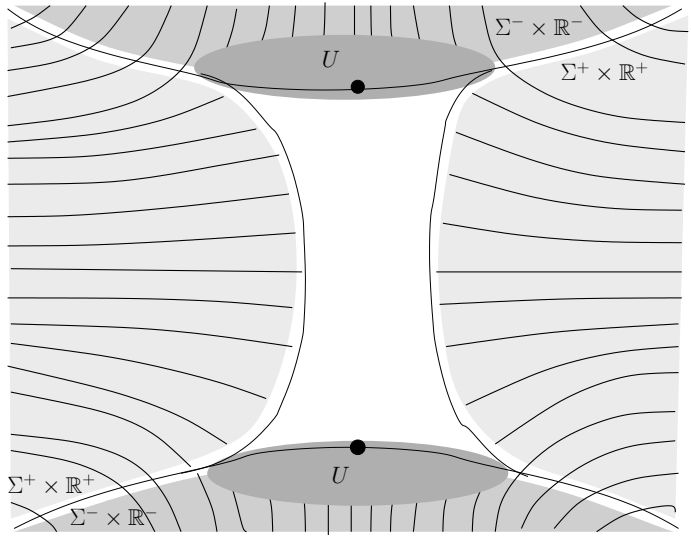}
 \caption{\label{fig4}Areas, where $\psi$ is linear together with the symplectizations of $\Sigma^\pm$}
\end{figure}
\begin{dis}
 It is the extension of $\psi$ beyond the handle, that is not quite correct in \cite{Cie}: It is stated there that one can extend $\psi$ on the positive symplectization of $\Sigma^+$ such that there is only one 1-periodic orbit of $X_\psi$ on the handle. To achieve this, $\psi$ has to be linear on the symplectizations away from the handle (as already stated above), i.e.\ $\psi$ has to be of the form $\psi=\alpha\cdot e^t+\beta$ for $\alpha\not\in spec(\Sigma^+)$. Moreover, it has to be of this form on the set $\{x=y=0\}$ and it has to be increasing for $y\rightarrow 0$ on the set $\{x=z=0\}$.\\
 Let us write for the moment $\psi= \alpha^+\cdot e^t + \beta^+$ on $\Sigma^+\times\mathbb{R}^+$ and $\psi =\alpha^-\cdot e^t+\beta^-$ on $\Sigma^-\times\mathbb{R}^-$. As $\Sigma^+$ and $\Sigma^-$ coincide together with their symplectizations on an open set, we find that $\alpha^+=\alpha^-$ and $\beta^+=\beta^-$. However, following the path depicted in Figure \ref{fig5} and keeping in mind that $\psi=\alpha^+\cdot e^t+\beta^+$ on $\{x=y=0\}$ and $\partial_y \psi\leq 0$ on $\{x=z=0\}$, we find that $\beta^+>\beta^-$, a contradiction.
 \begin{figure}[h]
\centering
 \resizebox{8.5cm}{!}{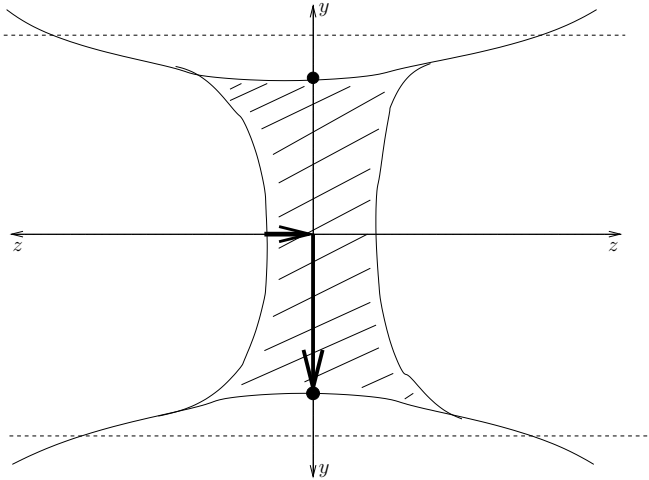}
 \caption{\label{fig5}Fig. 1: The problematic path}
\end{figure}
\newpage
Our solution to this dilemma is to allow $\psi$ to have a varying slope on $\{x=y=0\}$, first letting it grow very slowly coming from the origin and increasing the slope sharply near $\Sigma^+$. Using the Lyapunov function $f$, we can then show that this construction creates 1-periodic $X_\psi$-orbits only in the set $\{x=y=0\}$. These can be explicitly described and are hence still manageable.
\end{dis}
In order to construct such a $\psi$, we need the following two technical lemma:
\begin{lemme}
 Consider $\mathbb{R}^{2n}$ with the standard symplectic structure and the Liouville vector field given in \ref{handle} and 
 \[ \textstyle x:= \sum_{j=1}^k q_j^2 \qquad y:= \sum_{j=1}^k \frac{1}{2} p_j^2 \qquad z:= \sum_{j=k+1}^n \frac{A_j}{4}\left( q_j^2+p_j^2\right).\]
 Let $\Sigma\subset\mathbb{R}^{2n}$ be a smooth hypersurface transverse to the Liouville vector field $Y$ such that its outer normal $N$ is of the form $N=c_x\cdot \nabla_x+c_y\cdot \nabla_y + c_z\cdot \nabla_z$, where $c_x,c_y,c_z : \Sigma\rightarrow \mathbb{R}$ are functions with $c_x,c_z>0, c_y < 0$ and $\nabla_x,\nabla_y,\nabla_z$ are the gradients of $x,y,z$. Let $\tilde{h}_\Sigma$ denote the function $\tilde{h}_\Sigma(s,t)=\alpha e^t+\beta$ on the symplectization of $\Sigma$ and let $h_\Sigma=\tilde{h}_\Sigma\circ \Phi^{-1}$ be its pushforward onto $\mathbb{R}^{2n}$ by the flow $\Phi$ of $Y$.\\
 Then, the Hamiltonian vector field $X_h$ of $h_\Sigma$ is of the form
 \begin{align*}
  &X_h=C_x\cdot X_x + C_y\cdot X_y + C_z\cdot X_z,
 \end{align*}
 where $C_x,C_y,C_z\in C^{\infty}(\mathbb{R}^{2n})$ are functions satisfying $C_x,C_z>0,\; C_y <0$.
\end{lemme}
\begin{rem}~
\begin{itemize}
 \item Note that $C_x,C_z>0$ and $C_y<0$ guarantees that the Lyapunov function \linebreak $f=\sum_{j=1}^k q_jp_j$ satisfies $X_h\cdot f = 0$ only on $\{x= y= 0\}$.
 \item The assumptions on $\Sigma$ are satisfied, if $\Sigma=\psi^{-1}(c)$ for a function $\psi$ on $x,y,z$ with $\frac{\partial \psi}{\partial x}\big|_\Sigma,\frac{\partial \psi}{\partial z}\big|_\Sigma>0$ and $\frac{\partial \psi}{\partial y}\big|_\Sigma<0$.
\end{itemize} 
\end{rem}
\begin{proof}
 As $X_{\tilde{h}}=\alpha\cdot R_\lambda$ on $\Sigma\times(-\infty,\infty)$, it follows that on $\mathbb{R}^{2n}$ holds $X_h|_{\varphi^t(\Sigma)}=\alpha\cdot e^t\cdot R_t$, where $R_t$ is the Reeb field of $\lambda_t:=\lambda|_{T\varphi^t(\Sigma)}$, the contact form on $\varphi^t(\Sigma)$. By assumption, a normal $N$ to $\Sigma$ satisfies
 \begin{align*}
  N &= c_x\nabla_x+c_y\nabla_y+c_z\nabla_z \\
  &= c_x\sum_{j=1}^k 2q_j\frac{\partial}{\partial q_j}+c_y\sum_{j=1}^kp_j\frac{\partial}{\partial p_j} + c_z\sum_{j=k+1}^n \frac{1}{2}\left(q_j\frac{\partial}{\partial q_j}+p_j\frac{\partial}{\partial p_j}\right).
 \end{align*}
 The flow of $\varphi^t$ of $Y$ is given by 
 \[\varphi^t=\Big(\underbrace{...\,,e^{2t}\cdot q_j,e^{-t}\cdot p_j,...}_{j=1,...\,,k}\text{\huge,}\underbrace{...\,,e^{t/2}\cdot q_j,e^{t/2}\cdot p_j,...}_{j=k+1,...\,,n}\Big).\]
 A normal $N^t$ to $\varphi^t(\Sigma)$ is hence given by
 \[N^t = e^{2t}\cdot c_x\nabla_x + e^{-t}\cdot c_y\nabla_y+ e^{t/2}\cdot c_z\nabla_z.\]
 Let $J$ denote the standard almost complex structure on $\mathbb{R}^{2n}$, i.e.\ $J(\frac{\partial}{\partial q_j})=\frac{\partial}{\partial p_j}$ and $J( \frac{\partial}{\partial p_j})= -\frac{\partial}{\partial q_j}$. Using the definition of $\lambda_t$ and $N^t$, we find that the Reeb vector field $R_t$ is given by
 \begin{align*}
  R_t &= \frac{JN^t}{\lambda(JN^t)} = \frac{1}{C}\left( e^{2t}c_xX_x + e^{-t}c_yX_y+e^{t/2}c_zX_z\right)\\
  \text{with }\; C&=\lambda(JN^t) = e^{4t}(c_x)\cdot4x + e^{-2t}c_y \cdot(-2y)+e^{t}c_z\cdot z\,>0.
 \end{align*}
Using $X_h|_{\varphi^t(\Sigma)}=\alpha\cdot e^tR_t$, we see that $X_h$ is exactly of the announced form.
\end{proof}
\begin{lemme} \label{interpolation}
 Let $\veps,\delta,c>0$ be constants. Then there exists a smooth monotone function $g:\mathbb{R}\rightarrow[0,1]$ such that 
 \[g(t)=0\quad\text{ for }\quad t\leq 0\qquad g(t)=1\quad\text{ for }\quad t\geq \veps\tag{$\ast$}\]
 and for all $C^1$-functions $\phi,\psi$ with $\phi(0)=\psi(0)$ and $|\frac{\partial}{\partial t}\phi(t)-\frac{\partial}{\partial t}\psi(t)|<c$ for all $t\in[0,\veps]$ holds that
 \[\left|\left|\frac{\partial}{\partial t}\Big(\phi+\big(\psi-\phi\big)\cdot g\Big)-\Big(\frac{\partial}{\partial t}\phi + \big(\frac{\partial}{\partial t}\psi-\frac{\partial}{\partial t}\phi\big)\cdot g\Big)\right|\right|_\infty\leq \delta.\tag{$\ast\ast$}\]
 In other words, we can interpolate between $\phi$ and $\psi$, such that the slope of the interpolation is arbitrary close to the interpolation of the slopes of $\phi$ and $\psi$.
\end{lemme}
\begin{proof}
 We calculate that
 \[\frac{\partial}{\partial t}\Big(\phi+\big(\psi-\phi\big)g\Big)=\Big(\frac{\partial}{\partial t}\phi+\big(\frac{\partial}{\partial t}\psi-\frac{\partial}{\partial t}\phi\big)g\Big)+\big(\phi-\psi)\frac{\partial}{\partial t}g.\]
Therefore, $(\ast\ast)$ translates to
\[ \left|\big(\psi-\phi\big)\frac{\partial}{\partial t}g\right| = \left|\int_0^t\left(\frac{\partial}{\partial s}\psi - \frac{\partial}{\partial s}\phi\right) ds\cdot \frac{\partial}{\partial t}g\right| \leq c\cdot t\cdot \frac{\partial}{\partial t}g \leq \delta \qquad\forall t\in[0,\veps].\]
 So, $(\ast\ast)$ is satisfied, if $0\leq\frac{\partial}{\partial t}g(t) \leq \delta/ct\;\forall\,t\in[0,\veps]$. As $\int_0^\veps \delta/ct\,dt=\infty$, we can choose a smooth function $\tilde{g}$ satisfying
 \[0\leq \tilde{g}(t)\leq \frac{\delta}{c\cdot t},\qquad \tilde{g}\equiv 0 \;\,\text{ for } \;\,t\leq 0\,\;\text{ and }\;\,t\geq \veps \qquad \text{ and }\qquad \int_0^\veps \tilde{g}(t)\,dt = 1.\]
Setting $g(t):=\int_0^{t}\tilde{g}(s)\,ds$ then gives the desired function.
\end{proof}
Now, we construct $\psi$ in two steps:
\begin{itemize}
 \item First, recall that the isotropic sphere $S\subset\Sigma^-=\{\phi=-1\}$ is given by
\[ S:=\{x=z=0,\; y=1\}.\]
Consider the function $h_\Sigma^-$. As the Reeb vector field $R_{\Sigma^-}$ of $(\Sigma^-,\lambda|_{T\Sigma^-})$ coincides with the Hamiltonian vector field $X_\phi$, we find $X_{h_\Sigma^-}=R_{\Sigma^-}=X_\phi$ and $dh_\Sigma^-=d\phi$. As also $h_\Sigma^-(\Sigma^-)=\phi(\Sigma^-)=-1$, we find that $h_\Sigma^-$ and $\phi$ coincide up to first order on $\Sigma^-$. Therefore, given any neighborhood $U$ of $S$, there exists a function $\hat{\phi}$ of $x,y,z$ and a neighborhood $\hat{U}\subset U$, such that $\hat{\phi}\equiv h_\Sigma^-$ on $\mathbb{R}^{2n}\setminus U,\; \hat{\phi}\equiv\phi$ on $\hat{U}$ and $\hat{\phi}$ is arbitrarily $C^1$-close to $h_{\Sigma}^-$. Consequently, we can arrange that
\[X_{\hat{\phi}} = C_x\cdot X_x + C_y\cdot X_y + C_z\cdot X_z\quad\text{ with }\quad C_x,C_z>0,\;C_y<0.\]
Now choose a handle $H^{2n}_k$ so thin, such that $H^{2n}_k\cap\Sigma^-\subset \hat{U}$. See Figure \ref{fig6} for the different areas. Let $H^{2n}_k$ be defined by a function $\tilde{\psi}$ as in (\ref{eqXX}) and set
\begin{align*}
 \hat{\psi}& : \{\phi\leq -1\}\cup H^{2n}_k\rightarrow \mathbb{R}\qquad\qquad \hat{\psi}=
 \begin{cases}\tilde{\psi} & \text{ on }\big(\hat{U}\cap\{\phi\leq -1\}\big)\cup H^{2n}_k\\ 
 \hat{\phi} & \text{ on }\big(U\cap\{\phi\leq -1\}\big)\setminus \hat{U}\\
 h^-_\Sigma & \text{ on }\{\phi\leq -1\}\setminus U
\end{cases}.
\end{align*}
Since $\tilde{\psi}=\hat{\phi}$ outside a small neighborhood of $H^{2n}_k$, we find that $\hat{\psi}$ is smooth.
\begin{figure}[ht]
\centering
 \resizebox{10cm}{!}{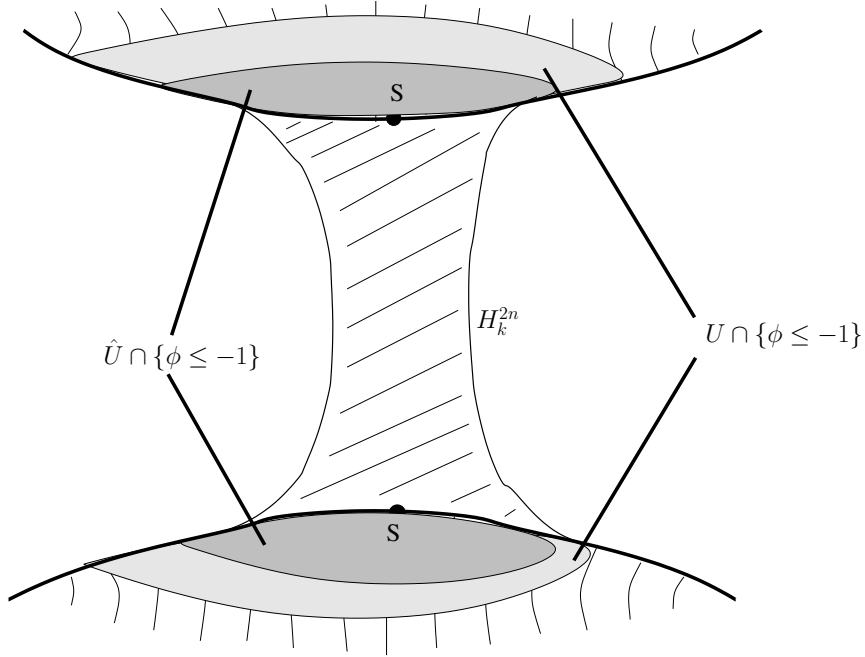}
 \caption{\label{fig6}Areas, where $\hat{\psi}$ is defined}
\centering
\end{figure}
\item For the second step consider the function $h_\Sigma^+$ associated to $\Sigma^+=\{\hat{\psi}=-1\}$. Use Lemma \ref{interpolation} to construct for $\delta>0$ small a smooth monotone function $g$ with
\begin{align}\label{eqpsi}
 && g(t) &\phantom{:}= 0\quad \text{ for } \quad t\leq -1-\delta,\qquad g(t) = 1\quad \text{ for } \quad t\geq -1\notag\\
 &\text{such that}& \psi &:= \hat{\psi} + (h_\Sigma^+-\hat{\psi})\cdot g(h_\Sigma^+)\quad \text{ satisfies }\notag\\
  && \Big.X_\psi &\phantom{:}= C_x\cdot X_x+C_y\cdot X_y+ C_z\cdot X_z,\qquad C_x,C_z>0,\;C_y<0.
\end{align}
Recall that $h_\Sigma^+=h_\Sigma^-=\hat{\psi}$ outside $U$, so that we are actually interpolating along a compact set. Moreover, it follows from this, that $\psi = h_\Sigma^+$ away from the handle.\pagebreak
\end{itemize}

\subsection{Closed orbits and Conley-Zehnder indices}\label{sec5.4}
Note that (\ref{eqpsi}) implies for the Lyapunov function $f$ that $f\cdot X_\psi=0$ only on $\{x=y=0\}$. This continues to hold, when we consider $\psi':=\alpha\cdot \psi+\beta$ and it guarantees that the only periodic orbits of $X_{\psi'}$ lie in $\{x=y=0\}$. In the following we will determine the 1-periodic orbits of $X_{\psi'}$ and calculate their Conley-Zehnder indices.
By (\ref{eqpsi}), the Hamiltonian vector field $X_{\psi'}$ is given by
\[X_{\psi'}=\alpha\cdot\big(C_x X_x+C_y X_y + C_z X_z\big),\]
where $C_x,C_y,C_z\in C^\infty(\mathbb{R}^{2n})$ are functions with $C_x,C_z>0,\,C_y<0$ and $X_x,X_y$ and $X_z$ are the Hamiltonian vector fields of $x,y,z$ and given by
\[X_x=\sum_{j=1}^k 2q_j\frac{\partial}{\partial p_j},\qquad X_y=\sum_{j=1}^k-p_j\frac{\partial}{\partial q_j},\qquad X_z=\sum_{j=k+1}^n\frac{A_j}{2}\left(q_j\frac{\partial}{\partial p_j}-p_j\frac{\partial}{\partial q_j}\right).\]
Note that on $\{x=y=0\}$, we have therefore
\[X_{\psi'}=\alpha \cdot C_z \sum_{j=k+1}^n \frac{A_j}{2}\left(q_j\frac{\partial}{\partial p_j}-p_j\frac{\partial}{\partial q_j}\right).\]
On $\Sigma^+\times\mathbb{R}^+\cap\{x=y=0\}$, we have by our construction $\psi'=\alpha\cdot h_\Sigma^+$ and hence that $X_{\psi'}=\alpha\cdot R_\lambda|_{\Sigma^+}$, where the Reeb vector field on $\Sigma^+\cap\{x=y=0\}$ is given by
\[R_\lambda|_{\Sigma^+} = \frac{2}{\veps}\cdot \sum_{j=k+1}^n\frac{A_j}{2}\left(q_j\frac{\partial}{\partial p_j}-p_j\frac{\partial}{\partial q_j}\right)=\frac{2}{\veps}X_z.\]
The constant $2/\veps$ comes from the fact that $\lambda(X_z)=z$ and that the value of $z$ on $\Sigma^+\cap\{x=y=0\}=\psi^{-1}(-1)\cap\{x=y=0\}$ is given by (\ref{eqXX}) as
\[-1=0-0+z-(1+\veps/2)\qquad\Leftrightarrow\qquad z=\veps/2.\]
On the other hand, for $z<\veps/2-O(\delta)$, we have by the second step that $\psi=\hat{\psi}$. By (\ref{eqXX}) we have for $y=0$ that
$\hat{\psi}=x-y+z-(1+\veps/2)$ and hence there $X_{\psi'}=\alpha X_z$. It follows that on $\{x=y=0\}$, we have $X_{\psi'}=\alpha\cdot C_z X_z$, where $C_z$ is a $z$-dependent interpolation between the constants 1 and $2/\veps$. Now we can calculate the flow $\varphi^t$ of $X_{\psi'}$ on $\{x=y=0\}$. First we calculate for $z(\varphi^t(p,q))$
\begin{align*}
 \frac{d}{dt}z\big(\varphi^t\big)&=dz(X_{\psi'})=\alpha\cdot C_z\, dz(X_z)=0.
\end{align*}
It follows that $z$ is constant along the flow lines of $\varphi$. Now, consider for $j=k+1,...\,,n$ the complex coordinates $z_j=q_j+ip_j$. Then, we have on $\{x=y=0\}$:
\[X_{\psi'}=\alpha C_z\cdot\Bigg(\underbrace{0\,,...\,,0}_{j=1,...\,,k}\text{\Huge ,}\underbrace{...\,,i A_{j}\cdot z_j,...}_{j=k+1,...\,,n}\Bigg).\]
As $z(\varphi^t)$ is independent from $t$ and hence $\frac{d}{dt}C_z(0,0,z(\varphi^t))=0$, we obtain that the flow $\varphi^t$ of $X_{\psi'}$ on $\{x=y=0\}$ is given by
\begin{equation}\label{eqA}
 \textstyle\varphi^t(0,z_{k+1},...\,,z_n)= \Big(0,...\,,0,\exp\big(i\alpha C_z A_{k+1} t\big)\cdot z_{k+1},...\,,\exp\big(i\alpha C_z A_n t\big)\cdot z_n\Big).
\end{equation}
By choosing the constants $A_j$ linear independent over $\mathbb{Q}$, we can arrange that the 1-periodic orbits of $X_{\psi'}$ on $\{x=y=0\}$ are isolated. They are given by values of $z$ and $j$ such that
\[\alpha  A_j\cdot C_z(0,0,z)\in 2\pi\mathbb{Z},\]
except for the one constant orbit at the origin. For $\alpha$ appropriately chosen, we can assume that there are only finitely many such orbits.\\
Now, we calculate the Conley-Zehnder indices of these orbits, using the definition of $\mu_{CZ}$ for a path of symplectic matrices from Section \ref{1.4}. Let $z^0\in\{x=y=0\}$ be such that $\gamma(t):=\varphi^t(z^0),0\leq t\leq1,\,\gamma(0)=\gamma(1),$ is a 1-periodic orbit of $X_{\psi'}$. In order to calculate $\mu_{CZ}(\gamma)$, we identify $T_{\gamma(t)}\mathbb{R}^{2n}$ with $\mathbb{R}^{2n}$ in the natural way. This yields a path $\Phi_\psi$ in $Sp(2n)$ given by $\Phi_\psi(t)=D\varphi^t(z^0)$. The $t$-derivative of $\Phi_\psi$ on $\{x=y=0\}$ is given by
\begin{align*}
 \frac{d}{dt}\Phi_\psi(t)&=\frac{d}{dt}D\varphi^t(z^0)=D\left(\frac{d}{dt}\varphi^t(z^0)\right)=D X_{\psi'}\big(\varphi^t(z^0)\big)\\
 &=D\big(C_x X_x+C_y X_y+ C_z X_z\big)\big(\varphi^t(z^0)\big)\\
 &=\alpha\cdot diag\Bigg(\underbrace{...\,,\begin{pmatrix} 0 & -C_y\\2C_x& 0\end{pmatrix},...}_{j=1,...\,,k}\text{\Huge , }\underbrace{...\,,iC_z A_j,...}_{j=k+1,...\,,n}\Bigg)\circ\Phi_\psi(t).
\end{align*}
Note that no derivatives of $C_x$ or $C_y$ are involved, as $X_x=X_y=0$ on $\{x=y=0\}$. It follows that $\Phi_\psi$ is of block form $\Phi_\psi = diag\big(\Phi^1_\psi,...\,,\Phi_\psi^n\big)$, where the paths of $2\times2$ matrices $\Phi^j_\psi$ are solutions of an ordinary differential equation with initial value $\Phi^j_\psi(0)=\mathbbm{1}$ and
\begin{align*}
 {\textstyle\frac{d}{dt}}\Phi^j_\psi(t)&=\alpha\begin{pmatrix}0 & -C_y\\ 2C_x & 0\end{pmatrix}\Phi^j_\psi(t) = \begin{pmatrix} 0 & -1\\ 1& 0\end{pmatrix}\begin{pmatrix} 2\alpha C_x & 0\\0& \alpha C_y\end{pmatrix}\Phi^j_\psi(t)& j&=1,...\,,k\\
 {\textstyle\frac{d}{dt}}\Phi^j_\psi(t)&=i\alpha A_jC_z\cdot \Phi^j_\psi(t)& j&=k+1,...\,,n.
\end{align*}
As the matrix $\left(\begin{smallmatrix}2\alpha C_x & 0\\ 0 & \alpha C_y\end{smallmatrix}\right)$ has for all $t$ one positive and one negative eigenvalue, it follows that its signature is always zero and hence by the Robbin-Salamon definition of $\mu_{CZ}$ that $\mu_{CZ}(\Phi^j_\psi)=0,\;j=1,...\,,k$. For $j=k+1,...\,,n$ on the other hand, we find by Lemma \ref{expliciteCZ} that
\begin{align}\label{CZ-Index}
&&\mu_{CZ}\big(\Phi^j_\psi\big)&= \left\lfloor\frac{\alpha A_j C_z}{2\pi}\right\rfloor+\left\lceil\frac{\alpha A_j C_z}{2\pi}\right\rceil,\qquad j=k+1,...\,,n\notag\\
&\text{and therefore}& \mu_{CZ}(\gamma)=\mu_{CZ}(\Phi_\psi) &=\sum_{j=k+1}^n \left(\left\lfloor\frac{\alpha A_j C_z}{2\pi}\right\rfloor+\left\lceil\frac{\alpha A_j C_z}{2\pi}\right\rceil\right).
\end{align}
Note again that $1\leq C_z\leq 2/\veps$ is constant for each $\gamma$ with $\alpha A_j C_z(0,0,z(\gamma(0))\in 2\pi\mathbb{Z}$. We have therefore $\mu_{CZ}(\gamma)\geq \alpha\cdot\sum\left\lceil A_j/2\pi\right\rceil$. It follows for $\alpha\rightarrow \infty$ that all possible values for $\mu_{CZ}(\gamma)$ go to $\infty$.
\newpage
\phantom{..}
\newpage
\section{Symplectic (co)homology}\label{secsymhom}
In this section, we roughly introduce symplectic (co)homology -- a Floer type homology similar to Rabinowitz-Floer homology. In fact, we will see in \ref{sec6.8} that both constructions are closely related.\\
In Rabinowitz-Floer homology, the chain complex is built solely on data coming from the contact hypersurface $\Sigma$, a fact that makes this homology very suitable for calculating contact invariants. The chain complex in symplectic (co)homology uses all 1-periodic orbits of a Hamiltonian vector field $X_H$ on $V$. Moreover, the definition of the homology involves direct limits indexed over the set of all admissible $H$. This makes the construction more flexible and allows us to prove the invariance of symplectic (co)homology under subcritical handle attachment, which we then transfer to Rabinowitz-Floer homology.
\subsection{Setup}
First, we describe shortly the setup for symplectic (co)homology. It is quite similar to the one for Rabinowitz-Floer homology, but has some important differences -- in particular the use of time-dependent Hamiltonians and the absence of the parameter $\eta$.\medskip\\
In this section, let $(V,\lambda)$ be a compact Liouville domain, with exact symplectic form $\omega=d\lambda$ and convex contact boundary $(\Sigma=\partial V,\alpha=\lambda|_{T\Sigma})$. As before, let $Y$ be the Liouville vector field of $\lambda$. That $(\Sigma,\alpha)$ is convex simply means that $Y$ points out of $V$ along $\partial V=\Sigma$. The completion of $V$ is still denoted $\widehat{V}$, as well as $R$ denotes the Reeb vector field of $\alpha$. The action spectrum of $(\Sigma,\alpha)$ contains in this section only \textit{positive} periods of closed Reeb orbits, so that
\[spec(\Sigma,\alpha):=spec(\Sigma,\alpha)\cap\mathbb{R}^+.\]
Hamiltonians are smooth $S^1$-families of functions $H_t: \widehat{V}\rightarrow \mathbb{R}$ with Hamiltonian vector fields $X_H^t$ given by
\[\omega(\cdot,X^t_H)=dH_t. \tag{for $t\in S^1$ fixed}\]
The action of a loop $x: S^1\rightarrow \widehat{V}$ is defined by
\[\mathcal{A}^H(x)=\int^1_0 x^\ast\lambda - \int^1_0 H_t(x(t)) dt.\]
The critical points of this functional are exactly the one-periodic solutions of
\begin{equation}\label{eqastast}
 \dot{x}(t)=X^t_H.
\end{equation}
We denote the set of these solutions by $\mathcal{P}(H)$.\\ Almost complex structures $J_t$ come in $S^1$-families. An $\mathcal{A}^H$-gradient trajectory $u:\mathbb{R}\times S^1\rightarrow \widehat{V}$ is in this section a solution of the Floer equation:
\begin{equation}\label{eqast}
 \begin{aligned}
  &&\partial_s u-\nabla\mathcal{A}^H&=0\qquad
  &\Leftrightarrow& \qquad\, u_s + J(u_t-X_H^t)&=0.
 \end{aligned}
\end{equation}
As before, we are also interested in homotopies $H_s$ of Hamiltonians. In this setting an $\mathcal{A}^{H_s}$-gradient trajectory is again a solution of (\ref{eqast}), but with $X_H^t$ depending on $s$.\pagebreak\\
Symplectic (co)homology can also be $\mathbb{Z}$-graded by the Conley-Zehnder index. For that, we require the same assumptions as for Rabinowitz-Floer homology, namely that $\Sigma=\partial V$ is simply connected and that the evaluation of the first Chern class $c_1(TV)$ vanishes on $\pi_2(V)$. One could also define the index under more relaxed conditions. However, imposing these assumptions makes it easier to compare $SH$ and $RFH$.
\begin{rem}
Note that $\mu_{CZ}(x)$ of a 1-periodic orbit $x$ of $X_H^t$ is for $SH$ calculated with respect to a trivialization of $TV$ over a capping disc $\bar{u}$ of $x$ and not with respect to a trivialization of the contact form $\xi$ as in $RFH$. Nevertheless, our assumptions imply for a non-constant Reeb orbit $x$ that $\mu_{CZ}(x)$ is the same for $TV$ and $\xi$. This is due to the fact that over $\Sigma$, we have the splitting
\[TV=\xi\oplus \mathbb{R}R\oplus\mathbb{R}Y.\]
As the action of the linearized flow of $X_H$ on $\mathbb{R}R\oplus\mathbb{R}Y$ is trivial (as both vector fields are preserved by $X_H$), we find by the product property of $\mu_{CZ}$ that $\mu_{CZ}(x)$ actually involves only the $\xi$ part of $TV$. This phenomenon is illustrated by the calculations of $\mu_{CZ}$ on Brieskorn manifolds $\Sigma_a$ in Section \ref{sec7.2}, where by Corollary \ref{lem18a} $Y_2$ is the Reeb vector field on $\Sigma_a$ and $X_2$ the Liouville vector field on the filling $W_\veps$.
\end{rem}
For the construction of symplectic (co)homology we look at solutions of (\ref{eqast}) such that  $\displaystyle\lim_{s\rightarrow\pm\infty}=x_{\pm}(t)$ are 1-periodic solutions of (\ref{eqastast}). In general, these solutions might not stay in a compact subset of $\widehat{V}$, even for $x_\pm$ fixed. Hence, it could be that the moduli space of these solutions is not compactifiable. To avoid this problem, we make the following restrictions:
\begin{itemize}
 \item We call a Hamiltonian $H$ \emph{admissible}, writing $H\in Ad(V,\Sigma)$, if all its 1-periodic orbits are non-degenerate and if $H$ is linear at infinity, that is if there exists an $R\in\mathbb{R}$ such that on $\Sigma\times[R,\infty)\subset \widehat{V}$ the Hamiltonian is of the form
 \[ H=h(e^r)=\alpha\cdot e^r+\beta\]
 with $\alpha,\beta\in\mathbb{R},\alpha>0$ and $\alpha\not\in Spec(\Sigma,\lambda)$.
 \item We call a homotopy $H_s$ between admissible Hamiltonians $H_\pm$ admissible if there exists an $R\in\mathbb{R}$ such that on $\Sigma\times[R,\infty)$ the homotopy has the form
 \[H_s=h_s(e^r) \qquad \text{ with }\qquad \partial_s\partial_r H_s \leq 0 \qquad\text{ on }\qquad\Sigma\times[R,\infty).\]
 \item We call a Hamiltonian/homotopy $H$ \emph{weakly admissible}, writing $H\in Ad^w(V,\Sigma)$, if there exists an $R$ such that on $\Sigma\times[R,\infty)$ it has the form
 \[H=h\big(e^{r-f(y)}\big)=\alpha\cdot e^{r-f(y)}+\beta \quad\text{ resp. }\quad H_s=h_s(e^{r-f_s(y)})\]
 for a function $f:\Sigma\rightarrow \mathbb{R}$. In the homotopy case we require that
 \[(\partial_s\partial_r h_s)\big(e^{r-f_s(y)}\big)-(\partial_r h_s)\big(e^{r-f_s(y)}\big)\cdot\partial_s f_s(y)\leq 0, \quad\text{with } <0 \text{ on $supp_s\;\partial_s f$}.\]
 If $\partial_r^2 h=0$ (e.g.\ if $h$ is linear), then this is equivalent  to $\partial_s\partial_r H_s\leq 0$.
 \item We call a possibly $s$-dependent almost complex structure $J$ (weakly) admissible, if it is cylindrical and time independent at infinity, that is if 
 \[d\big(e^{r-f_s}\big)\circ J_s = -\lambda \qquad\text{ on }\Sigma\times[R,\infty)\]
 for an $R\in\mathbb{R}$. We may write this shorter as $d(e^{r_s})\circ J=\lambda$ for $r_s:=r-f_s$.
\end{itemize}
Note that in Rabinowitz-Floer homology, we required that $H$ is constant outside a compact set. This implied by a Lemma of McDuff, \cite{Duff} Lem.\ 2.4, that all solutions of the Rabinowitz-Floer equation (\ref{eq3}) stayed in this compact set. The following lemma is a generalization of this crucial fact to Hamiltonians that are linear at infinity.
\begin{lemme}[\textbf{Maximum Principle}]\label{maxprinc}~\\
 Let $H$ be a (weakly) admissible Hamiltonian/homotopy and $J$ an admissible almost complex structure. Let $x_\pm\in P(H_\pm)$, where $H_\pm$ are the ends of the possibly constant homotopy $H_s$. Then there exists a constant $\sigma\leq 1$ such that for $H_{\sigma\cdot s}$ and $J_{\sigma\cdot s}$ any solution 
 \textbf{u} of (\ref{eqast}) with $\displaystyle\lim_{s\rightarrow\pm\infty}u(s)=x_\pm$ satisfies
 \[e^{r}\circ u(s,t)\leq e^{C}\qquad\forall (s,t)\in \mathbb{R}\times S^1\]
 for some constant $C\geq R$ not depending on $u$. If $H$ is a (weak) Hamiltonian or a non-weak homotopy, then we may choose $\sigma=1$, i.e.\ the Maximum principle holds already for $H$ and $J$.
\end{lemme}
\begin{proof} Our proof is a generalization of similar proofs by A.Oancea,\cite{Oan}, and P.Seidel, \cite{Sei}. We give the proof only for homotopies $H_s$, which includes the Hamiltonian case by constant $H_s=H$. Let us consider the function $\rho:\mathbb{R}\times S^1\rightarrow \mathbb{R}$ given by
 \[\rho:=e^{r-f_s}\circ u = e^{r_s}\circ u, \qquad\text{where }r_s:=r-f_s.\]
 To ease the notation, we will drop the index $s$, writing only $H,h,f$ and $J$. Moreover, we write $h'$ instead of $\partial_r h$. However, we keep $r_s$ and we write $u_s, u_t$ for $\partial_s u$ and $\partial_t u$.\medskip\\
 \underline{Calculation of $\Delta \rho $}
 \begin{align*}
  \partial_s \rho =d\big(e^{r_s}\big)(u_s)+\big(\partial_s e^{r-f_s}\big)(u) &=d\big(e^{r_s}\big)\big(-J(u_t-X_H)\big)+e^{r-f_s}(u)\cdot(-\partial_s f)(u)\\
  &=\lambda(u_t)+\lambda(X_H)-\rho\cdot(\partial_s f)(u)\\
  &=\lambda(u_t)-\rho\cdot h'(\rho)-\rho\cdot(\partial_s f)(u),
 \end{align*}
as $\lambda(X_H)=\omega(Y,X_H)=dH(\partial_r)=\partial_r H=\rho\cdot h'(\rho)$. Moreover, we have
\[\partial_t \rho = d\big(e^{r_s}\big)(u_t)=d\big(e^{r_s}\big)(Ju_s-X_H)=-\lambda(u_s),\]
as the orbits of $X_{H_s}$ stay in the level sets of $e^{r_s}$ and hence $d\big(e^{r_s}\big)(X_{H_s})=0$.\pagebreak\\
Therefore, we obtain for the Lapacian of $\rho$
 \begin{align*}
  \Delta \rho &=\partial_s\Big(\lambda(u_t)-\rho\cdot\big[h'(\rho)+(\partial_s f)(u)\big]\Big)-\partial_t\lambda(u_s)\\
  &=d\lambda(u_s,u_t)-\lambda(\,\underbrace{[u_s,u_t]}_{=0}\,)-\partial_s\rho\cdot(\partial_s f)(u)-\Big(-\rho(\partial_s f)(u)+\,\underbrace{d\big(e^{r_s}\big)(u_s)\Big)\cdot h'(\rho)}_{=dH(u_s)=d\lambda(u_s,X_H)}\\
  &\phantom{=\;}-\rho\cdot\Big[(\partial_s h')(\rho)+h''(\rho)\cdot\partial_s\rho+(\partial_s^2 f)(u)+d(\partial_s f)(u_s)\Big]\\
  &=\omega(u_s,\underbrace{u_t-X_H}_{=Ju_s}\,)-\partial_s\rho\cdot\big((\partial_s f)(u)+\rho\cdot h''(\rho)\big)-\rho\cdot d(\partial_s f)(u_s)\\
  &\phantom{=\;}-\rho\cdot\Big[(\partial_s h')(\rho)-h'(\rho)(\partial_s f)(u)+(\partial_s^2 f)(u)\Big]\\
  &=|u_s|^2-\partial_s\rho\cdot\big((\partial_s f)(u)+\rho\cdot h''(\rho)\big)-\rho\cdot d(\partial_s f)(u_s)\\
  &\phantom{=\;}-\rho\cdot\Big[(\partial_s h')(\rho)-h'(\rho)(\partial_s f)(u)+(\partial_s^2 f)(u)\Big].
 \end{align*}
 Abbreviating $g(u):=(\partial_s f)(u)+\rho\cdot h''(\rho)$, we find that this is equivalent to
 \[\Delta\rho+\partial_s\rho\cdot g(u)=|u_s|^2-\rho\cdot d(\partial_s f)(u_s)-\rho\cdot \Big[(\partial_s h')(\rho)-h'(\rho)(\partial_s f)(u)+(\partial_s^2 f)(u)\Big].\tag{$\ast$}\]
 Now if for $C>R$ holds  on $[C,\infty)\times\Sigma$ that the right-hand side of $(\ast)$ is non-negative, then $\rho$ satisfies on $[C,\infty)\times \Sigma$ a maximum principle and cannot have a local maximum at an interior point of $u^{-1}\big([C,\infty)\times \Sigma\big)$. As the asymptotics of $u$ lie outside of $[C,\infty)\times \Sigma$, it follows that $\rho=e^{r-f_s}\circ u\leq e^C$ everywhere.\medskip\\
 \underline{Estimate of $\kappa:=|u_s|^2-\rho\cdot d(\partial_s f)(u_s)$}\smallskip\\
 At first glance, this term might be an unbounded from below. However, as the Liouville form $\lambda=e^r\cdot\lambda_0$ grows exponentially in $r$, we will see that $\kappa$ is in fact bounded by a constant, independent of $u$. Indeed, as $d(\partial_s f)$ is an $r$-invariant 1-form, there exists a vector field $\xi_s$ on $\Sigma$, such that
 \[d(\partial_s f)(\cdot)=d\lambda({\textstyle\frac{1}{e^r}}\xi_s,\cdot)\;\Rightarrow\;\rho\cdot d(\partial_s f)(u_s)=d\lambda(\xi_s,u_s).\]
 For $c:=\sup_s |J\xi_s|$, we find that this last expression is bounded by $c\cdot |u_s|$. It will be usefull to introduce $\sigma$ at this point. Note that if we replace $f_s$ by $f_{\sigma\cdot s}$, then $\kappa$ becomes $|u_s|^2-\sigma\cdot\rho\cdot d(\partial_s f)(u_s)$. Then, we have
 \[\kappa=|u_s|^2-\sigma\cdot\rho\cdot d(\partial_s f)(u_s)\geq |u_s|^2-\sigma\cdot c\cdot |u_s|\geq -{\textstyle\frac{1}{4}}c^2\cdot \sigma^2.\tag{$\ast\ast$}\]
 Here, the last estimate is the minimum of the parabola $x^2-c\sigma x$. Finally note that outside the $s$-support of $\partial_s f$, we have $\kappa =|u_s|^2\geq 0$.\medskip\\
 \underline{Estimate of the whole right-hand side of $(\ast)$}\smallskip\\
 Let us introduce $\sigma$ everywhere in $(\ast)$. Then, we get the   following
 \begin{align*}
  \Delta\rho+\partial_s\rho\; g(u)&\,=\,|u_s|^2-\sigma \rho\; d(\partial_s f)(u_s)-\rho \Big[\sigma(\partial_s h')(\rho)-\sigma\; h'(\rho)(\partial_s f)(u)+\sigma^2(\partial_s^2 f)(u)\Big]\\
  &\overset{(\ast\ast)}{\geq}-\rho \Big[\sigma\big((\partial_s h')(\rho)-h'(\rho)(\partial_s f)(u)\big)+\sigma^2(\partial_s^2 f)(u)\Big]-{\textstyle\frac{1}{4}}c^2\cdot \sigma^2.\tag{$\ast\ast\ast$}
 \end{align*}
For weakly admissable, we assumed that $(\partial_s h')-h'(\rho)(\partial_s f)(u)\leq 0$ with $<0$ on the $s$-support of $\partial_s f$. As this support is bounded, we find for $\sigma$ sufficiently small that the expression in  the brackets is non-positive. Fixing such a $\sigma$, we find that for $\rho>R$ sufficiently large that the right-hand side is in fact non-negative. This proves the lemma.
\end{proof}
\begin{rem}~
\begin{itemize}
 \item By decreasing $\sigma$, we can in fact achieve that $C=R$.
 \item If $H$ is a Hamiltonian or a non-weak homotopy, then the term $(\partial_s^2 f)(u)$ does not exist and there is no need for a reparametrization by $\sigma$, i.e.\ we can choose $\sigma=1$.
\end{itemize}
\end{rem}

\subsection{Symplectic homology}
For a (weakly) admissible Hamiltonian $H$, we define the Floer homology $FH_\ast(H)$ as follows: The chain groups $FC_\ast(H)$ are the $\mathbb{Z}_2$-vector space generated by $\mathcal{P}(H)$. Note that due to $h'\not\in Spec(\Sigma,\alpha)$ and the non-degeneracy of the 1-periodic orbits, we find that $\mathcal{P}(H)$ is in fact a finite set. Thus, $FC_\ast(H)$ is a finite vector space of dimension $|\mathcal{P}(H)|$. For $x_\pm\in\mathcal{P}(H)$ let $\widehat{\mathcal{M}}(x_-,x_+)$ denote the space of solutions $u$ of (\ref{eqast}) with $\displaystyle\lim_{s\rightarrow\pm\infty} u = x_\pm$. There is an $\mathbb{R}$-action on this space given by time shift. The quotient under this action is called the moduli space of $\mathcal{A}^H$-gradient trajectories between $x_-$ and $x_+$ and denoted by $\mathcal{M}(x_-,x_+):=\widehat{\mathcal{M}}(x_-,x_+)/\mathbb{R}$.\\
For a generic $J$, the space $\mathcal{M}(x_-,x_+)$ is a manifold. Its zero-dimensional component $\mathcal{M}^0(x_-,x_+)$ is compact and hence a finite set. Let $\#_2\mathcal{M}^0(x_-,x_+)$ denote its cardinality modulo 2. We define the operator $\partial: FC_\ast(H)\rightarrow FC_\ast(H)$ as the linear extension of
\[\partial x:=\sum_{y\in\mathcal{P}(H)}\#_2\mathcal{M}^0(y,x)\cdot y.\]
A standard argument in Floer theory, involving the compactification of $\mathcal{M}^1(y,x)$, shows that $\partial^2=0$, so that $\partial$ is a boundary operator. We set as usual
\[FH_\ast(H):=\frac{\ker \partial}{\text{im }\partial}.\]
To a (weakly) admissible homotopy $H_s$ between admissible Hamiltonians $H_\pm$ we consider for $x_\pm\in\mathcal{P}(H_\pm)$ the moduli space of s-dependent $\mathcal{A}^{H_s}$-gradient trajectories $\mathcal{M}_s(x_-,x_+)$. Note that we have no time shift on this space, as equation (\ref{eqast}) now depends on $s$. We define the continuation map $\sigma_\ast(H_-,H_+):FC_\ast(H_+)\rightarrow FC_\ast(H_-)$ as the linear extension of 
\[\sigma_\ast(H_-,H_+)x_+=\sum_{x_-\in\mathcal{P}(H_-)}\#_2\mathcal{M}^0_s(x_-,x_+)\cdot x_-.\]
By considering homotopies of homotopies, one sees that $\sigma_\ast(H_-,H_+)$ is independent of the chosen homotopy. By considering the compactification of $\#_2\mathcal{M}^1_s(x_-,x_+)$, we obtain from Floer theory that $\partial\circ \sigma_\ast =\sigma_\ast\circ \partial$, so that $\sigma_\ast(H_-,H_+)$ is a chain map, which descends to a map $\sigma_\ast(H_-,H_+):FH(H_+)\rightarrow FH(H_-)$.\pagebreak\\
For three admissible Hamiltonians $H_1, H_2$ and $H_3$, we have the composition rule
\[\sigma_\ast(H_1,H_3)=\sigma_\ast(H_1,H_2)\circ\sigma_\ast(H_2,H_3).\]
Observe that admissibility of a homotopy $H_s$ between $H_-$ and $H_+$ implies that $H_->H_+$ on $\Sigma\times[R,\infty)$ for $R$ sufficiently large. We introduce a partial ordering $\prec$ on $Ad^w(V,\Sigma)$ by saying $H_+\prec H_-$ if and only if $H_+<H_-$ on $\Sigma\times[R,\infty)$ for a sufficiently large $R\in\mathbb{R}$. This ordering together with the maps $\sigma_\ast(H_-,H_+)$ turn $(FH(H),\sigma)$ into  a direct system over the directed set $(Ad^w(V,\Sigma),\prec)$. The symplectic homology groups $SH_\ast(V)$ are then defined to be the direct limit of this system:
\[SH_\ast(V):=\lim_{\longrightarrow}FH_\ast(H).\]
A cofinal sequence $(H_n)\subset Ad^w(V,\Sigma)$ is a sequence of Hamiltonians such that $H_n\prec H_{n+1}$ and for any $H\in Ad^w(V,\Sigma)$ there exists an $n\in\mathbb{N}$ such that $H\prec H_n$. It follows from Theorem \ref{cofinal} that we have for any cofinal sequence
\[SH_\ast(V)=\lim_{n\rightarrow \infty}FH_\ast(H_n).\]
Finally, for any cofinal sequence there exist sequences $(R_n),(\alpha_n),(\beta_n)\subset\mathbb{R}$  and $(f_n)\subset C^\infty(\Sigma)$ such that $(\alpha_n)$ and $(R_n)$ are monotone increasing and
\[H_n=\alpha_n\cdot e^{r-f_n}+\beta_n\qquad\qquad\text{ on } \Sigma\times[R_n,\infty).\]

\subsection{Truncation}\label{sectrunc}
For a (weakly) admissible Hamiltonian $H$ and $b\in\mathbb{R}$ consider the subchain groups
\[FC_\ast^{<b}(H)\subset FC_\ast(H)\]
which are generated by whose $x\in\mathcal{P}(H)$ with $\mathcal{A}^H(x)<b$. For $a<b$, we set 
\[FC^{[a,b)}_\ast(H):=\raisebox{.2em}{$FC^{<b}_\ast(H)$}\left/\raisebox{-.2em}{$FC^{<a}_\ast(H)$}\right.\]
We call $FC^{[a,b)}_\ast(H)$ truncated chain groups in the action window $[a,b)$. By setting $a=-\infty$, they include the cases $FC_\ast^{[-\infty,b)}(H)=FC_\ast^{<b}(H)$. Analogously one defines
\begin{align*}
 &FC_\ast^{\leq b}(B),\; FC^{>b}_\ast(H):=\raisebox{.2em}{$FC_\ast(H)$}\left/\raisebox{-.2em}{$FC^{\leq b}_\ast(H)$},\right.\; FC^{\geq b}_\ast(H),\\
 &FC^{(a,b]}_\ast(H),\;FC^{(a,b)}_\ast(H)\text{ and }FC^{[a,b]}_\ast(H).
\end{align*}
Note that $FC^{[a,b)}_\ast(H)=FC^{(a,b)}_\ast(H)$ if $a\not\in\mathcal{A}^H(\mathcal{P}(H))$. In the following, we restrict ourself for simplicity to $FC^{(a,b)}_\ast(H)$. However, most of the following results hold also for all other versions of action windows.\\
The following Lemma \ref{monolem} shows that the boundary operator $\partial$ reduces the action. It induces therefore a boundary operator $\partial$ on the truncated chain groups and for this $\partial$ we define
\[FH_\ast^{(a,b)}(H):=\frac{\ker \partial}{\text{im } \partial}.\]
\begin{lemme}\label{monolem}
 If $H$ is a Hamiltonian or a monotone decreasing homotopy and u a solution of (\ref{eqast}) with $\displaystyle\lim_{s\rightarrow\pm\infty}u=x_\pm\in\mathcal{P}(H)$, then $\mathcal{A}^H(x_+)\geq\mathcal{A}^H(x_-)$.
\end{lemme}
\begin{proof}
\begin{align*}
 \mathcal{A}^H(x_+)-\mathcal{A}^H(x_-)&=\int^\infty_{-\infty}\frac{d}{ds}\mathcal{A}^H(u(s))ds\\
 &=\int^\infty_{-\infty} ||\nabla\mathcal{A}^H||^2ds-\int^\infty_{-\infty}\int^1_0 \left(\frac{d}{ds}H\right)(u(s))dt\,ds\geq 0.
\end{align*}
 Note that the second term is zero, if $H$ does not depend on $s$, i.e.\ if $H$ is a Hamiltonian. This shows that the monotone decreasing condition is only needed for homotopies.
\end{proof}
Let $H_-,H_+$ be two (weakly) admissible Hamiltonians such that $H_->H_+$ everywhere. Then we may choose a monotone decreasing (weakly) admissible homotopy $H_s$ between them and it follows from Lemma \ref{monolem} that the associated continuation map $\sigma_\ast(H_-,H_+)$ also decreases action. We obtain hence a well-defined map
\[\sigma_\ast(H_-,H_+): FH_\ast^{(a,b)}(H_+)\rightarrow FH_\ast^{(a,b)}(H_-).\]
The truncated symplectic homology in the action window $(a,b)$ is then defined as the direct limit under these maps:
\[SH_\ast^{(a,b)}(V):=\lim_{\longrightarrow} FH_\ast^{(a,b)}(H).\]
\underline{Attention}: Without further restrictions, we have always
\[SH^{(a,b)}(V)=0,\qquad\qquad\text{ whenever }\quad a>-\infty.\]
To see this, take any cofinal sequence of Hamiltonians $(H_n)$ and take an increasing sequence $(\beta_n)\subset\mathbb{R}$ such that $\displaystyle\beta_n>\max_{x\in\mathcal{P}(H_n)}\mathcal{A}^{H_n}(x)$. Then we find that $K_n:=H_n+\beta_n-a$ yields also a cofinal sequence, but now 
\[\max_{x\in\mathcal{P}(K_n)}\mathcal{A}^{K_n}(x)=\max_{x\in\mathcal{P}(H_n)}\mathcal{A}^{H_n}(x)-\beta_n+a<a\]
such that $FC^{(a,b)}_\ast(K_n)=FH_\ast^{(a,b)}(K_n)=0$ for all $n$ and hence $SH^{(a,b)}(V)=0$. One can overcome this obstacle by restricting further the set of admissible Hamiltonians. For us, it will be enough to require that all Hamiltonians $H$ are smaller then $0$ on a fixed contact hypersurface $M\subset V$. We write $SH^{(a,b)}(V,M)$ for the direct limit of these Hamiltonians.\\
In addition, we remark that for the definition of $FH^{(a,b)}_\ast(H)$ it suffices that only the 1-periodic orbits $x$ of $X_H$ with $\mathcal{A}^H(x)\in(a,b)$ are non-degenerate, as the others are discarded. Therefore, we call a Hamiltonian $H$ admissible for $SH^{(a,b)}_\ast(V,M)$, writing $H\in Ad^{(a,b)}(V,M)$, if it satisfies
\begin{itemize}
 \item $H|_M<0$
 \item $H|_{\Sigma\times[R,\infty)}=h(e^r)$ for $R$ large
 \item all $x\in\mathcal{P}^{(a,b)}(H)=\{x\in\mathcal{P}(H)\;|\; \mathcal{A}^H(x)\in(a,b)\}$ are non-degenerate.\pagebreak
\end{itemize}
The partial ordering on $Ad^{(a,b)}(V,M)$ is given by $H\prec K$ if $H<K$ everywhere. Similar, one defines weakly admissible Hamiltonians. Note that we are free to choose for the computation of $SH^{(a,b)}(V,M)$ cofinal sequences $(H_n)$ which are also admissible for the whole symplectic homology or cofinal sequences, where the 1-periodic orbits of $X_{H_n}$ are only non-degenerate in the action window $(a,b)$.\\
When taking a cofinal sequence $(H_n)\subset Ad(V)$, we find that the projection
\[FC_\ast(H)\rightarrow FC_\ast^{> b}(H)=\raisebox{.2em}{$FC_\ast(H)$}\left/\raisebox{-.2em}{$FC^{\leq b}_\ast(H)$}\right.\]
or the short exact sequence
\[0\rightarrow FC_\ast^{(a,b)}(H)\rightarrow FC_\ast^{(a,c)}(H)\rightarrow FC_\ast^{(b,c)}(H)\rightarrow 0\]
induce in homology the map
\[FH_\ast(H)\rightarrow FH_\ast^{\geq b}(H)\]
respectively the long exact sequence
\[\dots\rightarrow FH_\ast^{(a,b)}(H)\rightarrow FH_\ast^{(a,c)}(H)\rightarrow FH_\ast^{(b,c)}(H)\rightarrow\dots\]
Applying the direct limit then yields the map
\[SH_\ast(V)\rightarrow SH_\ast^{> b}(V,M)\]
and (as $\displaystyle\lim_{\longrightarrow}$ is an exact functor) the long exact sequence
\[\dots\rightarrow SH^{(a,b)}_\ast(V,M)\rightarrow SH^{(a,c)}_\ast(V,M)\rightarrow SH^{(b,c)}_\ast(V,M)\rightarrow\dots\]
\subsection{Symplectic cohomology}
By dualizing the constructions from the previous section, we obtain the symplectic cohomology. Explicitly, we define for a (weakly) admissible Hamiltonian $H$ the cochain groups $FC^\ast(H)$ again as the $\mathbb{Z}_2$-vector space generated by $\mathcal{P}(H)$. The coboundary operator $\delta$ is then defined as the linear extension of
\[\delta x:=\sum_{y\in\mathcal{P}(H)} \#_2 \mathcal{M}^0(x,y)\cdot y.\]
Note that the operator $\delta$ increases action. The analogue construction of chain maps $\sigma^\ast(H_-,H_+)$ associated to an admissible homotopy $H_s$ between Hamiltonians $H_-$ and $H_+$ yields hence a map in the opposite direction (compared to $\sigma_\ast(H_-,H_+)$)
\[\sigma^\ast(H_-,H_+):FH^\ast(H_-)\rightarrow FH^\ast(H_+),\]
where $H_->H_+$ on $\Sigma\times[R,\infty)$ for $R$ sufficiently large. It obeys the composition rule
\[ \sigma^\ast(H_1,H_3)=\sigma^\ast(H_2,H_3)\circ\sigma^\ast(H_1,H_2).\]
By taking the same partial ordering on $Ad^w(V)$ as for homology, we obtain hence an inverse system. The symplectic cohomology $SH^\ast(V)$ is then defined to be the inverse limit of this system
\[SH^\ast(V):=\lim_{\longleftarrow} FH^\ast(H).\]
Again, it can be calculated using cofinal sequences $(H_n)$ of admissible Hamiltonians.\\
For the truncated version of symplectic cohomology, we now have to consider 
\[FC^\ast_{>a}(H)\subset FC^\ast(H)\]
generated by those 1-periodic orbits with action greater then $a$. Then, we define
\[ FC_{(a,b]}^\ast(H):=\raisebox{.2em}{$FC^\ast_{>a}(H)$}\left/\raisebox{-.2em}{$FC_{>b}^\ast(H)$}\right.\]
and all other truncated groups accordingly. As $\delta$ increases action, it is well-defined on the truncated chain groups and yields analogously $FH^\ast_{>a}(H)$ and $FH^\ast_{(a,b)}(H)$ as cohomology groups. Then considering only (globally) monotone decreasing homotopies, the chain maps $\sigma^\ast$ are also well-defined on truncated groups and we obtain as inverse limits
\[SH^\ast_{>a}(V,M)=\lim_{\longleftarrow}FH^\ast_{>a}(H),\qquad SH^\ast_{(a,b)}(V,M)=\lim_{\longleftarrow}FH^\ast_{(a,b)}(H),\]
where we restricted again to $H\in Ad^w(V,M)$.\\
Unlike to the homology case, the long exact sequence
\[\dots\rightarrow FH^\ast_{(b,c)}(H)\rightarrow FH^\ast_{(a,c)}(H)\rightarrow FH^\ast_{(a,b)}(H)\rightarrow\dots\]
induces in general not a long exact sequence in symplectic cohomology. This is due to the fact that, in general, the inverse limit is not an exact functor, but only left exact (see Section \ref{secalg} or \cite{bourbaki2} resp.\ \cite{eilenberg}). However, the inclusion $FC^\ast_{>a}(H)\rightarrow FC^\ast(H)$ still induces a map
\[SH^\ast_{>a}(V,M)\rightarrow SH^\ast(V).\]

\subsection{Transfer morphism and handle attaching}
In the following, we construct a map $\pi_\ast(W,V):SH_\ast(V)\rightarrow SH_\ast(W)$ for an exactly embedded Liouville subdomain $W\subset V$, as first suggested by Viterbo in \cite{Vit}. Analogously, we construct a map $\pi^\ast(W,V):SH^\ast(W)\rightarrow SH^\ast(V)$ in cohomology. Then we will show that $\pi_\ast(W,V)$ and $\pi^\ast(W,V)$ are isomorphisms if $V$ is obtained from $W$ by attaching a subcritical handle $H^{2n}_k$ as described in Section \ref{secsur}.\\
As shown above, we have always maps $SH_\ast(V)\rightarrow  SH_\ast^{>0}(V,\partial W)$ and $SH^\ast_{>0}(V,\partial W)\rightarrow SH^\ast(V)$. The maps $\pi_\ast(W,V)$ and $\pi^\ast(W,V)$ are obtained by showing the identities $SH_\ast^{>0}(V,\partial W)= SH_\ast(W)$ and $SH^\ast_{>0}(V,\partial W)= SH^\ast(W)$. This is done by giving an explicit cofinal sequence $(H_n)\subset Ad(V,\partial W)$.\\
The following proposition is based on ideas by Viterbo, \cite{Vit}. Its proof is taken from McLean, \cite{McLeanDis}. We include it here for completeness and to add a missing argument for the homotopy case. See also Cieliebak, \cite{Cie}, for a slightly different approach.
\begin{prop}\label{proptrans}
 There exists a cofinal sequence $(H_n)\subset Ad(V,\partial W)$ and a sequence of monotone decreasing admissible homotopies $(H_{n,n+1})$ between them such that
 \begin{enumerate}
  \item $K_n:=H_n|_W, \; K_{n,n+1}:=H_{n,n+1}|_W$ are sequences of admissible Hamiltonians / homotopies on $(W,\omega)$.
  \item all 1-periodic orbits of $X_{H_n}$ in $W$ have positive action and all 1-periodic orbits of $X_{H_n}$ in $V\setminus W$ have negative action.
  \item all $\mathcal{A}^H$-gradient trajectories of $H_n$ or $H_{n,n+1}$ connecting 1-periodic orbits in $W$ are entirely contained in $W$.
 \end{enumerate}
\end{prop}
\begin{proof}
 It will be convenient to use $z=e^r$ rather than $r$ for the second coordinate in the completion $(\widehat{W},\widehat{\omega})=\big(W\cup(\partial W\times[1,\infty),\,d(z\alpha)\big)$. Note that we can embed $\widehat{W}$ into $\widehat{V}$ using the flow of $Y_\lambda$, where $\lambda=z\alpha$. The cylindrical end $\partial W\times[1,\infty)$ is then a subset of $\widehat{V}$. The first coordinates will be denoted $z_W$ for $\partial W\times(0,\infty)$ and  $z_V$ for $\partial V\times(0,\infty)$.\medskip\\
 To begin, assume that $Spec(\partial W,\lambda)$ and $Spec(\partial V,\lambda)$ are discrete and let
 \[k:\mathbb{N}\rightarrow\mathbb{R}\setminus \Big(Spec(\partial W,\lambda)\cup 4\cdot Spec(\partial V,\lambda)\Big)\]
 be an increasing function such that $k(n)\rightarrow\infty$. Let $\mu:\mathbb{N}\rightarrow\mathbb{R}$ be defined by
 \[\mu(n)=dist\big(k(n),Spec(\partial W,\lambda)\big)=\min_{a\in Spec(\partial W,\lambda)}|k(n)-a|.\]
 Choose an increasing sequence $A=A(n)$ with
 \[A>\frac{2k}{\mu}>1\quad\text{ and }\quad A(n+1)>2A(n)\]
 which satisfies additionally the conditions $(\oplus)$ and $(\oplus\oplus)$ below. Note that we can always achieve $\frac{2k}{\mu}>1$, as we may choose $k$ arbitrarily large whilst making $\mu$ arbitrarily small. Let also $\veps(n)>0$ be a sequence tending to zero.\\
 We assume that $H_n|_W$ is a $C^2$-small negative Morse function inside $W\setminus\big(\partial W\times[1-\veps,1)\big)$ and for $1-\veps\leq z_W\leq A$ of the form $H_n=g(z)$ with $g(1)=-\veps,\;g'\geq 0$ and $g'\equiv k(n)$ for $1\leq z_W \leq A-\veps$. For $A\leq z_W\leq 2A$ we assume that $H_n\equiv B$ is constant with $B$ being arbitrarily close to $k\cdot(A-1)$.\\
 Now we describe $H_n$ on $\partial V\times[1,\infty)$: We keep $H_n$ constant until we reach $z_V=2A+P$, where $P$ is some constant such that $\{z_W\leq 1\}\subset\{z_V\leq P\}$, which implies $\{z_W\leq 2A\}\subset\{z_V\leq 2A+P\}$. Then let $H_n=f(z_V)$ for $z_V\geq 2A+P$ with $0\leq f'\leq \frac{1}{4}k(n)$ and $f'\equiv\frac{1}{4}k(n)$ for $z_V\geq 2A+P+\veps$. Figure \ref{fig7} gives a schematic illustration of $H_n$.
 \begin{figure}[ht]
\centering
 \resizebox{15cm}{!}{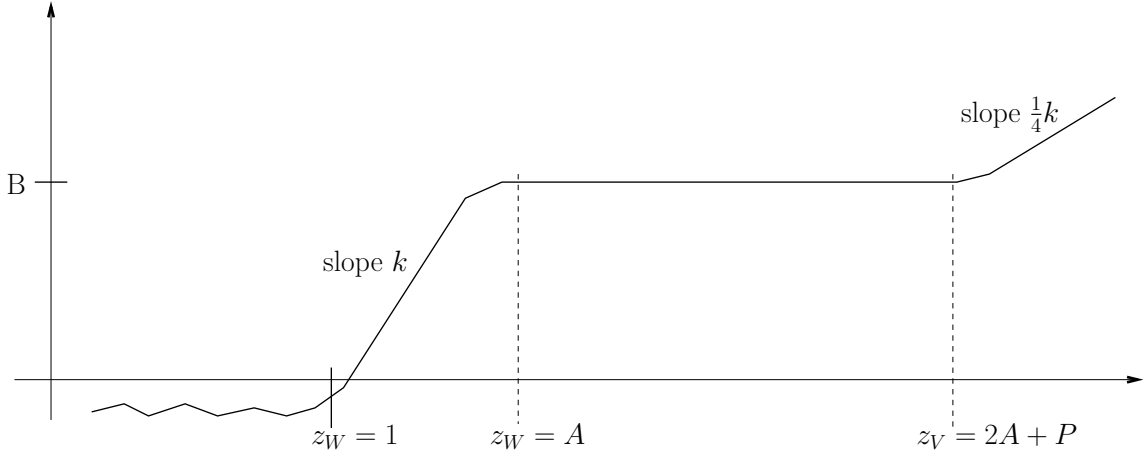}
 \caption{\label{fig7}The Hamiltonian $H_n$}
\end{figure}\pagebreak\\
 As the action of an $X_H$-orbit on level $z=const.$ is $h'(z)\cdot z-h(z)$, we distinguish five types of 1-periodic orbits of $X_H$:
 \begin{itemize}
  \item critical points inside $W$ of action $>0$ (as H is negative and $C^2$-small inside $W$)
  \item non-constant orbits near $z_W=1$ of action $\approx 1\cdot g'(z)>0$
  \item non-constant orbits on $z_W=a$ for $a$ near $A$ of action $\approx g'(a)\cdot a-B<\linebreak(k-\mu)\cdot A-B\approx -\mu\cdot A+k<-k\rightarrow -\infty$
  \item critical points in $A<z_W\,;\,z_V<2A+P$ of action $-B<0$
  \item non-constant orbits on $z_V=a$ for $a$ near $2A+P$ of action $\approx f'(a)\cdot a-B\leq\frac{1}{4}k\cdot(2A+P+\veps)-B\approx-\frac{1}{2}kA+k\cdot(\frac{1}{4}P+\frac{1}{4}\veps+1)<0$ for $A$ sufficiently large (this is condition $(\oplus)$).
 \end{itemize}
Obviously, $(H_n)$ satisfies 1. and 2. of the proposition's claims. It only remains to show that $\mathcal{A}^H$-gradient trajectories connecting two orbits of non-negative action are contained entirely inside $z_W\leq 1$. By Gromov's Monotonicity Lemma (see \cite{Sik}, Prop.\ 4.3.1 and \cite{Oan2}, Lem.\ 1) there exists a $K>0$ such that any $J$-holomorphic curve which intersects $z_W=A$ and $z_W=2A$ has area greater than $KA$. Note that inside $A\leq z_W\leq 2A$ the equation (\ref{eqast}) reduces to an ordinary $J$-holomorphic curve equation, as $X_H\equiv 0$ there. Any $\mathcal{A}^H$-gradient trajectory connecting two orbits of non-negative action which intersects $z_W=A$ and $z_W=2A$ has therefore area greater than $KA$ -- in other words the action difference between its ends is greater than $KA$.\\
For $k(n)$ fixed, the maximal action difference of two 1-periodic orbits in $W$ is bounded from above. So for $A(n)$ sufficiently large (this is condition $(\oplus\oplus)$) no such $\mathcal{A}^H$-gradient trajectory can touch 
$z_W=2A$. It follows 
then from the Maximum Principle that in fact all these $\mathcal{A}^H$-gradient trajectories have to remain inside $z_W\leq 1$.\\
For the construction of the homotopies $H_{n,n+1}$ we have to sharpen this argument. As $A(n+1)>2A(n)$, we can take for $H_{n,n+1}$ the following interpolations:\pagebreak\\
At first, in time $s\in[0,1/2]$, decrease $H_{n+1}$ in the area $z_W\leq 2A(n)$ to $H_n$ and keep it unchanged in $z_W\geq A(n+1)$. Then decrease in time $s\in[1/2,1]$ the remaining part to $H_n$ (see Figures \ref{fig8} and \ref{fig9}).\\
For $s\in[-\infty,1/2]$ the homotopy $H_{n,n+1}$ is then constant $B(n+1)$ in the area $A(n+1)\leq z_W\leq 2A(n+1)$ so that no $\mathcal{A}^H$-gradient trajectory can leave $z_W\leq 1$ in this time interval. For $s\in[1/2,\infty]$ the homotopy $H_{n,n+1}$ is constant $B(n)$ in the area $A(n)\leq z_W\leq 2A(n)$ so that again no $\mathcal{A}^H$-gradient trajectory can leave $z_W\leq 1$ in this time interval.
\begin{figure}[ht]
\centering
 \resizebox{13.4cm}{!}{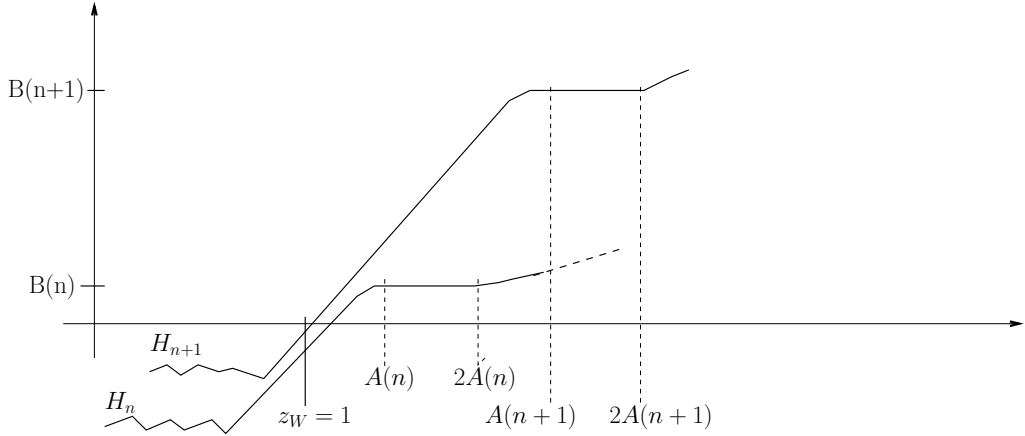}
 \caption{\label{fig8}Two Hamiltonian $H_n$ and $H_{n+1}$}
\end{figure}
\begin{figure}[ht]
\centering
 \resizebox{13.4cm}{!}{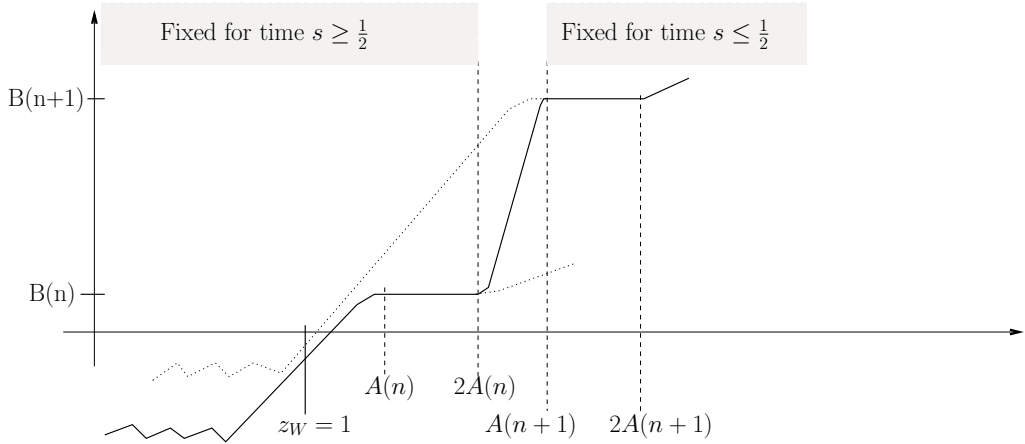}
 \caption{\label{fig9}The homotopy $H_{n,n+1}$ at time $s=\frac{1}{2}$}
\end{figure}
\end{proof}
\begin{cor}\label{transfer} $\quad\displaystyle SH^{>0}_\ast(V,\partial W)\simeq SH_\ast(W)\qquad\text{ and }\qquad SH^\ast_{>0}(V,\partial W)\simeq SH^\ast(W).$
\end{cor}
\begin{proof}
 We only prove the corollary for homology, cohomology being completely analog. Take the sequence of Hamiltonians $(H_n)$ constructed in Proposition \ref{proptrans}. Clearly it is cofinal and $(H_n)\subset Ad^{>0}(V,\partial W)$, as 1-periodic orbits with positive action are either isolated critical points inside $W$ (as $H$ is Morse and $C^2$-small there) or isolated Reeb-orbits near $z_W=1$ -- in both cases non-degenerate. Hence we have
 \[ SH^{>0}_\ast(V,\partial W)=\lim_{\longrightarrow} FH^{>0}_\ast(H_n).\pagebreak\]
 Let $\tilde{H}_n\in Ad(W)$ be the linear extension of $H_n|_W$ with slope $k(n)$. Due to \linebreak $k(n)\not\in Spec(\partial W,\lambda)$, we have obviously $FC^{>0}_\ast(H_n)=FC_\ast(\tilde{H}_n)$. As any $\mathcal{A}^H$-gradient trajectory connecting 1-periodic orbits in $W$ stays in $W$, the two boundary operators $\partial^{H_n}$ and $\partial^{\tilde{H}_n}$ coincide and we have $FH^{>0}_\ast(H_n)=FH_\ast(\tilde{H}_n)$. As the $\mathcal{A}^H$-gradient trajectories for the homotopies $H_{n,n+1}$ stay inside $W$, the continuation maps 
 \[\sigma(H_{n+1},H_n): FH^{>0}_\ast(H_n)\rightarrow FH^{>0}_\ast(H_{n+1})\] coincide with the continuation maps 
 \[\sigma(\tilde{H}_{n+1},\tilde{H}_n): FH_\ast(\tilde{H}_n)\rightarrow FH_\ast(\tilde{H}_{n+1}).\]
 Hence we have $\quad\displaystyle SH^{>0}_\ast(V,\partial W)= \lim_{\longrightarrow} FH^{>0}_\ast(H_n) = \lim_{\longrightarrow} FH_\ast(\tilde{H}_n)=SH_\ast(W)$.
\end{proof}
We finish this section with the following Invariance Theorem due to Cieliebak, \cite{Cie}.
\begin{theo}[Invariance of $SH$ under subcritical surgery]\label{theoinvsur}~\\
 Let $W$ be an exact Liouville domain and let $V$ be obtained from $W$  by attaching to $\partial W\times[0,1]$ a subcritical exact symplectic handle $H^{2n}_k,\,k<n$, as described in Section \ref{secsur}. Moreover, assume that the Conley-Zehnder index is well-defined on $W$. Then it holds that \[SH_\ast(V)\cong SH_\ast(W)\qquad\text{ and }\qquad SH^\ast(V)\cong SH^\ast(W).\]
\end{theo}
\begin{proof}
The idea of the proof is to construct yet another cofinal sequence of Hamiltonians $(H_n)\subset Ad^w(V)\cap Ad(V,\partial W)$ for which we can directly show that 
\begin{align*}
 SH_\ast(W)&\overset{(\ast)}{\cong} SH^{>0}_\ast(V,\partial W)&\overset{(1)}{=}\lim_{n\rightarrow\infty} FH^{>0}_\ast(H_n)\overset{(2)}{\simeq}\lim_{n\rightarrow\infty} FH_\ast(H_n)\overset{(3)}{=}SH_\ast(V)\\
 SH^\ast(W)&\overset{(\ast)}{\cong}SH_{>0}^\ast(V,\partial W)&\overset{(4)}{=}\lim_{n\rightarrow\infty} FH_{>0}^\ast(H_n)\overset{(5)}{\simeq}\lim_{n\rightarrow\infty} FH^\ast(H_n)\overset{(6)}{=}SH^\ast(V).
\end{align*}
Note that the isomorphisms $(\ast)$ have been shown in Corollary \ref{transfer}.\medskip\\
To start, fix sequences $k(n)\not\in Spec(\partial W,\lambda),\, k(n)\rightarrow\infty$ and $\veps(n)\rightarrow 0$. Then choose an increasing sequence of non-degenerate Hamiltonians $H_n$ on $W$ that is on $\partial W\times (-\veps(n),0]$ of the form
\[H_n|_{\partial W\times(-\veps(n),0]} = k(n)\cdot e^r-\big(1+\veps(n)\big)\]
and extend $H_n$ over the handle by a function $\psi$ with $\alpha=k(n)$ and $\beta=-1-\veps(n)$ as described in Section \ref{secsur}.\\
For each $n$ choose the handle so thin such that each trajectory of $X_{H_n}$ which leaves and reenters the handle has length greater than 1. Thus we obtain a cofinal weakly admissible sequence $(H_n)$, whose 1-periodic orbits having positive action are all contained in $W$. This shows already the identities $(1),(3),(4)$ and $(6)$.
Now recall that we had the long exact sequences
\begin{align*}
 \dots\rightarrow FH^{> 0}_{j+1}(H_n)\rightarrow FH^{\leq0}_j(H_n)\rightarrow FH_j(H_n)\rightarrow FH_j^{> 0}(H_n)\rightarrow\dots\\
 \dots\rightarrow FH_{\leq 0}^{j-1}(H_n)\rightarrow FH_{>0}^j(H_n)\rightarrow FH^j(H_n)\rightarrow FH^j_{\leq 0}(H_n)\rightarrow\dots
\end{align*}
Note that $FH^{>0}_j(H_n)$ is generated by all 1-periodic orbits of $H_n$ inside $W$, while \linebreak$FH^{\leq 0}_j(H_n)$ is generated by all other orbits. These are finitely many, lying on the handle, and are explicitly given in (\ref{eqA}). Observe that $H_n$ is on the handle time-independent. The orbits there are therefore of Morse-Bott type. We can now use either the definition of $SH$ with Morse-Bott techniques, as described in \cite{BourOan1}, or perturb $H_n$ near these orbits to make it non-degenerate, as described in \cite{CiFlHoWy}. In both cases we obtain for each orbit $\gamma$ two generators in the chain complex whose indices are $\mu_{CZ}(\gamma)$ and $\mu_{CZ}(\gamma)+1$. We have shown in Section \ref{sec5.4} that the possible values of $\mu_{CZ}(\gamma)$ increase to $\infty$ as the slope $\alpha=k(n)$ tends to $\infty$. Therefore, $FH^{\leq0}_j(H_n)$ becomes eventually zero for $n$ large enough, as well as $FH^{\leq0}_{j+1}(H_n)$. This implies for $n$ large enough that
\[FH_j(H_n)\rightarrow FH^{>0}_j(H_n)\]
is an isomorphism. As the direct limit is an exact functor, these maps converge to an isomorphism in the limit, proving (2). In the cohomology case, the line of arguments is the same. Even though taking inverse limits is not exact, it still takes the isomorphism
\[FH^j_{>0}(H_n)\rightarrow FH^j(H_n)\]
to an isomorphism in the limit, as it is left exact (see Theorem \ref{theoindilimits}). This proves (5).
\end{proof}

\subsection{Rabinowitz-Floer and symplectic (co)homology}\label{sec6.8}
In \cite{FraCieOan}, Cieliebak, Frauenfelder and Oancea showed that Rabinowitz-Floer homology and symplectic (co)homology are closely related. More precisely, they showed, under the assumption that all homologies are $\mathbb{Z}$-graded by the Conley-Zehnder index, that there exists the following exact sequence:
\begin{equation}\label{longexSHRFH}
 ...\rightarrow SH^{-\ast}(V)\rightarrow SH_\ast(V)\rightarrow RFH_\ast(V)\rightarrow SH^{-(\ast-1)}(V)\rightarrow SH_{\ast-1}\rightarrow...
\end{equation}
where the map $SH^{-\ast}(V)\rightarrow SH_\ast(V)$ fits into the commutative diagram
\[\begin{xy}\xymatrix{SH^{-\ast}(V) \ar[d]_{c^\ast} \ar[r] & SH_\ast(V)\\ H^{-\ast+n}(V,\partial V)\ar[r] & H_{\ast+n}(V,\partial V).\ar[u]_{c_\ast}}\end{xy}\]
Here, the bottom arrow is the composition of the map induced by the inclusion \linebreak $i:V\hookrightarrow(V,\partial V)$ together with the Poincaré duality isomorphism
\[H^{-\ast+n}(V,\partial V)\overset{PD}{\longrightarrow}H_{\ast+n}(V)\overset{i_\ast}{\longrightarrow}H_{\ast+n}(V,\partial V).\]
Moreover, there are the following commutative diagrams of long exact sequences, where $PD$ is the Poincaré duality and the top sequence is the (co)homological long exact sequence of the pair $(V,\partial V)$:
\begin{equation}\label{longexShRFHhom}
 \begin{xy}\xymatrix{
\ar[r] &  H_{\ast+n}(V)\ar@{=}[d]^{PD} \ar[r] & H_{\ast+n}(V,\partial V)\ar[d] \ar[r] & H_{\ast-1+n}(\partial V)\ar[d] \ar[r] & H_{\ast-1+n}(V) \ar@{=}[d]^{PD} \ar[r]&\\
\ar[r] & H^{-\ast+n}(V,\partial V)\ar[r] & SH_\ast(V) \ar[r] & RFH^{\geq 0}_\ast (V) \ar[r] & H^{-(\ast-1)+n}(V,\partial V) \ar[r] &
}\end{xy}
\end{equation}
and
\begin{equation}\label{longexShRFHcohom}
 \begin{xy}\xymatrix{
\ar[r] &  H^{-\ast+n}(V,\partial V) \ar[r] & H^{-\ast+n}(V)\ar@{=}[d]^{PD} \ar[r] & H^{-\ast+n}(\partial V) \ar[r] & H^{-(\ast-1)+n}(V,\partial V)  \ar[r]&\\
\ar[r] & SH^{-\ast}(V) \ar[u] \ar[r] & H_{\ast+n}(V,\partial V) \ar[r] & RFH^{\leq 0}_\ast (V) \ar[u] \ar[r] & SH^{-(\ast-1)}(V) \ar[u] \ar[r] &.
}\end{xy}
\end{equation}
In the sequence (\ref{longexSHRFH}), we find in particular for $\ast\geq n$ or $\ast\leq -n$ that $SH^{-\ast}(V)\rightarrow SH_\ast(V)$ is zero. This implies for field coefficients (for example $\mathbb{Z}_2$) and $\ast\in\mathbb{Z}\setminus[-n+1, n]$ that
\begin{equation}\label{RFHdirectsum}
 RFH_\ast(V)\cong SH_\ast(V)\oplus SH^{-\ast+1}(V).
\end{equation}
The sequences (\ref{longexShRFHhom}) and (\ref{longexShRFHcohom}) on the other hand imply for $|\ast|\geq n+1$ that
\[SH_\ast(V)\cong RFH^{\geq 0}_\ast(V)\qquad\text{ and }\qquad SH^{-(\ast-1)}(V)\cong RFH^{\leq 0}_\ast(V).\]
In view of the Invariance Theorem \ref{theoinvsur} for symplectic (co)homology, we obtain from (\ref{RFHdirectsum}) the Invariance Theorem for Rabinowitz-Floer homology.
\begin{theo}[Invariance of $RFH$ under subcritical surgery]\label{theoRFHinvsur}~\\
 Let $W$ be a Liouville domain and let $V$ be obtained from $W$  by attaching to $\partial W\times[0,1]$ a subcritical exact symplectic handle $H^{2n}_k,\,k<n$, as described in Section \ref{secsur}. Moreover, assume that the Conley-Zehnder index is well-defined on $W$. Then it holds for field coefficients that \[RFH_\ast(V)\cong RFH_\ast(W),\]
 at least for $\ast\in\mathbb{Z}\setminus[-n, n+1]$, i.e.\ away from the singular homology of $(W,\partial W)$.
\end{theo}
\begin{rem}~
 \begin{itemize}
  \item The restriction to $\ast\in\mathbb{Z}\setminus[-n+1, n]$ is technical. It is not clear at the moment what happens for $\ast\in[-n+1,n]$ but it is conjectured that $RFH$ is invariant there as well, just like $SH^\ast$ and $SH_\ast$.
  \item The use of field coefficients is also technical. However, we cannot drop this assumption, as our proof relies on the direct sum decomposition (\ref{RFHdirectsum}), which itself depends on a splitting of the exact sequence
  \[0\rightarrow SH_\ast(V)\rightarrow RFH(V)_\ast\rightarrow SH^{-\ast+1}(V)\rightarrow 0.\]
  Such a splitting exists in general only for field (or semi-simple) coefficients.
  \item In \cite{CieOan}, Rem. 9.15, Cieliebak and Oancea recently proved Theorem \ref{theoRFHinvsur} directly in the Rabinowitz-Floer setting resp.  the isomorphic setting of symplectic homology of trivial cobordisms. This approach avoids the two problems mentioned above.
 \end{itemize}
\end{rem}
\newpage\phantom{.}
\newpage

\section{Brieskorn manifolds and exotic contact structures}
In this section, we prove our Main Theorem \ref{maintheo1} and show some consequences. First, we introduce the Brieskorn manifolds $\Sigma_a$ with their canonical contact structure. We give explicit exact contact fillings and calculate for some $\Sigma_a$ their Rabinowitz-Floer homology.\\
Among these manifolds, there are many that are homeomorphic/diffeomorphic to the standard sphere. We use these to construct new contact structures on manifolds which support fillable contact structures.

\subsection{Brieskorn manifolds}\label{sec7.2}
In this subsection, we recall the construction of Brieskorn manifolds and their contact structures and fillings. It is a shortened version of similar sections in \cite{Fauck1} and my diploma thesis. We include it here for completeness and readability.\bigskip\\
Let $a = (a_0,a_1, ...\,, a_n)$ be a vector of natural numbers with $a_i\geq 2$ and define a complex polynomial $f\in C^\infty(\mathbb{C}^{n+1})$ by
\[f(z) = z_0^{a_0}+ z_1^{a_1} + ... + z_n^{a_n}.\]
The next lemma shows that its level sets $V_a(t):=f^{-1}(t)$ are smooth complex hypersurfaces except for $V_a(0)$, which has an isolated singularity at zero. The links of this singularity $\Sigma_a:=V_a(0)\cap S^{2n+1}$ are the \textit{\textbf{Brieskorn manifolds}}.
\begin{lemme}[cf. \cite{Hirz} or Fauck, diploma thesis]~\\
 The sets $\Sigma_a$ and $V_a(t),\, t\neq 0,$ are smooth manifolds.
\end{lemme}
\begin{proof}
 We set $\rho(z):=||z||^2=\sum z_k\bar{z}_k$ and consider the maps 
 \[ f:\mathbb{C}^{n+1}\rightarrow \mathbb{C}\qquad\text{ and }\qquad (f,\rho):\mathbb{C}^{n+1}\rightarrow \mathbb{C}\times\mathbb{R}.\]
 As $V_a(t)=f^{-1}(t)$ and $\Sigma_a=(f,\rho)^{-1}(0,1)$, it suffices to show that $t$ resp.\ $(0,1)$ are regular values. Using Wirtinger calculus, we obtain that
 \[D(f,\rho)=\begin{pmatrix}a_0 z_0^{a_0-1} & \dots & a_n z_n^{a_n-1} & 0 & \dots & 0\\
              0 & \dots & 0 & a_0\bar{z}_0^{a_0-1} & \dots & a_n\bar{z}_n^{a_n-1}\\
              \bar{z}_0 & \dots & \bar{z}_n & z_0 & \dots & z_n
             \end{pmatrix}.\]
 For $z\neq 0$, we find that the first two rows of this matrix are linear independent, which shows that $t\neq 0$ is a regular value of $f$. If $D(f,\rho)$ has not rank 3 and $z\neq 0$, then there exists $\lambda\neq 0$ such that $\bar{z}_k = \lambda a_k z_k^{a_k-1}$ for all $k$. Then, we find
 \[0<\sum_{k=0}^n \frac{z_k \bar{z}_k}{a_k} = \lambda \sum_{k=0}^n z_k^{a_k} = \lambda\cdot f(z),\]
 which is impossible for $z\in \Sigma_a \subset f^{-1}(0)$.
\end{proof}
Let us consider on $\mathbb{C}^{n+1}$ the following $a$-weighted Hermitian form given by
\[\langle\xi,\zeta\rangle_a:=\frac{1}{2}\sum_{k=0}^n a_k\xi_k\bar{\zeta}_k.\]
It defines an $a$-weighted symplectic 2-form $\omega_a=-\mathfrak{Im}\langle\cdot,\cdot\rangle_a$, explicitly given by
\[\omega_a:=\frac{i}{4}\sum^n_{k=0} a_kdz_k\wedge d\bar{z}_k.\]
Note that $Y_\lambda(z):=z/2$ is a Liouville vector field for $\omega_a$, whose Liouville 1-form is
\[\lambda_a :=\omega_a(Y_\lambda,\cdot)=\frac{i}{8} \sum^n_{k=0} a_k (z_kd\bar{z}_k - \bar{z}_kdz_k).\]
 \begin{prop}[Lutz \& Meckert, \cite{LuMe}] \label{prop15}~\\
  The restriction $\alpha_a:=\lambda_a|_\Sigma$ is a contact form on $\Sigma_a$ with Reeb vector field $R_a$ given by
  \[R_a = 4i\left( \frac{z_0}{a_0},\,\dots\,,\frac{z_n}{a_n}\right).\]
 \end{prop}
\begin{proof}
 The gradient $\nabla_a f$ of $f$ with respect to $\langle\cdot,\cdot\rangle_a$ if given by
 \[\nabla_a f := 2 \left(\bar{z}_0^{a_0-1},\,...\,,\bar{z}_n^{a_n-1}\right).\]
 The Liouville vector field $Y_V$ of the restricted 1-form $\lambda_a|_{V_a(0)}$ with respect to the restricted symplectic form $\omega_a|_{V_a(0)}$ is given by
 \[Y_V:=Y_\lambda - \frac{\langle \nabla_a f, Y_\lambda\rangle_a}{||\nabla_a f||_a^2}\cdot\nabla_a f.\]
 Indeed, $TV_a(t)=\ker df=\ker \langle \nabla_a f,\cdot\rangle_a$, which shows that $Y_V\in TV_a(0)$. Moreover, we calculate for any $\xi\in TV_a(0)$
 \[
  \omega_a(Y_V,\xi) = \omega_a(Y_\lambda,\xi)-\frac{\langle \nabla_a f, Y_\lambda\rangle_a}{||\nabla_a f||_a^2}\omega_a(\nabla_a f,\xi) = \lambda_a(\xi)+\frac{\langle \nabla_a f, Y_\lambda\rangle_a}{||\nabla_a f||_a^2}\underbrace{\mathfrak{Im}\langle\nabla_a f,\xi\rangle_a}_{=0}
  = \lambda_a(\xi).\]
This shows that $Y_V$ is the Liouville vector field for the pair $(\omega_a|_{V_a(0)},\lambda_a|_{V_a(0)})$. Now, note that $d\rho=\sum(\bar{z}_kd z_k+z_kd\bar{z}_z)$ and calculate
\[d\rho(Y_V)=\sum\frac{\bar{z}_k z_k}{2}-\frac{\langle \nabla_a f, Y_\lambda\rangle_a}{||\nabla_a f||_a^2}\sum 2\bar{z}_k\cdot \bar{z}_k^{a_k-1} = \frac{\rho(z)}{2}-\frac{\langle \nabla_a f, Y_\lambda\rangle_a}{||\nabla_a f||_a^2}\cdot 2\overline{f(z)}=\frac{1}{2}>0\]
as $\rho(z)=1$ and $f(z)=0$ for $z\in\Sigma_a$. It follows that $Y_V$ points out of the unit sphere $S^{2n+1}=\rho^{-1}(0)$ and hence out of $\Sigma_a$ in $V_a(0)$. We obtain from Lemma \ref{Lemma1} that $\Sigma_a$ is therefore a contact hypersurface in $V_a(0)$.\pagebreak\\
It remains to check that $R_a$ is the Reeb vector field of $\alpha_a$. We have for $z\in\Sigma_a$
\begin{align*}
 &&\left.\begin{aligned}\langle R_a, \nabla_a f\rangle_a &=\;& 4i\sum z_k^{a_k}\;=\; 4if(z)\; &= 0\\ d\rho(R_a)&=\;&\sum\bar{z}_k 4i z_k + z_k(-4i)\bar{z}_k&=0 \end{aligned}\right\}&\Rightarrow R_a(z)\in T_z\Sigma_a\\
 &\text{and}& \alpha_a(R_a) = \lambda_a(R_a) = \frac{i}{8} \sum a_k \left(z_k\frac{-4i}{a_k}\,\overline{z}_k-\overline{z}_k\frac{4i}{a_k}z_k\right)&= \frac{-i^2 8}{8}\rho(z) = 1.
\end{align*}
Finally, we calculate for $\xi\in T_z\Sigma_a$ that
\begin{align*}
 d\alpha_a(R_a,\xi) = \omega_a(R_a,\xi) = \frac{i}{4}\sum a_k\left(4i\frac{z_k}{a_k}\bar{\xi}_k-(-4i)\frac{\bar{z}_k}{a_k}\xi_k\right) &= -\sum\left(z_k\bar{\xi}_k+\bar{z}_k\xi_k\right)\\
 &=-d\rho(\xi) = 0.
\end{align*}
Hence, $R_a$ is the Reeb vector field of $\alpha_a$, as $\alpha_a(R_a)=1$ and $\iota(R_a)d\alpha_a = 0$.
\end{proof}
\begin{cor}\label{lem18a}
 The symplectic complement $\xi^\perp_a$ with respect to $\omega_a$ of the contact structure $\xi_a:=\ker \alpha_a$ is symplectically trivialized by the following 4 vector fields:
 \[ X_1 := \frac{\nabla_a f}{||\nabla_a f||_a},\quad Y_1 := i\cdot X_1,\quad X_2 := Y_V= Y_\lambda - \frac{\langle \nabla_a f, Y_\lambda\rangle_a}{||\nabla_a f||_a^2}\cdot\nabla_a f,\,\text{ and }\, Y_2 := R_a.\]
Explicitly, $X_1$ and $X_2$ are given by
\begin{align*}
 X_1 &= \sqrt{\frac{2}{\sum a_k|z_k|^{2(a_k-1)}}}\cdot \Big( \bar{z}_0^{a_0-1},\dots,\bar{z}_n^{a_n-1}\Big),\\
 X_2 &= \frac{1}{2}\cdot\left(z_0-\frac{\sum a_kz_k^{a_k}}{\sum a_k|z_k|^{2(a_k-1)}}\cdot \bar{z}_0^{a_0-1} , \dots,z_n-\frac{\sum a_kz_k^{a_k}}{\sum a_k|z_k|^{2(a_k-1)}}\cdot \bar{z}_n^{a_n-1}\right).
\end{align*}
\end{cor}
\begin{proof}
The explicit descriptions of $X_1, X_2$ are obtained by easy calculations from the definition of $\nabla_a f$ and $\langle\cdot,\cdot\rangle_a$. Note that $X_1, Y_1$ generate the complex complement of $TV_a(0)$ while $X_2, Y_2$ generate the symplectic complement of $\xi_a$ in $TV_a(0)$. This shows that 
\[\omega_a(X_1,X_2)=\omega_a(X_1,Y_2)=\omega_a(Y_1,X_2)=\omega_a(Y_1,Y_2)=0.\]
The norming guarantees $\omega_a(X_1,Y_1)=1$, while $\omega_a(X_2,Y_2)=1$ follows from the proof of Proposition \ref{prop15}.
\end{proof}
To define the Rabinowitz-Floer homology on $\Sigma_a$, we need a Liouville domain $(W,\lambda)$ with boundary $(\Sigma_a,\alpha_a)$. Unfortunately, we cannot take $V_a(0)\cap B_1(0)$, as it has a singularity at 0. We overcome this obstacle by constructing an interpolation between $V_a(0)$ and $V_a(\veps)$ for $\veps>0$ small. To do this, choose a smooth monotone decreasing cut-off function $\beta\in C^\infty (\mathbb{R})$ with $\beta(x)= 1$ for $x\leq 1/4$ and  $\beta(x)=0$ for $x\geq 3/4$. Then define
\begin{align*}
  V_\veps&:=\Big\{z\in \mathbb{C}^{n+1}\,\Big|\, z_0^{a_0}+z_1^{a_1}+...+z_n^{a_n} = \veps\cdot\beta\big(||z||^2\big)\Big\}\\
  \text{and }\quad W_\veps&:=V_\veps\cap B_1(0).
\end{align*}
\begin{prop}
 $(V_\veps,\lambda)$ is a Liouville domain with boundary $(\Sigma_a,\alpha_a)$ and vanishing first Chern class $c_1(TV)$, if $\veps$ is small enough.
\end{prop}
\begin{proof}~
\begin{itemize}
 \item 
 Consider on $\mathbb{C}^{n+1}$ the smooth complex valued function $f_\veps(z) := f(z)-\veps\cdot\beta\big(||z||^2\big)$. Its differential is given by
 \[D f_\veps = D f -\veps\cdot\beta'\big(||z||^2\big)\cdot D\big(||z||^2\big).\]
 As $Df$ is non-degenerate for $z\neq 0$, it follows that 0 is a regular value of $f_\veps$ for $\veps$ small enough, as $0\not\in f_\veps^{-1}(0)$ and $\beta'\big(||z||^2\big)\neq 0$ only for $1/4\leq ||z||^2\leq 3/4$. We find hence that $V_\veps=f_\veps^{-1}(0)$ is a smooth manifold.
 \item Next, we note that the two functions $\mathfrak{Re}(f_\veps)$ and $\mathfrak{Im}(f_\veps)$ have with respect to $\omega_a$ the following Hamiltonian vector fields
 \[X_{\mathfrak{Re}f_\veps} = X_{\mathfrak{Re}f} + \veps\cdot X_{\beta(||z||^2)},\qquad X_{\mathfrak{Im}f_\veps} = X_{\mathfrak{Im} f},\]
 as $\beta\big(||z||^2\big)$ is a real valued function. Moreover, it follows from the fact
 \[T_z V_\veps = \big\{ \xi\in T\mathbb{R}^{2n+2}\,\big|\,\omega_a(\xi,X_{\mathfrak{Re} f_\veps})=0 \text{ and }\omega_a(\xi,X_{\mathfrak{Im} f_\veps})=0\big\}\]
 that $span\left\{X_{\mathfrak{Re} f_\veps}(z),X_{\mathfrak{Im} f_\veps}(z)\right\}$ is the $\omega_a$-symplectic complement of $T_zV_\veps$. As $X_{\mathfrak{Re}f}$ and $X_{\mathfrak{Im}f}$ span the complex complement of $T_zV$, we find that $\omega_a(X_{\mathfrak{Re}f},X_{\mathfrak{Im}f})\neq 0$ and therefore that for $\veps$ small enough holds
 \[\omega_a(X_{\mathfrak{Re}f_\veps},X_{\mathfrak{Im}f_\veps})=\omega_a(X_{\mathfrak{Re}f},X_{\mathfrak{Im}f})+\veps\cdot\omega_a(X_{\beta(||z||^2)},X_{\mathfrak{Im}f})\neq 0.\]
 This implies that $span\left\{X_{\mathfrak{Re} f_\veps}(z),X_{\mathfrak{Im} f_\veps}(z)\right\}$ is a symplectic subspace of $T\mathbb{R}^{2n+2}$ and hence that its symplectic complement $TV_\veps$ is also a symplectic subspace, in other words $\omega_a|_{TV_\veps}$ is non-degenerate. As $\omega_a|_{TV_\veps}=d\lambda_a|_{TV_\veps}$, we know that $(V_\veps,\lambda_a|_{TV_\veps})$ is an exact symplectic manifold and that $(W_\veps,\lambda_a|_{TV_\veps})$ is
 a Liouville domain.
 \item Finally, $V_\veps$ is diffeomorphic to $V_a(\veps)$ as both are diffeomorphic to the set
 \[V_a(\veps)\cap B_{1/2} = \left\{z\in V_a(\veps)\;:\;||z||<1/2\right\} = V_\veps\cap B_{1/2}.\]
 To see this, consider on $V_\veps$ and $V_a(\veps)$ the function $\rho(z)= ||z||^2$, whose critical points lie for $V_\veps$ and $V_a(\veps)$ in $V_\veps\cap B_{1/2}$ if $\veps$ is small enough. It follows from classical Morse-theory that $V_\veps$ and $V_a(\veps)$ are hence diffeomorphic to $V_\veps\cap B_{1/2}$. Since $V_a(\veps)$ is parallelizable, $V_\veps$ is parallelizable as well and hence $c_1(TV)=0$ (for more details see \cite{Hirz}, § 14). \qedhere
\end{itemize}
\end{proof}
\begin{dis}\label{dissym}
 On $\Sigma_a$, we have symplectic symmetries $\sigma$ of the form
 \[\sigma(z)=\big(c_0\cdot z_0,\,\dots\,,c_n\cdot z_n\big),\qquad c_k\in\sqrt[\raisebox{0.2em}{$\scriptstyle a_k$}]{1}\subset\mathbb{C},\]
 where the $c_k$ are complex $a_k$-roots of 1. These symmetries extend to $V_\veps$, as
 \[0=\sum z_k^{a_k} - \phi\big(||z||^2\big) = \sum \big(c_k z_k\big)^{a_k} - \phi\left(\sum||c_k z_k||^2\right).\]
\end{dis}
The Reeb vector field $R_a=4i\big(z_0/a_0,...\,,z_n/a_n\big)$ generates the following flow:
\begin{equation} \label{eq6}
 \varphi^t_a(z)=\left(e^{4it/a_0}\cdot z_0,\,\dots\,,e^{4it/a_n}\cdot z_n\right).
\end{equation}
The sets of closed Reeb orbits of periodic $\eta$ are found by critical inspection as
\begin{equation*}
 \mathcal{N}^\eta =\left\{z\in\Sigma_a\,\left|\, z_k=0 \quad\text{ if }\quad \frac{4\eta}{a_k}\not\in2\pi\mathbb{Z}\right.\right\}.
\end{equation*}
Note that $\mathcal{N}^\eta=\emptyset$, if $4\eta/a_k\in 2\pi\mathbb{Z}$ does not hold for at least two different $k$, as vectors $z\in\Sigma_a$ have at least two non-zero entries. To ease the notation, we define the integer $L:=2\eta/\pi$, so that 
\begin{equation} \label{eq6a}
 \mathcal{N}^\eta =\mathcal{N}^{L\pi/2}=\left\{z\in\Sigma_a\,\left|\, z_k=0 \quad\text{ if }\quad \frac{L}{a_k}\not\in\mathbb{Z}\right.\right\}.
\end{equation}
Note that $\mathcal{N}^{L\pi/2}$ is the intersection $\Sigma_a\cap E(a,L)$ of $\Sigma_a$ with the complex linear subspace $E(a,L)\subset\mathbb{C}^{n+1}$ given by 
\[E(a,L):=\left\{z\in\mathbb{C}^{n+1}\,\left|\Big.\,z_k=0\quad\text{ if }\quad L/a_k\,\not\in\mathbb{Z}\right.\right\}.\]
 Its complex dimension is given by
 \[\dim_{\mathbb{C}} E(a,L) = n(a,L):= \#\left\{\,k\,\left|\Big.\,0\leq k\leq n\text{ and }L/a_k\,\in\mathbb{Z}\right.\right\}.\]
We find that $\mathcal{N}^{L\pi/2}$ is therefore isomorphic to the Brieskorn manifold $\Sigma_{a(L)}$ in \linebreak $\mathbb{C}^{n(a,L)}\cong E(a,L)$ where
\[a(L)=\big(a_{k_1},\dots,a_{k_{n(a,L)}}\big)\subset\big(a_0,\dots,a_n\big)=a\]
is the subvector of $a$ defined by $a_{k_i}\in a(L)$ if and only if $L/a_{k_i}\in\mathbb{Z}$. The 1-form $\alpha_a|_{T\mathcal{N}^{L\pi/2}}$ is hence isomorphic to the contact form $\alpha_{a(L)}$.\\
The differential of $\varphi_a$ at time $L\pi/2$ is given by
\[ D\varphi_a^{L\pi/2} = diag \left(e^{2\pi iL/a_0},\dots, e^{2\pi iL/a_n}\right).\]
It follows that
\[\ker\big(D_z\varphi^{L\pi/2}\big|_{T_z\Sigma_a}-Id\big) = T_z\Sigma_a\cap E(a,L) = T_z\mathcal{N}^{L\pi/2}.\]
We have thus proven the following proposition.
\begin{prop} \label{prop25} All sets $\mathcal{N}^{L\pi/2}$ of closed Reeb orbits on $(\Sigma_a,\alpha_a)$ satisfy (MB).
\end{prop}
Next, we give some topological facts about Brieskorn manifolds, as shown by Egbert Brieskorn in \cite{Brie}. To give them as precisely as possible, we introduce for every tuple $a=(a_0,...\,,a_n)$ the following graph $G_a$:
\begin{itemize}
 \item $G_a$ has $n+1$ vertices labeled $a_0,...\,,a_n$.
 \item $G_a$ contains an edge between $a_j,a_k$ if and only if $gcd(a_j,a_k)>1$.\pagebreak
\end{itemize}
Let $K$ be the connected subgraph of $G$ that consists of all even $a_k$. We say that $G_a$ satisfies condition (O) if 
\[|K| \text{\textit{ is odd and for all }}a_j,a_k\in K \text{\textit{ holds }} gcd(a_j,a_k)=2.\tag{O}\]
\begin{theo}[Brieskorn, cf. \cite{Brie}]\label{theoBries}
 Every Brieskorn manifold $\Sigma_a$ satisfies
 \begin{enumerate}
  \item[i.] $\Sigma_a$ is at least $(n-2)$-connected, i.e.\ $\pi_k(\Sigma_a)=0,\, 1\leq k\leq n-2,$ which implies in particular for the singular homology of $\Sigma_a$ that
  \[H_k(\Sigma_a)=0\qquad\text{ for }\quad k\neq 0,n-1,n,2n-1.\]
  \item[ii.] $\Sigma_a$ is homeomorphic to the sphere $S^{2n-1}$ if and only if $G_a$ contains two isolated vertices or $G_a$ has one isolated vertex and satisfies (O).
 \end{enumerate}
\end{theo}
\begin{rem}~
 \begin{itemize}
  \item Given the tuples $(a_0,...\,,a_n)$, it is possible to calculate $H_{n-1}(\Sigma_a)\cong H_n(\Sigma_a)$ (see for example \cite{vanKo}).
  \item It follows from $\pi_1(\Sigma_a)=0$ and $c_1(TV_\veps)=0$ that $(V_\veps,\Sigma_a)$ satisfies conditions (A) and (B), i.e.\ that Conley-Zehnder indices on $(V_\veps,\Sigma_a)$ are well-defined.
 \end{itemize}
\end{rem}
We conclude this subsection with the calculation of the indices of all closed Reeb orbits in $\mathcal{N}^{L\pi/2}$. This is mainly taken from my diploma thesis and follows \cite{vanKo} and \cite{Usti}.\\
Recall the definitions of $R_a$ (Proposition \ref{prop15}) and its flow $\varphi^t_a$ in (\ref{eq6}). We can obviously regard them as defined on $\mathbb{C}^{n+1}$ instead of $\Sigma_a$. This allows us to calculate indices directly on $T\mathbb{C}^{n+1}=\big(\mathbb{C}^{n+1}\big)^2$ instead of $T\Sigma_a$. The action of $D\phi_a^t$ on $T\mathbb{C}^{n+1}$ in terms of the standard trivialization is given by the following path $\Phi^t\in Sp(2n+2)$ of diagonal matrices:
\[D\phi^t_a = diag\left( e^{4it/a_0},\dots,e^{4it/a_n}\right)=:\Phi^t.\]
Recall that Corollary \ref{lem18a} gave the following symplectic trivialization of the symplectic complement $\xi^\bot_a$ of $\xi_a$:
\begin{align*}
 X_1 &= \sqrt{\frac{2}{\sum a_k|z_k|^{2(a_k-1)}}}\cdot \Big( \bar{z}_0^{a_0-1},\dots,\bar{z}_n^{a_n-1}\Big),\qquad Y_1 = i\cdot X_1,\\
 X_2 &= \frac{1}{2}\cdot\left(z_0-\frac{\sum a_kz_k^{a_k}}{\sum a_k|z_k|^{2(a_k-1)}}\cdot \bar{z}_0^{a_0-1} , \dots,z_n-\frac{\sum a_kz_k^{a_k}}{\sum a_k|z_k|^{2(a_k-1)}}\cdot \bar{z}_n^{a_n-1}\right),\\
 Y_2 &= \left(\frac{4i}{a_0}z_0,\dots,\frac{4i}{a_n}z_n\right).
\end{align*}
We find by some calculation that the action of $\phi_a^t$ on $\xi_a^\bot$ yields
\begin{align*}
 D\varphi^t_a\bigl( X_1(z)\bigr) &= e^{4it}\cdot X_1\bigl( \varphi^t_a(z)\bigr), & D\varphi^t_a \bigl( Y_1(z)\bigr) &= e^{4it}\cdot Y_1\bigl( \varphi^t_a(z)\bigr),\\
 D\varphi^t_a\bigl( X_2(z)\bigr) &= X_2\bigl(\varphi^t_a(z)\bigr), & D\varphi^t_a\bigl( Y_2(z)\bigr) &= Y_2\bigl(\varphi^t_a(z)\bigr).
\end{align*}
It follows that the action of $D\phi_a^t$ on $\xi_a^\bot$ in the trivialization given by $X_1,Y_1,X_2,Y_2$ is the following path $\Phi^t_2\in Sp(4)$ of diagonal matrices:
\[diag (e^{4it}, 1) =: \Phi^t_2.\]
Observe that $\Phi^t$ and $\Phi^t_2$ are linearizations of $\varphi^t_a$ on $T\mathbb{C}^{n+1}$ and $\xi^\bot_a$ respectively, which are both trivial bundles. A trivialization of $\xi_a$ over a disc $u\subset\Sigma_a$, with $\partial u=v$ being a Reeb trajectory, provides us with a linearization $\Phi^t_1\in Sp(2n-2)$ of $\varphi^t_a$ on $\xi_a$. Any trivialization of $\xi_a$ over any disc in $\Sigma_a$ followed by the above trivialization of $\xi^\bot_a$ gives again a trivialization of $T\mathbb{C}^{n+1}$, which is homotopic to the standard one, as the disc is contractible. We hence obtain that $\Phi^t = \Psi^t(\Phi^t_1\oplus \Phi^t_2)(\Psi^0)^{-1}$ for some contractible loop $\Psi\in Sp(2n+2)$.\\
Now, let $v\in\mathcal{N}^{L\pi/2}$ be a closed Reeb trajectory of length $L\pi/2$. Using Lemma \ref{expliciteCZ} and the product property of the Conley-Zehnder index, we find that
\begin{align*}
 \mu_{CZ}(v)=\mu_{CZ}(\Phi_1) = \mu_{CZ}(\Phi)-\mu_{CZ}(\Phi_2) &=\underbrace{\sum_{k=0}^n\left(\left\lfloor\frac{L}{a_k}\right\rfloor + \left\lceil\frac{L}{a_k}\right\rceil\right)}_{\Phi - comp.} - \underbrace{\left(\left\lfloor L\right\rfloor + \left\lceil L\right\rceil\right)}_{\Phi_2-comp.}\\
&=\sum_{k=0}^n\Big(\big\lfloor L/a_k\big\rfloor + \big\lceil L/a_k\big\rceil\Big) - 2L,
\end{align*}
where the last line holds as $L$ is an integer.\\
We have shown above that the manifold $\mathcal{N}^{L\pi/2}$ is isomorphic to the Brieskorn  manifold $\Sigma_a\cap E(a,L)$ with $\dim E(a,L)=2n(a,L)$. Its dimension is thus
\[\dim\mathcal{N}^{L\pi/2} = 2\cdot n(a,L) - 3 = 2\cdot \#\Big\{k\,\Big|\, L/a_k\in\mathbb{Z}\Big\} -3.\]
Note that $(n+1) - n(a,L)$ is the number of indices $k,$ where $\lceil L/a_k\rceil - \lfloor L/a_k\rfloor = +1 $. Thus, we find that
\begin{align*}
 \mu_{CZ}(v) - \frac{1}{2} \dim\mathcal{N}^{L\pi/2} &= \sum_{k=0}^n\left(\left\lfloor \frac{L}{a_k}\right\rfloor + \left\lceil\frac{L}{a_k}\right\rceil\right) - 2L -\frac{2n(a,L)-3}{2}\\
 &= \sum_{k=0}^n\left(\left\lfloor \frac{L}{a_k}\right\rfloor + \left\lceil\frac{L}{a_k}\right\rceil\right)  + (n+1)-n(a,L) - 2L  - (n+1) +\frac{3}{2}\\
 &= \sum_{k=0}^n 2\cdot \left\lceil\frac{L}{a_k}\right\rceil - 2L - (n-1) -\frac{1}{2}.
\end{align*}
Using the definition of the index $\mu$ (cf. Proposition \ref{Prop13}), we have shown the following:
\begin{prop} \label{lem26}
 Let $h$ be a Morse function on $\crita$ and $c=(v,\eta)\in\mathcal{N}^\eta\cap crit(h)$ with $\eta=L\pi/2$ and $v$ a closed Reeb orbit of length $\eta$. The index of $c$ is then given by
\[\mu(c) = 2\cdot\left(\sum_{k=0}^n \left\lceil\frac{L}{a_k}\right\rceil\right) - 2L + ind_h(c) - (n-1).\] 
\end{prop}
\begin{dis}\label{mubehaviour}
 As $\dim\mathcal{N}^{L\pi/2}=2n(a,L)-3\leq2n-1$, we find that $ind_h(c)$ lies in the interval $[0,2n-1]$ for every $c\in crit(h)$. Set $A:=\prod_{k=0}^n a_k$ and write $L=j\cdot A + l$ with $l\in[0,A-1]$. Then we have
 \begin{align*}
  \mu(c) &= 2\cdot\left(\sum_{k=0}^n \left\lceil\frac{L}{a_k}\right\rceil\right) - 2L + ind_h(c) - (n-1)\\
  &=2j\left(\sum_{k=0}^n \frac{A}{a_k} - A\right) + 2\cdot\left(\sum_{k=0}^n \left\lceil\frac{l}{a_k}\right\rceil\right) - 2l + ind_h(c) - (n-1)\\
  &=: j\cdot 2A\left(\sum_{k=0}^n \frac{1}{a_k} - 1\right) + D(l,c),
 \end{align*}
where $D(l,c)$ depends on $l$ and $c$ but $-2A-n\leq D(l,c)\leq 2nA+n$. Thus we find
\begin{itemize}
 \item If $\sum 1/a_k >1$, then $\mu\rightarrow \infty$ as $L\rightarrow \infty$.
 \item If $\sum 1/a_k <1$, then $\mu\rightarrow -\infty$ as $L\rightarrow \infty$.
 \item If $\sum 1/a_k =1$, then $\mu$ is uniformly bounded.
\end{itemize}
As the set $crit(h)$ generates $RFH(W,\Sigma_a)$ for any filling $W$ of $\Sigma_a$, we find the following:
\end{dis}
\begin{cor}
 For every Brieskorn manifold $\Sigma_a$ and every filling $W$ of $\Sigma_a$ holds that $RFH_\ast(W,\Sigma_a)$ is a finite dimensional $\mathbb{Z}_2$-vector space for all $\ast$ if $\sum 1/a_k\neq 1$. If $\sum 1/a_k= 1$, it is zero for almost all $\ast$, with the exceptions lying in $[-2A-n,2nA +n]$. (Actually, (\ref{muupest}) and (\ref{mudownest}) show that non-zero groups can only occur for $\ast\in[-n+1,n]$.)
\end{cor}

\subsection{Calculation of $RFH(W_\veps,\Sigma_a)$ for some $a$}\label{sec7.1}
Originally, the intention of this work was to calculate $RFH(W_\veps,\Sigma_a)$ for all $a$. By now, we are still far from achieving this goal. Hence, we restrict ourself to some subclasses of Brieskorn manifolds where calculations are doable and which are interesting in its own.\medskip\\
We start with the last case of Discussion \ref{mubehaviour}, i.e.\ we assume that $\sum_{k=0}^n 1/a_k = 1$. First we assume all $a_k$ to be equal, i.e.\ we consider
\[a=\big(n+1,n+1,...\,,n+1\big).\]
Here, the Reeb flow $\varphi_a$ (cf. (\ref{eq6})) is given by $\varphi^t_a(z)=e^{4it/(n+1)}\cdot z$. We find that the critical manifolds of closed Reeb orbits $\mathcal{N}^{L\pi/2}$ are non-trivial exactly if $L=l(n+1)$ for some $l\in\mathbb{Z}$. Moreover, all these manifolds are equal to $\Sigma_a$. Their singular homology groups with $\mathbb{Z}_2$-coefficients $H_\ast(\mathcal{N}^{l(n+1)\pi/2},\mathbb{Z}_2)$ are hence by Theorem \ref{theoBries} non-zero only for $\ast=0,n-1,n,2n-1$. Now let
\begin{itemize}
 \item $l\gamma_0$ denote \textit{the} generator of $H_0(\mathcal{N}^{l(n+1)\pi/2},\mathbb{Z}_2)$,
 \item $l\gamma_{n-1}^j$ denote the generators of $H_{n-1}(\mathcal{N}^{l(n+1)\pi/2},\mathbb{Z}_2)$,
 \item $l\gamma_n^j$ denote the generators of $H_n(\mathcal{N}^{l(n+1)\pi/2},\mathbb{Z}_2)$,
 \item $l\gamma_{2n-1}$ denote \textit{the} generator of $H_{2n-1}(\mathcal{N}^{l(n+1)\pi/2},\mathbb{Z}_2)$.\pagebreak
\end{itemize}
By Section \ref{secRedFil}, Theorem \ref{reductiontheo}, we know that we can calculate the homology groups $RFH_\ast(W_\veps,\Sigma_a)$ via a chain complex $RFC_\ast(V_\veps,\Sigma)$ generated by the elements $l\gamma^j_k$ (or equivalently, we can pretend that there is a perfect Morse function $h$ on $\mathcal{N}^{l(n+1)\pi/2}$). The index $\mu(l\gamma_k^j)$ is due to Proposition \ref{lem26} given by
\begin{equation}\label{eqindexlg}
\begin{aligned}
 \mu(l\gamma_0)&=\underbrace{2\cdot\left(\sum_{k=0}^n \left\lceil\frac{l(n+1)}{n+1}\right\rceil\right) - 2l(n+1)}_{=0} + \underbrace{0}_{=ind_h(l\gamma_0)} - (n-1)&&=-n+1\\
 \mu(l\gamma_{n-1}^j)&=0 + (n-1) - (n-1) &&=0\\
 \mu(l\gamma_n^j)&=0 + n - (n-1) &&=1\\
 \mu(l\gamma_{2n-1})&=0 + (2n-1) - (n-1) &&=n.
 \end{aligned}
\end{equation}
Hence $RFC_\ast(W_\veps,\Sigma_a)$ has for $a=(n+1,...\,,n+1)$ an infinite number of generators in degrees $\ast=-(n-1),0,1$ and $n$ and no generators in all other degrees. From these observations, we obtain directly the following Proposition.
\begin{prop}\label{RFHforaeq}
 Let $a=(n+1,n+1,...\,,n+1)$ and $n\geq 3$. Then $RFH_\ast(W_\veps,\Sigma_a)$ is independent of the filling and satisfies
 \[\dim_{\mathbb{Z}_2} RFH_\ast(W_\veps,\Sigma_a) = \begin{cases}
  \infty & \text{if } \ast=-n+1 \text{ or } n\\
  0 & \text{if } \ast\neq -n+1,0,1,n. 
 \end{cases}\]
\end{prop}
The independence from the filling is obvious for $\ast\neq 0,1$ as the chain complex is independent from the filling. For $\ast=0,1$ it follows from Theorem \ref{theoinvar} as we can estimate the Conley-Zehnder part of the index $\mu$ by
\[\mu_{CZ}(\gamma)=2\cdot\left(\sum \left\lceil\frac{L}{a_k}\right\rceil\right) - 2L\geq 2L\Big(\sum 1/a_k -1\Big) = 0>3-n.\]
This shows also that $RFH_\ast(W_\veps,\Sigma_a)$ is independent of the filling whenever $\sum 1/a_k=1$. 
\begin{rem}
 By now, we have no means to calculate the groups $RFH_\ast(W_\veps,\Sigma_a)$ for $\ast=0,1$. For that, we would need the operator $\partial^F$ explicitly, as it might be non-zero here.
\end{rem}
If not all $a_k$ are equal, we can still get from $\sum 1/a_k = 1$ the following estimates on the index $\mu(\gamma)$ for a generator $\gamma\in H_\ast(\mathcal{N}^{L\pi/2},\mathbb{Z}_2)$. If $\big\lceil L/a_k\big\rceil\neq L/a_k$, then we have surely $L/a_k<\big\lceil L/a_k\big\rceil< L/a_k +1$. Recalling that $(n+1)-n(a,L)$ is the number of indices with $\big\lceil L/a_k\big\rceil\neq L/a_k$ and that the dimension of $\mathcal{N}^{L\pi/2}$ is $2\cdot n(a,L)-3$, we get
\begin{align}
 \mu(\gamma) &= 2\cdot\Big(\sum\big\lceil L/a_k\big\rceil\Big) - 2L + \ast - (n-1)\notag\\
 &\leq 2\cdot\Big(\sum L/a_k + (n+1)-n(a,L)\Big) - 2L + 2\cdot n(a,L)-3 - (n-1)\notag\\
 &= 2(n+1)-3-(n-1)\notag\\
 &=n\label{muupest}\\
 \text{and }\qquad \mu(\gamma)&\geq 2\cdot\Big(\sum L/a_k\Big)-2L + 0 -(n-1)\notag\\
 &=-n+1.\label{mudownest}
\end{align}
Note that in these estimates equality can only hold if equality holds for the Conley-Zehnder part, i.e.\ if $\big\lceil L/a_k\big\rceil = L/a_k$ for all $k$ that is if $a_k|L$ for all $k$. If this is not satisfied, then the estimates (\ref{muupest}) and (\ref{mudownest}) can be sharpened to
\begin{equation}\label{mu+est}
 -(n-1)+2 \leq \mu(\gamma)\leq n-2
\end{equation}
as the Conley-Zehnder part of $\mu$ is always even. If $a_k|L$ holds for  all $k$, then we have that $\mathcal{N}^{L\pi/2}$ equals $\Sigma_a$. In this case we have as above four classes of generators $L\gamma_\ast^j$ in $RFC(W_\veps,\Sigma_a)$ corresponding to generators of the four non-vanishing singular homology groups $H_\ast(\mathcal{N}^{L\pi/2},\mathbb{Z}_2), \; \ast=0,n-1,n,2n-1$. Note that the index calculations (\ref{eqindexlg}) for $l\gamma_\ast^j$ are also valid for $L\gamma^j_\ast$. It follows from this observation and the estimates (\ref{muupest}), (\ref{mudownest}), (\ref{mu+est}) that
\[\dim_{\mathbb{Z}_2} RFC_\ast(W_\veps,\Sigma_a)=
\begin{cases}
  \infty & \text{if } \ast = -n+1,\, n\\
  0 & \text{if } \ast\leq -n,\; \ast\geq n+1,\quad\text{ or }\quad\ast=-n+2,\,n-1.
\end{cases}\]
As before, we hence get:
\begin{prop}
 Let $a=(a_0,a_1,...\,,a_n)$ be such that $\sum 1/a_k=1$. Then it holds that $RFH_\ast(W_\veps,\Sigma_a)$ is independent from the filling and satisfies
 \[\dim_{\mathbb{Z}_2} RFH_\ast(W_\veps,\Sigma_a)=
 \begin{cases}
  \infty & \text{if } \ast = -n+1,\, n\\
  0 & \text{if } \ast\leq -n,\; \ast\geq n+1,\quad\text{ or }\quad\ast=-n+2,\,n-1.
\end{cases}\]
\end{prop}
Again, the independence from the filling for $\ast \leq -n+2$ or $\ast\geq n-1$ follows already from the specific form of $RFC_\ast(W_\veps,\Sigma_a)$.\medskip\\
To conclude the case $\sum 1/a_k=1$, let $A=lcm(a_0,...\,,a_n)$. We find for $\ast=-n+1$ or $n$ and every $N\in\mathbb{N}$ that there are exactly $N$ generators in $RFC_\ast(W_\veps,\Sigma_a)$ whose period lies in $(0,NA\pi/2]$. Thus, we find that
\[\dim_{\mathbb{Z}_2}RFH_\ast^{(0,NA\pi/2]}(W_\veps,\Sigma_a)=N\]
grows linearly in $N$. A similar result holds for the action window $[-NA\pi/2,0)$. Using Definition \ref{defngrowth} for the growth rates $\Gamma^\pm(W_\veps,\Sigma_a,f)$ we find therefore
\begin{prop}
 Let $a=(a_0,...\,,a_n)$ be such that $\sum 1/a_k=1$. Then it holds for $\ast=-n+1$ or $n$ that
 \[\Gamma^\pm_\ast(W_\veps,\Sigma_a,id)=1.\]
\end{prop}
\begin{rem}
 We observe that the boundedness of the Conley-Zehnder index for all closed Reeb orbits implies that the mean index (see \ref{1.4}) for all closed Reeb orbits is zero. This implies in particular that all Brieskorn manifolds $\Sigma_a$ with $\sum 1/a_k=1$ are neither Ustilovsky index positive, nor weakly index positive in the sense of Espina, \cite{Esp}, nor have convenient dynamics in the sense of Kwon and van Koert, \cite{KwKo}, Defn.\ 5.14. Using only results known until now, we can therefore not decide whether the contact homology, the $S^1$-equivariant symplectic homology or the mean Euler characteristic of these examples are invariant under subcritical surgery.
\end{rem}
As a second class of Brieskorn manifolds let us now turn to tuples $a=(2,2,2,a_3,...\,,a_n)$. In his diploma thesis (see also \cite{Fauck1}), the author considered in particular
\begin{align*}
 a&=(2,...\,,2,p\phantom{,q})\in\mathbb{N}^{n+1}, n \text{ odd, } p \text{ odd}\\
 \text{and }\qquad a&=(2,...\,,2,p,q)\in\mathbb{N}^{n+1}, n \text{ even, } p,q \text{ odd and } gcd(p,q)=1.
\end{align*}
 The corresponding Brieskorn manifolds $\Sigma_a$ are by Theorem \ref{theoBries} homeomorphic to the sphere $S^{2n-1}$ and the author calculated $RFH_\ast(W_\veps,\Sigma_a)$ for $n\geq 5$. He used that the chain complex $RFC_\ast(W_\veps,\Sigma_a)$ is zero for many values of $\ast$ and that the boundary operator $\partial^F$ reduces the action.\\
For $n=3$ however, we encounter the problem that $RFC_\ast(W_\veps,\Sigma_a)$ has generators for almost all $\ast$ so that $\partial^F$ is not trivially zero. Using a symplectic $\mathbb{Z}_2$-symmetry on the filling $V_\veps$ we will now overcome this problem.\\
For convenience, assume that for $a=(2,2,2,a_3,...\,,a_n)$ holds that $2< a_k\leq a_{k+1}$ for all $k\geq 3$. We find that for $0<L<a_3$ or $0>L>-a_3$ exactly the even $L$ give nontrivial critical manifolds $\mathcal{N}^{L\pi/2}$, all of which are diffeomorphic to $\Sigma_{(2,2,2)}$.\\ Let us have a closer look at this particular Brieskorn manifold. Writing $z_k=x_k+iy_k$, we obtain from the defining equation for $\Sigma_{(2,2,2)}=V_{(2,2,2)}(0)\cap S^{2n+1}$ that
\begin{align*}
 && 0 &= z_0^2+z_1^2+z_2^2& &\text{and}& 1&=|z_0|^2+|z_1|^2+|z_2|^2\\
 &\Leftrightarrow& 0&=\sum(x_k^2-y_k^2) + 2i\sum x_ky_k& &\text{and}& 1&=\sum(x_k^2+y_k^2)\\
 &\Leftrightarrow& \frac{1}{2}&=x_0^2+x_1^2+x_2^2=y_0^2+y_1^2+y_2^2& &\text{and}& 0&=x_0y_0+x_1y_1+x_2y_2.
\end{align*}
These equations describe the half unit tangent bundle $\frac{1}{2}S^\ast(\frac{1}{2}S^2)$ of the half unit sphere, which is of course naturally diffeomorphic to the true unit tangent bundle $S^\ast S^2$ of the unit sphere $S^2$. In Appendix \ref{app1}, we show that there is on $S^\ast S^2$ a perfect Morse function $\psi$ having exactly 4 critical points such that 
\[H_\ast(S^\ast S^2,\mathbb{Z}_2)\cong
\begin{cases}
  \mathbb{Z}_2 & \text{if } \ast=0,1,2,3\\
  0 & \text{otherwise.}
\end{cases}\]
So if $\mathcal{N}^{L\pi/2}\cong S^\ast S^2$, in particular if $|L|<a_3, L\neq 0$ and $L$ even, we get four generators for $RFC(W_\veps,\Sigma_a)$. We denote them by $L\gamma_0,L\gamma_1,L\gamma_2,L\gamma_3$, where $L\gamma_j$ is the generator of $H_j(\mathcal{N}^{L\pi/2},\mathbb{Z}_2)$. Note that these generators are exactly given by the 4 critical points of the perfect Morse function $\psi$. Hence, we do not need Theorem \ref{reductiontheo} here.\\
In order to describe $RFC_\ast(W_\veps,\Sigma_a)$ for $\ast$ around $0$, we calculate $\mu(L\gamma_j)$ for $|L|\leq a_3$ explicitly and give estimates for all other $L$. For $0<L<a_3$, the indices are given by Proposition \ref{lem26} as
\begin{align}
 \mu(L\gamma_j) &= 2\Big(\sum\big\lceil L/a_k\big\rceil\Big)-2L+ind_\psi(L\gamma_j)-(n-1)\notag\\
 &=2\Big(\underbrace{3\cdot L/2}_{a_0,a_1,a_2}+\underbrace{n-2}_{a_3,...\,,a_n}\Big)-2L+j-n+1\notag\\
 &=L+n-3+j\label{muLgammaj}.
\end{align}
For $L\leq 0$, we estimate for any critical point $c\in\mathcal{N}^{L\pi/2}$ that 
\begin{align}
 \mu(c) &= \textstyle 2\Big(\sum\big\lceil L/a_k\big\rceil\Big)-2L+ind_h(c)-(n-1)\notag\\
 &\leq  \textstyle 2\Big(\sum L/a_k+(n+1)-n(a,L)\Big)-2L+2\cdot n(a,L)-3-n+1\notag\\
 &=2 \textstyle \Big(\sum 1/a_k-1\Big)\cdot L + n\notag\\
 &\leq n \label{muLgammadownest}.
\end{align}
For $L\geq a_3$, we estimate for any critical point $c\in\mathcal{N}^{L\pi/2}$ that 
\begin{align}
 \mu(c) &= \textstyle 2\Big(\sum_{k\geq 2}\big\lceil L/a_k\big\rceil\Big)+\underbrace{4\big\lceil L/2\big\rceil}_{z_0,z_1}-2L+ind_h(c)-(n-1)\notag\\
 &\geq \textstyle 2\Big(\sum_{k\geq 2}\big\lceil a_3/a_k\big\rceil\Big)+2L-2L+0-n+1\notag\\
 &\geq2\Big(\underbrace{a_3/2}_{z_2} + \underbrace{n-2}_{z_3,...\,,z_n}\Big)-n+1\notag\\
 &=a_3+n-3.\label{muLgammaupest}
\end{align}
For $-a_3<L<0$, we find by analogue calculations for the index that 
\begin{align}
 \mu(L\gamma_j)=L-(n-1)+j.\label{muLgammaj-}
\end{align}
and for any $c\in\mathcal{N}^{L\pi/2}$ the analogue estimates
\begin{align}
 \mu(c)\leq -a_3-n+4 \quad\text{ if }L\leq - a_3\qquad\text{ and }\qquad \mu(c)\geq -(n-1)\quad\text{ if } L\geq 0.\label{muLgammaest-}
\end{align}
These calculations and estimates imply for degrees $\ast\in(n,a_3+n-3)$ respectively $\ast\in(-a_3-n+4,-n+1)$ that the only generators in $RFC_\ast(W_\veps,\Sigma_a)$ are the $L\gamma_j$ for even $L$. The distribution of these elements among the different values of $\ast$ is shown in the following table
\begin{equation}\label{tableLgen}
 \begin{array}{l|c|c|c|c|c|c|c}
   \begin{aligned}\ast& \scriptstyle-(n-3) \text{ if } L>0\\\ast&\scriptstyle+(n-1) \text{ if } L<0\end{aligned} & \vdots&  L &   L+1 &  (L+2) &  (L+2)+1 &  (L+4) &   \vdots\\\hline
  \text{generators } & \vdots & \begin{aligned}& \;\,L\gamma_0\\&(L-2)\gamma_2\end{aligned}  & \begin{aligned}& \;\,L\gamma_1\\& (L-2)\gamma_3\end{aligned}  & \begin{aligned}& \;\,L\gamma_2\\& (L+2)\gamma_0\end{aligned}  & \begin{aligned}& \;\,L\gamma_3\\& (L+2)\gamma_1\end{aligned}  & \begin{aligned}& (L+4)\gamma_0\\& (L+2)\gamma_2\end{aligned}  & \vdots 
  \end{array}  
\end{equation}
\begin{theo}\label{theobriesspheres}
 Let $a=(2,2,2,a_3,...\,,a_n)$ be as above. Then $RFH(W_\veps,\Sigma_a)$ is independent of the filling and satisfies
 \[RFH_\ast(W_\veps,\Sigma_a)\cong\big(\mathbb{Z}_2\big)^2  \phantom{\Leftrightarrow}\]
  for $n+2\leq \ast\leq a_3+n-5$ or $-a_3-n+6\leq\ast\leq -n-1$, which is equivalent to $n+1\leq|\ast-1/2|-1/2\leq a_3+n-6$.
\end{theo}
\begin{proof}
 For Reeb orbits $\gamma$ with positive period $L\pi/2>0$, we can estimate the Conley-Zehnder index $\mu_{CZ}(\gamma)$ by
 \[\mu_{CZ}(\gamma)=2\left(\sum \big\lceil L/a_k\big\rceil\right)-2L\geq 2L\left(\sum 1/a_k-1\right)\geq 2L\big(3/2-1\big) = L >3-n.\]
 Theorem \ref{theoinvar} then implies that $RFH(W_\veps,\Sigma_a)$ is independent of the filling $V_\veps$.\pagebreak\\
 For the explicit calculations of homology groups, we first note that our calculations and estimates (\ref{muLgammaj}) - (\ref{muLgammaest-}) and the table (\ref{tableLgen}) tell us the following about the Rabinowitz-Floer chain complex. For $L>0$ and $n+1\leq \ast\leq a_3+n-4$ we have two cases depending on the parity of $\ast$
 \begin{itemize}
  \item[i.)] $\ast-(n-3)=L$ is even, then $RFC_\ast(W_\veps,\Sigma_a)$ is $\mathbb{Z}_2$-generated by
  \[\begin{aligned}RFC_\ast(W_\veps,\Sigma_a)&=\mathbb{Z}_2\langle L\gamma_0,(L-2)\gamma_2\rangle\\RFC_{\ast+1}(W_\veps,\Sigma_a)&=\mathbb{Z}_2\langle L\gamma_1,(L-2)\gamma_3\rangle\end{aligned},\text{ i.e.\ }\quad
  \begin{tabular}{l|c|c}
   $\ast$ & $\scriptstyle L+(n-3)$ & $\scriptstyle L+(n-3)+1$\\\hline
   gen. & $L\gamma_0$ & $L\gamma_1$\\
   & $(L-2)\gamma_2$ & $(L-2)\gamma_3$
  \end{tabular}\]
  \item[ii.)] $\ast-(n-3)=L+1$ is odd, then $RFC_\ast(W_\veps,\Sigma_a)$ is $\mathbb{Z}_2$-generated by
  \[\begin{aligned}RFC_\ast(W_\veps,\Sigma_a)&=\mathbb{Z}_2\langle L\gamma_1,(L-2)\gamma_3\rangle\\RFC_{\ast+1}(W_\veps,\Sigma_a)&=\mathbb{Z}_2\langle L\gamma_2,(L+2)\gamma_0\rangle\end{aligned},\text{ i.e.\ }\quad
  \begin{tabular}{l|c|c}
   $\ast$ & $\scriptstyle L+(n-3)+1$ & $\scriptstyle L+(n-3)+2$\\\hline
   gen. & $L\gamma_1$ & $L\gamma_2$\\
   & $(L-2)\gamma_3$ & $(L+2)\gamma_0$
  \end{tabular}\]
 \end{itemize}
A similar result holds for $L<0$ and $-a_3-n+5\leq \ast\leq -n$ if we replace $+(n-3)$ by $-(n+1)$. As all following arguments are completely analogue for $L<0$ or $L>0$, we restrict ourself from now on to $L>0$. To prove the theorem, it suffices to show that in both of the above cases the boundary operator $\partial^F$ is zero. We will do so with the help of the following order 2 symplectic symmetry:
\[\sigma:\mathbb{C}^{n+1}\rightarrow\mathbb{C}^{n+1},\qquad(z_0,z_1,z_2,z_3,...\,,z_n)\mapsto (z_0,z_1,-z_2,z_3,...\,,z_n).\]
As $a_2=2$, it follows from Discussion \ref{dissym} that $\sigma$ restricts to a well-defined symplectic symmetry on $V_\veps$. Its fixed point set $V_{fix}$ is given by
\[V_{fix}=\big\{z\in V_\veps\,\big|\,z_2=0\big\}=V_\veps\cap\big(\mathbb{C}^2\times\{0\}\times\mathbb{C}^{n-2}\big),\]
which is obviously a symplectic submanifold of $V_\veps$. In fact, it is an exact filling of the contact manifold
\[\Sigma_{fix}=\partial V_{fix}=\big\{z\in \Sigma_a\,\big|\,z_2=0\big\}=\Sigma_a\cap\big(\mathbb{C}^2\times\{0\}\times\mathbb{C}^{n-2}\big),\]
which is contactomorphic to the Brieskorn manifold $\Sigma_{\bar{a}}$ with $\bar{a}=(2,2,a_3,...\,,a_n)$.\\
Recall that for $0<L<a_3$ all critical manifolds $\mathcal{N}^{L\pi/2}$ are diffeomorphic to $S^\ast S^2$ and given by
\[\mathcal{N}^{L\pi/2}=\Sigma_a\cap\big(\mathbb{C}^3\times\{0\}\big)\cong\Sigma_{(2,2,2)}.\]
The symmetry $\sigma$ restricts therefore on $\mathcal{N}^{L\pi/2}\cong S^\ast S^2$ to the involution
\[r: S^\ast S^2\rightarrow S^\ast S^2,\qquad (x_0,x_1,x_2,y_0,y_1,y_2)\mapsto(x_0,x_1,-x_2,y_0,y_1,-y_2),\]
where $z_k=x_k+iy_k$. In Appendix \ref{app1}, we show that there exists on $S^\ast S^2$ a Morse-Smale pair $(\psi,g)$ of a perfect Morse function  and a metric which are invariant under $r$. Now, if the Global Transversality Theorem \ref{theotranssym} holds, then it follows from Corollary \ref{symhom} that the number of all unparametrized flow-lines with cascades between relevant generators is 0 mod 2. This would then imply that $\partial^F\equiv 0$.\medskip\\
It remains to show that Theorem \ref{theotranssym} can be applied, i.e. that we can find a $\sigma$-symmetric almost complex structure $J$ for which transversality holds for $\ast\in[n+1,a_3+n-4]$.\linebreak Note that we can restrict ourself to this degree window, as we are only interested in $RFH_\ast(W_\veps,\Sigma_a)$ for $\ast\in[n+2,a_3+n-5]$ and as the definition of the homology group $RFH_\ast(W_\veps,\Sigma_a)$ only involves the chain groups $RFC_{\ast-1}(W_\veps,\Sigma_a), RFC_\ast(W_\veps,\Sigma_a)$ and $RFC_{\ast+1}(W_\veps,\Sigma_a)$. We proceed by considering the two cases $i.)$ and $ii.)$ separately.
\begin{itemize}
 \item In $i.)$, when $\ast-(n-3)=L$ is even, the argument goes as follows. As $RFC_\ast(W_\veps,\Sigma_a)$ is freely generated by $L\gamma_0$ and $(L-2)\gamma_2$ and $RFC_{\ast+1}(W_\veps,\Sigma_a)$ is freely generated by $L\gamma_1$ and $(L-2)\gamma_3$, we may write the boundary operator $\partial^F_{\ast+1}:RFC_{\ast+1}\rightarrow RFC_\ast$ as a 2 by 2 matrix with $\mathbb{Z}_2$-entries such that
 \[\partial^F \begin{pmatrix} L\gamma_1\\(L-2)\gamma_3\end{pmatrix}=\begin{pmatrix}a_{10} & a_{12}\\ a_{30} & a_{32}\end{pmatrix}\begin{pmatrix}L\gamma_0\\ (L-2)\gamma_2\end{pmatrix}.\]
 Here, $a_{ij}$ is the $\mathbb{Z}_2$-count of flow lines with cascades (flwc) from $L\gamma_j$ to $L\gamma_i$. For $a_{10}$ and $a_{32}$ these are flow lines from $L\gamma_0$ to $L\gamma_1$ resp.\ from $(L-2)\gamma_2$ to $(L-2)\gamma_3$. As these are between orbits of the same period, they are actually Morse flow lines and the number of such flow lines is even, as we have a perfect Morse function and thus $0=a_{10}=a_{32}$. The coefficient $a_{30}$ counts flwc going from $L\gamma_0$ to $(L-2)\gamma_3$. As the period can only increase along flow lines (Lemma \ref{lem2}), we know that there are no such flow lines and hence that $a_{30}=0$ as well.\\
 So it only remains to show that Theorem \ref{theotranssym} can be applied to flow lines from $(L-2)\gamma_2$ to $L\gamma_1$ counted by $a_{12}$. Note that there is no closed Reeb orbit on $\Sigma_a$ having a period between $(L-2)\pi/2$ and $L\pi/2$. This shows that every flow line between $(L-2)\gamma_2$ and $L\gamma_1$ has exactly 1 cascade. For Theorem \ref{theotranssym}, it suffices therefore to show that $\widehat{\mathcal{M}}\big((L-2)\gamma_2,L\gamma_1,1\big)\big|_{V_{fix}}$ is empty for a generic choice of $J$ on $V_{fix}$. For that, we calculate the indices of $(L-2)\gamma_2$ and $L\gamma_1$ in $V_{fix}$. As $V_{fix}$ is $V_\veps$ with the $3^{rd}$-coordinate omitted, we have by Proposition \ref{lem26} that
 \begin{align*}
  \mu\big(L\gamma_1\big)\big|_{V_{fix}} &= 2\Big(\underset{k\neq 2}{\textstyle\sum} \big\lceil L/a_k\big\rceil\Big)-2L +\underbrace{0}_{ind_h(L\gamma_1)|_{V_{fix}}} - (n-2)\\
  &=2\big(\underbrace{2\cdot L/2}_{a_0,a_1} + \underbrace{n-2}_{a_3,...\,,a_n}\big)-2L+0-(n-2)&&=n-2\\
  \mu\big((L-2)\gamma_2\big)\big|_{V_{fix}} &= 2\cdot\Big(\underset{k\neq 2}{\textstyle\sum} \big\lceil (L-2)/a_k\big\rceil\Big)-2(L-2) +1 - (n-2)\\
  &=2\cdot\big(2\cdot (L-2)/2+n-2\big)-2(L-2)+1-(n-2)&&=n-1.
 \end{align*}
Hence we obtain from Theorem \ref{theotrans} that for a generic $J$ on $V_{fix}$ we have
\begin{align*}
 \dim \widehat{\mathcal{M}}\big((L-2)\gamma_2,L\gamma_1,1\big)\big|_{V_{fix}}&=\mu\big(L\gamma_1\big)\big|_{V_{fix}}-\mu\big((L-2)\gamma_2\big)\big|_{V_{fix}}+(1-1)\\
 &=(n-2)-(n-1)=-1,
\end{align*}
which shows that this space is generically empty.\pagebreak
\item  In the second case $ii.)$, when $\ast-(n-3)=L+1$ is odd, we argue similar. We write $\partial^F$ again as a matrix with $\mathbb{Z}_2$-coefficients such that
\[\partial^F\begin{pmatrix} L\gamma_2\\(L+2)\gamma_0\end{pmatrix} = \begin{pmatrix} a_{21}& a_{23}\\ a_{01} & a_{03}\end{pmatrix}\begin{pmatrix}L\gamma_1\\(L-2)\gamma_3\end{pmatrix}.\]
Here, $a_{21}=0$ as it counts flwc from $L\gamma_1$ to $L\gamma_2$ which are Morse flow lines whose number is again even. For the remaining three coefficients, we first calculate the restricted indices as
\begin{align*}
 \mu\big(L\gamma_1\big)\big|_{V_{fix}}&=2\cdot\Big(\underset{k\neq 2}{\textstyle\sum}\big\lceil L/a_k\big\rceil\Big) -2L + 0 -(n-2)\\
 &=2\cdot\big(2\cdot L/2 + n-2\big)-2L-(n-2)&&=n-2\\
 \mu\big(L\gamma_2\big)\big|_{V_{fix}}&=\mu\big(L\gamma_1\big)\big|_{V_{fix}}+1 &&= n-1\\
 \mu\big((L-2)\gamma_3\big)\big|_{V_{fix}}&=2\cdot\Big(\underset{k\neq 2}{\textstyle\sum}\big\lceil (L-2)/a_k\big\rceil\Big) -2(L-2) + 1 -(n-2)\\
 &=2\cdot\big(2\cdot (L-2)/2 + n-2\big)-2(L-2)+1-(n-2)&&=n-1\\
 \mu\big((L+2)\gamma_0\big)\big|_{V_{fix}}&=2\cdot\Big(\underset{k\neq 2}{\textstyle\sum}\big\lceil (L+2)/a_k\big\rceil\Big) -2(L+2) + 0 -(n-2)\\
 &=2\cdot\big(2\cdot (L+2)/2 + n-2\big)-2(L+2)-(n-2)&&=n-2.
\end{align*}
Calculating again dimensions, we find by Theorem \ref{theotrans} that
\begin{align*}
 \dim \widehat{\mathcal{M}}\big(L\gamma_1,(L+2)\gamma_0,1\big)\big|_{V_{fix}}&=(n-2)-(n-2)+0=0\\
 \dim \widehat{\mathcal{M}}\big((L-2)\gamma_3,L\gamma_2,1\big)\big|_{V_{fix}}&=(n-1)-(n-1)+0=0\\
 \dim \widehat{\mathcal{M}}\big((L-2)\gamma_3,(L+2)\gamma_0,1\big)\big|_{V_{fix}}&=(n-2)-(n-1)+0=-1.
\end{align*}
Note that $\mathbb{R}$ acts freely on $\widehat{\mathcal{M}}(\cdot,\cdot,1)$ by time shift on the non-constant cascade. This shows that all three above spaces have to be empty. As a direct consequence, we find $a_{01}=a_{23}=0$ as there is no closed Reeb orbit whose period lies strictly between $(L-2)\pi/2$ and $L\pi/2$ resp.\ $L\pi/2$ and $(L+2)\pi/2$.\\
For $a_{03}$ however, we have to exclude flow lines with two cascades between \linebreak $(L-2)\pi/2$ and $(L+2)\pi/2$ passing through the intermediate critical manifold $\mathcal{N}^{L\pi/2}$. Here, we have to show for Theorem \ref{theotranssym} that $\widehat{\mathcal{M}}\big((L-2)\gamma_3,\mathcal{N}^{L\pi/2}\big|_{V_{fix}}\big)$ and $\widehat{\mathcal{M}}\big(\mathcal{N}^{L\pi/2}\big|_{V_{fix}},(L+2)\gamma_0\big)$ are both empty. Noting that 
\[\dim \mathcal{N}^{(L-2)\pi/2}\big|_{V_{fix}}=\dim \mathcal{N}^{L\pi/2}\big|_{V_{fix}}= \dim \mathcal{N}^{(L+2)\pi/2}\big|_{V_{fix}}=1,\]
we calculate with the dimension formulas (\ref{dimcN}) and (\ref{dimNc}) that
\begin{align*}
 \dim \widehat{\mathcal{M}}\big((L-2)\gamma_3,\mathcal{N}^{L\pi/2}\big|_{V_{fix}}\big) &=\underbrace{2(n-2)}_{\mu_{CZ}(\mathcal{N}^{L\pi/2})} - \underbrace{2(n-2)}_{\mu_{CZ}((L-2)\gamma_3)} + \frac{1+1}{2} -1=0\\
 \dim \widehat{\mathcal{M}}\big(\mathcal{N}^{L\pi/2}\big|_{V_{fix}},(L+2)\gamma_0\big) &=2(n-2) - 2(n-2) + \frac{1-1}{2} +0=0.
\end{align*}
As $\mathbb{R}$ acts also freely on these spaces, we know again that they are empty. Hence there cannot be a flow line with two cascades between $(L-2)\gamma_3$ and $(L+2)\gamma_0$ with at least one cascade in $V_{fix}$. As $L\pi/2$ is the only period of a closed Reeb orbit between $(L-2)\pi/2$ and $(L+2)\pi/2$, we may hence apply Theorem \ref{theotranssym} and find also $a_{03}=0$.\qedhere
\end{itemize}
\end{proof}
\begin{rem}
 The result of Theorem \ref{theobriesspheres} (and the ideas of its proof) together with the calculations in \cite{Fauck1}, Thm.2.10, can be used to show that the Brieskorn manifolds $\Sigma_{(2,2,2,p)}$ for $p$ odd are all non-contactomorphic for different values of $p$. This improves a result shown by the author in his diploma thesis, where he could only prove that $\Sigma_{(2,2,2,p)}$ and $\Sigma_{(2,2,2,q)}$ are non-contactomorphic if $gcd(p+2,q+2)=1$.
\end{rem}
\begin{cor}
 For the unit cotangent bundle $S^\ast S^2\cong\Sigma_{(2,2,2)}$ holds that its Rabinowitz-Floer homology is independent from the filling and given by
 \[RFH_\ast(S^\ast S^2)\cong \big(\mathbb{Z}_2\big)^2\qquad\forall\;\ast\in\mathbb{Z}.\]
\end{cor}
\begin{proof}
 The independence from the filling follows again from Theorem \ref{theoinvar}.\\
 The critical manifolds are here all of the form $\mathcal{N}^{L\pi/2}\cong \Sigma_{(2,2,2)}$ with $L$ even. In particular, we get for each $L$ the four generators $L\gamma_0,L\gamma_1,L\gamma_2,L\gamma_3$ for $RFC(S^\ast S^2)$ as in the proof of Theorem \ref{theobriesspheres}. Their indices are here always
 \[\mu(L\gamma_j)=2\cdot\Big(3\cdot\big\lceil L/2\big\rceil\Big) -2L + j -(2-1) = L+j-1.\]
 Now if $\ast=L-1$ is odd, then $RFC_\ast(S^\ast S^2)$ is generated by $L\gamma_0$ and $(L-2)\gamma_2$, while for $\ast=L$ even $RFC_\ast(S^\ast S^2)$ is generated by $L\gamma_1$ and $(L-2)\gamma_3$. Thus we can use the same arguments as in the proof of Theorem \ref{theobriesspheres} to show that $\partial^F_\ast$ is zero for all $\ast$ and hence that 
 \[RFH_\ast(S^\ast S^2)\cong RFC_\ast(S^\ast S^2)\cong\big(\mathbb{Z}_2\big)^2.\]
\end{proof}
Let us finish this section with a generalization of this corollary  to the Brieskorn manifolds $\Sigma_l:=\Sigma_{(2,2,2,2l)},\,l\in\mathbb{N},$ which is due to Peter Uebele, \cite{Ueb}. In \cite{DurKau}, Prop.\ 6.1, it is shown that the diffeomorphism type of $\Sigma_l$ is given by
\[\Sigma_l\cong\begin{cases}
                S^2\times S^3 & \text{if } l\equiv 0 \mod 4\\
                S^\ast S^3 & \text{if } l\equiv 1 \mod 4\\
                \big(S^2\times S^3\big)\# K & \text{if } l\equiv 2 \mod 4\\
                S^\ast S^3\# K & \text{if } l\equiv 3 \mod 4,
               \end{cases}\]
where $K$ denotes the Kervaire sphere of dimension 5 and $\#$ denotes the connected sum. As $K$ is diffeomorphic to $S^5$ (see \cite{KerMil}) and as the cotangent bundle of $S^3$ is trivial, so that $S^\ast S^3\cong S^2\times S^3$, we get
\begin{lemme}
 $\Sigma_l=\Sigma_{(2,2,2,2l)}$ is diffeomorphic to $S^2\times S^3$ for all $l\geq 1$.
\end{lemme}
In \cite{vanKo} it is shown that all $\Sigma_l$ have the same contact homology and the same holds true for their equivariant symplectic homology by \cite{KwKo}. Moreover, their underlying formal homotopy classes/almost contact structures coincide, as follows from \cite{Geiges},8.1.1 together with the fact that the first Chern classes $c_1(T\Sigma_l)$ vanishes. However, we still have:
\begin{theo}[Uebele, \cite{Ueb}]
 The manifolds $\Sigma_{(2,2,2,2l)},\,l\geq 1,$ are non-contactomorphic.
\end{theo}
\begin{proof}
 Our proof relies on the calculation of $RFH_\ast(V_\veps,\Sigma_l)$. Note that it does not depend on the filling, again by Theorem \ref{theoinvar}.\\
 For $\Sigma_l$ we have two types of critical manifolds $\mathcal{N}^{L\pi/2}$, which are given by
 \begin{align*}
  \qquad\qquad\mathcal{N}^{L\pi/2} &\cong \Sigma_{(2,2,2)}\cong S^\ast S^2& &\text{if } l\nmid L\qquad\qquad\\
  \mathcal{N}^{L\pi/2} &\cong \Sigma_{(2,2,2,2l)}\cong S^\ast S^3& &\text{if } l\mid L,
 \end{align*}
where $L$ is always even. This gives us for each $L$ four generators of $RFC(\Sigma_l)$, which we denote again by $L\gamma_0,L\gamma_1,L\gamma_2,L\gamma_3$. Note that for $l\nmid L$, we have $L\gamma_j\in H_j(S^\ast S^2)$, while for $l\mid L$, we have $L\gamma_0\in H_0(S^2\times S^3),\, L\gamma_1\in H_2(S^2\times S^3), L\gamma_2\in H_3(S^2\times S^3)$ and $L\gamma_3\in H_5(S^2\times S^3)$. The indices of these generators are given by
\begin{itemize}
 \item $l\nmid L$
 \[\mu(L\gamma_j) = 2\cdot\Big(3\cdot L/2 + \big\lceil L/2l\big\rceil\Big) -2L + j -(3-1) = L+2\big\lceil L/2l\big\rceil + j -2\]
 \item $L=N(2l)$
 \[\mu(L\gamma_j)=2\left(3\frac{N2l}{2}+\frac{N2l}{2l}\right)-2(N2l)+ind(L\gamma_j)-(3-1) = N(2l+2)+ind(L\gamma_j) -2,\]
 where $ind(L\gamma_j)\in\{0,2,3,5\}$.
\end{itemize}
For the degree $\ast$ around $N(2l+2)=L+2N$, we find that the chain groups $RFC_\ast(\Sigma_l)$ have the following generators:
\[\begin{tabular}{l|c|c|c|c|c|c|c|c|c|c}
   $\ast$ & $\vdots$ &$\scriptstyle L+2N-3$ & $\scriptstyle L+2N-2$ & $\scriptstyle L+2N-1$ & $\scriptstyle L+2N$ & $\scriptstyle L+2N+1$ & $\scriptstyle L+2N+2$ & $\scriptstyle L+2N+3$ & $\scriptstyle L+2N+4$ & $\vdots$ \\\hline
   gen. & $\vdots$ & $\scriptstyle(L-2)\gamma_1$ & $\scriptstyle(L-2)\gamma_2$ & $\scriptstyle(L-2)\gamma_3$ & & & $\scriptstyle(L+2)\gamma_0$ & $\scriptstyle(L+2)\gamma_1$ & $\scriptstyle(L+2)\gamma_2$ & $\vdots$\\
   & $\vdots$& $\scriptstyle(L-4)\gamma_3$ & $\scriptstyle L\gamma_0$ & & $\scriptstyle L\gamma_1$ & $\scriptstyle L\gamma_2$ & & $\scriptstyle L\gamma_3$ & $\scriptstyle (L+4)\gamma_0$ &$\vdots$
  \end{tabular}\]
Away from these values for $\ast$, we have only generators $L\gamma_j$ living on $S^\ast S^2$ and it follows that the situation there looks similar to the table (\ref{tableLgen}). Using this explicit description of $RFC_\ast(\Sigma_l)$ and the arguments used in Theorem \ref{theobriesspheres}, we get
\begin{align*}
 \dim_{\mathbb{Z}_2} RFH_\ast(\Sigma_l)&\leq 1& &\text{if } \ast=N(2l+2)+j,\text{ for any }N\in\mathbb{Z},\,j\in\{-1,0,1,2\}\\
 \dim_{\mathbb{Z}_2} RFH_\ast(\Sigma_l)&= 2& &\text{if } \ast=N(2l+2)+5\leq \ast \leq (N+1)(2l+2)-4\text{ for any }N\in\mathbb{Z}.
\end{align*}
Note that the second case is only non-empty if $l\geq 4$.\pagebreak\\
Now we can show that $\Sigma_l$ and $\Sigma_{l+k}$ are not contactomorphic for $l\geq 4$ and $k\geq 1$. As the Rabinowitz-Floer homology of $\Sigma_l$ does not depend on the filling, it suffices to find degrees $\ast$, where $RFH_\ast(\Sigma_l)\neq RFH_\ast(\Sigma_{l+k})$. For $k\geq 2$, we consider $\ast=(2l+2)-1$ and find that
\[\dim_{\mathbb{Z}_2} RFH_\ast(\Sigma_l)\leq 1 \quad\text{ while }\quad \dim_{\mathbb{Z}_2} RFH_\ast(\Sigma_{l+k})= 2.\]
For $k=1$, we consider $\ast=2(2l+2)-1$ and find that
\[\dim_{\mathbb{Z}_2} RFH_\ast(\Sigma_l)\leq 1 \quad\text{ while }\quad \dim_{\mathbb{Z}_2} RFH_\ast(\Sigma_{l+1})= 2.\]
This shows that all $\Sigma_l$ for $l\geq 4$ are non-contactomorphic.
\end{proof}
\begin{rem}~
 \begin{itemize}
  \item For $l=1$, i.e.\ for $\Sigma_{(2,2,2,2)}$ we can use the symplectic symmetry 
  \[\sigma:(z_0,z_1,z_2,z_3)\mapsto (z_0,z_1,-z_2,-z_3)\]
  to show that $RFH_\ast(\Sigma_1)\cong\mathbb{Z}_2$ for all $\ast\in\mathbb{Z}$. This shows that $\Sigma_1$ is also not contactomorphic to $\Sigma_l$ for any $l\geq 4$.
  \item For $l=2,3$, we can use the same arguments as above to show that they are not contactomorphic to $\Sigma_l$ for any $l\geq 4$. However, our methods here do not suffice to distinguish $\Sigma_1,\Sigma_2$ and $\Sigma_3$. This can be achieved either by disturbing the contact structure (as did Uebele in \cite{Ueb}) or one would have to find a symmetric Morse function and metric on $\Sigma_2$ resp.\ $\Sigma_3$. Note that even though $\Sigma_2\cong\Sigma_3\cong S^\ast S^3$, the symmetry $\sigma$ does here not obviously restrict to the symmetry 
  \[(x_0,x_1,x_2,x_3,y_0,y_1,y_2,y_3)\mapsto (x_0,x_1,-x_2,-x_3,y_0,y_1,-y_2,-y_3).\]
 \end{itemize}
\end{rem}

\subsection{Exotic contact structures and fillings}\label{sec7.3}
In this subsection, we are going to prove our Main Theorem \ref{maintheo1}. Moreover, we show some Reeb-dynamical consequences which can be concluded from it and we give some mild generalizations.\bigskip\\
Let us begin with the following existence result for certain contact structures on the standard sphere. It is a consequence of Theorem \ref{theobriesspheres} and the handle attachment construction.
\begin{theo}\label{theocontactonsphere}
 For every $n\geq 3$ and any $k\in\mathbb{Z}$ such that $k\not\in[-n+1,n]$ resp.\ $|k-1/2|\geq n+1/2$ there exists on the standard sphere $S^{2n-1}$ a contact structure $\xi$ with filling $W$ such that $\mu_{CZ}$ is well-defined on $W$, i.e.\ $W$ satisfies (A) and (B), and
 \[2\leq \dim_{\mathbb{Z}_2}RFH_k\big(W,(S^{2n-1},\xi)\big)<\infty.\]
\end{theo}
\begin{proof}
 Consider tuples $a=(2,2,2,a_3,...\,,a_n),$ where $a_3,...\,,a_n$ are all odd and for all $k>3$ holds $a_3<a_k$ and $gcd(a_3,a_k)=1$. It follows from Theorem \ref{theoBries} that the Brieskorn manifold $\Sigma_a$ for such $a$ is homeomorphic to the sphere $S^{2n-1}$. As the diffeomorphism types of the topological sphere $S^{2n-1}$ form a finite group under the connected sum construction with the standard differentiable structure as neutral element (see \cite{KerMil}), we can find an $m\in\mathbb{N}$ such that the $m$-fold connected sum $\#_m\Sigma_a$ is diffeomorphic to $S^{2n-1}$. Let $\#_m\xi_a$ denote the resulting contact structure on $S^{2n-1}$. The $m$-fold boundary connected sum $\#_mW_\veps$ of the filling $W_\veps$ of $\Sigma_a$ (obtained by attaching $(m-1)$ 1-handles to $m$ copies of $W_\veps$) is then an exact contact filling for $(S^{2n-1},\#_m\xi_a)$. It follows from Theorem \ref{theobriesspheres} that for $a_3\geq |k-1/2|-n+5+1/2$ holds that
 \[\dim_{\mathbb{Z}_2}RFH_k(W_\veps,(\Sigma_a,\xi_a))=2\]
 and it follows from Theorem \ref{theoRFHinvsur} for the $m$-fold connected sum that
 \[\dim_{\mathbb{Z}_2}RFH_k\big(\#_mW_\veps,(S^{2n-1},\#_m\xi_a)\big)=\sum_{j=1}^m \dim_{\mathbb{Z}_2} RFH_k (W_\veps,(\Sigma_a,\xi_a))=2m.\]
 Note that $(W_\veps,\Sigma_a)$ satisfies (A) and (B). Hence, it follows from Lemma \ref{lembehaviourA} and \ref{lembehaviourB} that $\#_m(W_\veps,\Sigma_a)=(\#_mW_\veps,S^{2n-1})$ also satisfies (A) and (B).
\end{proof}
\begin{theo}[Main Theorem]\label{maintheo1}~\\
 Suppose that $\Sigma$ is a differentiable manifold, $\dim \Sigma =2n-1\geq 5$, which supports at least one fillable contact structure with filling for which the conditions (A) and (B) are true. Then $\Sigma$ satisfies at least one of the following alternatives:
 \begin{itemize}
  \item[a)] For every fillable contact structure $\xi$ on $\Sigma$ and any filling $W$ of $(\Sigma,\xi)$, which satisfies (A) and (B), holds true that
  \[\dim_{\mathbb{Z}_2} RFH_\ast(W,(\Sigma,\xi)) =\infty\qquad\qquad\forall\,\ast\in\mathbb{Z}\setminus[-n+1,n].\]
  \item[b)] There is (at least) one contact structure on $\Sigma$  for which there exist infinitely many different fillings.
  \item[c)] There exist infinitely many different fillable contact structures on $\Sigma$.
 \end{itemize}
\end{theo}
\begin{proof}~\\
 If $\Sigma$ satisfies case a) of the Main Theorem, then nothing has to be proven.\\
 If $\Sigma$ does not satisfy a), then there exists a contact structure $\xi$ with filling $W$ and a degree $k\in\mathbb{Z}\setminus[-n+1,n]$ such that
 \[b^\Sigma_k:=\dim_{\mathbb{Z}_2}RFH_k\big(W,(\Sigma,\xi)\big)<\infty.\]
 By Theorem \ref{theocontactonsphere}, we can find a fillable contact structure $\xi_0$ on $S^{2n-1}$ and a filling $W_0$ such that
 \[2\leq \dim_{\mathbb{Z}_2}RFH_k\big(W_0,(S^{2n-1},\xi_0)\big)=:b_k^0<\infty.\]
 By Theorem \ref{theoRFHinvsur} we know for the connected sum of $(\Sigma,\xi)$ and $(S^{2n-1},\xi_0)$ that
 \[2m\leq\dim_{\mathbb{Z}_2} RFH_k\Big(\big(W,(\Sigma,\xi)\big)\#\big(\#_m(W_0,(S^{2n-1},\xi_0))\big)\Big)=b^\Sigma_k+m\cdot b^0_k\leq\infty.\pagebreak\]
 Note that it follows from the fact that $S^{2n-1}$ is the standard sphere that $\Sigma\#\big(\#_mS^{2n-1})$ is diffeomorphic to $\Sigma$. Hence we get on $\Sigma$ an infinite number of contact structures $\xi_m:=\xi\#\big(\#_m\xi_0\big)$ each with an exact contact filling $W_m:=W\#\big(\#_m W_0\big)$.\\
 Now, we have two cases: If an infinite number of the contact structures $\xi_m$ is pairwise non-contactomorphic, then $\Sigma$ satisfies case c) of the Main Theorem and we are done.\\
 If an infinite number of the contact structures $\xi_m$ is contactomorphic to one contact structure $\xi_\infty$, then we find that the corresponding fillings $W_m$ of $\xi_\infty$ cannot be equal, as the groups $RFH_k\big(W_m,(\Sigma,\xi_\infty)\big)$
 are all different as their Betti-numbers $b^\Sigma_k+m\cdot b^0_k$ are all different. This implies that $\Sigma$ satisfies case b) of the Main Theorem.
\end{proof}
\begin{rem}
 One can sharpen the Main Theorem slightly by considering also the formal homotopy class $[\xi]$ of a contact structure $\xi$ (see \ref{sec1.1} for a definition). Morita showed in \cite{Mor} that there are only finitely many formal homotopy classes on the standard sphere $S^{4m+1}$ and countable infinitely many on $S^{4m+3}, m\geq 1$. His calculations also indicate that these homotopy classes on the standard sphere should form a group under connected sums with the homotopy class of the standard structure as neutral element. Hence by taking perhaps more connected sums, we can arrange in Theorem \ref{theocontactonsphere}  that the contact structure $\xi_0$ on $S^{4m+1}$ is in the standard formal homotopy class. When taking connected sums of $(S^{4m+1},\xi_0)$ with any contact manifold $(\Sigma,\xi)$, we then find that $\xi\#\xi_0$ is still in the same formal homotopy class as $\xi$.\\
 On $S^{4m+3}$, the situation is more complicated, as the group of formal homotopy classes is infinite. However, calculations made by Uebele, \cite{Ueb}, show that at least on $S^7, S^{11}$ and $S^{15}$ one can find fillable contact structures $\xi$ with fillings $W$ lying in the standard formal homotopy class such that
 \[0<\dim_{\mathbb{Z}_2}RFH_k\big(W,(S^{4m+3},\xi)\big)<\infty.\]
 I conjecture that the same holds true for any $S^{4m+3}$.
\end{rem}
Let us finish this section with some dynamical consequences if $\Sigma$ should satisfy $a)$ or $b)$ in the Main Theorem.
\begin{theo}\label{theo112}
 Let $(\Sigma,\xi)$ be a compact fillable contact manifold, $\dim\Sigma=2n-1$, such that for every $N\in\mathbb{N}$ there exists a filling $W_N$ satisfying (A) and (B) and a degree $k_N\in\mathbb{Z}$ with $|k_N|\geq 3n$ such that
 \[\dim_{\mathbb{Z}_2} RFH_{k_N}(W_N,(\Sigma,\xi))>N.\]
 Then it holds for any contact form $\alpha$ defining $\xi$ and satisfying (MB) that its Reeb field $R_\alpha$ has for every $L>0$ a simple closed Reeb trajectory whose period is greater than $L$.
\end{theo}
\begin{rem}~
 \begin{itemize}
  \item A closed Reeb trajectory $v$ of period $\eta$ is called simple if $v:[0,\eta)\rightarrow \Sigma$ is injective.
  \item Recall that according to Cieliebak and Frauenfelder, \cite{FraCie}, appendix B, condition (MB) is satisfied for any generic contact form, i.e.\ for a set of second category within all contact structures defining $\xi$.
 \end{itemize}
\end{rem}
\begin{proof}~
\begin{enumerate}
 \item Let $\alpha$ on $(\Sigma,\xi)$ be an arbitrary contact form defining $\xi$ and satisfying the Morse-Bott assumption (MB). Recall that the chain complex $RFC_k(W_N,\Sigma)$ is generated by all critical points of a Morse function $h$ on the critical manifold
 \[\crita = \bigcup_{\eta\in spec(\Sigma,\alpha)}\mathcal{N}^\eta.\]
 The index of such a critical point $c\in\mathcal{N}^\eta$ is given by
 \[\mu(c)=\mu_{CZ}(c)+ind_h(c)-\frac{1}{2}\dim_c\big(\crita\big)+\frac{1}{2}.\]
 If we assume that $c$ corresponds not to a simple trajectory but is an $l$-fold iteration of a shorter closed trajectory $c_0$, then we can estimate $\mu(c)$ with the help of the Iteration Formula for $\mu_{CZ}$ (Lemma \ref{lemitformula}) as follows
 \begin{equation}\label{meanindexesti}
 \begin{aligned}
  \mu(c)&= l\cdot \Delta(c_0)+R_c+ind_h(c)-\frac{1}{2}\dim_c\big(\crita\big)+\frac{1}{2}\\
  &= l\cdot \Delta(c_0) + C_c,\\
  \text{where }\qquad |C_c|&\leq |R_c|+\dim_c\big(\crita\big)-\frac{1}{2}\dim_c\big(\crita\big)+\frac{1}{2}\\
  &\leq 2n +\frac{1}{2}(2n-1)+\frac{1}{2}\\
  &=3n.
  \end{aligned}
 \end{equation}
 \item Assume that the period of every simple Reeb trajectory $c_0$ lies in the interval $[-L,L]$ for some $L>0$. As $\alpha$ satisfies (MB), we know by Theorem \ref{theo17} that there are only finitely many $\eta\in[-L,L]$ such that $\mathcal{N}^\eta\neq 0$. As the period of simple Reeb orbits is  bounded by $\pm L$, we know for every $\eta\in spec(\Sigma,\alpha)$ with $|\eta|>L$ that $\mathcal{N}^\eta$ consists solely of iterated trajectories. By considering perhaps the connected components of $\mathcal{N}^\eta$ separately, we find that there exists an $l\in\mathbb{Z}$ and a $\eta_0\in spec(\Sigma,\alpha)$ such that $\eta=l\eta_0$ and $|\eta_0|\leq L$.\\
 Note that $\mathcal{N}^\eta=\mathcal{N}^{\eta_0}$ as a submanifold of $\Sigma$, as the images of iterated trajectories coincide with the images of the underlying simple trajectories. This allows us to choose the same Morse function on $\mathcal{N}^\eta$ as on $\mathcal{N}^{\eta_0}$. Hence, we may assume for every $c\in crit(h)$, whose period is not in $[-L,L]$, that it is an iteration of a closed Reeb trajectory $c_0\in crit(h)$, whose period lies in $[-L,L]$.
 \item As $\Sigma$ is compact and there are only finitely many $\eta_0\in spec(\Sigma,\alpha)$ with $|\eta_0|\leq L$, we know that there are only finitely many critical points of $h$ whose period lies in $[-L,L]$. Let us number them $c_1,...\,,c_m$ and consider the set of all their absolute mean indices $D:=\{ |\Delta(c_1)|,...\,,|\Delta(c_m)|\}$. We define the number $\delta$ by
 \[\delta:=\begin{cases} 0 & \text{if } D=\{0\}\\ \min\big(D\setminus\{0\}\big) & \text{otherwise}\end{cases}.\]
 If $\delta=0$, then all mean indices $\Delta(c_i)$ are 0 and we obtain by (\ref{meanindexesti}) that the degree of all closed Reeb trajectories $c$ is bounded by $\pm 3n$. Hence, $RFH_k(W_N,(\Sigma,\xi))\neq 0$ can hold only for $|k|\leq 3n$, which contradicts the assumptions of the theorem.\\
 If $\delta>0$, then we can estimate for $|k|>3n$ the number of critical points $c\in crit(h)$ with degree $\mu(c)=k$ as follows. The index of an $l$-fold iterate $lc_j$ of $c_j$ is by (\ref{meanindexesti}) given as
 \[\mu(lc_j)=l\cdot\Delta(c_j)+C_{lc_j}.\]
 As $|C_{lc_j}|\leq 3n$, we can have  $\mu(lc_j)=k$ only if $l\cdot\Delta(c_j)\in[k-3n,k+3n]$. This is possible for at most
 \[\frac{(k+3n)-(k-3n)}{\Delta(c_j)}\leq\left\lceil\frac{6n}{\delta}\right\rceil\quad\text{numbers $l$.}\]
  Hence, there are at most $m\cdot\big\lceil 6n/\delta\big\rceil$ points $c\in crit(h)$ with $\mu(c)=k$. By assumption, we have that $\dim RFH_{k_N} (W_N,(\Sigma,\xi))>N$. By the construction of $RFH$, we know that there have to be at least $N$ different $c\in crit(h)$ generating $RFC_{k_N}(W_N,(\Sigma,\xi))$, i.e.\ where $\mu(c)=k_N$. However, $N$ can be arbitrarily large, which contradicts the fact that there are at most $m\cdot\big\lceil 6n/\delta\big\rceil$ such critical points.\qedhere
\end{enumerate}
\end{proof}
\begin{cor}\label{lastcor}~
 \begin{itemize}
  \item If $\,\Sigma$ satisfies alternative a) of the Main Theorem, then every fillable contact structure on $\Sigma$ has for any generic contact form simple Reeb trajectories of arbitrary length.
  \item If $\,\Sigma$ satisfies alternative b) but not a) of the Main Theorem, then there is at least one contact structure on $\Sigma$ which has simple closed Reeb trajectories of arbitrary length for every generic contact form.
 \end{itemize}
\end{cor}
\begin{proof}
 The first statement is a direct consequence of Theorem \ref{theo112}. For the second statement note that we showed in the proof of the Main Theorem that if $\Sigma$ satisfies b) but not a), then there exists a contact structure $\xi_{\infty}$ with infinitely many fillings $W_m$ and a degree $k$ such that
 \[\dim_{\mathbb{Z}_2} RFH_k\big(W_m,(\Sigma,\xi_\infty)\big)\geq 2m.\]
 Then, the second statement is again a direct consequence of Theorem \ref{theo112}.
\end{proof}
\begin{cor}\label{lastlastcor}
 Every Brieskorn manifold $\Sigma_a$ supports at least 2 non-contactomorphic, exactly fillable contact structures.
\end{cor}
\begin{proof}
 Note that on $\Sigma_a$ the length of a simple closed Reeb trajectory is bounded from above by $\big(\prod a_k\big)\cdot \pi/2$ for the standard contact from $\lambda_a$. Therefore, it follows from Corollary \ref{lastcor} that $\Sigma_a$ cannot satisfy alternative a) of the Main Theorem and if for $\Sigma_a$ holds c), then there are infinitely many different contact structures and we are done. If $\Sigma_a$ satisfies b) in the Main Theorem, then we know from Corollary \ref{lastcor} that it supports at least one contact structure $\xi$ that has simple closed Reeb trajectories of arbitrary length for any contact form satisfying (MB). However, as $\lambda_a$ satisfies (MB), we know that $\lambda_a$ cannot be a contact form for $\xi$, which shows that $\xi$ and $\xi_a$ have to be different.
\end{proof}
\begin{rem}~
\begin{itemize}
 \item In \cite{McLean}, Mark McLean showed that any manifold $\Sigma$, which supports at least one fillable contact structure with trivial Chern class, admits a contact structure $\xi$ with filling $W$ such that $\dim_{\mathbb{Z}_2} SH_\ast\big(W,(\Sigma,\xi)\big)=\infty$ for all $\ast$. Together with the long exact sequence (\ref{longexSHRFH}) we find that the same holds true for Rabinowitz-Floer homology. This gives again Corollary \ref{lastlastcor}.
 \item Using local Floer homology as in \cite{McLean} or \cite{GinGuer}, it should be possible to sharpen Corollary \ref{lastcor} so that there are infinitely many closed simple Reeb trajectories for every contact form, not just the ones which satisfy (MB).
 \item Using the mean Euler characteristic for $S^1$-equivariant symplectic homology $SH^{S^1}$, it should be possible to distinguish all the contact structures on Brieskorn manifolds $\Sigma_a$ that are obtained by our connected sum construction, provided that $\Sigma_a$ satisfies any form of index positivity. This should in particular be possible if $\sum a_k>1$. Consult \cite{BourOan2} or \cite{KwKo} for the mean Euler characteristic and its behaviour under handle-attachment.
\end{itemize}
\end{rem}

\newpage
\phantom{..}
\newpage
\appendix
\section{A perfect Morse function on $S^\ast S^{n-1}$}\label{app1}
In this appendix, we show the existence of a Morse-Smale pair $(\psi,g)$ of a Morse function $\psi$ and a metric $g$ on $S^\ast S^{n-1}$ with $\psi$ having exactly four critical points and $(\psi,g)$ being invariant under the reflection of the last $4n-4$ coordinates:
\[r  : \mathbb{R}^{2n}\rightarrow\mathbb{R}^{2n},\quad (x_1,\dots x_n\;;\;y_1,\dots,y_n)\mapsto (x_1,x_2,-x_3,\dots,-x_n\;;\;y_1,y_2,-y_3,\dots,-y_n).\]
In the first part, we construct the Morse function $\psi$ and calculate its critical points with their indices. In the second part, we construct $g$ on $S^\ast S^2$ and show that $(\psi,g)$ is Morse-Smale there. The third part finally contains the generalization of $g$ to higher dimensions.

\subsection{The Morse function $\mathbf{\psi}$}
This first part was (with some minor mistakes) already included in the author's diploma thesis. We repeat it here for completeness and to give a corrected version.\bigskip\\
We consider the unit tangent bundle $S^\ast S^{n-1}\subset\mathbb{R}^{2n}$ of the unit sphere $S^{n-1}$, i.e.\ the set
\begin{equation} \label{eq7}
  S^\ast S^{n-1} := \left\{z=(x,y)\in\mathbb{R}^n\times\mathbb{R}^n\,\Big|\, ||x||^2 = 1 = ||y||^2,\, \langle x,y\rangle = 0\right\}.
\end{equation}
The tangent space $T_zS^\ast S^{n-1}\subset\mathbb{R}^{2n}$ at a point $z=(x,y)\in S^\ast S^{n-1}$ is given by
\begin{align*}
 T_zS^\ast S^{n-1} &= \left\{(\xi_x,\xi_y)\in\mathbb{R}^n\times\mathbb{R}^n\,\Big|\, \langle\xi_x,x\rangle = 0 = \langle\xi_y,y\rangle, \langle\xi_x,y\rangle + \langle x,\xi_y\rangle = 0\right\}.
\end{align*}
We choose $a\in\mathbb{R}\setminus\{-1,0,1\}$ and define the function $\psi : \mathbb{R}^{2n}\rightarrow\mathbb{R}$ as follows:
 \begin{align*}
 z_a &:= (x_a;y_a)=\big(\underbrace{a,0,...\,,0}_{x_a}\;;\,\underbrace{0,1,0,...\,,0}_{y_a}\big)\\
  \psi(z) := \frac{1}{2}||z-z_a||^2 &\phantom{:}= \frac{1}{2}\Big(||x-x_a||^2+||y-y_a||^2\Big)\\
  &\phantom{:}=\frac{1}{2}\left((x_1-a)^2+x_2^2+y_1^2+(y_2-1)^2+\sum_{k=3}^n (x_k^2+y_k^2)\right).
 \end{align*}
\begin{prop}
 $\psi$ is a Morse function with four critical points, whose Morse indices are $0, n-2, n-1$ and $2n-3$. Moreover $\psi\circ r=\psi$. 
\end{prop}
\begin{proof} The fact that $\psi\circ r=\psi$ is obvious. The proof that $\psi$ is a Morse function is organized in two parts:\medskip\\
\textbf{$1^{\text{st}}$ claim:} $\psi$ has four critical points.
\begin{proof}
We calculate, that the gradient of $\psi$ on $\mathbb{R}^{2n}$ is given by
\[ \nabla_z \psi = (x_1-a,x_2,...\,,x_n\;\,;\;\,y_1,y_2-1,y_3,...\,,y_n)=(x,y)-(x_a,y_a).\]
Using the theorem of extrema with constraints, we find that in a critical point $(x,y)$ there exist real numbers $\alpha,\beta,\gamma\in\mathbb{R}$, such that with $k\geq 3$ holds
\begin{align*}
 x_1-a&=\alpha\cdot x_1+\gamma\cdot y_1 & y_1 &= \beta\cdot y_1+\gamma\cdot x_1\\
x_2 &= \alpha\cdot x_2 + \gamma\cdot y_2 & y_2-1 &= \beta\cdot y_2 + \gamma\cdot x_2\\
x_k &= \alpha\cdot x_k + \gamma\cdot y_k & y_k &= \beta\cdot y_k + \gamma\cdot x_k.
\end{align*}
These equations are equivalent to
\begin{align*}
 \text{I : } (1-\alpha)\cdot x_1 &= a + \gamma\cdot y_1& \text{IV : } (1-\beta)\cdot y_1 &= \gamma\cdot x_1 \\
 \text{II : } (1-\alpha)\cdot x_2 &= \gamma\cdot y_2 & \text{V : } (1-\beta)\cdot y_2 &= 1+ \gamma\cdot x_2 \\
 \text{III : } (1-\alpha)\cdot x_k &= \gamma\cdot y_k & \text{VI : } (1-\beta)\cdot y_k &= \gamma\cdot x_k.
\end{align*}
Using III and VI, we obtain
\[(1-\alpha)(1-\beta)\cdot x_k = \gamma^2\cdot x_k \quad\text{ and }\quad (1-\alpha)(1-\beta)\cdot y_k = \gamma^2\cdot y_k.\]
If $x_k\neq 0$ or $y_k\neq 0$ for any $k\geq 3$, we find that $(1-\alpha)(1-\beta)=\gamma^2$ and hence
\begin{align*}
 (1-\alpha)\cdot\text{V}-\gamma\cdot\text{II} \qquad& \Rightarrow & (1-\alpha) &= 0\hspace{4cm}\\
 \gamma\cdot\text{V}-(1-\beta)\cdot\text{II} \qquad &\Rightarrow & \gamma &= 0.
\end{align*}
Inserting this in I yields $a=0$, a contradiction. Therefore, we have $x_k=y_k = 0$ for all $k\geq 3$.
Hence, we are on the set 
\[U_0:=S^\ast S^{n-1}\cap (\mathbb{R}^2\times\mathbf{0})^2,\]
which is easily identified with $S^\ast S^1$. This manifold is the disjoint union of two circles and it follows that
\[y_2=\pm x_1,\quad x_2=\mp y_1,\quad x_1^2+y_1^2=1.\]
Inserting this in II and V yields
\begin{align*}
 \text{II}^\ast : (1-\alpha)\cdot y_1&=-\gamma\cdot x_1,& \text{V}^\ast : (1-\beta)\cdot\pm x_1&=1\mp\gamma\cdot y_1.
\end{align*}
Using these, we calculate that
\begin{align*}
 y_1\cdot\text{I\phantom{V}} + x_1\cdot\text{II}^\ast \qquad &\Rightarrow& -\gamma&=a\cdot y_1\\
 x_1\cdot\text{IV}\pm y_1\cdot\text{V}^\ast \qquad &\Rightarrow& -\gamma&=\;\;\mp y_1\hspace{4cm}.
\end{align*}
Thus $a\cdot y_1=\mp y_1$. As $a\neq \pm 1$, this implies that $y_1=x_2=0$ and hence that $x_1=\pm 1$ and $y_2=\pm 1$. Therefore, we have the following four critical points:
\begin{align*}
 z^+_+&= (+1,0,...\,,0\,;\,0,+1,0,...\,,0);\quad \psi(z^+_+)=\frac{(a-1)^2}{2}\\
 z^+_-&= (+1,0,...\,,0\,;\,0,-1,0,...\,,0);\quad \psi(z^+_-)=\frac{(a-1)^2}{2}+2\\
 z^-_+&= (-1,0,...\,,0\,;\,0,+1,0,...\,,0);\quad \psi(z^-_+)=\frac{(a+1)^2}{2}\\
 z^-_-&= (-1,0,...\,,0\,;\,0,-1,0,...\,,0);\quad \psi(z^-_-)=\frac{(a+1)^2}{2}+2.\qedhere
\end{align*}
\end{proof}
\textbf{$2^{\text{nd}}$ claim:} All four critical points are non-degenerate.
\begin{proof}
To prove this statement, we have to calculate the Hessian of $\psi$ at the 4 critical points. For this purpose, we need charts of $S^\ast S^{n-1}$. Recall that the inverse of the stereographic projection gives charts on $S^{n-1}$. These charts for the ``north-pole''$=(1,0,...\,,0)$ and the ``south-pole''$=(-1,0,...\,,0)$ are of the form:
\[u^{\pm}:\mathbb{R}^{n-1}\rightarrow S^{n-1},\quad u^{\pm}(x)=\frac{1}{1+||x||^2}\left(\pm 1\mp ||x||^2,2x_1,...\,,2x_{n-1}\right)^T.\]
Their differentials yield charts for the tangent bundle $TS^{n-1}$. Explicitly:
\[D_xu^{\pm} = \frac{1}{\rho(x)^2}
 \begin{pmatrix}
\mp 4x_1 & \mp 4x_2 & \dots &  \mp 4x_{n-1} \\
2\rho(x)-4x_1^2 & -4x_1x_2 & \dots & -4x_1x_{n-1}\\
-4x_2x_1 & 2\rho(x) - 4x_2^2 & \dots & -4x_2x_{n-1}\\
\vdots & & \ddots  & \vdots \\
-4x_{n-1}x_1 & -4x_{n-1}x_2 & \dots & 2\rho(x) - 4x_{n-1}^2                    
\end{pmatrix},\]
where $\rho(x):=1+||x||^2$. Short calculation shows $\left(Du^\pm\right)^T\left(Du^\pm\right) = \frac{4}{\rho(x)^2}\cdot Id$. This implies that the following map is an affine isometry for each $x\in\mathbb{R}^{n-1}$:
\[U^\pm(x):\mathbb{R}^{n-1}\rightarrow T_{u^\pm(x)}S^{n-1}\subset\mathbb{R}^n,\qquad U^\pm(x):=\frac{\rho(x)}{2}D_xu^\pm.\]
It follows that $U^\pm(x)\big( S^{n-2}\big)$ is the unit sphere in the tangent space $T_{u^\pm(x)}S^{n-1}$. Using the following charts given by stereographic projections:
\[v^\pm : \mathbb{R}^{n-2}\rightarrow S^{n-2}\subset\mathbb{R}^{n-1},\quad v^\pm(y)=\frac{1}{1+||y||^2}\left(\pm 1\mp||y||^2, 2y_1,...\,,2y_{n-2}\right)^T,\]
we obtain four charts around $z^+_+,z^+_-,z^-_+,z^-_-$ by
\begin{align*}
 w^\pm_\pm&:=u^\pm\times (U^\pm\cdot v^\pm) :\mathbb{R}^{n-1}\times\mathbb{R}^{n-2}\rightarrow S^\ast S^{n-1}, \text{ where}\\
w^+_+(0) &\phantom{:}= \left(u^+(0),U^+(0)\cdot v^+(0)\right) = z^+_+\\
w^+_-(0) &\phantom{:}= \left(u^+(0),U^+(0)\cdot v^-(0)\right) = z^+_-\\
w^-_+(0) &\phantom{:}= \left(u^-(0),U^-(0)\cdot v^+(0)\right) = z^-_+\\
w^-_-(0) &\phantom{:}= \left(u^-(0),U^-(0)\cdot v^-(0)\right) = z^-_-.
\end{align*}
Here, $U^\pm(x)\cdot v^\pm(y)$ denotes the matrix multiplication.\\
Now, we can express $\psi$ in these charts:
\begin{align*}
 \psi(w^\pm_\ast(x,y)) &= \frac{1}{2}\Big( ||u^\pm(x)-x_a||^2+
||U^\pm(x)\cdot v^\ast(y)-y_a||^2\Big)\\
&=\frac{1}{2}\Big(||u^\pm(x)||^2-2\left\langle u^\pm(x),x_a\right\rangle + ||x_a||^2\\
&\phantom{=}\quad+||U^\pm(x)\cdot v^\ast(y)||^2-2\left\langle U^\pm(x)\cdot v^\ast(y), y_a\right\rangle + ||y_a||^2\Big)\\
&=\frac{1}{2}\Big(1-2\left\langle u^\pm(x),x_a\right\rangle + a^2+1-2\left\langle v^\ast(y),\big( U^\pm(x)\big)^T\cdot y_a\right\rangle+1\Big)\\
&=\frac{1}{2}\Big(3+a^2+(\pm)2a\frac{||x||^2-1}{1+||x||^2}+(\ast)2\frac{||y||^2-1}{1+||y||^2}\left(1-\frac{2x_1^2}{1+||x||^2}\right)\\
&\phantom{=}\quad+\frac{8x_1}{1+||y||^2}\cdot\frac{\sum y_jx_{j+1}}{1+||x||^2}\Big).
\end{align*}
The third equation holds, since $||u^\pm(x)||=||v^\ast(y)|| = 1$. Here and above, $\ast\in\{+,-\}$ represents the choice of signs for $y$. Then, some calculation shows that
\begin{align*}
\frac{\partial^2(\psi\circ w^\pm_\ast)}{\partial x_k\,\partial y_l}(0,0) &= \quad0\qquad\qquad\quad \text{for any $k,l$}\\
\frac{\partial^2(\psi\circ w^\pm_\ast)}{\partial y_k\,\partial y_l}(0,0) &= \begin{cases}\ast \,4\phantom{a} & \qquad\quad\text{if } k=l\\ \phantom{\ast}\,0 & \qquad\quad\text{otherwise} \end{cases}\\
 \frac{\partial^2(\psi\circ w^\pm_\ast)}{\partial x_k\,\partial x_l}(0,0) &= \begin{cases}\phantom{\pm}0 & \text{if } k\neq l\\
 \pm 4a & \text{if } k=l\neq 1\\ \pm 4a+(\ast)4 & \text{if } k=l=1 \end{cases}. 
\end{align*}
 Thus we see that all four critical points $z^+_+,z^+_-,z^-_+,z^-_-$ are non-degenerate.
\end{proof}
If we assume that $a>1$, we obtain that
 \[ind_{\psi}(z^+_+)= 0,\quad ind_\psi(z^+_-)=n-2,\quad ind_\psi(z^-_+)=n-1,\quad ind_\psi(z^-_-)=2n-3.\qedhere\]
\end{proof}

\subsection{The metric $g$ on $S^\ast S^2$}
We remind the reader of the well-known fact that $S^\ast S^2$ is diffeomorphic to $\mathbb{R}P^3$. The latter space has a 2-covering by $S^3$. The corresponding covering of $S^\ast S^2$ can be explicitly obtained by restricting the following smooth map to $S^3$:
\begin{align*}
 \Phi=\begin{pmatrix}\Phi_1\\\Phi_2\end{pmatrix}: \mathbb{R}^4\rightarrow\mathbb{R}^6,\qquad
 \Phi(s)=\Phi\begin{pmatrix}s_1\\s_2\\s_3\\s_4\end{pmatrix}=\begin{pmatrix}\Phi_1(s)\\\Phi_2(s)\end{pmatrix}:=
 \begin{pmatrix} s_1^2 +s_2^2-s_3^2-s_4^2\\ \phantom{-}2(s_2s_3+s_1s_4)\\\phantom{-}2(s_1s_3-s_2s_4)\\ \phantom{-}2(s_2s_3-s_1s_4)\\ s_1^2+s_3^2-s_2^2-s_4^2\\-2(s_1s_2+s_3s_4)\end{pmatrix}.
\end{align*}
It is not difficult to see that $\Phi$ maps $S^3$ into $S^\ast S^2$. In fact, $\Phi_1$ and $\Phi_2$ restricted to $S^3$ yield both the Hopf-fibration and are orthogonal to each other. Moreover, $\Phi(-s)=\Phi(s)$.
\begin{lemme}
 $\Phi|_{S^3}$ is a differentiable 2-covering of $S^\ast S^2$.
\end{lemme}
\begin{proof}~\\
 \textbf{$1^{\text{st}}$ claim:} $\Phi: S^3\rightarrow S^\ast S^2$ is surjective and $\Phi^{-1}(x,y)$ contains two preimages for every point $(x,y)\in S^\ast S^2$.
 \begin{proof}
 Let $x\in S^2$ be an arbitrary point. Its preimage under the Hopf-fibration $\Phi_1^{-1}(x)$ is a circle in $S^3$. We parametrize this circle by the following path:
 \begin{align*}
  \gamma(\alpha)=\begin{pmatrix}\begin{matrix}\phantom{-}\cos \alpha & \sin \alpha \\ -\sin \alpha & \cos \alpha\end{matrix} & 0\\ 0 & \begin{matrix}\cos\alpha & -\sin \alpha\\ \sin\alpha & \phantom{-}\cos \alpha \end{matrix}\end{pmatrix}\begin{pmatrix}s_1\\s_2\\s_3\\s_4\end{pmatrix}.
 \end{align*}
 Here, $s=(s_1,s_2,s_3,s_4)\in S^3$ is a point with $\Phi_1(s)=x$. The image of this circle under $\Phi$ is given by
 \begin{align*}
  \Phi(\gamma(\alpha))&=\begin{pmatrix} \begin{pmatrix}s_1^2 +s_2^2-s_3^2-s_4^2\\ 2(s_2s_3+s_1s_4)\\2(s_1s_3-s_2s_4)\end{pmatrix}\\\cos(2\alpha)\begin{pmatrix}\phantom{-}2(s_2s_3-s_1s_4)\\ s_1^2+s_3^2-s_2^2-s_4^2\\-2(s_1s_2+s_3s_4)\end{pmatrix}+\sin(2\alpha)\begin{pmatrix}-2(s_2s_4+s_1s_3)\\\phantom{-}2(s_1s_2-s_3s_4)\\s_1^2+s_4^2-s_2^2-s_3^2\end{pmatrix}\end{pmatrix}\\
  &=\Bigg.\begin{pmatrix}\Phi_1(s)\\\cos(2\alpha)\cdot\Phi_2(s)+\sin(2\alpha)\cdot\big(\Phi_1(s)\times\Phi_2(s)\big)\end{pmatrix}.
 \end{align*}
Here, $\times$ denotes as usual the cross-product in $\mathbb{R}^3$. Note that while the first part of $\Phi$ maps $\gamma(\alpha)$ to $x=\Phi_1(s)$, the second part of $\Phi$ runs twice through the unit circle of vectors orthogonal to $x$. Thus, we see that $\Phi$ is surjective and the preimage $\Phi^{-1}(x,y)$ for $(x,y)\in S^\ast S^2$ contains exactly two points $s$ and $-s$.
 \end{proof}
 \textbf{$2^{\text{nd}}$ claim:} $\Phi$ is a local diffeomorphism.
 \begin{proof}
  The differential of $\Phi$ is given by
  \[ D\Phi = 2\cdot\begin{pmatrix}s_1& s_2 & -s_3 & -s_4\\s_4 &  s_3 & s_2 & s_1\\s_3 & -s_4 & s_1 & -s_2\\ -s_4 & s_3 & s_2 & -s_1\\ s_1 & -s_2 & s_3 & -s_4\\ -s_2 & -s_1 & -s_4 & -s_3 \end{pmatrix}.\]
  Recall that an orthonormal basis of $T_sS^3$ at a point $s=(s_1,s_2,s_3,s_4)\in S^3$ is given by
  \begin{align*}
   v_1(s)&:=(s_2,-s_1,-s_4,s_3)^T\\
   v_2(s)&:=(-s_3,-s_4,s_1,s_2)^T\\
   v_3(s)&:=(-s_4,s_3,-s_2,s_1)^T.
  \end{align*}
An easy calculation shows that
\begin{equation}\label{Dphibasis}\begin{aligned}
 D\Phi(v_1)=2\cdot\begin{pmatrix}0\\\Phi_1(s)\times\Phi_2(s)\end{pmatrix}&,\qquad\qquad
 D\Phi(v_2)=2\cdot\begin{pmatrix}\Phi_1(s)\times\Phi_2(s)\\0\end{pmatrix},\\
 D\Phi(v_3)&=2\cdot\begin{pmatrix}\phantom{-}\Phi_2(s)\\-\Phi_1(s)\end{pmatrix}. \end{aligned}
\end{equation}
The image of the orthonormal basis $\{v_1,v_2,v_3\}$ of $T_sS^3$ under $D\Phi$ is hence an orthogonal basis of $T_{\Phi(s)}S^\ast S^2$ and $D\Phi_{TS^3}$ is therefore pointwise an isomorphism. It follows that $\Phi$ is a local diffeomorphism. These calculations also show that the pushforward $\Phi_\ast g_{S^3}$ of the standard metric $g_{S^3}$ on $S^3$ is not a multiple of the standard metric on $S^\ast S^2$ coming from $\mathbb{R}^{2n}$ (as $||D\Phi(v_3)||^2=8$, while $||D\Phi(v_1)||^2=4$).
 \end{proof}
 It follows from claim 1 and 2 and the fact that $\Phi(-s)=\Phi(s)$ that $\Phi$ is a 2-covering of $S^\ast S^2$. Moreover, it induces a diffeomorphism between $\mathbb{R}P^3=\raisebox{.2em}{$S^3$}\hspace{-.2em}\Big/\hspace{-.2em}\raisebox{-.2em}{$\scriptstyle\sim$}$ and $S^\ast S^2$.
\end{proof}
The metric $g$ on $S^\ast S^2$ mentioned in the introduction is the pushforward $\Phi_\ast g_{S^3}$. In order to see that $(\psi,g)$ satisfies the Morse-Smale condition, we first consider on $S^3$ the function
\[f:\mathbb{R}^4\supset S^3\rightarrow\mathbb{R},\qquad f(s)=A_1\cdot s_1^2+A_2\cdot s_2^2+A_3\cdot s_3^2 + A_4\cdot s_4^2,\]
where $A_i>0$ are pairwise different positive real numbers. Note that $f(-s)=f(s)$, so that $f$ induces a well-defined function $\Phi_\ast f$ on $S^\ast S^2$. Easy calculations show that for
\begin{equation}\label{valuesofAi}
  A_1=\frac{(a-1)^2}{2},\quad A_2=\frac{(a-1)^2}{2}+2,\quad A_3=\frac{(a+1)^2}{2},\quad A_4=\frac{(a+1)^2}{2}+2
\end{equation}
the functions $\psi$ and $\Phi_\ast f$ do coincide on $S^\ast S^2$.
\begin{lemme}
 For $A_i$ pairwise different positive real numbers, the function $f$ is a Morse function on $S^3$ having the following 8 critical point
 \[c_1^\pm=(\pm1,0,0,0),\;c_2^\pm=(0,\pm1,0,0),\;c_3^\pm=(0,0,\pm1,0),\;c_4^\pm=(0,0,0,\pm1).\]
 Moreover, $f$ and the standard metric $g_{S^3}$ on $S^3$ are Morse-Smale, i.e.\ the stable and unstable manifolds of the critical points intersect transversally.
\end{lemme}
\begin{proof}~
 \begin{itemize}
  \item \underline{critical points}\\
  The gradient $\nabla^\mathbb{R}$ of $f$ on $\mathbb{R}^4$ is given by $\;\, \nabla^\mathbb{R} f=2\cdot\big(A_1s_1,A_2s_2,A_3s_3,A_4s_4\big)^T$.\\
  The tangent space $T_sS^3$ at $s\in S^3$ is given by $\;\, T_sS^3=\big\{\xi\in\mathbb{R}^4\,\big|\,\big\langle\xi,s^T\rangle=0\big\}$.\\
  The gradient $\nabla f:=\nabla^{S^3}f$ of $f$ on $S^3$ is therefore given by
  \begin{align*}
   \nabla f&=2\begin{pmatrix}A_1s_1\\A_2s_2\\A_3s_3\\A_4s_4\end{pmatrix}-\left\langle2\begin{pmatrix}A_1s_1\\A_2s_2\\A_3s_3\\A_4s_4\end{pmatrix}\,,\,\begin{pmatrix}s_1\\s_2\\s_3\\s_4\end{pmatrix}\right\rangle\cdot\begin{pmatrix}s_1\\s_2\\s_3\\s_4\end{pmatrix}=
   \end{align*}
   \begin{align*}
   &=2\begin{pmatrix}\Big.s_1\cdot\big(A_1-(A_1s_1^2+A_2s_2^2+A_3s_3^2+A_4s_4^2)\big)\\\Big.s_2\cdot\big(A_2-(A_1s_1^2+A_2s_2^2+A_3s_3^2+A_4s_4^2)\big)\\\Big.s_3\cdot\big(A_3-(A_1s_1^2+A_2s_2^2+A_3s_3^2+A_4s_4^2)\big)\\\Big.s_4\cdot\big(A_4-(A_1s_1^2+A_2s_2^2+A_3s_3^2+A_4s_4^2)\big)\end{pmatrix}.
  \end{align*}
Assuming for example $0<A_1<A_2<A_3<A_4$, it is not difficult to see that $\nabla f=0$ only in the 8 points $c_1^\pm, c_2^\pm, c_3^\pm, c_4^\pm$ (Hint: Use that $s_1^2+s_2^2+s_3^2+s_4^2=1$).
\item \underline{non-degeneracy}\medskip\\
Around the critical points $c_1^\pm, c_2^\pm, c_3^\pm, c_4^\pm$ we have the following charts in $S^3$ coming from stereographic projection:
\begin{align*}
 u_1^\pm:\mathbb{R}^3&\rightarrow S^3,& u_1^\pm(x)&=\frac{1}{1+||x||^2}\big(\pm1\mp||x||^2, 2x_1, 2x_2, 2x_3\big)\\
 u_2^\pm:\mathbb{R}^3&\rightarrow S^3,& u_2^\pm(x)&=\frac{1}{1+||x||^2}\big(2x_1, \pm1\mp||x||^2, 2x_2, 2x_3\big)\\
 u_3^\pm:\mathbb{R}^3&\rightarrow S^3,& u_3^\pm(x)&=\frac{1}{1+||x||^2}\big(2x_1, 2x_2, \pm1\mp||x||^2, 2x_3\big)\\
 u_4^\pm:\mathbb{R}^3&\rightarrow S^3,& u_4^\pm(x)&=\frac{1}{1+||x||^2}\big(2x_1, 2x_2, 2x_3, \pm1\mp||x||^2 \big)
\end{align*}
Using these charts, easy calculations show that
\[\frac{\partial ^2 f(u^\pm_i)}{\partial x_j \partial x_k}(0)=\begin{cases}0 & \text{if } j\neq k\\ 8A_j-8A_i & \text{if } j=k<i\\8A_{j+1}-8A_i & \text{if } j=k\geq i\end{cases}.\]
Hence, if the $A_i$ are pairwise different, we see that all critical points are non-degenerate.
\item \underline{Morse-Smale}\medskip\\
As above, we consider $S^3$ equipped with the standard metric coming from $\mathbb{R}^4$. The ordinary differential equation $\dot{\gamma}=\nabla f$ for a gradient trajectory $\gamma$ reads as
\begin{equation}\label{gradientODE}\begin{aligned}
 \dot{\gamma_1} &= 2s_1\cdot\big(A_1-(A_1s_1^2+A_2s_2^2+A_3s_3^2+A_4s_4^2)\big)=:2s_1\cdot B_1\\
 \dot{\gamma_2} &= 2s_2\cdot\big(A_2-(A_1s_1^2+A_2s_2^2+A_3s_3^2+A_4s_4^2)\big)=:2s_2\cdot B_2\\
 \dot{\gamma_3} &= 2s_3\cdot\big(A_3-(A_1s_1^2+A_2s_2^2+A_3s_3^2+A_4s_4^2)\big)=:2s_3\cdot B_3\\
 \dot{\gamma_4} &= 2s_4\cdot\big(A_4-(A_1s_1^2+A_2s_2^2+A_3s_3^2+A_4s_4^2)\big)=:2s_4\cdot B_4. \end{aligned}
\end{equation}
Let $\gamma$ be a solution of (\ref{gradientODE}). As $S^3$ is compact without boundary, we know that $\gamma$ converges asymptotically at both ends to a critical point of $f$. Assuming $A_4>A_3>A_2>A_1>0$ as in (\ref{valuesofAi}), we find with $s_1^2+s_2^2+s_3^2+s_4^2=1$ that $B_4\geq 0$, where $B_4=0$ exactly for $s=(0,0,0,\pm 1)$. So if for a time $t_0$ holds that $\gamma_4(t_0)\neq 0$, then $\gamma_4(t)$ strictly increases/decreases in $t$ to $\pm 1=sign\,\gamma_4(t_0)$ and hence $\displaystyle\lim_{t\rightarrow\infty}\gamma(t)=(0,0,0,\pm1)=(0,0,0,sign\,\gamma_4(t_0))$.\pagebreak\\
If $s_4=0$, then $B_3\geq 0$, where $B_3=0$ exactly for $s=(0,0,\pm1,0)$. So if $\gamma_4(t)=0$ for all $t$ and if $\gamma_3(t_0)\neq 0$ for one $t_0$, then the same reasoning shows that $\displaystyle\lim_{t\rightarrow\infty}\gamma(t)=(0,0,\pm1,0)=(0,0,sign\,\gamma_3(t_0),0)$.\\
Analogously, we find that if $\gamma_4(t)=\gamma_3(t)=0$ for all $t$ and if $\gamma_2(t_0)\neq0$ for one $t_0$, then $\displaystyle\lim_{t\rightarrow\infty}=(0,\pm1,0,0)=(0,sign\,\gamma_2(t_0),0,0)$.\\
In complete analogy, we show that if $\gamma_1(t_0)\neq 0$ for one $t_0$, then $\displaystyle\lim_{t\rightarrow -\infty}\gamma(t)=(\pm1,0,0,0)=(sign\,\gamma_1(t_0),0,0,0)$. The same, if $\gamma_1(t)=0$ for all $t$, but $\gamma_2(t_0)\neq 0$, then $\displaystyle\lim_{t\rightarrow -\infty}\gamma(t)=(0,\pm1,0,0)=(0,sign\,\gamma_2(t_0),0,0)$ and finally if $\gamma_1(t)=\gamma_2(t)=0$ for all $t$, but $\gamma_3(t_0)\neq 0$, then $\displaystyle\lim_{t\rightarrow-\infty}\gamma(t)=(0,0,\pm1,0)=(0,0,sign\,\gamma_3(t_0),0)$.\\
This allows us to read off the stable and unstable manifolds as follows
\begin{align*}
 W^s(c_1^\pm)&=\left\{s\in S^3\,\left|\, s=(\pm1,0,0,0)\right.\right\},\\
 W^s(c_2^\pm)&=\left\{s\in S^3\,\left|\, s_4=s_3=0, sign(s_2)=\pm1\right.\right\},\\
 W^s(c_3^\pm)&=\left\{s\in S^3\,\left|\, s_4=0, sign(s_3)=\pm1\right.\right\},\\
 W^s(c_4^\pm)&=\left\{s\in S^3\,\left|\, sign(s_4)=\pm1\right.\right\},\\
 W^u(c_1^\pm)&=\left\{s\in S^3\,\left|\,sign(s_1)=\pm 1\right.\right\},\\ 
 W^u(c_2^\pm)&=\left\{s\in S^3\,\left|\,s_1=0, sign(s_2)=\pm 1\right.\right\},\\
 W^u(c_3^\pm)&=\left\{s\in S^3\,\left|\,s_1=s_2=0, sign(s_3)=\pm 1\right.\right\},\\
 W^u(c_4^\pm)&=\left\{s\in S^3\,\left|\,s=(0,0,0,\pm1)\right.\right\}.
\end{align*}
It is now obvious, that all stable and unstable manifolds intersect transversally.\qedhere
 \end{itemize}
\end{proof}
We have seen that $(f,g_{S^3})$ is a Morse-Smale pair on $S^3$ and that $\Phi : S^3\rightarrow S^\ast S^2$ is a local diffeomorphism. This implies that $(\Phi_\ast f, \Phi_\ast g_{S^3})$ is also a Morse-Smale pair on $S^\ast S^2$.\\
To conclude this section, we note that $(f,g_{S^3})$ is invariant under the reflection
\[\mathfrak{r}:\mathbb{R}^4\rightarrow\mathbb{R}^4, (s_1,s_2,s_3,s_4)\mapsto (s_1,-s_2,-s_3,s_4).\]
Note that $\mathfrak{r}$ on $S^3$ is conjugate via $\Phi$ to the reflection $r$ on $S^\ast S^2$, as defined in the introduction of this appendix. It folows that $(\Phi_\ast f, \Phi_\ast g_{S^3})$ is invariant under $r$. Moreover, for the $A_i$ chosen as in (\ref{valuesofAi}) such that $\Phi_\ast f=\psi$, we find that with respect to the metric $\Phi_\ast g_{S^3}$ there are exactly two gradient trajectories between each pair $z_-^-$ and $z_+^-$, $z_+^-$ and $z_-^+$ and $z_-^+$ and $z_+^+$ corresponding to the four gradient trajectories between $(0,0,0,\pm1)$ and $(0,0,\pm1,0)$, $(0,0,\pm1,0)$ and $(\pm1,0,0,0)$ and $(\pm1,0,0,0)$ and $(0,\pm1,0,0)$ respectively. These latter gradient trajectories can be read off as the 1-dimensional intersections of the stable and unstable manifolds given above.

\subsection{The metric $g$ on $S^\ast S^{n-1}$}
Recall from the calculations in (\ref{Dphibasis}) that $D\Phi$ maps an orthonormal basis of $T_sS^3$ to an orthogonal basis of $T_{\Phi(s)}S^\ast S^2$, where two vectors have length 2 and one length $2\sqrt{2}$. At a point $(x,y)\in S^\ast S^2$, the latter vector, $D\Phi(v_3)$, is given by $2\cdot(y,-x)^T$.\pagebreak\\
Note that we can define a global vector field $X$ on $\mathbb{R}^{2n}$ by
\[X(x,y):=(y,-x)^T.\]
Let $\big(\mathbb{R}X\big)^\perp$ denote the orthogonal complement of $\mathbb{R}X$ with respect to the standard metric $g_{std}$ on $\mathbb{R}^{2n}$. Then, we have for $(x,y)\neq0$ the splitting
\[\mathbb{R}^{2n}=\mathbb{R}\oplus\big(\mathbb{R}X\big)^\perp.\]
We define a Riemannian metric $g$ on $\mathbb{R}^{2n}\setminus\{0\}$ by requiring that 
\begin{equation}\label{metriconSS}
 g(X,X)=\frac{1}{8}, \quad g\big|_{(\mathbb{R}X)^\perp}=\frac{1}{4}g_{std}\big|_{(\mathbb{R}X)^\perp}\quad\text{and}\quad g(X,Y)=0\quad\forall Y\in\big(\mathbb{R}X\big)^\perp.
\end{equation}
Then, $g$ coincides with $\Phi_\ast g_{S^3}$ in the following sense: Consider for $i\geq 3$ the sets
\begin{align*}
 S^\ast S^2_i&:=S^\ast S^{n-1}\cap(\mathbb{R}^2\times\mathbf{0}\times\underbrace{\mathbb{R}}_{i^{\text{th}}-coord.}\times\mathbf{0})^2\\
 &\phantom{:}=\left\{(x,y)\in S^\ast S^{n-1}\,\left|\,x_j=y_j=0 \text{ for } j\geq 3, j\neq i\right.\right\}.
\end{align*}
They are easily identified with $S^\ast S^2$. Let us denote by $\Phi_i:S^3\rightarrow S^\ast S^2_i$ the maps given by $\Phi$ composed with this identification. Then we have that $(\Phi_i)_\ast (g_{S^3})=g|_{TS^\ast S^2_i}$.\medskip\\
In the remainder of this section we show that $(\psi,g)$ is Morse-Smale on $S^\ast S^{n-1}$. First, we calculate the gradient $\nabla_g^\mathbb{R} \psi$ of $\psi$ on $\mathbb{R}^{2n}$ with respect to $g$. Let $\xi\in T_z\big(\mathbb{R}^{2n}\setminus\{0\}\big)$ be any tangent vector and let $\xi=\xi_X+\xi_\perp$ be its  decomposition with respect to the splitting $\mathbb{R}X\oplus\big(\mathbb{R}X\big)^\perp$. Let $\nabla^\mathbb{R} \psi$ be the gradient of $\psi$ with respect to $g_{std}$ and let $\nabla^\mathbb{R} \psi=\nabla^\mathbb{R} \psi_X + \nabla^\mathbb{R} \psi_\perp$ be its decomposition. Then
\begin{align*}
 d\psi(\xi)\overset{\phantom{(58)}}{=}g_{std}(\nabla^\mathbb{R} \psi,\xi) &\overset{\phantom{(58)}}{=}g_{std}(\nabla^\mathbb{R} \psi_X,\xi_X) + g_{std}(\nabla^\mathbb{R} \psi_\perp,\xi_\perp)\\
 &\overset{(\ref{metriconSS})}{=} 8 g(\nabla^\mathbb{R} \psi_X,\xi_X) + 4 g(\nabla^\mathbb{R} \psi_\perp,\xi_\perp)\\
 &\overset{(\ref{metriconSS})}{=} 4 g(2\nabla^\mathbb{R} \psi_X + \nabla^\mathbb{R}\psi_\perp,\xi_X+\xi_\perp)\\
 &\overset{\phantom{(58)}}{=} 4 g(\nabla^\mathbb{R}\psi + \nabla^\mathbb{R}\psi_X,\xi).
\end{align*}
Thus, we find that $\nabla_g^\mathbb{R}\psi = 4(\nabla^\mathbb{R} \psi +\nabla^\mathbb{R} \psi_X)$. To get the gradient $\nabla_g \psi$ of $\psi$ on $S^\ast S^{n-1}$, we now have to project $\nabla^\mathbb{R}_g\psi$ orthogonally (with respect to $g$) to $TS^\ast S^{n-1}$. In other words, we calculate
\begin{align*}
 \nabla_g \psi = \nabla_g^\mathbb{R}\psi 
 - g\left(\Big.\nabla_g^\mathbb{R}\psi,{\textstyle\binom{x}{0}}\right)\frac{\binom{x}{0}}{\left|\left|\binom{x}{0}\right|\right|^2_g}
 - g\left(\nabla_g^\mathbb{R}\psi,{\textstyle\binom{0}{y}}\right)\frac{\binom{0}{y}}{\left|\left|\binom{x}{0}\right|\right|^2_g}
 - g\left(\Big.\nabla_g^\mathbb{R}\psi,{\textstyle\binom{y}{x}}\right)\frac{\binom{y}{x}}{\left|\left|\binom{y}{x}\right|\right|^2_g}.
\end{align*}
As $(x,0)^T,\, (0,y)^T$ and $(y,x)^T$ are all orthogonal to $X=(y,-x)^T$ for $(x,y)\in S^\ast S^{n-1}$, we can replace in the above equation $g$ by $\frac{1}{4}g_{std}$. Recalling that $\nabla^\mathbb{R}\psi$ was given by
\[\nabla^\mathbb{R}\psi=(x_1-a,x_2,\dots,x_n\;;\;y_1,y_2-1,y_3,\dots,y_n)^T=(x,y)^T-(x_a,y_a)^T,\]
we then calculate
\begin{align*}
 \nabla_g\psi=4\left({\textstyle\binom{x}{y}}-{\textstyle\binom{x_a}{y_a}} + \frac{x_2-ay_1}{2}{\textstyle\binom{y}{-x}}-(1-ax_1){\textstyle\binom{x}{0}} -(1-y_2) {\textstyle\binom{0}{y}}-\frac{-ay_1-x_2}{2}{\textstyle\binom{y}{x}}\right).
\end{align*}
Next, we show that all solutions of $\dot{\gamma}=\nabla_g\psi$ with $\displaystyle\lim_{t\rightarrow\infty}\gamma(t)=z^-_+$ and $\displaystyle\lim_{t\rightarrow-\infty}\gamma(t)=z^+_-$ lie in the region $U_0$, where $x_k=y_k=0$ for $k\geq 3$. To this purpose, consider the Lyapunov function
\[F(x,y):=(x_1-1)^2+\sum_{i\neq 1} {x_i}^2 + (y_2-1)^2+\sum_{i\neq 2}{y_i}^2.\]
We shall see that $F$ increases along $\gamma$, unless the trajectory lies entirely in $U_0$. As $F(z^+_-)=F(z^-_+)=4$, this proves the statement. For the derivative of $F$ along $\gamma$, we have
\begin{align*}
\nabla F&=(x,y)^T-(1,0,\dots,0\;;\;0,1,0,\dots,0)^T\\
 \frac{d}{dt}F(\gamma(t))=dF(\dot{\gamma}(t))&=g_{std}(\nabla F,\nabla_g \psi)\\
 &= 8\cdot(a-ax_1^2-x_2y_1+1-y_2^2-ay_1x_2)\\
 &= 8a(1-x_1^2-y_1x_2)+8(1-y_2^2-y_1x_2).
\end{align*}
This is positive (so that $F$ increases along $\gamma$) if $h_1(x,y):=1-x_1^2-y_1x_2\geq 0$ and $h_2(x,y):=1-y_2^2-y_1x_2\geq 0$ and at least one of them is truly positive. By the theorem of extrema with constrains, we find that $h_1$ has a critical point on $S^\ast S^{n-1}$ if there exist real numbers $\alpha, \beta,\gamma$ such that
\begin{align*}
 &\text{\phantom{II}I} :&  -2x_1 &= \alpha \cdot x_1+\gamma\cdot y_1& &\text{IV} :&  -x_2 &= \beta\cdot y_1+\gamma\cdot x_1\\
 &\text{\phantom{I}II} :&  -y_1 &= \alpha\cdot x_2 + \gamma\cdot y_2& &\text{\phantom{I}V} :&  0 &= \beta\cdot y_2 + \gamma \cdot x_2\\
 &\text{III} :&  0 &= \alpha\cdot x_k + \gamma\cdot y_k& 
 &\text{VI} :&  0 &= \beta\cdot y_k + \gamma\cdot x_k,
\end{align*}
where $k\geq 3$. From III and VI follows
\[\alpha\beta\cdot x_k=-\beta\gamma\cdot y_k=\gamma^2\cdot x_k,\qquad \alpha\beta \cdot y_k=-\alpha\gamma\cdot x_k=\gamma^2\cdot y_k.\]
Hence, whenever $x_k\neq 0$ or $y_k\neq 0$ for at least one $k\geq 3$, i.e.\ if we are not on $U_0$, we have $\alpha\beta=\gamma^2$. Then II and V yield
\[-\beta\cdot y_1=\alpha\beta\cdot x_1+\beta\gamma y_2 = \gamma^2 x_2-\gamma^2 x_2=0.\]
So if we are not on $U_0$, then either $y_1=0$ or $(\beta=0)\,\Rightarrow\, (\gamma=0)\,\Rightarrow\, (x_2=0)$. But for these points we have $h_1=1 - x_1^2\geq 0$ with equality holding for $x_1=\pm 1$. If we are on $U_0$, then $x_2=\pm y_1$ and $x_1^2+y_1^2=0$. Then we also have $h_1\geq 0$ with equality holding for $y_1=x_2$. As $h_1\geq 0$ holds for its critical points and $S^\ast S^{n-1}$ is compact, we find that $h_1\geq 0$ everywhere with equality for $x_1=\pm 1$ or on $U_0$ with $y_1=x_2$.\\
Analog calculations show that $h_2\geq 0$ everywhere on $S^\ast S^{n-1}$ with equality holding for $y_2=\pm1$ or on $U_0$ with $y_1=x_2$. Hence $\frac{d}{dt}F(\gamma(t))\geq 0$ with equality holding only if $\gamma(t)\in U_0$.\medskip\\
Recall that $U_0$ is diffeomorphic to $S^\ast S^1$ -- the disjoint union of two circles, one of them containing $z^+_-$ and $z^-_+$, the other containing $z^+_+$ and $z^-_-$. Moreover, $U_0\subset S^\ast S^2_i$ for all $i$. Hence, we know that there are exactly two gradient trajectories between $z^+_-$ and $z^-_+$ in $U_0$, corresponding to the four gradient trajectories between $(\pm1,0,0,0)$ and $(0,0,\pm1,0)$ on $S^3$ via the maps $\Phi_i$. As there are no gradient trajectories connecting $z^+_-$ and $z^-_+$ outside $U_0$, we know that these two are the only ones.\\
Therefore, we know that $W^u(z^+_-)$ and $W^s(z^-_+)$ intersect only in $U_0$. In order to show that $(\psi,g)$ is Morse-Smale, it suffices to show that these spaces intersect transversally there, as all other intersections of stable/unstable manifolds are trivially transversal due to dimensions. Let $\gamma$ be one of the two trajectories in $U_0$. Then the discussion in the previous part shows that $T_{\gamma(t)}W^u(z^+_-)$ and $T_{\gamma(t)}W^s(z^-_+)$ do span the subspaces $T_{\gamma(t)}S^\ast S^2_i$ for every $3\leq i\leq n$. As these subspaces span the whole tangent space $T_{\gamma(t)}S^\ast S^{n-1}$, we know that $T_{\gamma(t)}W^u(z^+_-)$ and $T_{\gamma(t)}W^s(z^-_+)$ do span the whole space as well. Hence $W^u(z^+_-)\pitchfork W^s(z^-_+)$.\bigskip\\
As a nice bonus we get the Morse-homology of $S^\ast S^{n-1}$ with $\mathbb{Z}_2$-coefficients.
\begin{cor}
\begin{align*}
 &\text{If } n\geq 3, \text{ then:}& H_k(S^\ast S^{n-1},\mathbb{Z}_2) &= \begin{cases} \mathbb{Z}_2 & \quad\text{ if } k\in\{0,n-2,n-1,2n-3\}\\ 0 & \quad\text{ otherwise }\end{cases}.\\
 &\text{If } n= 2, \text{ then:}& H_k(S^\ast S^1,\mathbb{Z}_2) &= \begin{cases} \left(\mathbb{Z}_2\right)^2 & \text{ if } k\in\{0,1\}\\ 0 & \text{ otherwise }\end{cases}.
\end{align*}
\end{cor}
\newpage
\phantom{..}
\newpage
\section{Some properties of convolutions}\label{appB}
If $f, g$ are functions on $\mathbb{R}$, then their convolution $f\ast g$, if existent, is defined by
\[(f\ast g)(x):=\int_{-\infty}^{\infty}f(x-y)\cdot g(y) \,dy = \int_{-\infty}^{\infty}f(y)\cdot g(x-y) \,dy=(g\ast f)(x).\]
We fix a smooth non-negative function $\rho:\mathbb{R}\rightarrow\mathbb{R}$ such that
\[\text{supp}\, \rho\subset B_1(0),\quad \rho(x)=\rho(-x)\qquad\text{ and }\qquad \int_{-\infty}^\infty \rho(x) \,dx = 1\]
and define for any $\delta>0$ the function $\rho_\delta$  by $\quad\rho_\delta(x):= \frac{1}{\delta}\rho\left(\frac{1}{\delta}x\right)$.\medskip\\
For any $g\in L^p$ the convolution $\rho_\delta\ast g$ is well-defined and smooth, since
\begin{align*}
 && \int_{-\infty}^\infty \rho_\delta(x-y)\cdot g(y)\,dy &= \int_{x-\delta}^{x+\delta}\rho_\delta(x-y)\cdot g(y)\,dy\\
 &\text{and}& \frac{d}{dx}\big(\rho_\delta\ast g\big)(x) &= \int_{-\infty}^\infty\left(\frac{d}{dx}\rho_\delta(x-y)\right)\cdot g(y)\,dy = \Big(\big(\frac{d}{dx}\rho_\delta\big)\ast g\Big)(x).
\end{align*}
Note that if $\text{supp}\,g\subset[a,b]$ then $\text{supp}(\rho_\delta\ast g)\subset[a-\delta,b+\delta]$.
\begin{lemme}[cf. \cite{Dieu}\footnote{The proof in \cite{Dieu} is done only for $p=1,2$. However, we apply this lemma solely for $p=2$ anyway.}, 14.10.6, or \cite{Adams}, Lem.\ 2.18]\label{lemconvconv}~\\
 For any $g\in L^p$ holds that $\rho_\delta\ast g$ converges in $L^p$ to $g$ as $\delta\rightarrow 0$.
\end{lemme}
Observe that if $f$ is a $T$-periodic function, then $g\ast f$ is also $T$-periodic for any $g$, since
\[(f\ast g)(x+T)=\int_{-\infty}^\infty f(x+T-y)\cdot g(y)\,dy= \int_{-\infty}^\infty f(x-y)\cdot g(y)\,dy= (f\ast g)(x).\]
\begin{lemme}\label{lemconvolution}
 Let $f,g$ be two 1-periodic functions in $L^2$ and let $r_\delta\in L^2$ be such that $r_\delta(x)=r_\delta(-x)$ and $\text{\emph{supp}}\,r_\delta\subset [-\delta,\delta]$ for $\delta<1$. Then it holds that
 \[\int_0^1(r_\delta\ast f)\cdot g \,dx = \int_0^1 f\cdot(r_\delta\ast g)\,dx.\]
\end{lemme}
\begin{proof}
 As $\text{supp}\,r_\delta\subset[-\delta,\delta]$, we have
 \[(r_\delta\ast f)(x)=\int_{x-\delta}^{x+\delta} r_\delta(x-y)\cdot f(y)dy\qquad\text{ and }\qquad (r_\delta\ast g)(x)=\int_{x-\delta}^{x+\delta} r_\delta(x-y)\cdot g(y)dy.\]
 We calculate
 \[\int_0^1(r_\delta\ast f)(x)\cdot g(x)dx=\int_0^1\int_{x-\delta}^{x+\delta}r_\delta(x-y)\cdot f(y)\cdot g(x) dydx = (\ast).\]
 Abbreviate $\Box:=r_\delta(x-y)\cdot f(y)\cdot g(x) dxdy$. Then we get from Fubini's Theorem
 \begin{align*}
  (\ast)&=\int_\delta^{1-\delta}\int_{y-\delta}^{y+\delta}\Box + \int_{-\delta}^0\int_0^{y+\delta}\Box + \int_0^\delta\int_0^{y+\delta}\Box + \int_{1-\delta}^1\int_{y-\delta}^1\Box + \int_1^{1+\delta}\int_{y-\delta}^1\Box\\
  &= \int_\delta^{1-\delta}\int_{y-\delta}^{y+\delta}\Box + \int_{1-\delta}^1\int_1^{y+\delta}\Box + \int_0^\delta\int_0^{y+\delta}\Box + \int_{1-\delta}^1\int_{y-\delta}^1\Box + \int_0^{\delta}\int_{y-\delta}^0\Box\\ &=\int_0^1\int_{y-\delta}^{y+\delta}\Box.
 \end{align*}
Here, we added 1 to $x$ and $y$ in the $2^{\text{nd}}$ term and subtracted $1$ from $x$ and $y$ in the last term. Equality holds, as $f$ and $g$ are $1$-periodic and $r_\delta(x-y)=r_\delta\big((x+1)-(y+1)\big)=r_\delta\big((x-1)-(y-1)\big)$. Using $r_\delta(x-y)=r_\delta(y-x)$, we conclude
\begin{align*}
 (\ast)&=\int_0^1\int_{y-\delta}^{y+\delta} r_\delta(x-y)\cdot f(y)\cdot g(x) dxdy\\
 &=\int_0^1\int_{y-\delta}^{y+\delta} f(y)\cdot r_\delta(y-x)\cdot g(x) dxdy &&=\int_0^1 f(y)\cdot\big(r_\delta\ast g\big)(y)dy.\qedhere
\end{align*}
\end{proof}
If $f:\mathbb{R}\rightarrow\mathbb{R}^n$ is an $L^p$-function with components $f=(f^1,...\,,f^n)$, we define the convolution $r_\delta\ast f$ componentwise via $\big(r_\delta\ast f)^i:=r_\delta \ast f^i$.
\begin{cor}\label{cor121}
 Let $A=(a_{ij})$ be an $n\times n$-matrix and $f,g,r_\delta$ as in Lemma \ref{lemconvolution}. Then 
 \[\int_0^1\big(r_\delta\ast f\big)^T A\, g\, dt = \int_0^1 f^T A \big(r_\delta\ast g\big)\,dt.\]
\end{cor}
\begin{proof}
 This is a consequence of Lemma \ref{lemconvolution} and linear algebra:
 \begin{align*}
  \int_0^1\big(r_\delta\ast f\big)^T A\, g\, dt &=&\int_0^1\sum_{i,j}\big(r_\delta\ast f\big)^i\cdot a_{ij}\cdot g^j dt &= \sum_{i,j}a_{ij}\cdot\int_0^1\big(r_\delta\ast f^i\big)\cdot  g^j dt\\
  &=& \sum_{i,j}a_{ij}\cdot\int_0^1 f^i\cdot\big(r_\delta\ast g\big)^j dt &=
  \int_0^1\sum_{i,j} f^i\cdot a_{ij}\cdot\big(r_\delta\ast g\big)^j dt\\
  && &= \int_0^1 f^T A \big(r_\delta\ast g\big)\,dt.\qedhere
 \end{align*}
\end{proof}
\begin{cor}\label{corconvolution}
 If $\omega$ is a time independent 2-form and $f,g$ and $r_\delta$ are as above, then
 \[\int_0^1\omega(r_\delta \ast f,g)dt=\int_0^1\omega( f,r_\delta \ast g)dt.\]
\end{cor}
\begin{proof}
 As $\omega$ is time independent, we can write $\omega(x,y)=x^T A y$ for a time independent antisymmetric matrix $A$. Then apply Corollary \ref{cor121}.
\end{proof}

\newpage
 \section{Automatic transversality}\label{auttrans}
 In the proof of the Local Transversality Theorem \ref{theoloctrans}, we always assumed that the solution $(v,\eta)$ of the Rabinowitz-Floer equation (\ref{eq3}) is non-constant. When defining the boundary operator $\partial^F$, this assumption is satisfied, as each $\mathcal{A}^H$-gradient trajectory reduces the action which excludes constant solutions. However, if we consider homotopies $H_s$, we cannot avoid stationary trajectories.\\
 Note that the proof of the transversality fails in this situation, as for solutions $(v,\eta)$ of (\ref{eq3}) which are constant in $s$, we find that $\partial_t v-\eta X_H=0$ for all $s$. This forces the second equation in $(\ast)$ (on page \pageref{eqasttrans}) to be zero. As a consequence, the choice of $Y$ has no influence on the surjectivity of $D_{J,(v,\eta)}$.\bigskip\\
 The purpose of this appendix is to show that we have always transversality along constant $(v,\eta)$. For that, it suffices to show that the vertical differential $D_{J,(v,\eta)}$ of $\mathcal{F}$ is always surjective. The domain and range of $D_{J,(v,\eta)}$ are given by
 \begin{align*}
  D_{J,(v,\eta)}&\,:\, T_J\mathcal{J}^\ell\times T_{(v,\eta)}\mathcal{B}\rightarrow \mathcal{E}_{(v,\eta)},\\
  \text{where }\quad T_{(v,\eta)}\mathcal{B}&=W^{1,p}_\delta(\mathbb{R}\times S^1,\mathbb{R}^{2n})\oplus W^{1,p}_\delta(\mathbb{R},\mathbb{R}) \oplus T_{(v^-,\eta^-)}C^-\oplus T_{(v^+,\eta^+)}C^+,\\
  \mathcal{E}_{(v,\eta)}&=L^p_\delta(\mathbb{R}\times S^1,\mathbb{R}^{2n})\oplus L^p_\delta(\mathbb{R},\mathbb{R}).
 \end{align*}
 In fact, we show that already the restriction $D_{(v,\eta)}$ of $D_{J,(v,\eta)}$ to $T_{(v,\eta)}\mathcal{B}$ is surjective. We prove this by following closely a similar proof by Salamon, \cite{Sal}, Lem.\ 2.4. We consider only the case $p=2$, the first step in \cite{Sal}. For the transition to general $p$ see there.\bigskip\\
 To start, choose with Lemma \ref{lemcoord} a tubular neighborhood around $(v,\eta)$. This fixes a trivialization of $v^\ast TV$, so that we may write $\mathbb{R}^{2n}$ instead of $v^\ast TV$ in the expressions for $T_{(v,\eta)}\mathcal{B}$ and $\mathcal{E}_{(v,\eta)}$.\\
 Note that we have $(v^-,\eta^-)=(v^+,\eta^+)$, as $(v,\eta)$ is constant in $s$. This implies in particular that $C^\pm=C$. However, the two spaces $T_{(v^-,\eta^-)}C$ and $T_{(v^+,\eta^+)}C$ are not the same (recall that the identification of a complement of $\ker d(ev^+)\cap\ker d(ev^-)$ with these spaces is not natural, cf. page \pageref{nonnatural}). In fact, we should think of $T_{(v^\pm,\eta^\pm)}C$ as being generated by maps $\xi^\pm$ of the form
 \[\xi^-(s):=(1-\beta(s))\cdot\xi_0\qquad\text{ and }\qquad \xi^+(s)=\beta(s)\cdot \xi_0,\]
 where $\xi_0: S^1\rightarrow TC$ is a tangent vector to $C\subset\mathscr{L}$ and $\beta$ is a fixed monotone cut-off function such that $\beta(s)=1$ for $s\geq 0$ and $\beta(s)=0$ for $s\leq -1$.\\
 Now, recall that the operator $D_{(v,\eta)}$ can be written as
 \[D_{(v,\eta)}=\partial_s + A,\]
 where $A$ is a self-adjoint operator on the Hilbert space $H:=L^2(S^1,\mathbb{R}^{2n})\times\mathbb{R}$ with domain $W:=W^{1,2}(S^1,\mathbb{R}^{2n})\times\mathbb{R}$. Note that $A$ is $s$-independent, as $(v,\eta)$ is constant in $s$ and that the  kernel of $A$ consists of constant maps $\mathfrak{v}:S^1\rightarrow \mathbb{R}^{k+1}\times\{0\}$. This holds, as $\mathbb{R}^{k+1}\times\{0\}$ is identified with $TC$ in our chosen trivialization. Hence there exists a splitting
 \[H=E^+\oplus E^-\oplus\ker A\]
 into the positive, negative and zero eigenspaces of $A$.\pagebreak\\ Write $A^\pm:=A|_{E^\pm}$ and denote by $P^\pm: H\rightarrow E^\pm$ the orthogonal projections. The operator $-A^+$ generates a strongly continuous semigroup of operators on $E^+$ and $A^-$ generates a strongly semigroup of operators on $E^-$. Let us denote these semigroups by $s\mapsto e^{-A^+s}$ and $s\mapsto e^{A^-s}$, where both are defined for $s\geq 0$. Let $P^0: H\rightarrow \ker A$ denote the orthogonal projection to the kernel of $A$. Now define $K:\mathbb{R}\rightarrow\mathcal{L}(H)$ by
 \[K(s):=\begin{cases} \phantom{-}e^{-A^+s}P^+ + P^0& \text{for } s\geq 0\\ -e^{-A^-s}P^- &\text{for } s<0.\end{cases}\]
 This function is discontinuous at $s=0$, strongly continuous for $s\neq 0$ and satisfies
 \[\Big|\Big|K(s)\big|_{E^-\oplus E^+}\Big|\Big|\leq e^{-\delta|s|}\tag{$\ast$}.\]
 Consider the operator $Q:L^2_\delta(\mathbb{R},H)\rightarrow \big(W^{1,2}_\delta(\mathbb{R},H)\cap L^2_\delta(\mathbb{R},W)\big)\oplus T_{(v^-,\eta^-)}C\oplus T_{(v^+,\eta^+)}C$, which is defined for $\mu\in L^2_\delta(\mathbb{R},H)$ by
 \[Q\mu(s):=\int_{-\infty}^\infty K(s-\tau)\mu(\tau)d\tau.\]
 We claim that $Q$ is a right inverse for $D_{(v,\eta)}$, thus showing that $D_{(v,\eta)}$ is surjective. To see this, note that $\xi=Q\mu=\xi^++\xi^-+\xi^0$ with respect to the splitting of $H$, where
 \begin{align*}
  \xi^+(s)&=\phantom{-}\int_{-\infty}^s e^{-A^+(s-\tau)}\mu^+(\tau)d\tau\\
  \xi^-(s)&=-\int_s^\infty e^{-A^-(s-\tau)}\mu^-(\tau)d\tau\\
  \xi^0(s)&=\phantom{-}\int_{-\infty}^s\mu^0(\tau)d\tau.
 \end{align*}
A simple calculation shows that $\dot{\xi}^\pm+A^\pm\xi^\pm=\mu^\pm$ and $\dot{\xi}^0+A\xi^0=\dot{\xi}^0=\mu^0$, as $\xi^0(s)\in\ker A$ for all $s$. Hence $\dot{\xi}+A\xi=\mu$. It only remains to show that $\xi$ is actually in the right weighted space. The exponential convergence $\xi^\pm\overset{s\rightarrow \pm\infty}{\longrightarrow} 0$ follows from the exponential convergence of $\mu$ and $(\ast)$. Similarly, it follows that $\xi^0\overset{s\rightarrow -\infty}{\longrightarrow} 0$ exponentially. For $s\rightarrow \infty$ however, it might happen that $\xi^0$ does not even converge to 0. Nevertheless, as $\mu^0\overset{s\rightarrow \infty}{\longrightarrow} 0$ exponentially, we certainly have that $\xi^0\rightarrow c$ exponentially for some $c\in TC$. This allows us to write
\[\xi^0(s)=\beta(s)\cdot\xi_c + (\xi^0(s)-\beta(s)\cdot\xi_c),\]
where $\xi_c:S^1\rightarrow TC,\,t\mapsto c$ is a constant map. It follows that $\beta\cdot\xi_c\in T_{(v^+,\eta^+)}C$ and that $(\xi^0-\beta\cdot\xi_c)$ converges at both ends exponentially to zero.\\
Finally note that the space $W^{1,2}_\delta(\mathbb{R},H)\cap L^2_\delta(\mathbb{R},W)$ agrees with $W^{1,2}_\delta(\mathbb{R}\times S^1,\mathbb{R}^{2n})\oplus W^{1,2}_\delta(\mathbb{R},\mathbb{R})$.
\begin{rem}
 Note that $T_{(v^-,\eta^-)}C$ lies not in the image of $Q$. This is due to our specific construction of $Q$. In general, we cannot construct a surjective $Q$, as $D_{(v,\eta)}$ has a kernel consisting of $\xi$ which are constant in $s$ with $\xi(s)\in \ker A$.
\end{rem}
 \newpage
\section{Conventions}\label{conventions}
\begin{itemize}
 \item[Setup:] $(V,\lambda)$ is the completion of a Liouville domain $\tilde{V}$ with contact boundary $M=\partial\tilde{V}$ and symplectic form $\omega=d\lambda$. $(\Sigma,\alpha)\subset V$ is an exact contact hypersurface bounding a Liouville domain $W\subset V$. $H\in C^\infty(V)$ is a defining Hamiltonian for $\Sigma$, i.e.\ $H^{-1}(0)=\Sigma$, $H$ constant outside a compact set and $X_H|_\Sigma=R_\alpha$.
 \item[$X_H$:] Hamiltonian vector fields are defined by $dH=\omega(\cdot,X_H)=-\omega(X_H,\cdot)$.
 \item[$J$:] Almost complex structures depend on $(t,n)\in S^1\times\mathbb{R}$ with $\sup_n ||J_t(\cdot,n)||_{C^\ell}<\infty$. They are $\omega$-compatible and $t$-, $n$-independent on $M\times[R,\infty)\subset V$ for $R\gg0$ with  
  \[J|_{\xi_M}=J_0\qquad\text{ and }\qquad J\frac{\partial}{\partial r}=R_\lambda.\]
  \item[$\mathcal{A}^H$:] The action functional on the free loop space $\mathscr{L}$ times $\mathbb{R}$ is defined by
  \[\mathcal{A}^H(v,\eta)=\int_0^1\Big(\lambda(\big(\dot{v}(t)\big)-\eta H\big(v(t)\big)\Big)dt.\]
  \item[$\mathcal{N}^\eta$:] Denotes the set of all closed $\eta$-periodic Reeb orbits on $\Sigma$.
  \item[$\nabla h$:] We consider for the Morse complex \textit{positive} gradient flow lines $\dot{y}=\nabla h(y)$ of a Morse function $h$ on $\crita$ with respect to a Riemannian metric $g_h$ on $\crita$.
  \item[$g$:] The Riemannian metric on $\mathscr{L}\times\mathbb{R}$ is given by
  \[g\left(\begin{pmatrix}\mathfrak{v}_1\\ \hat{\eta}_1 \end{pmatrix}\,,\,\begin{pmatrix}\mathfrak{v}_2\\\hat{\eta}_2\end{pmatrix}\right):=\int^1_0\omega\Big(\mathfrak{v}_1(t)\,,\,J_t\big(v(t),\eta\big)\mathfrak{v}_2(t)\Big)dt+\hat{\eta}_1\cdot\hat{\eta}_2.\]
  \item[(\ref{eq3}):] The Rabinowitz-Floer equation (\ref{eq3}) is the \textit{positive} gradient equation $\partial_s(v,\eta)=\nabla\mathcal{A}^H(v,\eta)$. It has (\ref{eqRabFlHom}) as a counterpart for homotopies, which is not a gradient equation. Explicitly, they read as
  \begin{align*} 
  \partial_s v +J_t (v,\eta)\bigl(\partial_t v-\eta X_H(v)\bigr) &=0 \quad\text{ and }&  \partial_s\eta \;+\; \int^1_0 H\bigl(v(s,t)\bigr) dt &=0,\tag{\ref{eq3}}\\
  \partial_s v +J_t (v,\eta)\bigl(\partial_t v-\eta X_{H_s}(v)\bigr) &=0 \quad\text{ and }&  \partial_s\eta \;+\; \int^1_0 H_s\bigl(v(s,t)\bigr) dt &=0.\tag{\ref{eqRabFlHom}}
\end{align*}
  \item[(MB):] The Morse-Bott assumption is\bigskip\\
(MB)\hfill\begin{minipage}{13.5cm}
 \textit{The set $\mathcal{N}^\eta\subset\Sigma$ formed by the $\eta$-periodic Reeb orbits is a closed sub-manifold for each $\eta\in\mathbb{R}$ and $T_p\,\mathcal{N}^\eta = \ker \left(D_p\phi^\eta - id\right)$ holds for all $p\in\mathcal{N}^\eta$.}
\end{minipage}
  \item[(A)\&(B):] The grading assumptions are
  \begin{enumerate}
 \item[(A)] \textit{The map $i_\ast : \pi_1(\Sigma)\rightarrow\pi_1(V)$ induced by the inclusion is injective.}
 \item[(B)] \textit{The integral $I_{c_1}: \pi_2(V)\rightarrow \mathbb{Z}$ of $c_1(TV)$ vanishes on spheres.}
\end{enumerate}
  \item[$\mu$:] The grading of the Rabinowitz-Floer homology is defined by
  \[\mu(c)=\mu_{CZ}(v)+ind_h(c)- \frac{1}{2} dim_c\, \mathcal{N}^\eta +\frac{1}{2}.\pagebreak\]
\end{itemize}

\newpage
\phantom{..}
\newpage
\bibliographystyle{plain}
\bibliography{References.bib}

\end{document}